\documentclass{dis}

\usepackage{amssymb}
\usepackage{amsmath, amsfonts, amsthm, amscd}
\usepackage{cases}
\usepackage{enumerate, paralist}
\usepackage{appendix}

\usepackage{graphicx}
\usepackage{subfigure}

\usepackage[msc-links, alphabetic, abbrev, nobysame]{amsrefs}

\RequirePackage{amsthm, aliascnt}

\theoremstyle{theorem}
\newtheorem{thm}{Theorem}[chapter]

\newtheorem*{thm*}{Theorem}

\newaliascnt{corollary}{thm}
\newtheorem{cor}[corollary]{Corollary}
\aliascntresetthe{corollary}

\newaliascnt{lemma}{thm}
\newtheorem{lem}[lemma]{Lemma}
\aliascntresetthe{lemma}

\newaliascnt{sublemma}{thm}
\newtheorem{slem}[sublemma]{Sublemma}
\aliascntresetthe{sublemma}

\theoremstyle{definition}
\newaliascnt{definition}{thm}
\newtheorem{defi}[definition]{Definition}
\aliascntresetthe{definition}

\newaliascnt{example}{thm}
\newtheorem{exa}[example]{Example}
\aliascntresetthe{example}

\newaliascnt{remark}{thm}
\newtheorem{rmk}[remark]{Remark}
\aliascntresetthe{remark}

\theoremstyle{theorem}
\newtheorem{thmA}{Theorem}

\theoremstyle{theorem}
\newtheorem{thmI}{Theorem}

\newaliascnt{corI}{thmI}
\newtheorem{corI}[corI]{Corollary}
\aliascntresetthe{corI}

\theoremstyle{definition}

\RequirePackage{etoolbox}
\makeatletter
\patchcmd{\hyper@makecurrent}{%
	\ifx\Hy@param\Hy@chapterstring
	\let\Hy@param\Hy@chapapp
	\fi
}{%
	\iftoggle{inappendix}{%
		\@checkappendixparam{chapter}%
		\@checkappendixparam{section}%
		\@checkappendixparam{subsection}%
		\@checkappendixparam{subsubsection}%
		\@checkappendixparam{paragraph}%
		\@checkappendixparam{subparagraph}%
	}{}%
}{}{\errmessage{failed to patch}}

\newcommand*{\@checkappendixparam}[1]{%
	\def\@checkappendixparamtmp{#1}%
	\ifx\Hy@param\@checkappendixparamtmp
	\let\Hy@param\Hy@appendixstring
	\fi
}

\newtoggle{inappendix}
\togglefalse{inappendix}
\apptocmd{\appendix}{\toggletrue{inappendix}}{}{\errmessage{failed to patch}}
\@ifpackageloaded{appendix}
{
	\apptocmd{\appendices}{\toggletrue{inappendix}}{}{\errmessage{failed to patch}}
	\apptocmd{\subappendices}{\toggletrue{inappendix}}{}{\errmessage{failed to patch}}
}
{}
\makeatother

\DeclareMathOperator*{\lip}{Lip}
\newcommand{\hol}{\mathrm{ Hol }}
\newcommand{\supp}{\mathrm{ supp }}
\newcommand{\graph}{\mathrm{Graph}}
\newcommand{\id}{\mathrm{id}}

\newcommand{\mymatrix}[4]{\left(\begin{array}{cc} #1 & #2 \\[0pt] #3 & #4\end{array}\right)}

\allowdisplaybreaks

\numberwithin{equation}{chapter}

\begin{document}

\title{Invariant manifolds of partially normally hyperbolic invariant manifolds in Banach spaces}

\author{Deliang Chen}
\address{College of Mathematics and Physics \\
	Wenzhou University \\
	Wenzhou 325035 \\
	People's Republic of China \\
	E-mail: chernde@wzu.edu.cn}

\abbrevauthors{D. Chen}
\abbrevtitle{Invariant manifolds}

\mathclass{Primary 37D10; Secondary 37D30, 37C05, 37L10, 58B99}


\thanks{The author would like to thank Prof. Yannick Sire for his helpful discussions and particularly drawing the author's attention to the application of the paper's results to the whiskered tori. The author also thanks Prof. Chongchun Zeng for his useful discussions when the paper was prepared. The author is deeply indebted to the referee and editors who have thoroughly revised the text and provided valuable comments. Part of this work was done at Shanghai Jiao Tong University, where the author benefited greatly from the guidance and encouragement of Prof. Dongmei Xiao. The research is supported by the National Natural Science Foundation of China (No. 12101461).}

\maketitledis





\setcounter{tocdepth}{2}
\tableofcontents

\begin{abstract}
	We investigate the existence and regularity of locally invariant manifolds near an approximately invariant set that satisfies a geometric hyperbolicity condition with respect to an abstract ``generalized" dynamical system in Banach spaces. This hyperbolicity framework, which we term partial normal hyperbolicity, bridges the gap between normal hyperbolicity and partial hyperbolicity--concepts previously studied in finite dimensions and specific PDE contexts. Our generalized dynamical system accommodates non-smooth, non-Lipschitz, and even ``non-mapping" dynamics, making it applicable to both well-posed and ill-posed differential equations. As an illustrative application, we employ our results to analyze the dynamics of whiskered tori.
\end{abstract}

\makeabstract

\chapter{Introduction}

\section{Motivation}

This paper is a sequel to our previous works \cite{Che18a}, which aims to expand the scope of invariant manifold theory making it applicable to both well-posed and ill-posed differential equations and abstract infinite-dimensional dynamical systems.
In \cite{Che18a}, we gave a detailed investigation of the existence and regularity of invariant graphs for (discrete) cocycles or \emph{bundle correspondences} (see \autoref{sec:corr}) in general bundles; this is a global version of invariant manifold theory.

In this article, we study the theory of \emph{partially normal hyperbolicity}, a geometric hyperbolicity condition between normal hyperbolicity and partial hyperbolicity. To provide motivation for this study, we first survey relevant existing results before presenting a detailed illustration of partially normal hyperbolicity.

It is well known that a (fully) normally hyperbolic invariant manifold (without boundary) is a generalization of a hyperbolic equilibrium. Such a manifold has many good properties; for instance, it is persistent under small $ C^1 $ perturbations, and its regularity could be the same as that of dynamic if some spectral gap condition holds. For more details, see e.g. \cite{Fen72, HPS77, Wig94, PS01, BLZ98, BLZ08}.

A partially normally hyperbolic invariant manifold now can be considered as a natural extension of a finite-dimensional non-hyperbolic equilibrium or an infinite-dimensional equilibrium admitting \emph{exponential trichotomy} (see \cite{CL99}).

In the case of an equilibrium, it is important to note that its normal bundle is \emph{trivial}, a property that enables the application of global invariant manifold theory. More generally, for a (partially) normally hyperbolic invariant manifold, if its normal bundle is either trivial or can be embedded into a trivial bundle (preserving the metric structure, i.e., the embedding is bi-Lipschitz in some sense), then the global version of invariant manifold theory can be employed to derive relevant results; see e.g. \cite{HPS77, Eld13} and \cite[Theorem 4.17 and Remark 4.18]{Che18a}.

In the context of normally hyperbolic manifolds with boundary, the \emph{inflowing} (or \emph{overflowing}) property of manifolds with respect to the dynamics is crucial for obtaining persistence results (see \cite{HPS77, Fen72, BLZ99}). This property implies that the boundary itself is not invariant. In fact, for a normally hyperbolic inflowing (or overflowing) manifold, the manifold may possess boundaries, corners, or other singularities, provided it exhibits uniform behavior away from these features; see also \autoref{sub:setup}.
However, if a normally hyperbolic invariant manifold has an invariant boundary, then the boundary cannot be normally hyperbolic with respect to the dynamic. To some extent, this is a special case of partially normal hyperbolicity.

Consider a continuous dynamic generated by some smooth tangent field and its periodic orbit with period $ T $. If the associated time-$ T $ solution map $ \mathcal{P} $ of its linearized dynamic along this orbit has a non-simple eigenvalue $ 1 $, then obviously the periodic orbit is not normally hyperbolic. It is well known that if $ \mathcal{P} $ possesses certain compactness properties (for instance, when the essential spectral radius of $ \mathcal{P} $ is strictly less than $ 1 $) or, more generally, admits exponential trichotomy, then center-(un)stable and center manifolds exist near this orbit. This result can be established by considering the first return map and applying global invariant manifold theory, or by employing direct methods such as the integral equation approach. For further details in the case of normally hyperbolic periodic orbits (i.e., non-degenerate cases), see \cite{HR13} and the references therein. It is worth mentioning that similar results were also obtained in earlier works for \emph{ill-posed} PDEs (see e.g. \cite{SS99}).

For a normally hyperbolic invariant manifold, it is known that its tangent space (approximately) coincides with the center direction of the partially hyperbolic splitting at the manifold. To illustrate this more precisely, consider a simple example: Let $ H: X \to X $ be a $ C^1 $ diffeomorphism and $ K $ a boundaryless, compact, smooth invariant submanifold of $ X $, where $ X $ is a Banach space. Suppose $ K $ admits a \emph{partially hyperbolic splitting} (or equivalently, \emph{exponential trichotomy}) $ X = X^{s} \oplus X^{c} \oplus X^{u} $, with three continuous bundles $ X^{\kappa} $ ($ \kappa = s, c, u $) over $ K $, such that $ DH $ is uniformly contractive in $ X^{s} $ and uniformly expanding in $ X^{u} $, while in $ X^{c} $ it is strictly weaker than the contraction in $ X^{s} $ and the expansion in $ X^{u} $. In this framework, the normal hyperbolicity of $ K $ corresponds to $ TK = X^{c} $. Clearly, $ K $ would \emph{not} be normally hyperbolic if $ TK \subset X^{c} $, in which case we say that $ K $ is \emph{partially normally hyperbolic} (with respect to $ H $).

In the context of differential equations, this scenario has been extensively studied by many authors. For example (the list is by no means exhaustive):

\begin{enumerate}[(i)]
	\item Chow, Liu, Yi \cite{CLY00} studied center-(un)stable and center manifolds of a boundaryless compact and smooth invariant submanifold $ K $ for ODEs in $ \mathbb{R}^{n} $;

	\item \label{it:ex2} Zeng \cite{Zen00} gave a locally invariant center-stable manifold (in $ H^{n} $) of a circle of equilibriums $ K $ for some nonlinear Schrödinger equations;

	\item Schlag \cite{Sch09} exhibited a (globally invariant) center-stable manifold near ground states $ K $ for some nonlinear Schrödinger equation in $ W^{1,1}(\mathbb{R}^3) \cap W^{1,2}(\mathbb{R}^3) $; later, Beceanu \cite{Bec12} improved this in the critical space $ \dot{H}^{1/2} (\mathbb{R}^3) $;

	\item Nakanishi and Schlag \cite{NS12} constructed a (globally invariant) center-stable manifold near a soliton manifold $ K $ for some nonlinear Klein–Gordon equation in $ \mathbb{R}^{d} $;

	\item Krieger, Nakanishi and Schlag \cite{KNS15} obtained (globally invariant) center-stable manifolds of ground states (invariant under scale and translation) and ground state solitons (invariant under scale, translation and Lorentz transform), respectively, for some critical wave equation in $ \dot{H}^1(\mathbb{R}^{d}) $;

	\item \label{it:ex6} Jin, Lin and Zeng \cite{JLZ17} build (locally invariant) center-(un)stable and center manifolds for a smooth manifold $ K $ arising from spatial translations of a certain traveling wave solution of the Gross--Pitaevskii equation; moreover, under appropriate non-degeneracy conditions, these manifolds become globally invariant and hence unique.
\end{enumerate}

\begin{rmk}\label{rmk:PDEs}
	We make the following remark about the above items \eqref{it:ex2}--\eqref{it:ex6}.
	\begin{enumerate}[$ \bullet $]
		\item The linear parts of the differential equations (in appropriate forms) exhibit Hamiltonian structures with finite Morse index, implying that they possess \emph{finite}-dimensional (un)stable spaces and \emph{infinite}-dimensional center spaces, and that they generate $ C_0 $ groups (see \cite{LZ17} for a detailed characterization);

		\item The manifold $ K $ is typically taken as equilibriums of the differential equations; its geometry resembles $ \mathcal{M} $, $ \mathbb{R}^{n} $, or $ \mathbb{R}^{n} \times \mathcal{M} $, where $ \mathcal{M} $ is a compact and smooth boundaryless submanifold, which naturally satisfies (H1)--(H4) in \autoref{sub:setup}; the exponential trichotomy of $ K $ is usually obtained from that of a point in $ K $ (due to the construction of $ K $ under certain symmetry structures of the differential equations, such as scale and spatial translation invariance);

		\item The non-linear parts of the differential equations are typically unbounded, making the analysis highly complex, which in turn makes it extremely challenging to verify the non-linear version of exponential trichotomy with respect to the original non-linear dynamics (i.e., the (A) (B) condition defined in \autoref{defAB} and \autoref{defi:ABk}).
	\end{enumerate}
\end{rmk}

The manifold $ K $ considered above is taken as a smooth embedding manifold, though this is not strictly necessary. In the 1970s, Fenichel \cite{Fen79} first considered invariant manifolds of a compact \emph{set} $ K $ consisting of equilibriums and applied this framework to study the geometric singular perturbation problem for ODEs in $ \mathbb{R}^{n} $. Approximately two decades later, Chow, Liu, and Yi \cite{CLY00a} generalized this to the case where $ K $ is a compact invariant set satisfying certain admissibility conditions. More recently, Bonatti and Crovisier \cite{BC16} obtained an analogous result for diffeomorphisms on Riemannian manifolds (see also \autoref{cor:compact} in the infinite-dimensional setting); in addition, they also obtained an equivalent characterization (see also \autoref{thm:eqInvariant}). It is worth mentioning that one can define the ``pre-tangent space'' of a set (in the sense of Whitney), so the notation $ TK \subset X^{c} $ remains meaningful even when $ K $ is a set; see e.g. \cite{Whi34, CLY00a, BC16} or our \autoref{def:tangent}.

The theory of invariant manifolds near a single equilibrium for \emph{abstract} differential equations including semi-linear and quasi-linear parabolic, hyperbolic PDEs, delay equations, etc., has been extensively investigated by many researchers, see e.g. \cite{Hen81, CL88, BJ89, DPL88, PS01, MR09a, Zel14}.
It seems that there is no general result parallel to \cite{CLY00} for (well-posed) PDEs or abstract differential equations in Banach spaces.

The differential equations discussed previously are \emph{well-posed}, i.e., they generate $ C_0 $ semiflows (or even flows). However, there are numerous differential equations that are \emph{ill-posed}, such as the (bad) Boussinesq equations, the elliptic problem on the cylinder, certain semi-linear wave equations (with non-vanished convective term), and the \emph{spatial dynamics} (in the sense of K. Kirchg\"assner) induced by reaction-diffusion equations (e.g., the Swift-Hohenberg equations or KdV equations) (see \cite{EW91,Gal93,SS99, LP08, dlLla09,ElB12,BCJ21}). Ill-posedness means that for most initial data, there is no local solution. Eckmann and Wayne \cite{EW91}, Gally \cite{Gal93}, de la Llave \cite{dlLla09}, and ElBialy \cite{ElB12} considered various types of invariant manifolds (e.g., (un)stable, center, center-(un)stable manifolds) of equilibriums for some classes of ill-posed PDEs; de la Llave and Sire \cite{dlLS19} and Cheng and de la Llave \cite{CdlL19} also investigated time-dependent invariant manifolds for potentially ill-posed PDEs. In \cite{Che18a}, we introduced the notions of \emph{(bundle) correspondence} and \emph{cocycle correspondence}, originally due to Akin \cite{Aki93} and Chaperon \cite{Cha08}, which can be viewed as generalizations of the concept of (hyperbolic) dynamical systems (e.g., maps, bundle maps, cocycles (skew-product semiflows), semiflows), and provided a unified treatment of global invariant manifold theory for ill-posed differential equations. However, there are significantly fewer results concerning the local invariant manifold theory for general manifolds or sets in the context of ill-posed differential equations, which constitutes the \emph{purpose} of this paper.

Turning now to abstract dynamical systems, various types of invariant manifolds (both local and global) associated with fixed points have been extensively studied for diffeomorphisms and $ C^1 $ maps. For instance, \cite[\S 5-5A]{HPS77} and \cite{Irw80, dlLW95, Cha04} provide comprehensive treatments; correspondences with generating maps are also considered in \cite{Cha08}. Furthermore, \cite{Che18a} extends the results of \cite{Cha08} to a more general setting than that of \cite[\S 5-6]{HPS77}. The theory of normally hyperbolic invariant manifolds has been thoroughly investigated for diffeomorphisms on compact Riemannian manifolds in \cite{HPS77}, and for $ C^1 $ maps in Banach spaces in \cite{BLZ98, BLZ08}; these results can be modified to the corresponding (semi)flow version (see \cite{HPS77, BLZ98, BLZ08}). For the partially hyperbolic case, the study was undertaken in \cite{BC16} for diffeomorphisms on Riemannian manifolds near compact sets.

Nowadays, invariant manifold theory serves as a fundamental tool for understanding non-linear dynamical systems, with broad theoretical and practical applications. In the context of partially normal hyperbolicity, it is very useful to find some special interesting orbits; see e.g. \cite{Zen00} and \cite{LLSY16} for the search of homoclinic and heteroclinic orbits in practical models. It is important to emphasize that this theory yields an extremely detailed characterization of global dynamical behaviors; see e.g. \cite{Sch09, Bec12, NS12, KNS15}. For some notable dynamical consequences in finite dimensions, see also \cite[Section 1.2]{BC16}.

Recently, the study of $ C^2 $ (non-invertible) maps in Banach spaces has received extensive attention, especially regarding their dynamical properties and ergodic theory, e.g., Lyapunov exponents, horseshoes, entropy, and SRB measures; see e.g. \cite{LL10, LY11, BY17}. In this paper, we aim to provide an infinite-dimensional extension of \cite{BC16} for both $ C^1 $ maps and correspondences in Banach spaces, with certain generalizations even in the finite-dimensional case.

In the 1960s, V. I. Arnol{\cprime}d \cite{Arn63, Arn64} introduced ``whiskered tori'' in symplectic systems, which are regarded as key geometric structures leading to the instability of nearly integrable systems. The existence of whiskered tori was further investigated by Fontich, de la Llave, and Sire \cite{FdlLS09, FdlLS15, dlLS19}. As an illustration of our main results, we continue to study the non-linear dynamical behaviors near whiskered tori, establishing stable-(un)stable manifolds near them (see \autoref{sec:application}). This analysis encompasses both maps on Hilbert spaces or lattices and correspondences generated by the $1$-D (bad) Boussinesq equation with periodic boundary conditions for discrete time $ t_0 > 0 $ (see \autoref{exa:map1}, \ref{exa:map2}, and \ref{exa:Bq}). It is worth mentioning that in general, the whiskered tori constructed in \cite{FdlLS09, FdlLS15, dlLS19} are \emph{partially normally hyperbolic} in the sense of this paper.

\section{A nontechnical overview of main results for $ C^1 $ maps}
In this subsection, we present our main results for $ C^1 $ \emph{maps} in a nontechnical and special version.
In \cite{BC16}, Bonatti and Crovisier obtained the following.
\begin{thm*}[\cite{BC16}]
	Let $ H $ be a $ C^1 $ diffeomorphism on a smooth Riemannian manifold $ X $ and $ K $ a partially hyperbolic invariant compact set such that $ TX|_{K} = X^{s} \oplus X^{c} \oplus X^{u} $ (see e.g. the assumptions in \autoref{sub:dyntan}). Then the following are equivalent.
	\begin{enumerate}[(1)]
		\item There is a $ C^1 $ submanifold $ \Sigma^{c} $ such that $ K $ belongs to the interior of $ \Sigma^{c} $, $ \Sigma^{c} $ is locally invariant under $ H $, and $ T\Sigma^{c}|_{K} = X^{c} $.
		\item The strong stable manifolds and strong unstable manifolds at $ m \in K $ intersect $ K $ only at $ m $.
		\item $ TK \subset X^{c} $.
	\end{enumerate}
\end{thm*}

This result also holds in infinite dimensions but with an additional assumption. See also \autoref{def:tangent} for the meaning of the symbol $ TK \subset X^{c} $ when $ X $ is a Banach space.
\begin{thmA}\label{thm:A}
	Suppose $ X $ is either a separable Hilbert space, or a Banach space in which the fibers of $ X^{c} $ over $ K $ are finite-dimensional. Then the above theorem remains valid.
\end{thmA}
See \autoref{sub:maps} and \autoref{sub:dyntan} for more details pertaining to \autoref{thm:A}.
The compactness assumption on $ K $ imposes limitations in practical applications. To address this constraint, we now proceed to examine more general situations.

(S) Consider a $ C^1 $ map $ H: X \to X $, where $ X $ is a Banach space and $ K \subset X $. Suppose that the derivative $ DH $ is (almost) uniformly continuous in a neighborhood of $ K $ in $ X $.

(Q1) The first situation is that $ K $ is a uniformly $ C^{0,1} $ immersed submanifold of $ X $ (see (H1)--(H4) in \autoref{sub:setup} in detail when $ K = \Sigma $) and \emph{invariant} under $ H $ (meaning $ H(K) \subset K $ and $ H^{-1}(K) \subset K $); in addition, $ H: K \to K $ is an invertible $ C^{0} $ (and almost uniformly continuous) map in the immersed topology of $ K $.
A description of the qualitative behavior of $ H $ in a neighborhood of $ K $ is needed. Assume $ X $ has a decomposition $ X = X^{s}_{m} \oplus X^{c}_{m} \oplus X^{u}_{m} $, $ m \in K $, such that
\begin{enumerate}[$ \bullet $]
	\item $ X^{\kappa}_{m} $ ($ \kappa = s, c, u $) are invariant under $ DH(m) $ for $ m \in K $;

	\item $ \sup_{m}|A^{s}_{m} \oplus A^{c}_{m}||(A^{u}_{m})^{-1}| < 1 $ and $ \sup_{m}|A^{s}_{m} ||(A^{u}_{m} \oplus A^{c}_{m})^{-1}| < 1 $ with 
	\[
	\sup_{m}\{|A^{s}_{m} \oplus A^{c}_{m}|,|(A^{u}_{m} \oplus A^{c}_{m})^{-1}|\} < \infty,
	\]
	where $ A^{\kappa}_{m} = DH(m)|_{X^{\kappa}_{m}} $, $ \kappa = s, c, u $;

	\item $ T_{m}K \subset X^{c}_{m} $ with $ T_{m}K \oplus X^{c_0}_{m} = X^{c}_{m} $ for $ m \in K $;

	\item $ K \to \mathbb{G}(X): m \mapsto X^{\kappa}_{m} $ ($ \kappa = s, c, u, c_0 $) are Lipschitz in the immersed topology of $ K $ where $ \mathbb{G}(X) $ is the Grassmann manifold of $ X $ (see \autoref{sub:Grassmann}).
\end{enumerate}

Heuristically, we have the following. We emphasize that $ \sup_{m}\{|A^{s}_{m}|,|(A^{u}_{m})^{-1}|\} < 1 $ is not assumed. Note also that in general, $ H $ is not Lipschitz.
\begin{thmA}\label{thm:B}
	Under (S) and (Q1), there are three $ C^{0,1} $ immersed submanifolds of $ X $, denoted by $ W^{cs}_{loc}(K) $, $ W^{cu}_{loc}(K) $ and $ \Sigma^{c} = W^{cs}_{loc}(K) \cap W^{cu}_{loc}(K) $, such that
	\begin{enumerate}[$ \bullet $]
		\item $ K \subset \Sigma^{c} $, $ T_{m}W^{cs}_{loc}(K) = X^{cs}_{m} $, $ T_{m}W^{cu}_{loc}(K) = X^{cu}_{m} $, and $ T_{m}\Sigma^{c} = X^{c}_{m} $, $ m \in K $;

		\item $ W^{cs}_{loc}(K) $, $ W^{cu}_{loc}(K) $ and $ \Sigma^{c} $ are locally positively invariant, locally negatively invariant and locally invariant under $ H $, respectively.

		\item If, in addition, $ K \in C^{1} $, $ m \mapsto X^{\kappa}_{m} $ ($ \kappa = s, c, u, c_0 $) are $ C^1 $ and the norm of $ X $ is $ C^1 $, then one can choose $ W^{cs}_{loc}(K) $, $ W^{cu}_{loc}(K) $, $ \Sigma^{c} $ to be $ C^1 $. In particular, if $ K \in C^1 $ and $ X $ is a separable Hilbert space, then $ W^{cs}_{loc}(K) $, $ W^{cu}_{loc}(K) $, $ \Sigma^{c} $ can be chosen to be $ C^1 $.
	\end{enumerate}
\end{thmA}

For a more general version of \autoref{thm:B}, see \autoref{thm:invariant} and \autoref{app:im}.

\begin{rmk}
	Suppose $ \sup_{m}\{|A^{s}_{m}|,|(A^{u}_{m})^{-1}|\} < 1 $ and $ T_{m}K = X^{c}_{m} $. Then $ K $ becomes normally hyperbolic with respect to $ H $ in the sense of \cite{HPS77, BLZ98, BLZ08}. Consequently, we have $ \Sigma^c = K $ and the manifolds $ W^{cs}_{loc}(K), W^{cu}_{loc}(K), \Sigma^{c} $ automatically become $ C^1 $, even without the assumption that $ K \in C^{1} $; see also \cite{Che18b}.
\end{rmk}

In the second situation, we do not assume that $ K $ is a manifold or invariant. Instead, we consider the following assumptions.

(Q2) Let $ K \subset \Sigma $ where $ \Sigma $ is a uniformly $ C^{0,1} $ immersed submanifold of $ X $ near $ K $ (see (H1)--(H4) in \autoref{sub:setup} for details) and $ u: K \to K $ an invertible $ C^0 $ (and almost uniformly continuous) map (in the immersed topology of $ \Sigma $).
Assume that $ \sup_{m \in K}|H(m) - u(m)| $ is small and $ X $ has a decomposition $ X = X^{s}_{m} \oplus X^{c}_{m} \oplus X^{u}_{m} $, $ m \in K $, such that
\begin{enumerate}[$ \bullet $]
	\item $ X^{\kappa}_{m} $ ($ \kappa = s, c, u $) are approximately invariant under $ DH(m) $ (over $ u $) for $ m \in K $ (see also (IV$ ' $) (a$ ' $) in \autoref{sub:maps});

	\item $ DH(m) $ in $ X^{s}_{m} $ and in $ X^{u}_{m} $ is uniformly contractive and expanding, respectively, and in $ X^{c}_{m} $ is strictly weaker than the contraction in $ X^{s}_{m} $ and expansion in $ X^{u}_{m} $ for $ m \in K $;

	\item $ T_{m}\Sigma = X^{c}_{m} $ for $ m \in K $;

	\item $ K \to \mathbb{G}(X): m \mapsto X^{\kappa}_{m} $ ($ \kappa = s, c, u $) are Lipschitz (in the immersed topology of $ \Sigma $).
\end{enumerate}

Now, we can replace $ \Sigma $ by a locally invariant $ C^{0,1} $ submanifold $ \Sigma^c $ (or more intuitively, modify a neighborhood of $ K $ in $ \Sigma $ such that the new one is locally invariant). This result can be considered as a version of persistence result (see also \cite{CLY00, CLY00a, BC16}), but it is markedly different from normal hyperbolicity, since it is not uniquely persistent and the persistent manifold is only \emph{locally} invariant.

\begin{thmA}
	Under (S) and (Q2), there are three $ C^{0,1} $ immersed submanifolds in a small neighborhood of $ K $ in $ X $, denoted by $ W^{cs}_{loc}(K) $, $ W^{cu}_{loc}(K) $ and $ \Sigma^{c} = W^{cs}_{loc}(K) \cap W^{cu}_{loc}(K) $, such that
	\begin{enumerate}[$ \bullet $]
		\item $ W^{cs}_{loc}(K) $, $ W^{cu}_{loc}(K) $ and $ \Sigma^{c} $ are locally modeled on $ X^{cs}_{m}, X^{cu}_{m}, X^{c}_{m} $, $ m \in \Sigma $, respectively;

		\item $ W^{cs}_{loc}(K) $, $ W^{cu}_{loc}(K) $ and $ \Sigma^{c} $ are locally positively invariant, locally negatively invariant and locally invariant under $ H $, respectively.

		\item If, in addition, $ \Sigma \in C^1 $ admitting $ C^{1} \cap C^{0,1} $ bump functions (see also \autoref{def:C1Lbump}) and $ m \mapsto X^{\kappa}_{m} $ ($ \kappa = s, c, u $) are $ C^1 $, then one can choose $ W^{cs}_{loc}(K) $, $ W^{cu}_{loc}(K) $, $ \Sigma^{c} $ to be $ C^1 $. In particular, if $ \Sigma \in C^1 $ and if $ X $ is a separable Hilbert space or $ \Sigma $ is finite-dimensional, then $ W^{cs}_{loc}(K) $, $ W^{cu}_{loc}(K) $, $ \Sigma^{c} $ can be chosen to be $ C^1 $.
	\end{enumerate}
\end{thmA}

See \autoref{thm:tri0} and \autoref{cor:mapG} for more details.

\begin{rmk}
	If $ K = \Sigma $, then $ K $ is approximately invariant and approximately normally hyperbolic with respect to $ H $ in the sense of \cite{BLZ08}. In this case, $ W^{cs}_{loc}(K), W^{cu}_{loc}(K), \Sigma^{c} $ automatically become $ C^1 $ (even without assuming $ \Sigma \in C^{1} $), and $ \Sigma^{c} $ is invariant and normally hyperbolic with respect to $ H $; see also \cite{BLZ08, Che18b}.
\end{rmk}

The general results are formulated in \autoref{sec:statement} after some preliminaries in \autoref{sec:PreI} are set up.

\section{A unified framework for differential equation applications}

\begin{exa}\label{exa:diff}
	Let $ X, Y $ be two Banach spaces. Let $ T(t): X \to X $, $ S(-t): Y \to Y $, $ t \geq 0 $, be $ C_0 $ semigroups with generators $ A, -B $, respectively, and $ |T(t)| \leq e^{\mu_s t} $, $ |S(-t)| \leq e^{-\mu_u t} $, $ \forall t \geq 0$; see \cite{Paz83}. Let $ F_1 : X \times Y \to X $, $ F_2 : X \times Y \to Y $ be Lipschitz maps with $ \lip F_i \leq \varepsilon_{i} $. Consider the following differential equation
	\[ \tag{DE} \label{equ:ill}
		\begin{cases}
			\dot{x} = A x + F_1 (x, y), \\
			\dot{y} = B y + F_2 (x, y),
		\end{cases}
	\]
	or its integral form (called \emph{variation of constants formula}),
	\[ \tag{$ \mathrm{DE}_{\mathrm{int}} $} \label{equ:integral}
		\begin{cases}
			x(t) = T(t - t_1)x_1 + \int_{t_1}^{t} T(t-s)F_1(x(s),y(s)) ~\mathrm{d} s, \\
			y(t) = S(t - t_2)y_2 - \int_{t}^{t_2} S(t-s)F_2(x(s),y(s)) ~\mathrm{d} s,
		\end{cases}
		t_1 \leq t \leq t_2.
	\]
	(Usually, the solutions of \eqref{equ:integral} are called the \emph{mild solutions} of \eqref{equ:ill}. Equation \eqref{equ:ill} is usually \emph{ill-posed}, meaning that for arbitrarily given $ (x_0, y_0) \in X \times Y $, there might be no (mild) solution $ (x(t), y(t)) $ satisfying \eqref{equ:ill} with $ x(0) = x_0, y(0) = y_0 $.)

	\begin{enumerate}[$ \bullet $]

		\item Note that if $ S(-\cdot) $ is a $ C_0 $ group, then equation \eqref{equ:ill} is well-posed and classical, and has been investigated at length by many authors; see e.g. \cite{Paz83}. In this case, there is a natural setting to which we can apply the discrete dynamics results in \autoref{sec:statement}; for example, the study of first return map defined on certain cross-section of some orbits of \eqref{equ:ill}.

		\item For some concrete examples of \eqref{equ:ill} where $ S(-\cdot) $ is not a $ C_0 $ group, see e.g. \cite{LP08}.

		\item By the standard argument, it is easy to see that for each $ (x_1, y_2) \in X \times Y $, \eqref{equ:integral} has a unique $ C^0 $ solution $ (x(t), y(t)) $, $ t_1 \leq t \leq t_2 $, with $ x(t_1) = x_1, y(t_2) = y_2 $ (see e.g. \cite{ElB12}); let
		      \[
			      F_{t_1, t_2}(x_1, y_2) = x(t_2),~ G_{t_1, t_2}(x_1, y_2) = y(t_1).
		      \]

		\item Our results in \autoref{sec:statement} will show that the maps $ F_{0,t_0}, G_{0,t_0} $ for fixed $ t_0 > 0 $ can reflect some dynamical properties of \eqref{equ:ill}.
	\end{enumerate}
\end{exa}

\begin{asparaenum}[(Step I).]
	\item From our perspective, invariant manifolds for both ill-posed and well-posed differential equations can be treated in a unified framework. Rather than viewing \eqref{equ:ill} as an initial value problem, we approach it by solving \eqref{equ:integral} with the boundary conditions $ x(t_1) = x_1 $ and $ y(t_2) = y_2 $. It is important to note that the integral formula \eqref{equ:integral} depends critically on the non-linear maps $ F_1 $ and $ F_2 $ (and hence on $ A $ and $ B $). In some instances, it may not be possible to express the non-linear map as $ F = (F_1, F_2) $; the familiar form \eqref{equ:integral} is presented here due to the simplifying assumptions on $ A $, $ B $, $ F_1 $, and $ F_2 $.
	
	Let $ H $ denote the \emph{continuous correspondence} induced by \eqref{equ:ill} (i.e., \eqref{equ:integral}), so that $ H(t) \sim (F_{0,t}, G_{0,t}) $ for all $ t \geq 0 $ (see \autoref{sec:corr} for details). Intuitively, $ H(t) $ defines a \emph{correspondence} between $ X \times Y $ and itself, where $ (x_1, y_1) \in H(t)(x_0, y_0) $ if and only if there is a solution $ (x(s), y(s)) $ ($ 0 \leq s \leq t $) of \eqref{equ:ill} satisfying $ (x(0), y(0)) = (x_{0}, y_{0}) $ and $ (x(t), y(t)) = (x_{1}, y_{1}) $.
	
	Let $K \subset X$. Assume that \eqref{equ:ill} induces a flow $ t $ on $K$, meaning that for each $ \omega \in K $, there is a unique orbit $ \{z(t)\}_{t \in \mathbb{R}} $ of \eqref{equ:ill} such that $ z(t) \in K $ for all $ t \in \mathbb{R} $ and $ z(0) = \omega $. The flow $ t $ is then defined by $ t(\omega) = z(t) $. We write $ t(\omega) $ as $ t \omega $ for simplicity.
	
	\item To analyze the dynamical properties of \eqref{equ:ill} near $ K $, assume $ g = (F_1, F_2) \in C^1 $ and linearize the equation along $ K $, yielding
	\[
	\dot{z}(t) = Cz(t) + L(t\omega)z(t),
	\]
	where $ C = A \oplus B $, $ Z = X \times Y $, and $ L(\omega) = Dg(\omega) \in L(Z, Z) $ for $ \omega \in K $. In certain cases, the \emph{uniform (exponential) trichotomy} (or \emph{uniform (exponential) dichotomy}) is known to hold (see, e.g., \cite{CL99} for well-posed differential equations and \cite{LP08} for ill-posed ones, as well as \cite{dlLS19, CdlL19}).
	
	\item To capture the ``dynamical behavior'' of \eqref{equ:ill}, we must further investigate the uniform (exponential) trichotomy (or dichotomy) of the non-linear equation
	\[
	\dot{z}(t) = Cz(t) + L(t\omega)z(t) + f(t\omega)z(t),
	\]
	where $ f(\omega)z = g(z + \omega) - L(\omega)z - g(\omega) $ for $ z \in Z $ and $ \omega \in K $. More precisely, we require knowledge of the \emph{(A) (B) condition} for the \emph{cocycle correspondence} generated by this equation, which has been verified in some cases in \cite{Che18c}. This provides the verification of (B3) (a) in \autoref{sub:tri} (or (A3) (a) in \autoref{subsec:main}).
	
	\item  We must also demonstrate that $ K $ possesses suitable geometric properties (i.e., (H1)--(H4) in \autoref{sub:setup}). Correspondingly, the continuity of the flow $ t $ (in the immersed topology) and its (almost) uniform continuity must be established (i.e., (B2) in \autoref{sub:tri} or (A2) in \autoref{subsec:main}).
	
	\item Finally, the geometric properties of $ K $ must be compatible with the uniform (exponential) trichotomy (or dichotomy) assumption (i.e., (B3) (b) (ii) in \autoref{sub:tri} or (A3) (b) (ii) in \autoref{subsec:main}). At this stage, we can attempt to apply the main results in \autoref{sec:statement} to $ H(t_0) $ for a fixed $ t_0 > 0 $.
	
		The perturbed equation of \eqref{equ:ill} can also be considered. For example,
		\[
		\dot{z} = Cz + \widetilde{g}(z),
		\]
		where $ |\widetilde{g} - g|_{C^{1}} $ is small. In this case, replace $ f $ in (Step III) with $ \widetilde{f}(\omega)z = g(z + \omega) - L(\omega)z - \widetilde{g}(\omega) $. Note that $ |\widetilde{g} - g|_{C^{0}} $ must be sufficiently small (i.e., (B3) (b) (i) in \autoref{sub:tri} or (A3) (b) (i) in \autoref{subsec:main}). The required smallness of $ |D\widetilde{g} - Dg|_{C^{0}} $ depends solely on the analysis in (Step II). In fact, $ |D\widetilde{g} - Dg|_{C^{0}} $ may not be as small as in the theory of \emph{inertial manifold} (see, e.g., \cite{MS88, Zel14}); this phenomenon also occurs in ``\emph{weak hyperbolicity}'' as discussed in \cite[Section 7.3]{Wig94} (see also \cite{Yan09}). Here, \autoref{thm:invariant} (along with \autoref{app:im}) gives such a generalization.
\end{asparaenum}

\begin{rmk}
	(Step II) is inherently connected to the spectral theory of differential equations (see, e.g., \cite{CL99, LP08, LZ17} and references therein), which falls outside the scope of this series of papers. We note that the general spectral theory for ill-posed differential equations remains underdeveloped.

	(Step I) and (Step III) rely heavily on the availability of a \emph{variation of constants formula} along with associated estimates (such as Strichartz estimates). The difficulty of these steps varies depending on the assumptions regarding linear and non-linear terms. In particular, for \emph{critical} non-linear perturbations, the analysis becomes highly complex (see, e.g., \cite{Sch09, Bec12}); however, we are confident that our results can be applied once an abstract formulation of such perturbations is established, as in \cite{Che18c}. 

	Given the generality of our assumptions on the submanifold, which encompass many interesting examples, (Step IV) is often not problematic, especially when $ K $ is a compact set (see \autoref{thm:geo}) or a (uniformly Lipschitz invariant) submanifold (see \autoref{sub:invariant}); refer to \autoref{rmk:geo} for details. In the context of partially normal hyperbolicity, (Step V) presents the main challenge, though it typically holds naturally in physical applications.
\end{rmk}

For a concrete application, see \autoref{sec:application} and, in particular, \autoref{exa:Bq}.

\section{Structure of this paper}

In \autoref{sec:PreI}, we introduce basic notions for ``generalized'' dynamical systems that are non-smooth, non-Lipschitz, and ``non-mapping'', following \cite{Che18a}. \autoref{sec:application} is an application of our main results to invariant whiskered tori. Our main results are stated in \autoref{sec:statement}. In \autoref{sec:facts}, we give essential background on Grassmann manifolds and uniform submanifolds in Banach spaces. \autoref{sec:existence} contains the proof of existence results and \autoref{thm:I}; the core step is the construction of graph transforms that combine ideas from \cite{Fen72, BLZ08, BC16} with new techniques from \cite{Che18a} to handle ``generalized'' dynamical systems. In \autoref{sec:smooth}, we prove regularity results and \autoref{thm:smooth} by combining methods from \cite{HPS77, BLZ08} and \cite{Che18a}. \autoref{sub:bump} includes a brief discussion on the existence of $ C^{0,1} \cap C^{1} $ bump functions. \autoref{sec:invariant} and \autoref{sec:tri} contain proofs of \autoref{thm:invariant} and \autoref{thm:tri0}, respectively. In \autoref{sec:whitney}, we give a geometric version of the Whitney extension theorem in infinite dimensions (see also \cite{CLY00a, BC16} for finite-dimensional cases) and prove \autoref{cor:compact}. Finally, in \autoref{sec:C1map}, we discuss invariant manifolds for $ C^1 $ maps.

\vspace{.5em}
\noindent{\textbf{Notations}}: Throughout this paper, the following notations will be used.
\begin{enumerate}[$ \bullet $]
	\item $\lip f$: the Lipschitz constant of $f$; $ \hol_{\theta} f $: the $ \theta $-H\"older constant of $f$. $ \lip f = \hol_{1} f $.

	\item $\mathbb{R}_{+} \triangleq \{x\in \mathbb{R}: x \geq 0\}$, $ \mathbb{N} = \{ 0, 1,2, \ldots \} $, $ \mathbb{Z} = \{ 0, \pm 1, \pm 2, \ldots \} $.

	\item $ \mathbb{B}_r(K) \triangleq O_{r}(K) \triangleq \{ x: d(x, K) < r \} $, $ \mathbb{B}_r(m) \triangleq \mathbb{B}_r(\{m\}) $, if $ (X, d) $ is a metric space with metric $ d $.

	\item For a normed space $ X $, write $ \mathbb{S}_X = \{ x: |x| = 1 \} $, $ \mathbb{B}_r = X(r) \triangleq \{ x: |x| < r \} $.

	\item For two normed spaces $ X $ and $ Y $, $ L(X, Y) $ denotes the space of all bounded linear operators from $ X $ to $ Y $, and $ L(X) \triangleq L(X, X) $.

	\item $R(L) = \{Lx:x \in X\}$, $\ker(L) = \{x:Lx=0\}$ where $L:X \to Y$ is a linear operator.

	\item $ \mathbb{G}(X) $: the Grassmann manifold of $ X $ (i.e., all complemented linear subspaces of $ X $); see \autoref{sub:Grassmann}.

	\item $ \overline{\Pi}(X) $: the set of projections of $ X $, if $ X $ is a Banach space.

	\item $ X_1 \oplus X_2 $: $ X_i $ ($ i = 1,2 $) are closed subspaces of a normed space with $ X_1 \cap X_2 = \{0\} $. Without explicit mentioned, the norm of $ X_1 \oplus X_2 \cong X_1 \times X_2 $ is taken as $ |(x_1, x_2)| = \max \{ |x_1|, |x_2| \} $ (when $ X_1 \oplus X_2 $ is close).

	\item $ \Pi_{X_2}(X_1) $: the projection with $ R(\Pi_{X_2}(X_1)) = X_1 $, $ \ker (\Pi_{X_2}(X_1)) = X_2 $ if $ X_1 \oplus X_2 $ is closed.

	\item For a correspondence $ H: X \to Y $ (defined in \autoref{sec:corr}),
	      \begin{asparaitem}
		      \item $ \graph H $, the graph of the correspondence $ H $,
		      \item $ H(x) \triangleq \{ y: \exists (x, y) \in \graph H \} $, $ H(A) \triangleq \bigcup_{x \in A} H(x) $,
		      \item $A \subset H^{-1}(B)$, if $ \forall x \in A $, $ \exists y \in B $ such that $ y \in H(x) $,
		      \item $ H^{-1}: Y \to X $, the inversion of $ H $ defined by $ (y,x) \in \graph H^{-1} \Leftrightarrow (x, y) \in \graph H $.
		      \item $ H \sim (F, G) $ means that the correspondence has a generating map $ (F, G) $.
	      \end{asparaitem}

	\item (A) (B) condition, (A)$({\alpha; \alpha', \lambda_u})$ (B)$({\beta; \beta', \lambda_s})$ condition, or (A$ ' $) (B) condition, etc., defined in \autoref{defAB}.

	\item Write $ f = O_{\epsilon}(1) $ if $ f \to 0 $ as $ \epsilon \to 0 $; $ f = o(\epsilon) $ if $ f/\epsilon = O_{\epsilon}(1) $; $ f = O(\epsilon) $ if $ |f/\epsilon| < \infty $ as $ \epsilon \to 0 $.

	\item $ D_1F(x,y) = D_x F(x,y) $ and $ D_2F(x,y) = D_y F(x,y) $: the derivatives of $ F $ with respect to $ x, y $, respectively, where $ F $ is differentiable at $ (x,y) \in X \times Y $.

	\item Let $ r_\varepsilon(x) $ be the \emph{radial retraction} of a Banach space, defined by
	      \begin{equation}\label{equ:radial}
		      r_\varepsilon(x) =
		      \begin{cases}
			      x,                   & ~ \text{if}~ |x| \leq \varepsilon, \\
			      \varepsilon x / |x|, & ~ \text{if}~ |x| \geq \varepsilon.
		      \end{cases}
	      \end{equation}
\end{enumerate}

\chapter{Basic notions: set-up for generalized dynamical systems}\label{sec:PreI}

To provide a unified treatment of differential equations in Banach spaces for both ill-posed and well-posed cases, we introduce the notion of correspondence with generating map, which can be regarded as ``generalized dynamical system''. We also use the (A) (B) condition (see \cite{Che18a}) to describe the hyperbolicity of correspondences.

Let us first recall some terminology related to \emph{bundle}.
A triple $ (X, M, \pi_1) $ (or simply $ X $) is called a (set) \emph{bundle} (over $ M $) if $ \pi_1: X \to M $ is a surjection. Typically, we call $ X_{m} = \pi_1^{-1}(m) $, $ m \in M $, the \emph{fibers} of $ X $, $ M $ the \emph{base space} of $ X $, and $ \pi_1 $ the \emph{projection}. Sometimes, the elements of $ X $ are written as $ (m,x) $ where $ x \in X_{m} $, $ m \in M $.
The Whitney sum $ X \times Y $ of two bundles $ X $ and $ Y $ over $ M $ is defined by
\[
X \times Y = \{ (m,x,y): x \in X_{m}, y \in Y_{m}, m \in M \}.
\]

Let $ (X, M_1, \pi_1) $ and $ (Y, M_2, \pi_2) $ be two bundles and $ u: M_1 \to M_2 $ a map. A map $ f: X \to Y $ is called a bundle map over $ u $ if $ f(X_{m}) \subset Y_{u(m)} $ for all $ m \in M_1 $; we write $ f(m,x) = (u(m), f_{m}(x)) $ and call $ f_{m}: X_{m} \to Y_{u(m)} $ a fiber map of $ f $.
A bundle is called a \emph{vector bundle} if each of its fibers is a linear space. A bundle map between two vector bundles is called a \emph{vector bundle map} if each fiber map is linear.

\section{Correspondence with generating map, dual correspondence} \label{sec:corr}

Let $X, Y$ be two sets. A \emph{correspondence} $H: X \rightarrow Y$ (see \cite{Cha08}) is defined to be a non-empty subset of $X \times Y$, called the graph of $ H $ and denoted by $ \graph H $; in this case, we write $ y \in H(x) $ if $(x, y) \in \graph H$.
Various operations between correspondences can be defined.
\begin{enumerate}[$ \bullet $]
	\item (Inversion) The inversion $ H^{-1} : Y \rightarrow X $ of a correspondence $ H: X \rightarrow Y $ is defined by
	\[
	(y,x) \in \graph H^{-1} \Leftrightarrow (x,y) \in \graph H.
	\]

	\item (Composition) For two correspondences $ H_1: X \rightarrow Y $ and $ H_2: Y \rightarrow Z $, define $H_2 \circ H_1: X \rightarrow Z$ by
	\[
	\graph H_2 \circ H_1 = \{ (x,z): \exists y \in Y ~\text{such that}~ (x,y) \in \graph H_1, (y,z) \in \graph H_2 \}.
	\]

	\item (Linear operation) If $X, Y$ are vector spaces and $H_1, H_2: X \rightarrow Y$ are correspondences, then $H_1 - H_2: X \rightarrow Y$ is defined by
	\[
	\graph (H_1 - H_2) = \{ (x,y): \exists (x, y_i) \in \graph H_i ~\text{such that}~ y = y_1-y_2 \}.
	\]
	In particular, $H_m \triangleq H(m+\cdot) - \widehat{m}: X \rightarrow Y$ is well defined, i.e.,
	\[
	\graph H_m = \{ (x, y - \widehat{m}) : \exists (x+m, y) \in \graph H \}.
	\]
\end{enumerate}

The following notations for a correspondence $ H: X \to Y $ will be used frequently.
\begin{enumerate}[$ \bullet $]
	\item $ H(x) \triangleq \{ y \in Y: \exists (x, y) \in \graph H \} $; here, $ H(x) $ may be empty; if $ H(x) = \{y\} $, we write $ H(x) = y $.
	\item $ H(A) \triangleq \bigcup_{x \in A} H(x) $, where $ A \subset X $.
	\item Thus, $ A \subset H^{-1}(B) $ means that $ \forall x \in A $, $ \exists y \in B $ such that $ y \in H(x) $ ($ \Leftrightarrow x \in H^{-1}(y) $).
	\item When $ X = Y $, we say $ A \subset X $ is \emph{invariant} under $ H $ if $ A \subset H^{-1}(A) $.
	\item If $ A \subset H^{-1}(B) $, then $ H: A \to B $ can be regarded as a correspondence $ H|_{A \to B} $ defined by
	\[
	(x, y) \in \graph H|_{A \to B} \Leftrightarrow y \in H(x) \cap B, x \in A.
	\]
	\item We say $ H: A \to B $ induces (or defines) a map, denoted by $ H|_{A \to B} $ (sometimes also written as $ H|_{A} $), if $ \forall x \in A $, $ H(x) \cap B $ contains exactly one element.
\end{enumerate}

In \cite{Aki93}, Akin used the term ``relation'' instead of ``correspondence'' and applied it to unify various properties of topological dynamical systems (e.g., recurrence, attractor, chain recurrence, etc.).

Let us focus on the description of $ \graph H $.
\begin{enumerate}[$ \bullet $]
	\item We say a correspondence $H: X_1 \times Y_1 \rightarrow X_2 \times Y_2$ has a \emph{generating map} $(F,G)$, and we write $H\sim(F,G)$, if there are maps $F: X_1 \times Y_2 \rightarrow X_2$ and $G: X_1 \times Y_2 \rightarrow Y_1$ such that
	\[
	(x_2,y_2) \in H(x_1,y_1) \Leftrightarrow y_1 = G(x_1,y_2), ~x_2 = F(x_1, y_2).
	\]
\end{enumerate}
Note that given such maps $ F, G $, one can determine a correspondence $ H $ such that $ H \sim (F, G) $.

The following type of maps induce correspondences with generating maps; see also \cite{Cha08}.

\begin{enumerate}[$ \bullet $]
	\item Let $H = (f,g): X_1 \times Y_1 \rightarrow X_2 \times Y_2$ be a map. Suppose for every $x_1 \in X_1$, $g_{x_1}(\cdot) \triangleq g(x_1, \cdot): Y_1 \rightarrow Y_2$ is a bijection. Let $G(x_1, y_2) = g^{-1}_{x_1}(y_2)$ and $F(x_1, y_2) = f(x_1, G(x_1, y_2))$. Then we have $H \sim (F,G)$.
\end{enumerate}

\emph{A map exhibiting hyperbolic behavior typically falls into this category}, regardless of whether the map is bijective or not; see \autoref{lem:mapAB} and \cite[Section 3]{Che18a} for more details.

\begin{defi}\label{defi:dcorr}
	Let $ (X, M, \pi_1), (Y, N, \pi_2)$ be two bundles, and $u: M \rightarrow N $ a map.
	$H: X \times Y \rightarrow X \times Y$ is called a \emph{bundle correspondence} over a map $u$, if $\graph H \triangleq \bigcup_{m \in M}(m, \graph H_m$), i.e.,
	\[
	(u(m), x_{u(m)}, y_{u(m)}) \in H(m, x_m, y_m) \Leftrightarrow (x_{u(m)}, y_{u(m)}) \in H_m(x_m, y_m),
	\]
	where $H_m: X_m \times Y_m \rightarrow X_{u(m)} \times Y_{u(m)}$ is a correspondence, $m \in M$.

	If for each $m \in M$, $H_m \sim (F_m, G_m)$, where $F_m: X_m \times Y_{u(m)} \rightarrow X_{u(m)}$ and $G_m: X_m \times Y_{u(m)} \rightarrow Y_m$ are maps, then we say $H$ has a \emph{generating bundle map} $(F,G)$ over $u$, denoted by $H \sim (F,G)$.

	Let $ M = N $. The \emph{kth bundle composition} $ H^{(k)} $ of $ H $ is defined by
	\[
	H^{(k)}_m = H_{u^{k-1}(m)} \circ H_{u^{k-2}(m)} \circ \cdots \circ H_m.
	\]
	This is a bundle correspondence over $ u^k $.
\end{defi}

\begin{defi}\label{defi:dual}
	Let $ H: X_1 \times Y_1 \to X_2 \times Y_2 $ be a correspondence with a generating map $ (F, G) $. The \emph{dual correspondence} of $ H $, denoted by $ \widetilde{H}: \widetilde{X}_1 \times \widetilde{Y}_1 \to \widetilde{X}_2 \times \widetilde{Y}_2 $, is defined as follows. Set $ \widetilde{X}_1 = Y_2 $, $ \widetilde{X}_2 = Y_1 $, $ \widetilde{Y}_1 = X_2 $, $ \widetilde{Y}_2 = X_1 $ and
	\[
	\widetilde{F}(\widetilde{x}_1, \widetilde{y}_2) = G(\widetilde{y}_2, \widetilde{x}_1), ~ \widetilde{G}(\widetilde{x}_1, \widetilde{y}_2) = F(\widetilde{y}_2, \widetilde{x}_1).
	\]
	Now $ \widetilde{H} $ is uniquely determined by $(\widetilde{x}_2,\widetilde{y}_2) \in \widetilde{H}(\widetilde{x}_1,\widetilde{y}_1) \Leftrightarrow \widetilde{y}_1 = \widetilde{G}(\widetilde{x}_1,\widetilde{y}_2), ~\widetilde{x}_2 = \widetilde{F}(\widetilde{x}_1, \widetilde{y}_2)$, i.e., $ \widetilde{H} \sim (\widetilde{F}, \widetilde{G}) $.
\end{defi}

Usually, $ \widetilde{H} $ can reflect some properties of $ H^{-1} $; but $ H^{-1} $ and $ H $ are not in duality. For example, if $ H $ satisfies the (A)$(\alpha; \alpha', \lambda_u)$ condition (see \autoref{defAB} below), then $ \widetilde{H} $ satisfies the (B)$(\alpha; \alpha', \lambda_u)$ condition. In particular, if we obtain the ``stable results'' with respect to $ H $, then we also get the ``unstable results'' of $ H $ by using the ``stable results'' of $ \widetilde{H} $.
We state all the results in the ``stable direction'' in this paper, leaving the corresponding statements of ``unstable results'' to the readers.

\section{Hyperbolicity and (A) (B) condition} \label{sub:AB}

We now present a method for describing the hyperbolicity of a correspondence (see \cite{Che18a}).
Let $X_i, Y_i$ ($i=1,2$) be metric spaces. \emph{For notational convenience, we denote metrics by $d(x,y) \triangleq |x-y|$}.

\begin{defi}\label{defAB}
	A correspondence $H: X_1 \times Y_1 \rightarrow X_2 \times Y_2$ is said to satisfy the \emph{(A) (B) condition}, or more specifically the \emph{(A)$(\alpha; \alpha', \lambda_u)$ (B)$(\beta; \beta', \lambda_s)$ condition}, if the following conditions hold for all $(x_1, y_1) \times (x_2, y_2), (x'_1, y'_1) \times (x'_2, y'_2) \in \graph H$:
	\begin{enumerate}[(A)]
		\item If $|x_1 - x'_1| \leq \alpha |y_1 - y'_1|$, then $|x_2 - x'_2| \leq \alpha' |y_2 - y'_2|$;

		\noindent if $|x_1 - x'_1| \leq \alpha |y_1 - y'_1|$, then $ |y_1 - y'_1| \leq \lambda_u |y_2 - y'_2|$;

		\item If $|y_2 - y'_2| \leq \beta |x_2 - x'_2| $, then $ |y_1 - y'_1| \leq \beta' |x_1 - x'_1|$;

		\noindent if $|y_2 - y'_2| \leq \beta |x_2 - x'_2| $, then $ |x_2 - x'_2| \leq \lambda_s |x_1 - x'_1|$.
	\end{enumerate}
	If $\alpha = \alpha'$ and $\beta = \beta'$, we also use the notation \emph{(A) ($\alpha, \lambda_u$) (B) ($\beta, \lambda_s$) condition}. If $ H \sim (F, G) $, we also say $ (F, G) $ satisfies the (A) (B) condition; in this case, $F$ and $G$ satisfy the following Lipschitz conditions:
	\begin{enumerate}[(A$'$)]
		\item $\sup_{x}\lip F(x,\cdot) \leq \alpha'$, $\sup_{x}\lip G(x,\cdot) \leq \lambda_u$.
		\item $\sup_{y}\lip G(\cdot,y) \leq \beta'$, $\sup_{y}\lip F(\cdot,y) \leq \lambda_s$.
	\end{enumerate}
	If $F, G$ satisfy the above Lipschitz conditions, then we say $H \sim (F, G)$ satisfies the \emph{(A$'$)($\alpha', \lambda_u$) (B$'$)($\beta', \lambda_s$) condition}, or simply the \emph{(A$'$) (B$'$) condition}. Similarly, one can define (A$'$) (B) condition, (A) (B$'$) condition, etc.
\end{defi}

Roughly speaking, in \autoref{defAB}, the numbers $ \lambda_{s}, \lambda_{u} $ and the spaces $ X_i, Y_i $ ($ i = 1,2 $) are related to \emph{Lyapunov numbers} and \emph{spectral spaces}, respectively; while the numbers $ \alpha, \alpha' $ and $ \beta, \beta' $ describe how the spaces $ X_i, Y_i $ ($ i = 1,2 $) are approximately invariant.
This connection can be intuitively understood in the context of differential equations (see \cite[Section 3]{Che18c}), where we established a relation between the \emph{uniform (exponential) dichotomy} (see \cite{CL99, LP08}) and the (A) (B) condition; see also \cite[Sections 4.2--4.3]{Che18a}. In what follows, we collect some useful facts about the (A) (B) condition; for the simple proofs of Lemmas \ref{lem:a4} and \ref{lem:c1} below, see \cite[Sections 4.2--4.3]{Che18a}.

It is straightforward to verify that if for each $n$, $H_n: X_n \times Y_n \rightarrow X_{n+1} \times Y_{n+1}$ satisfies the (B) ($\beta_{n}; \beta'_{n}, \lambda_{n, s}$) condition with $ \beta'_{n+1} \leq \beta_{n} $, then $H_n \circ \cdots \circ H_1$ satisfies the (B) ($\beta_{n}; \beta'_{1}, \lambda_{1,s} \cdots \lambda_{n,s}$) condition. This fact will be used in \autoref{sec:tri}.

\begin{lem}[Lipschitz condition $ \Rightarrow $ (A) (B) condition]\label{lem:a4}
	Let $ H \sim (F, G): X_1 \times Y_1 \to X_2 \times Y_2 $.
	\begin{enumerate}[(a)]
		\item \label{it:ab0} If for all $ x_1, x'_1 \in X_1$ and $ y_2, y'_2 \in Y_2 $,
		\begin{gather*}
		|F(x_1, y_2)-F(x'_1, y'_2)| \leq \max\{ \lambda_s |x_1 - x'_1|,~ \alpha |y_2 - y'_2| \}, \\
		|G(x_1, y_2)-G(x'_1, y'_2)| \leq \max\{ \beta |x_1 - x'_1|,~ \lambda_u |y_2 - y'_2| \},
		\end{gather*}
		and $ \alpha\beta < 1 $, $ \lambda_s \lambda_u < 1 $, then $ H $ satisfies the (A) ($\alpha, \lambda_u$) (B) ($\beta, \lambda_s$) condition. Moreover, if $ \alpha \beta < \lambda_s \lambda_u $, then $H$ satisfies the (A) ($c^{-1} \alpha; \alpha, \lambda_u$) (B) ($c^{-1} \beta; \beta, \lambda_s$) condition, where $ c = \lambda_s \lambda_u < 1 $.

		\item \label{it:ab1} If $ H $ satisfies the (A$'$) ($\widetilde{\alpha}, \widetilde{\lambda}_u$) (B$'$) ($\widetilde{\beta}, \widetilde{\lambda}_s$) condition with $ \widetilde{\lambda}_s \widetilde{\lambda}_u < c^2 $ and $ \widetilde{\alpha} \widetilde{\beta} < (c - \sqrt{\widetilde{\lambda}_s \widetilde{\lambda}_u})^2 $, where $ 0 < c \leq 1 $, then $H$ satisfies the (A)$(\alpha; c\alpha, \lambda_u)$ (B)$(\beta; c\beta, \lambda_s)$ condition, where
		\begin{gather*}
					\alpha = \frac{b-\sqrt{b^2 - 4c\widetilde{\alpha}\widetilde{\beta}} }{2\widetilde{\beta}}, \quad
			\beta = \frac{b-\sqrt{b^2 - 4c\widetilde{\alpha}\widetilde{\beta}} }{2\widetilde{\alpha}}, \\
			\lambda_s = \frac{\widetilde{\lambda}_s}{1-\alpha\widetilde{\beta}},\quad
			\lambda_u = \frac{\widetilde{\lambda}_u}{1-\beta\widetilde{\alpha}},\quad
			b = c - \widetilde{\lambda}_s \widetilde{\lambda}_u + \widetilde{\alpha}\widetilde{\beta}.
		\end{gather*}
	\end{enumerate}
\end{lem}

Note that in \autoref{lem:a4} \eqref{it:ab1}, we also have $\alpha \beta < 1$ and $\lambda_s \lambda_u < 1$.

For a Banach space $ X $ and $ r > 0 $, we denote $ X(r) = \{ x \in X: |x| < r \} $.

\begin{lem}[Invariant cone condition $ \Rightarrow $ (A) (B) condition]\label{lem:c1}
	Let $ X_i, Y_i $ ($ i = 1, 2 $) be Banach spaces, and $ H \sim (F, G): X_1 \times Y_1 \rightarrow X_2 \times Y_2 $ a correspondence. Assume $F,G \in C^{1}$ and $\alpha \beta < 1$.
	Then the following statements are equivalent:
	\begin{enumerate}[(a)]
		\item $H$ satisfies the (A) ($\alpha; \alpha', \lambda_s$) condition and $\sup_{y} \lip G(\cdot, y) \leq \beta$.
		\item For every $(x_1,y_2) \in X_1 \times Y_2$, $(DF(x_1,y_2), DG(x_1, y_2))$ satisfies the (A) ($\alpha; \alpha', \lambda_u$) condition and
		\[
		\sup_{(x_1,y_2)}|D_1G(x_1,y_2)| \leq \beta.
		\]
	\end{enumerate}

	In particular, if $ H \sim (F,G): X_1(r_1) \times Y_1(r'_1) \rightarrow X_2(r_2) \times Y_2 (r'_2) $ satisfies the (A) ($\alpha; \alpha', \lambda_u$) condition with 
	\[
	\sup_{y_2}\lip G(\cdot, y_2) \leq \beta, \quad \alpha \beta < 1/2,
	\]
	then $(DF(x_1, y_2), DG(x_1, y_2))$ also satisfies (A) ($\alpha; \alpha', \lambda_u$) and $|D_1G(x_1, y_2)| \leq \beta$ for any $(x_1, y_2) \in X_1(r_1) \times Y_2(r'_2)$.
\end{lem}

The condition (b) of \autoref{lem:c1}, which is satisfied by the linear operators $(DF(x_1,y_2)$, $DG(x_1, y_2))$ for $(x_1,y_2) \in X_1 \times Y_2$, is known as the \emph{invariant cone condition} (see, e.g., \cite{MS88, BLZ98, LYZ13}).

\begin{exa}
	Continuing with \autoref{exa:diff}, let $ H(t) $ be the continuous correspondence induced by \eqref{equ:ill}, i.e., $ H(t) \sim (F_{0,t}, G_{0,t}) $ for all $ t \geq 0 $, where $ F_{0,t}, G_{0,t} $ are given in \autoref{exa:diff}.
	If $ \mu_u - \mu_s - \varepsilon_1 - \varepsilon_2 > 0 $, then there are constants $ 0 < k_{1, t}, k_{2, t} < 1 $, $ \alpha, \beta $, and $ \lambda_s = e^{\mu_s + \varepsilon_1}, \lambda_u = e^{-\mu_u + \varepsilon_2} $ such that $ \alpha \beta < 1 $ and $ H(t) $ satisfies the (A)($ \alpha; k_{1, t}\alpha, \lambda^{t}_u $) (B)($ \beta; k_{2, t} \beta, \lambda^{t}_s $) condition with $ k_{i, t} \to 0 $ as $ t \to \infty $ and $ \varepsilon_{i} \to 0 $; for details, see \cite[Section 3]{Che18c}.
\end{exa}

\chapter{Applications to study dynamics of whiskered tori}\label{sec:application}

Center-stable, center-unstable, and center manifolds near ground states or ground state solitons for various Hamiltonian PDEs have been constructed by many authors; see e.g. \cite{Sch09, NS12, Bec12, KNS15, JLZ17} etc. and \autoref{rmk:PDEs}. In what follows, to emphasize our motivation, analogous to \cite[Theorem 1.2]{JLZ17}, we consider the existence of center-stable, center-unstable, and center manifolds near a \emph{whiskered torus} (i.e., a hyperbolic invariant torus). Heuristically, such results are direct consequences of our main results in \autoref{sec:statement} and the definition of a whiskered torus. The concept of whiskered torus was introduced by V. I. Arnol{\cprime}d \cite{Arn63, Arn64}.
The existence of invariant whiskered tori has been extensively studied in e.g. \cite{FdlLS09} (for exact symplectic maps and Hamiltonian flows), \cite{FdlLS15} (for coupled Hamiltonian systems on an infinite lattice), \cite{dlLS19} (for possibly ill-posed Hamiltonian PDEs), and \cite{CCdlL19} (for dissipative systems), etc. To provide applications to ill-posed PDEs (such as \cite{dlLS19}), we need to describe invariant whiskered tori for correspondences. Write $ \mathbb{T}^{d} = \mathbb{R}^{d} / \mathbb{Z}^{d} $.

\begin{defi}[Whiskered(-like) torus]\label{defi:torus}
	Given a correspondence $ H: X \to X $ where $ X $ is a Banach space, we say a $ C^1 $-smooth embedding $ K: \mathbb{T}^{d} \to X $ is an invariant \emph{whiskered(-like) torus} of $ H $ if $ K $ satisfies the following conditions.
	\begin{enumerate}[(a)]
		\item $ H \circ K = K \circ u_{\omega} $, where $ u_{\omega}: \mathbb{T}^{d} \to \mathbb{T}^{d} $ is a rotation with frequency $ \omega $, i.e., $ u_{\omega} (\theta) = \theta + \omega $. That is, $ H: R(K) \to R(K) $ induces a map conjugate to a rotation $ u_{\omega} $, where $ R(K) $ denotes the range of $ K $.

		\item (Uniform trichotomy) For each $ \theta \in \mathbb{T}^{d} $, there is a splitting
		\[
		T_{K(\theta)} X = X = X^s_{\theta} \oplus X^{c}_{\theta} \oplus X^{u}_{\theta},
		\]
		with associated projections $ \Pi^{s, c, u}_{\theta} $ such that

		\noindent (i) $ \theta \mapsto \Pi^{s, c, u}_{\theta} $ are $ C^1 $,

		\noindent (ii) $ X^{c}_{\theta} $ is finite-dimensional with $ \dim X^{c}_{\theta} \geq d $,

		\noindent (iii) For $ \kappa_1 = cs $, $ \kappa_2 = u $, $ \kappa = cs $, or $ \kappa_1 = s $, $ \kappa_2 = cu $, $ \kappa = cu $, we have
		\[
		H(\cdot - K(\theta)) - K(\theta + \omega) \sim (F^{\kappa}_{\theta}, G^{\kappa}_{\theta}): X^{\kappa_1}_{\theta} (r) \oplus X^{\kappa_2}_{\theta} (r_1) \to X^{\kappa_1}_{u_{\omega} (\theta)} (r_2) \oplus X^{\kappa_2}_{u_{\omega} (\theta)} (r),
		\]
		where $ r, r_1, r_2 > 0 $, such that

		(iii$ _1 $) $ F^{\kappa}_{\theta}(\cdot), G^{\kappa}_{\theta}(\cdot) \in C^{1} $ with $ DF^{\kappa}_{\theta}(\cdot), DG^{\kappa}_{\theta}(\cdot) $, $ \theta \in \mathbb{T}^{d} $, being equicontinuous,

		(iii$ _2 $) $ DF^{\kappa}_{\theta}(0)|_{X^{\kappa_2}_{u_{\omega} (\theta)}} = 0 $, $ DG^{\kappa}_{\theta}(0)|_{X^{\kappa_1}_{\theta}} = 0 $; and

		(iii$ _3 $) $ \sup_{\theta}|DF^{\kappa}_{\theta}(0)| < \lambda_{\kappa_1} $, $ \sup_{\theta}|DG^{\kappa}_{\theta}(0)| < \lambda_{\kappa_2} $ with $ \lambda_{s} $, $ \lambda_{u} $, $ \lambda_{\kappa_1} \lambda_{\kappa_2} < 1 $.

	\end{enumerate}

	Consider the case when $ H $ is a map. Then (b)(iii) is implied by the following classical assumption, which was used in e.g. \cite{FdlLS09, FdlLS15, CCdlL19}; see e.g. \autoref{lem:mapAB} for a proof.
	\begin{enumerate}[({iii}$ '_1 $)]
		\item $ H $ is a $ C^1 $ map in a neighborhood of $ K $ with $ DH $ uniformly continuous;
		\item $ A^{s, c, u}_{\theta} \triangleq DH(K(\theta))|_{X^{s, c, u}_{\theta}}:X^{s, c, u}_{\theta} \to X^{s, c, u}_{u_{\omega}(\theta)} $ with $ A^{c,u}_{\theta} $ invertible; and
		\item $ |A^{s}_{\theta}| < \lambda_{s} $, $ |(A^{u}_{\theta})^{-1}| < \lambda_{u} $, $ |A^{c}_{\theta}| < \lambda_{cs} $ and $ |(A^{c}_{\theta})^{-1}| < \lambda_{cu} $ with $ \lambda_{s} $, $ \lambda_{u} $, $ \lambda_{cs} \lambda_{u} $, $ \lambda_{cu} \lambda_{s} < 1 $.
	\end{enumerate}

	Note that, in this case, $ DF^{cs, cu}_{\theta}(0) = A^{cs, s}_{\theta} $ and $ DG^{cs, cu}_{\theta}(0) = (A^{u, cu}_{\theta})^{-1} $.
\end{defi}

\begin{rmk}
	\begin{asparaenum}[(a)]
		\item From \autoref{defi:torus} (a) and (b) (iii$ _2 $), combined with the conditions $ \lambda_{u} < 1 $ and $ \sup_{\theta}|DK(\theta)| < \infty $, it follows that $ DK(\theta) T_{\theta} \mathbb{T}^{d} \subset X^{cs}_{\theta} $; similarly, $ DK(\theta) T_{\theta} \mathbb{T}^{d} \subset X^{cu}_{\theta} $. This implies that $ DK(\theta) T_{\theta} \mathbb{T}^{d} \subset X^{c}_{\theta} $, and consequently, the invariant torus $ K $ is partially normally hyperbolic.
		
		\item If $ \dim X^{c}_{\theta} = d $, then the invariant torus $ K $ is normally hyperbolic; if $ \dim X^{c}_{\theta} = 2d $, then $ K $ corresponds to the whiskered torus studied in references such as \cite{FdlLS09, FdlLS15, dlLS19, CCdlL19}.
		
		\item The finite-dimensionality of $ X^{c}_{\theta} $ is not essential for \autoref{thm:torus}, except that without it we cannot guarantee the $ C^1 $ smoothness of $ W^{cs}_{loc}(K), W^{cu}_{loc}(K), \Sigma^c $ or the validity of \autoref{thm:torus} \eqref{it:toriLast}. In this case, the interpretation of \autoref{thm:torus} \eqref{it:tangent} is in the sense of Whitney, and the proof relies on \autoref{lem:selection}. Additionally, if the norm of $ X $ is $ C^1 $, then also $ W^{cs}_{loc}(K), W^{cu}_{loc}(K), \Sigma^c $ are $ C^1 $. Furthermore, if $ K \in C^{k+1} $ and the norm of $ X $ is $ C^{k+1} $, then \autoref{thm:torus} \eqref{it:toriLast} holds without assuming the finite-dimensionality of $ X^{c}_{\theta} $.
		
		\item The essential requirement in \autoref{defi:torus} (b) (iii$ _1 $) is that $ DF^{\kappa}_{\theta}(\cdot), DG^{\kappa}_{\theta}(\cdot) $, $ \theta \in \mathbb{T}^{d} $, are $ \xi $-almost equicontinuous with sufficiently small $ \xi $, which can be ensured by, for example, assuming that $ H $ is a $ C^1 $ map in a neighborhood of $ K $ (without requiring uniform continuity of $ DH $); this is a standard assumption. We adopt this condition because the existence of whiskered tori in the cited references requires it in a stronger form, such as when $ H $ is an analytic map. The condition in \autoref{defi:torus} (b) (iii$ _3 $) is also not the most general: (1) a more natural formulation would be, for instance, $ |(DF^{\kappa}_{\theta}(0))^{n}| \leq C\lambda^{n}_{\kappa_1} $ for some constant $ C $ independent of $ n $ and $ \theta $, but this can be achieved without loss of generality by choosing an appropriate metric (see \cite{Gou07}); (2) the independence of $ \lambda_{\kappa_1}, \lambda_{\kappa_2} $ from $ \theta $ implies that the hyperbolicity of $ K $ is described in an absolute sense (see \cite{HPS77}), which is related to Sacker--Sell spectral theory (see \cite[Section 6.3]{CL99}); note that in some concrete examples, such whiskered tori are constructed near an equilibrium (see e.g. \cite{dlLS19}). In \autoref{defi:torus} (b) (i), it suffices to assume that $ \theta \mapsto \Pi^{s, c, u}_{\theta} $ are $ C^0 $ (which implies that they are H\"older continuous; see also \cite[Section 6.4]{Che18a}). In fact, since $ K \in C^1 $ and $ \mathbb{T}^{d} $ is finite-dimensional, these projections can be approximated by $ C^1 $ ones; see \cite[Theorem 6.9]{BLZ08} or \autoref{cor:C1app}.
	\end{asparaenum}
\end{rmk}

In analogy with \cite[Theorem 1.2]{JLZ17}, we establish the existence of locally invariant manifolds near a torus. For brevity, if $ K: \mathbb{T}^{d} \to X $ is a $ C^1 $ embedding, we identify $ K $ with its range $ R(K) $.

\begin{thm}[Dynamical behavior near an invariant whiskered(-like) torus] \label{thm:torus}
	Let $ K: \mathbb{T}^{d} \to X $ be a $ C^1 $ invariant whiskered(-like) torus of a correspondence $ H: X \to X $ as defined in \autoref{defi:torus}, where $ X $ is a Banach space. Then there are $ C^1 $ embedded submanifolds $ W^{cs}_{loc}(K)$, $W^{cu}_{loc}(K)$, $\Sigma^c $ of $ X $ in a neighborhood of $ K $, called a \emph{local center-stable manifold}, a \emph{local center-unstable manifold}, and a \emph{local center manifold} of $ K $, respectively, satisfying the following properties:
	\begin{enumerate}[(1)]
		\item \label{it:toria} $ W^{cs}_{loc}(K), W^{cu}_{loc}(K), \Sigma^c $ contain $ K $ in their interiors; $ W^{cs}_{loc}(K) \cap W^{cu}_{loc}(K) = \Sigma^c $ is the transversal intersection of $ W^{cs}_{loc}(K) $ and $ W^{cu}_{loc}(K) $.

		\item \label{it:tangent} $ T_{K(\theta)}W^{cs}_{loc}(K) = X^{cs}_{\theta} $, $ T_{K(\theta)}W^{cu}_{loc}(K) = X^{cu}_{\theta} $, $ T_{K(\theta)}\Sigma^c = X^{c}_{\theta} $, $ \forall \theta \in \mathbb{T}^{d} $.

		\item $ W^{cs}_{loc}(K) $, $ W^{cu}_{loc}(K) $, and $ \Sigma^c $ are locally positively invariant, locally negatively invariant, and locally invariant under $ H $, respectively. That is, there are open sets $ \Omega_{cs} $, $ \Omega_{cu} $, and $ \Omega_{c} = \Omega_{cs} \cap \Omega_{cu} $ of $ W^{cs}_{loc}(K) $, $ W^{cu}_{loc}(K) $, and $ \Sigma^c $, respectively, containing $ K $ in their interiors, such that $ \Omega_{cs} \subset H^{-1}(W^{cs}_{loc}(K)) $, $ \Omega_{cu} \subset H(W^{cu}_{loc}(K)) $, and $ \Omega_{c} \subset H^{\pm 1}(\Sigma^{c}) $.

		\item $ H: \Omega_{cs} \to W^{cs}_{loc}(K) $ and $ H^{-1}: \Omega_{cu} \to W^{cu}_{loc}(K) $ induce Lipschitz maps, denoted by $ \mathcal{H}^{cs}_0 $ and $ \mathcal{H}^{-cu}_0 $, with Lipschitz constants less than $ \lambda_{cs} $ and $ \lambda_{cu} $, respectively; and $ H: \Omega_{c} \to \Sigma^{c} $ induces a bi-Lipschitz map, denoted by $ \mathcal{H}^{c}_0 $. Here, $ \mathcal{H}^{c}_0 = \mathcal{H}^{cs}_0|_{\Omega_{c}} $ and $ (\mathcal{H}^{c}_0)^{-1} = \mathcal{H}^{-cu}_0|_{\Omega_{c}} $. (See also \autoref{rmk:mapmeaning} \eqref{it:maps} for the precise meanings).

		\item \label{it:char} There are small constants $ \delta, \varepsilon > 0 $ such that the following hold. If 
		\[
		\{z_{n} = (K(\theta_{n}), x_{n})\}_{n \in \mathbb{N}} \subset X^{h}_{K}(\delta)
		\]
		(where $ X^{h}_{K}(\delta) $ is a tubular neighborhood of $ K $ in $ X $ with radius $ \delta $) is a forward orbit of $ H $ (i.e., $ z_{n+1} \in H(z_{n}) $ for all $ n \in \mathbb{N} $) satisfying $ |K(\theta_{n+1}) - K(\theta_{n} + \omega)| \leq \varepsilon $, then $ \{z_{n}\} \subset \Omega_{cs} $ and $ |x_{n+1}| \leq \lambda_{cs} |x_{n}| $; a similar statement holds for $ \Omega_{cu} $. In particular, if $ \{z_{n} = (K(\theta_{n}), x_{n})\}_{n \in \mathbb{Z}} \subset X^{h}_{K}(\delta) $ is an orbit of $ H $ (i.e., $ z_{n+1} \in H(z_{n}) $ for all $ n \in \mathbb{Z} $) satisfying $ |K(\theta_{n+1}) - K(\theta_{n} + \omega)| \leq \varepsilon $ and $ |K(\theta_{n-1}) - K(\theta_{n} - \omega)| \leq \varepsilon $ for all $ n \in \mathbb{Z} $, then $ \{z_{n}\} \subset \Omega_{c} $.

		\item \label{it:toriLast} Let $ k \in \mathbb{N} $ and $ r \in [0,1] $. If $ \lambda^{k+r}_{cs} \lambda_{u} $, $ \lambda^{k+r}_{cu} \lambda_{s} < 1 $, and $ F^{\kappa}_{\theta}(\cdot), G^{\kappa}_{\theta}(\cdot) \in C^{k, r} $ uniformly for $ \theta \in \mathbb{T}^{d} $ (see \autopageref{page:uniformdiff}), then we can take $ W^{cs}_{loc}(K), W^{cu}_{loc} (K), \Sigma^c \in C^{k,r} $; in particular, $ \theta \mapsto \Pi^{s, c, u}_{\theta} $ are $ C^{k-1,r} $.

		\item \label{it:lamination} There are two (H\"older) bundles $ \mathcal{W}^{ss}, \mathcal{W}^{uu} $ over $ \mathbb{T}^{d} $, called the strong stable and strong unstable laminations of $ K $, respectively, satisfying the following:
		\begin{enumerate}[(i)]
			\item $ \mathcal{W}^{ss} \subset \Omega_{cs} $, $ \mathcal{W}^{uu} \subset \Omega_{cu} $;
			\item each fiber $ W^{ss}_{\theta} $ (resp. $ W^{uu}_{\theta} $) of $ \mathcal{W}^{ss} $ (resp. $ \mathcal{W}^{uu} $) at $ \theta \in \mathbb{T}^{d} $, called the strong stable (resp. unstable) manifold of $ K(\theta) $, contains $ K(\theta) $ and is $ C^{1} $ with $ T_{K(\theta)} W^{ss}_{\theta} = X^{s}_{\theta} $ (resp. $ T_{K(\theta)} W^{uu}_{\theta} = X^{u}_{\theta} $);
			\item $ \mathcal{W}^{ss}, \mathcal{W}^{uu} $ are invariant under $ H $, i.e., $ \mathcal{H}^{cs}_0(\mathcal{W}^{ss}_{\theta}) \subset \mathcal{W}^{ss}_{\theta + \omega} $ and $ \mathcal{H}^{-cu}_0(\mathcal{W}^{uu}_{\theta}) \subset \mathcal{W}^{uu}_{\theta - \omega} $;
			\item Moreover, the following properties hold for all $ \theta \in \mathbb{T}^{d} $:

			(1) $ \mathcal{W}^{ss}_{\theta} \cap K = K(\theta) $ and $ \mathcal{W}^{uu}_{\theta} \cap K = K(\theta) $;

			(2) $ \mathcal{H}^{cs}_0|_{\mathcal{W}^{ss}_{\theta}}: \mathcal{W}^{ss}_{\theta} \to \mathcal{W}^{ss}_{\theta + \omega} $ (resp. $ \mathcal{H}^{-cu}_0|_{\mathcal{W}^{uu}_{\theta}}: \mathcal{W}^{uu}_{\theta} \to \mathcal{W}^{uu}_{\theta - \omega} $) is Lipschitz with Lipschitz constant less than $ \lambda_{s} $ (resp. $ \lambda_{u} $);

			(3) $ z \in W^{ss}_{\theta} $ if and only if $ z \in \Omega_{cs} $ and $ \sup_{n \geq 0} (\lambda^{n}_{s})^{-1} |(\mathcal{H}^{cs}_0)^{n} (z) - K(u^{n}_{\omega}(\theta))| < \infty $, and $ z \in W^{uu}_{\theta} $ if and only if $ z \in \Omega_{cu} $ and $ \sup_{n \geq 0} (\lambda^{n}_{u})^{-1} |(\mathcal{H}^{-cu}_0)^{n} (z) - K(u^{-n}_{\omega}(\theta))| < \infty $;

			\item \label{it:r} If, in addition, $ \lambda_{cs} \lambda_{s} \lambda_{cu} $, $ \lambda_{cu} \lambda_{u} \lambda_{cs} < 1 $, and $ F^{\kappa}_{\theta}(\cdot), G^{\kappa}_{\theta}(\cdot) \in C^{1,1} $ uniformly for $ \theta \in \mathbb{T}^{d} $, i.e.,
			\[
			\sup_{\theta} \{\lip_{z}D_{z}F^{\kappa}_{\theta}(z), \lip_{z}D_{z}G^{\kappa}_{\theta}(z): z \in X^{\kappa_1}_{\theta} (r) \times X^{\kappa_2}_{u_{\omega} (\theta)} (r) \} < \infty,
			\]
			(e.g., $ H $ is a $ C^2 $ map), then $ \mathcal{W}^{ss}, \mathcal{W}^{uu} $ are $ C^1 $ bundles.
		\end{enumerate}
	\end{enumerate}
\end{thm}

\begin{rmk}
	\begin{asparaenum}[(a)]
		\item Unless the torus $ K $ is normally hyperbolic, the strong stable (resp. unstable) lamination $ \mathcal{W}^{ss} $ (resp. $ \mathcal{W}^{uu} $) of $ K $ is not open in $ W^{cs}_{loc} (K) $ (resp. $ W^{cu}_{loc} (K) $), and consequently, it does not generally constitute a foliation of $ W^{cs}_{loc} (K) $ (resp. $ W^{cu}_{loc} (K) $). However, one can construct a foliation $ \widetilde{\mathcal{W}}^{ss} $ (resp. $ \widetilde{\mathcal{W}}^{uu} $) of $ W^{cs}_{loc} (K) $ (resp. $ W^{cu}_{loc} (K) $) that contains $ \mathcal{W}^{ss} $ (resp. $ \mathcal{W}^{uu} $) and is locally invariant under $ \mathcal{H}^{cs}_0 $ (resp. $ \mathcal{H}^{-cu}_0 $); i.e., for the fibers $ \widetilde{\mathcal{W}}^{ss}(z) $ through $ z \in W^{cs}_{loc} (K) $, we have $ \mathcal{H}^{cs}_0(\widetilde{\mathcal{W}}^{ss}(z) \cap \Omega_{cs}) \subset \widetilde{\mathcal{W}}^{ss}(\mathcal{H}^{cs}_0(z)) $ whenever $ z \in \Omega_{cs} $. This construction provides a clear understanding of the dynamical behavior of $ H $ in $ W^{cs}_{loc}(K) $ and $ W^{cu}_{loc}(K) $.

		\item If $ H $ is merely a correspondence, the preceding theorem provides sufficient information about the nonlinear dynamics of $ H $ near the torus $ K $. Essentially, when $ H $ is not an invertible map, one cannot expect the existence of long-time orbits beyond $ W^{cs}_{loc}(K) $ and $ W^{cu}_{loc}(K) $. To fully characterize the dynamical behavior near the torus $ K $ when $ H $ is an \emph{invertible map}, one must additionally obtain locally invariant center-(un)stable(-like) foliations $ \widetilde{\mathcal{W}}^{cs}, \widetilde{\mathcal{W}}^{cu} $ in a neighborhood of $ K $ that contain $ W^{cs}_{loc}(K), W^{cu}_{loc}(K) $ as their leaves. This leads to a decoupling of $ H $ into a product structure of the form like ``$ \mathcal{H}^{cs}_0 \times (\mathcal{H}^{-cu})^{-1} $ over $ \mathcal{H}^{c}_0 $''; see \cite[Corollary 4.19]{Che18a} for a global version.

		\item Consider the characterization of $ W^{cs}_{loc} (K) $ in \autoref{thm:torus} \eqref{it:char}. Now assume $ H $ is a \emph{map}. Then $ H \in C^1 $ in a neighborhood of $ K $. As $ K $ is compact, $ H $ is Lipschitz near $ K $. In this case, if $ \{z_{n} = (K(\theta_{n}), x_{n})\}_{n \in \mathbb{N}} \subset X^{h}_{K}(\delta) $ is a forward orbit of $ H $, the condition $ |K(\theta_{n+1}) - K(\theta_{n} + \omega)| \leq \varepsilon $ becomes redundant for sufficiently small $ \delta $, as can be seen from the estimate:
		\begin{align*}
		|K(\theta_{n+1}) - K(\theta_{n} + \omega)| & = |H(K(\theta_{n}) + x_{n}) - x_{n+1} - H(K(\theta_{n}))| \\
		& \leq \lip H |x_{n}| + |x_{n+1}| \leq (\lip H + 1) \delta.
		\end{align*}
		However, for a correspondence $ H $ that is not a map, or for a map that is not Lipschitz, the condition $ |K(\theta_{n+1}) - K(\theta_{n} + \omega)| \leq \varepsilon $ is needed to characterize $ W^{cs}_{loc}(K) $; see \autoref{defi:orbit} for the general case. We note that our main results in \autoref{sec:statement} also address \emph{non-Lipschitz} maps.

		\item While it is possible to consider the global invariance of $ W^{cs}_{loc}(K), W^{cu}_{loc}(K), \Sigma^c $, the additional conditions required are highly restrictive and are not discussed here. However, we mention that if $ K $ is \emph{stable} in $ \Sigma^c $ (in both forward and backward directions), then $ W^{cs}_{loc}(K), W^{cu}_{loc}(K), \Sigma^c $ are globally invariant (and hence locally unique) by \autoref{thm:torus} \eqref{it:char}. Finally, we emphasize that \autoref{thm:torus} remains valid when the torus is only \emph{immersed}, which is also an interesting case.
	\end{asparaenum}
\end{rmk}

\begin{proof}[Proof of \autoref{thm:torus}]
	The conclusions in \autoref{thm:torus} \eqref{it:lamination} are direct consequences of \cite[Theorem 4.6]{Che18a} (for the existence and characterization of $ \mathcal{W}^{ss}, \mathcal{W}^{uu} $) and \cite[Section 6]{Che18a} (for the regularities of $ \mathcal{W}^{ss}, \mathcal{W}^{uu} $); for item \eqref{it:r}, see also \cite[Section 5.5.3]{Che18b}. So we omit the details here but refer to \cite[Theorem 3.1]{CdlL19} in the (complex) analyticity setting or \cite[Theorem 3.1]{FM00} in the differentiable maps setting.

	Write $ X^{c_0}_{\theta} = DK(\theta) T_{\theta} \mathbb{T}^{d} $. Since $ K \in C^1 $ is compact, one gets $ \theta \mapsto X^{c_0}_{\theta} \in \mathbb{G}(X) $ is $ C^0 $ where $ \mathbb{G}(X) $ denotes the Grassmann manifold of $ X $ (see \autoref{sub:Grassmann}), and then by $ \dim X^{c}_{\theta} < \infty $, one can find $ X^{h_0}_{\theta} \in \mathbb{G}(X) $ such that $ \theta \mapsto X^{h_0}_{\theta} $ is $ C^0 $ and $ X^{c_0}_{\theta} \oplus X^{h_0}_{\theta} = X^{c}_{\theta} $ (see e.g. \autoref{lem:selection}). Now by $ C^1 $ approximation of $ \theta \mapsto X^{c_0}_{\theta}, X^{h_0}_{\theta} $ (see e.g. \autoref{cor:C1app} or \cite[Theorem 6.9]{BLZ08}), without loss of generality, we assume $ \theta \mapsto X^{c_0}_{\theta}, X^{h_0}_{\theta} $ are $ C^1 $; note also that in this step, \autoref{thm:geo} can be applied. As $ X^{c_0}_{\theta} \subset X^{c}_{\theta} $, by \autoref{lem:a4}, we see that \autoref{thm:torus} \eqref{it:toria}--\eqref{it:toriLast} are corollaries of the main results in \autoref{sec:statement} (see \autoref{thm:invariant} case (2) and \autoref{app:im}); for a proof of \autoref{thm:torus} \eqref{it:toriLast}, see also \autoref{cor:compact} \eqref{it:ht}.
\end{proof}

In \cite{FdlLS09, FdlLS15, dlLS19}, the authors established the existence of whiskered(-like) tori as defined in \autoref{defi:torus}. We now present more specific details.

\begin{exa}[Exact symplectic analytic maps in Hilbert spaces]\label{exa:map1}
	Let $ H: U \subset \mathcal{M} \to \mathcal{M} $ be an \emph{exact symplectic} analytic map with respect to the exact symplectic form $ d \alpha $ for some $ 1 $-form $ \alpha \in T^{*}\mathcal{M} $, where $ \mathcal{M} = X^{su} \times \mathbb{R}^{2d} \triangleq X $ or $ \mathcal{M} = X^{su} \times \mathbb{R}^{d} \times \mathbb{T}^{d} \subset X $, $ U $ is an open connected subset of $ \mathcal{M} $, and $ X^{su} = X^{s} \times (X^{s})^{*} $ with $ X^s $ a Hilbert space.

	In \cite{FdlLS09}, Fontich, de la Llave and Sire gave conditions under which $ H $ admits a whiskered torus $ K \subset \mathcal{M} $ satisfying \autoref{defi:torus} for the maps case (i.e., \autoref{defi:torus} (a), (b) (i) (ii) (iii$ '_1 $)--(iii$ '_3 $)) (and in a certain sense is unique); see \cite[Theorem 3.11]{FdlLS09}. Although their work primarily addresses the case where $ X^{su} $ is finite-dimensional, the proof remains valid when $ X^{su} $ is a Hilbert space (see also \cite{dlLS19}).

	A standard example is provided in \cite[Section 7.6.1]{FdlLS09}. Let $ \mathcal{M} \triangleq X^{su} \times \mathbb{R}^{d} \times \mathbb{T}^{d} $ be endowed with a canonical exact symplectic form, and consider
	\[
	F(x^s, x^u, x^c, \theta) = (\mu A_{\theta}x^s, \mu^{-1}(A^{-1}_{\theta})^{*} x^u, x^c, \theta + \omega), \quad (x^s, x^u, x^c, \theta) \in \mathcal{M},
	\]
	where $ A_{\theta}: X^s \to X^s $ is an invertible linear operator, $ |A_{\theta}| < \mu^{-1} $, and $ \theta \mapsto A_{\theta} $ is analytic in $ \mathbb{T}^{d}_{\rho} = \{ z = (z_1, \ldots, z_{d}) \in \mathbb{C}^{d} / \mathbb{Z}^{d}: \sup_{i} |\mathrm{Im} z_{i}| < \rho \} $.
	If $ H $ is an exact symplectic, analytic map sufficiently close to $ F $ in a small complex neighborhood of $ \mathcal{M} $, then $ H $ admits a whiskered torus satisfying \autoref{defi:torus}. Note that in this case, for any $ k \in \mathbb{N} $, the center-(un)stable manifolds $ W^{cs}_{loc}(K), W^{cu}_{loc}(K) $ can be constructed as $ C^{k} $ manifolds; while it is unknown whether they can be $ C^{\infty} $, any center-(un)stable manifolds are $ C^{\infty} $ at $ K $ in the sense of Whitney.

	Moreover, $ H $ can be taken as the time-$ t $ solution map of an exact symplectic analytic vector field $ \mathcal{X} $ on $ \mathcal{M} $ (i.e., $ \mathcal{L}_{\mathcal{X}} \alpha = d W $ for some analytic function $ W $ on $ \mathcal{M} $); see also \cite[Section 8]{FdlLS09}.
\end{exa}

\begin{exa}[Exact symplectic analytic maps on lattices]\label{exa:map2}
	In \cite{FdlLS15}, Fontich, de la Llave and Sire studied the existence of whiskered tori for an exact symplectic analytic map $ H $ defined on a lattice $ M^{\mathbb{Z}^{N}} $, where $ M = \mathbb{R}^{2n - d} \times \mathbb{T}^{d} $; see Theorems 3.6 and 3.11 therein. Thus, \autoref{thm:torus} can be applied to such a map $ H $ to describe the dynamical behavior near the torus. Without delving into technical details, we present a special case: the so-called $ 1 $-D Klein–Gordon system described by the formal Hamiltonian
	\[
	\mathcal{H}(q, p) = \sum_{n = -\infty}^{\infty} \left(\frac{1}{2} p^2_n + W(q_{n}) + \frac{\gamma}{2}(q_{n+1} - q_{n})^2\right) ~\text{on}~ \mathcal{M},
	\]
	where $ \mathcal{M} = l^{\infty}(\mathbb{Z}^{N}, M) \subset M^{\mathbb{Z}^{N}} $ is endowed with the canonical exact symplectic form given by $ \Omega_{\infty} = \sum_{n \in \mathbb{Z}^{N}} \mathrm{d} q_{n} \wedge \mathrm{d} p_{n} $. Let $ X = l^{\infty}(\mathbb{Z}^{N}, \mathbb{R}^{2n}) $ (a Banach space). Note that $ \mathcal{M} \subset X $, and any functions defined on $ \mathcal{M} $ can be naturally considered as defined on $ X $; moreover, the descriptions of hyperbolicity in $ \mathcal{M} $ and in $ X $ are identical. Let $ \mathcal{X} \in T\mathcal{M} $ denote the Hamiltonian vector field associated with $ \mathcal{H} $, and let $ H: X \to X $ be the time-$ t $ solution map of this field $ \mathcal{X} $. Suppose
	\begin{enumerate}[$ \bullet $]
		\item $ W: M \to \mathbb{R} $ is analytic, and the system $ \ddot{q} + DW(q) = 0 $ admits a hyperbolic equilibrium;
		\item there exists a set $ \Xi_0 \subset \mathbb{R}^{d} $ of positive Lebesgue measure such that for each $ \omega \in \Xi_0 $, there exists a KAM torus with frequency $ \omega $ invariant under the flow of $ \ddot{q} + DW(q) = 0 $ and non-degenerate in the sense of standard KAM theory (twist condition).
	\end{enumerate}
	Then, for a sufficiently small $ \gamma_{*} > 0 $, there is a set $ \Xi_0(\gamma_{*}) \subset \Xi_0 $ such that if $ |\gamma| < \gamma_{*} $ and $ \omega \in \Xi_0(\gamma_{*}) $, there is a whiskered torus $ K: \mathbb{R}^{d} \to \mathcal{M} \subset X $ with frequency $ \omega $ satisfying \autoref{defi:torus} for $ H $; see \cite[Theorem 3.11]{FdlLS15}. Therefore, \autoref{thm:torus} can be applied to such $ H $ with $ K $.
\end{exa}

\begin{exa}[$ 1 $-D (bad) Boussinesq equation]\label{exa:Bq}
	In \cite{dlLS19}, de la Llave and Sire investigated the existence of whiskered tori for some possibly ill-posed Hamiltonian PDEs. As noted in \cite{Che18c, ElB12} (see also \autoref{exa:diff}), ill-posed differential equations are not evolution equations but can generate continuous correspondences. So \autoref{thm:torus} provides an effective tool for studying the dynamical behaviors of ill-posed differential equations near whiskered tori at fixed times $ t > 0 $, a result not found in the existing literature. \autoref{thm:torus} can be applied to the general abstract Hamiltonian PDEs studied in \cite[Theorem 3.5]{dlLS19}.

	For simplicity, we consider a specific example from \cite[Theorem 3.6]{dlLS19}, namely the one-dimensional (bad) Boussinesq equation subject to the periodic boundary condition, i.e.,
	\[\label{equ:bBou}
	\partial^2_t u  = \partial^2_x u + \mu\partial^4_x u + \partial^2_x (u^2), \quad t \in \mathbb{R}, ~ x \in \mathbb{T}, ~ \mu > 0. \tag{bBou}
	\]
	This equation describes two-dimensional flow of shallow-water waves with small amplitudes. 
	Following \cite{dlLla09}, let $ H^{\sigma, m} $ denote the space of analytic functions $ h $ in $ \mathbb{T}_{\sigma}  = \{ z \in \mathbb{C} / \mathbb{Z}: |\mathrm{Im} z| < \sigma \} $ such that
	\[
	|h|^2_{\sigma, m} = \sum_{k \in \mathbb{Z}} |h_{k}|^2 e^{4\pi \sigma |k|} (|k|^{2m} + 1)
	\]
	is finite, with the norm $ |h|^2_{\sigma, m} $, where $ \{ h_{k} \} $ are the Fourier coefficients of $ h $.
	
	Define $ X = H^{\sigma, m} \times H^{\sigma, m-2} $ for $ m > 5/2 $ and $ \sigma > 0 $, and its subspace $ X_0 $ consisting of symmetric functions with zero average (i.e., $ h \in X_0 $ if and only if $ h \in X $ satisfies $ h(x) = h(-x) $ for $ x \in \mathbb{T} $, and $ \int_{0}^{1} h (x) ~\mathrm{d}x = 0 $).
	We first rewrite \eqref{equ:bBou} in abstract form:
	\[\label{equ:acp}
	\dot{z} = Az + N(z), \quad z = (u, v) \in X, \tag{ACP}
	\]
	where
	\[
	A = \begin{pmatrix}
	0 & 1 \\
	\partial^2_x + \mu\partial^4_x & 0
	\end{pmatrix},
	\quad N(u, v)(x) = \left( 0, \partial^2_x (u^2)\right): X \to X.
	\]

	Note that $ N: X \to X $ is analytic with $ N(0) = 0 $ and $ DN(0) = 0 $, and $ D(A) = H^{\sigma, m+2} \times H^{\sigma, m} $. The spectrum of $ A $ in $ X $ is given by
	\[
	\sigma (A) = \left\{ \pm 2 \pi |k| \sqrt{\mu 4\pi^2 k^2 - 1} \right\}_{k \geq 1} \triangleq \{ \lambda_{\pm k}(\mu) \}_{k \geq 1},
	\]
	with corresponding eigenfunctions $ \phi_{\pm k} = (u_{\pm k}, v_{\pm k}) $ satisfying $ A \phi_{\pm k} = \lambda_{\pm k}(\mu) \phi_{\pm k} $ for $ k \geq 1 $, where
	\[
	(u_{\pm k}, v_{\pm k}) = (e^{\pm 2\pi i kx}, \lambda_{\pm k}(\mu) e^{\pm 2\pi i kx}), \quad k = 1,2,\ldots.
	\]
	Let $ X^{s}_{*} = \overline{\mathrm{span}\{ \phi_{k}: \mathrm{Re} \lambda_{k} (\mu) < 0 \}} $ (in $X$), $ X^{u}_{*} = \overline{\mathrm{span}\{ \phi_{k}: \mathrm{Re} \lambda_{k} (\mu) > 0 \}} $ (in $X$), and $ X^{c}_{*} = \mathrm{span}\{ \phi_{k}: \mathrm{Re} \lambda_{k} (\mu) = 0 \} $. Note that $ X^{c}_{*} $ is finite-dimensional. Since $ \partial^2_x + \mu\partial^4_x: H^{\sigma, m+2} \to H^{\sigma, m-2} $ is selfadjoint, we obtain the decomposition $ X = X^{s}_{*} \oplus X^{c}_{*} \oplus X^{u}_{*} $ with associated projections $ \Pi^{s, c, u}_{*} $. Define $ X^{s, c, u}_{0} = X^{s, c, u}_* \cap X_0 $, yielding $ X_{0} = X^{s}_{0} \oplus X^{c}_{0} \oplus X^{u}_{0} $ with projections $ \Pi^{s, c, u}_{0} $. For example,
	\[
	X^{c}_{0} = \mathrm{span}\{ (\cos(2\pi kx), \lambda_{k}(\mu) \cos(2\pi kx)): \mathrm{Re} \lambda_{k} (\mu) = 0 \}.
	\]

	Let $ A_{s, c, u} = A_{X^{s, c, u}_{*}} $ denote the part\footnote{The part of $ A $ in $ Y \subset X $, denoted by $ A_{Y} $, is defined by $ A_{Y} x = Ax $, $ x \in D(A_{Y}) = \{ x \in D(A) \cap Y: Ax \in Y \} $.} of $ A $ in $ X^{s, c, u}_{*} $. Note that $ \dim X^{c}_0 $ is even and $ 0 \notin \sigma(A_{X^{c}_0}) $. One can verify that there are $ 0 < \lambda_{s}, \lambda_{u} < 1 $ and $ \lambda_{c} > 1 $, with $ \lambda_{c} $ sufficiently close to $ 1 $, which depend only on $ \mu $, such that
	\begin{equation*}
	|e^{t A_{s}}| \leq \lambda^{t}_{s}, \quad |e^{-t A_{u}}| \leq \lambda^{t}_{u}, \quad \text{for } t \geq 0, \quad \text{and} \quad |e^{tA_{c}}| \leq \lambda^{t}_{c}, \quad \text{for } t \in \mathbb{R};
	\end{equation*}
	see also \cite[Section 6.1.3]{dlLla09}. In particular, mild solutions $ z(t) = (x_{s}(t), x_{c}(t), x_{u}(t)) \in X^{s}_{*} \oplus X^{c}_{*} \oplus X^{u}_{*} $ ($ t_1 \leq t \leq t_2 $) of \eqref{equ:acp} satisfy the following variation of constants formula:
	\[
	\begin{cases}
	x_{s, c}(t) = e^{(t - t_1) A_{s, c}}x_{s, c}(t_1) + \int_{t_1}^{t} e^{(t - r) A_{s, c}} \Pi^{s, c}_{*} N(z(r)) ~\mathrm{d} r, \\
	x_{c, u}(t) = e^{(t - t_2) A_{c, u}}x_{c, u}(t_2) - \int_{t}^{t_2} e^{(t - r) A_{c, u}} \Pi^{c, u}_{*} N(z(r)) ~\mathrm{d} r,
	\end{cases}
	\quad t_1 \leq t \leq t_2.
	\]

	For $ t > 0 $, define the \emph{correspondence} $ H(t): X \to X $ by $ (x_1, y_1) \in H(t)(x_0, y_0) $ if and only if there is a mild solution $ (x(s), y(s)) $ ($ 0 \leq s \leq t $) of \eqref{equ:acp} satisfying $ (x(0), y(0)) = (x_{0}, y_{0}) $ and $ (x(t), y(t)) = (x_{1}, y_{1}) $.

	For any $ t > 0 $, if $ r_{t} > 0 $ is sufficiently small, then for any $ (x_{c,s}, y_{u}) \in X^{c,s}_{*}(r_{t}) \times X^{u}_{*} (r_{t}) $ (or $ (x_{s}, y_{c, u}) \in X^{s}_{*}(r_{t}) \times X^{c, u}_{*} (r_{t}) $), the above equation has a unique solution 
	\[
	z(t') = (x_{s}(t'), x_{c}(t'), x_{u}(t')) \in X^{s}_{*}(r_{t}) \oplus X^{c}_{*}(r_{t}) \oplus X^{u}_{*}(r_{t}), \quad 0 \leq t' \leq t,
	\]
	with $ x_{s, c}(0) = x_{s, c}, x_{c}(t) = y_{u} $ (or $ x_{s}(0) = x_{s}, x_{c, u}(t) = y_{c, u} $); see e.g. \cite{ElB12}. Set
	\[
	F^{cs}_{*, t}(x_{s}, x_{c}, y_{u}) = (x_{s}(t), x_{c}(t)), \quad G^{cs}_{*, t} (x_{s}, x_{c}, y_{u}) = x_{u} (0),
	\]
	and
	\[
	F^{cu}_{*, t}(x_{s}, y_{c}, y_{u}) = x_{s}(t), \quad G^{cu}_{*, t} (x_{s}, y_{c}, y_{u}) = (x_{c} (0), x_{u} (0)).
	\]
	Then $ H(t) \sim (F^{cs}_{*, t}, G^{cs}_{*, t}): X^{cs}_{*} (r_{t}) \times X^{u}_{*}(r_{t}) \to X^{cs}_{*} (r_{t}) \times X^{u}_{*}(r_t) $ and $ H(t) \sim (F^{cu}_{*, t}, G^{cu}_{*, t}): X^{s}_{*} (r_{t}) \times X^{cu}_{*}(r_{t}) \to X^{s}_{*} (r_{t}) \times X^{cu}_{*}(r_t) $. Since $ e^{tA_{s}}: X^{s}_{0} \to X^{s}_{0} $ ($ t \geq 0 $) and similarly for $ e^{tA_{c, u}} $, and \eqref{equ:bBou} is symmetric in phase space and conserves the quantity $ \int_{0}^{t} \partial_{t} u(t, x) ~\mathrm{d} x $, these statements also hold when restricting $ H(t) $ to $ X_0 $ (i.e., replacing $ X^{s, c, u}_{*} $ with $ X^{s, c, u}_{0} $). Thus $ H(t): X_0 \to X_0 $.

	Now fix $ t_0 > 0 $ and $ r_0 = r_{t_0} $. Write $ H_0 = H(t_0) $ and $ u_{\omega} (\theta) = \theta + \omega $.

	By \cite[Theorem 3.5]{dlLS19}, for fixed $ d \in \mathbb{N} $, there are an open set $ I'_{d} \subset \mathbb{R}_{+} $ and a Cantor set $ \mathcal{C} \subset \mathbb{R}^{d} $ such that for each $ \omega \in \mathcal{C} $ and $ \mu \in I'_{d} $, $ \dim X^{c}_{0} = 2d $ and there is an analytic whiskered torus $ K_{\omega}: \mathbb{T}^{d} \to X_{0} \subset X $ with frequency $ \omega $, close to $ 0 $, solving $ (A + N) \circ K_{\omega} = DK_{\omega} \circ u_{\omega} $; that is, for any $ \theta \in \mathbb{T}^{d} $, $ z(t) = (u(t), \dot{u}(t)) = K_{\omega}(\omega t + \theta) $ is a quasi-periodic solution of \eqref{equ:acp} (i.e., $ u(\cdot) $ is a quasi-periodic solution of \eqref{equ:bBou}) and $ \sup_{\theta}|K_{\omega}(\theta)| < r_0  / 4 $ is sufficiently small (as follows from the construction in \cite[Theorem 3.5]{dlLS19}; see \cite[Section 10.6]{dlLS19}). In particular,
	\[
	H_0 \circ K_{\omega} = K_{\omega} \circ u_{\omega t_0}.
	\]

	Moreover, $ K_{\omega} $ is \emph{spectrally non-degenerate} in the sense of \cite[Definition 3.3]{dlLS19} (or equivalently satisfies uniform trichotomy), i.e., there are splittings $ X_{0} = X^{s}_{\theta} \oplus X^{c}_{\theta} \oplus X^{u}_{\theta} $ with associated projections $ \Pi^{s, c, u}_{\theta} $, $ \theta \in \mathbb{T}^{d} $, and three \emph{cocycles} $ U_{s,c,u} $ over $ u_{\omega} $ such that
	\begin{enumerate}[$ (\bullet 1) $]
		\item The projections $ \theta \mapsto \Pi^{s, c, u}_{\theta} $ are analytic;
		\item $ \sup_{\theta}|\Pi^{s, c, u}_{\theta} - \Pi^{s, c, u}_{0}| $ is small, so in particular $ \dim X^{c}_{\theta} = 2d $;
		\item $ U_{s}(t, \theta): X^{s}_{\theta} \to X^{s}_{\theta + \omega t} $, $ U_{u}(-t, \theta): X^{u}_{\theta} \to X^{u}_{\theta + \omega t} $ for $ t \geq 0 $, and $ U_{c}(t, \theta): X^{c}_{\theta} \to X^{c}_{\theta + \omega t} $ for $ t \in \mathbb{R} $;
		\item $ |U_{s}(t, \theta)| \leq (\lambda^{*}_{s})^t $, $ |U_{u}(-t, \theta)| \leq (\lambda^{*}_{u})^t $ for $ t \geq 0 $, and $ |U_{c}(t, \theta)| \leq (\lambda^{*}_{c})^t $ for $ t \in \mathbb{R} $, where $ \lambda^{*}_{s, c, u} = \lambda_{s, c, u} + \epsilon $ with $ \epsilon $ sufficiently small and depending only on $ \sup_{\theta}|K_{\omega}(\theta)| $;
		\item $U_{s, c, u}$ satisfy the following variational equations
		\[
		\begin{cases}
		U_{s}(t,\theta) = \id + \int_{0}^{t} A(\theta + \omega r) U_{s}(r, \theta) ~\mathrm{d}r, \quad t > 0, \\
		U_{u}(t,\theta) = \id + \int_{0}^{t} A(\theta + \omega r) U_{u}(r, \theta) ~\mathrm{d}r, \quad t < 0, \\
		U_{c}(t,\theta) = \id + \int_{0}^{t} A(\theta + \omega r) U_{c}(r, \theta) ~\mathrm{d}r, \quad t \in \mathbb{R},
		\end{cases}
		\]
		where $ A(\theta) = A + DN (K_{\omega}(\theta)) $; more precisely,
		\[
		\begin{cases}
		U_{s}(t,\theta)x = x + A \int_{0}^{t} U_{s}(r, \theta)x ~\mathrm{d}r + \int_{0}^{t} DN(\theta + \omega r) U_{s}(r, \theta)x ~\mathrm{d}r, ~~ t > 0, x \in X^{s}_{\theta}, \\
		U_{u}(t,\theta)x = x + A \int_{0}^{t} U_{u}(r, \theta)x ~\mathrm{d}r + \int_{0}^{t} DN(\theta + \omega r) U_{u}(r, \theta)x ~\mathrm{d}r, ~~ t < 0, x \in X^{u}_{\theta}, \\
		U_{c}(t,\theta)x = x + A \int_{0}^{t} U_{c}(r, \theta)x ~\mathrm{d}r + \int_{0}^{t} DN(\theta + \omega r) U_{c}(r, \theta)x ~\mathrm{d}r, ~~ t \in \mathbb{R}, x \in X^{c}_{\theta},
		\end{cases}
		\]
		where $ DN(\theta) = DN (K_{\omega}(\theta)) $.
		That is, for $ z(t) = (U_{cs}(t, \theta)x_{cs}, U_{u}(t - r, \theta + r\omega)y_{u}) $ or $ z(t) = (U_{s}(t, \theta)x_{s}, U_{cu}(t - r, \theta + r\omega)y_{cu}) $ ($ 0 \leq t \leq r $), it is a mild solution of the linearized equation along $ K_{\omega} $:
		\[
		\label{equ:Vareq}
		\dot{z}(t) = Az(t) + DN(K_{\omega}(\theta + t\omega))z(t). \tag{Vareq}
		\]
	\end{enumerate}

	Set
	\[
	f(\theta)z = N(K_{\omega}(\theta) + z) - DN(K_{\omega}(\theta))z - N(z).
	\]
	One can verify that $ z(t) $ is a mild solution of \eqref{equ:acp} if and only if $ w(t) = z(t) - K_{\omega}(\theta + t\omega) $ is a mild solution of the following equation in the cocycle form:
	\[
	\dot{w}(t) = Aw(t) + DN(K_{\omega}(\theta + t\omega))w(t) + f(\theta + t\omega) w(t).
	\]
	The mild solutions $ w(t) = (w_{s}(t), w_{c}(t), w_{u}(t)) \in X^{s}_{\theta+t\omega} \oplus X^{c}_{\theta+t\omega} \oplus X^{u}_{\theta+t\omega} $ ($ 0 \leq t \leq t_1 $) satisfy the following variation of constants formula: for $ 0 \leq t \leq t_1 $,
	\[
	\begin{cases}
	w_{s,c}(t) = U_{c,s}(t,\theta)w_{s,c}(0) + \int_{0}^{t} U_{s, c}(t - r,\theta+r\omega) f_{s, c}(\theta+r\omega) w(r) ~\mathrm{d}r, \\
	w_{c, u}(t) = U_{c, u}(t - t_1,\theta+t_1\omega)w_{c, u}(t_1) - \int_{t}^{t_1} U_{c, u}(t - r,\theta+r\omega) f_{c, u}(\theta+r\omega) w(r) ~\mathrm{d}r,
	\end{cases}
	\]
	where $ f_{s, c, u} (\theta) z = \Pi^{s, c, u}_{\theta} f(\theta) z $. For further details, we refer to \cite[Section 3.1]{Che18c}.

	Since $ \sup_{\theta} |K_{\omega}(\theta)| $ is small, $ \sup_{\theta}\lip f(\theta)(\cdot)|_{X(r)} $ is small when $ r > 0 $ is small (e.g., $ r < r_0 / 4 $). This implies that for $ t_1 = t_0 $, the above equation has a unique solution 
	\[
	w(t) = (w_{s}(t), w_{c}(t), w_{u}(t)) \in X^{s}_{\theta+t\omega} (r_0) \oplus X^{c}_{\theta+t\omega} (r_0) \oplus X^{u}_{\theta+t\omega} (r_0), \quad 0 \leq t \leq t_0,
	\]
	such that $ w_{s, c}(0) = x_{s, c} $, $ w_{u}(t_0) = y_{u} $ (or $ w_{s}(0) = x_{s} $, $ w_{c,u}(t_0) = y_{c,u} $).
	Therefore, for $ \kappa_1 = cs $, $ \kappa_2 = u $, $ \kappa = cs $, or $ \kappa_1 = s $, $ \kappa_2 = cu $, $ \kappa = cu $, we have
	\[
	H_0 - K(\theta + t_0\omega) \sim (F^{\kappa}_{\theta}, G^{\kappa}_{\theta}): X^{\kappa_1}_{\theta} (r_0) \oplus X^{\kappa_2}_{\theta} (r_0) \to X^{\kappa_1}_{\theta + t_0\omega} (r_0) \oplus X^{\kappa_2}_{\theta + t_0\omega} (r_0),
	\]
	with $ F^{\kappa}_{\theta}(\cdot), G^{\kappa}_{\theta}(\cdot) $ analytic,
	\[
	\sup_{\theta} \{|D^{2}_{z}F^{\kappa}_{\theta}(z)|, |D^{2}_{z}G^{\kappa}_{\theta}(z)|: z \in X^{\kappa_1}_{\theta} (r_0) \times X^{\kappa_2}_{\theta + t_0\omega} (r_0) \} < \infty,
	\]
	and
	\begin{gather*}
	DF^{cs}_{\theta}(0) = U_{s}(t_0, \theta) \oplus U_{c}(t_0, \theta), \quad DG^{cs}_{\theta}(0) = U_{u}(-t_0, t_0\theta),\\
	DF^{cu}_{\theta}(0) = U_{s}(t_0, \theta), \quad DG^{cu}_{\theta}(0) = U_{c}(-t_0, t_0\theta) \oplus U_{u}(-t_0, t_0\theta).
	\end{gather*}

	Applying \autoref{thm:torus} to $ H_0: X_0 \to X_0 $, we get the dynamical behaviors of $ H_0 $ in $ X_0 $ (i.e.,  \eqref{equ:bBou} for discrete time $ t_0 > 0 $ in $ X_0 $) near $ K_{\omega} $.

	The whiskered torus $ K_{\omega} $ of \eqref{equ:bBou} constructed in \cite{dlLS19} lies in $ X_0 $ (with no general result available for $ K_{\omega} \in X \setminus X_0 $). However, it is more natural to consider the dynamical behaviors of $ H_0 $ in $ X $ rather than $ X_0 $.

	The uniform trichotomy of $ K_{\omega} $ (i.e., ($ \bullet 1 $)--($ \bullet 5 $)) is constructed in $ X_0 $ as shown in \cite[Lemma 6.1]{dlLS19}. Similarly (see also \cite[Appendix A]{CdlL19}), one can construct splittings $ X = \widetilde{X}^{s}_{\theta} \oplus \widetilde{X}^{c}_{\theta} \oplus \widetilde{X}^{u}_{\theta} $ with associated projections $ \widetilde{\Pi}^{s, c, u}_{\theta} $, $ \theta \in \mathbb{T}^{d} $, and three \emph{cocycles} $ \widetilde{U}_{s,c,u} $ over $ u_{\omega} $ such that
	\begin{enumerate}[$ (\bullet 1' ) $]
		\item $ \theta \mapsto \widetilde{\Pi}^{s, c, u}_{\theta} $ are analytic;
		\item $ \sup_{\theta}|\widetilde{\Pi}^{s, c, u}_{\theta} - \widetilde{\Pi}^{s, c, u}_{*}| $ is small, so $ \dim \widetilde{X}^{c}_{\theta} = \dim X^{c}_{*} \geq 2d $;
		\item $ U_{s}(t, \theta): \widetilde{X}^{s}_{\theta} \to \widetilde{X}^{s}_{\theta + \omega t} $, $ U_{u}(-t, \theta): \widetilde{X}^{u}_{\theta} \to \widetilde{X}^{u}_{\theta + \omega t} $ for $ t \geq 0 $, and $ U_{c}(t, \theta): \widetilde{X}^{c}_{\theta} \to \widetilde{X}^{c}_{\theta + \omega t} $ for $ t \in \mathbb{R} $;
		\item $ |U_{s}(t, \theta)| \leq (\lambda^{*}_{s})^t $, $ |U_{u}(-t, \theta)| \leq (\lambda^{*}_{u})^t $ for $ t \geq 0 $, and $ |U_{c}(t, \theta)| \leq (\lambda^{*}_{c})^t $ for $ t \in \mathbb{R} $, where $ \lambda^{*}_{s, c, u} = \lambda_{s, c, u} + \epsilon $ with $ \epsilon $ sufficiently small depending only on $ \sup_{\theta}|K_{\omega}(\theta)| $;
		\item For $ z(t) = (U_{cs}(t, \theta)x_{cs}, U_{u}(t - r, \theta + r\omega)y_{u}) $ or $ z(t) = (U_{s}(t, \theta)x_{s}, U_{cu}(t - r, \theta + r\omega)y_{cu}) $ ($ 0 \leq t \leq r $), it is a mild solution of \eqref{equ:Vareq}.
	\end{enumerate}
	By uniqueness of the construction, we have $ X^{s, c, u}_{\theta} \subset \widetilde{X}^{s, c, u}_{\theta} $. This construction also follows from \cite[Theorem 4.13]{Che18c} applied to \eqref{equ:Vareq}. More precisely (considering the stable direction), since $ \sup_{\theta} |K_{\omega}(\theta)| $ is small, $ \sup_{\theta}|DN(K_{\omega}(\theta))| $ is small, and by \cite[Theorem 4.13]{Che18c}, there are unique $ h_{\theta}: X^{s}_{*} \to X^{cu}_{*} $ and $ x_{\theta, t}: X^{s}_{*} \to X^{s}_{*} $ with $ \lip h_{\theta}(\cdot) $ small and $ \lip x_{\theta, t}(\cdot) \leq (\lambda^{*}_{s})^t $ such that
	\[
	( x_{\theta, t}(x), h_{\theta + t\omega}( x_{\theta, t}(x) ) ) \in H(t) (x, h_{\theta}(x));
	\]
	since \eqref{equ:Vareq} is linear, $ h_{\theta}(\cdot) $ is linear; we can take $ \widetilde{X}^{s}_{\theta} = \graph h_{\theta} $ and 
	\[
	\widetilde{U}_{s}(t, \theta): (x, h_{\theta}(x)) \mapsto ( x_{\theta, t}(x), h_{\theta + t\omega}( x_{\theta, t}(x) ) );
	\]
	moreover, since $ \mathbb{T}^{d}_{\sigma} \to L(X, X): \theta \mapsto DN(K_{\omega}(\theta)) $ is analytic, the analyticity property of the projections follows easily (see e.g. \cite[Appendix A]{CdlL19}).

	\begin{enumerate}[$ \bullet $]
		\item Therefore, applying \autoref{thm:torus} to $ H_0: X \to X $, we get the dynamical behaviors of $ H_0 $ in $ X $ (i.e., equation \eqref{equ:bBou} for discrete time $ t_0 > 0 $ in $ X $) near $ K_{\omega} $.
	\end{enumerate}
	Note that in this case $ DK_{\omega}(\theta) T_{\theta} \mathbb{T}^{d} = X^{c}_{\theta} \subset \widetilde{X}^{c}_{\theta} $.
\end{exa}

\begin{rmk}
	\begin{asparaenum}[(a)]
		\item The local center-stable, center-unstable, and center manifolds of $ K_{\omega} $ constructed in this section exhibit dependence on the discrete time parameter $ t_0 $. This result possesses independent mathematical interest. Nevertheless, employing appropriate truncation techniques for the nonlinearity $ N $ (see, e.g., \cite{CLY00}), one can obtain local center-stable, center-unstable, and center manifolds of $ K $ that are independent of the time parameter.
		
		\item The center-(un)stable and center-stable manifolds of the equilibrium point $ 0 $ for the $ 1 $-D (bad) Boussinesq equation \eqref{equ:bBou} have been previously constructed in \cite{dlLla09} (see also \cite{ElB12}); our paper \cite{Che18a} contains further developments in this direction. To the best of our knowledge, the present work provides the first construction of center-(un)stable and center manifolds of $ K_{\omega} $ for \eqref{equ:bBou} with discrete time $ t_0 > 0 $.
		
		\item While we have primarily demonstrated how the main results in \autoref{sec:statement} apply to dynamics near invariant whiskered tori, in \cite{FdlLS09, FdlLS15, dlLS19}, the invariant whiskered tori are constructed through the posteriori ones (i.e., given tori solving the invariance equations approximately, true solutions exist nearby under some non-degeneracy conditions). Our results establish the existence of locally invariant manifolds in the vicinity of such tori (see \autoref{thm:tri0}). We observe that the invariant whiskered tori necessarily lie within the center manifold; the conditions required by our main results are usually weaker than the existence of invariant whiskered tori. So we hope that our findings will facilitate the discovery of different types of invariant whiskered tori under more relaxed conditions, analogous to the search for homoclinic and heteroclinic orbits as well as periodic orbits in practical models (see \cite{Zen00,LLSY16}); in addition, these results are expected to be valuable for identifying orbits that are homoclinic or heteroclinic to tori.
	\end{asparaenum}
\end{rmk}

\chapter{Approximately partially normal hyperbolicity: statements}\label{sec:statement}

\section{Set-up}\label{sub:setup}

(I) (\emph{Submanifold}).
We need the following hypotheses on the $ C^{0,1} $ (immersed) submanifold $ \Sigma $ of $ X $.
Let $ \widehat{\Sigma} $ be a $ C^0 $ manifold and $ \phi: \widehat{\Sigma} \to X $ a $ C^0 $ map with $ \phi(\widehat{\Sigma}) = \Sigma $.

\begin{enumerate}[$ \bullet $]
	\item For $ \epsilon > 0 $ and any $ \widehat{m} \in \widehat{\Sigma} $, let $ \widehat{U}_{\widehat{m}} (\epsilon) $ be the component of $ \phi^{-1} (\Sigma \cap \mathbb{B}_{\epsilon} (\phi(\widehat{m})) ) $ containing $ \widehat{m} $. Set $ \phi ( \widehat{U}_{\widehat{m}} (\epsilon) ) = U_{m,\gamma} (\epsilon) $, where $ m = \phi(\widehat{m}) $, $ \gamma \in \Lambda(m) \triangleq \phi^{-1}(m) $. The map $ \phi: \widehat{U}_{\widehat{m}}(\epsilon) \to U_{m,\gamma}(\epsilon) $ is a homeomorphism, and $ \widehat{U}_{\widehat{m}}(\epsilon) $ is open.

	\item There is a family of projections $ \{ \Pi^{\kappa}_{m}: m \in \Sigma, \kappa = s, c, u \} $ such that $ \Pi^{s}_{m} + \Pi^{c}_{m} + \Pi^{u}_{m} = \id $. Set $ \Pi^{h}_{m} = \id - \Pi^{c}_m $, and $ X^{\kappa}_{m} = R(\Pi^{\kappa}_{m}) $, $ \kappa = s, c, u, h $; also $ X^{\kappa_1\kappa_2}_{m} = X^{\kappa_1}_{m} \oplus X^{\kappa_2}_{m} $, where $ \kappa_1 \neq \kappa_2 \in \{ s,c, u \} $. For $ K \subset \Sigma $, $ \widehat{K}' \subset \widehat{\Sigma} $ and $ \kappa = s, c, u,h $, define
	\begin{equation}\label{equ:notationM}
	\begin{gathered}
	\widehat{K} = \phi^{-1}(K) \subset \widehat{\Sigma}, ~\widehat{K}_{\epsilon} = \bigcup_{\widehat{m} \in \widehat{K}} \widehat{U}_{\widehat{m}}(\epsilon), ~ K_{\epsilon} = \phi(\widehat{K}_{\epsilon}), \\
	{X}^{\kappa}_{\widehat{K}'} = \{ (\widehat{m}',x): x \in {X}^{\kappa}_{\phi(\widehat{m}')}, \widehat{m}' \in \widehat{K}' \}, ~
	{X}^{\kappa}_{\widehat{K}'} (r) = \{ (\widehat{m}',x): x \in {X}^{\kappa}_{\phi(\widehat{m}')}(r), \widehat{m}' \in \widehat{K}' \}.
	\end{gathered}
	\end{equation}
\end{enumerate}

We make the following assumptions on $ \Sigma $, which are essentially due to \cite{BLZ99, BLZ08}. 
\begin{enumerate}[(H1)]
	\item ($ C^{0,1} $ immersed submanifold). For any $ m \in \Sigma $, there exist $ \epsilon_m > 0 $ and $ \delta_0(m) > 0 $ such that
	\[
	\sup\left\{ \frac{|m_1 - m_2 - \Pi^c_{m}(m_1 - m_2)|}{|m_1 - m_2|} : m_1 \neq m_2 \in U_{m,\gamma} (\epsilon) \right\} \leq \chi_{m}(\epsilon) < 1,
	\]
	where $ \gamma \in \Lambda(m) $, $ 0 < \epsilon \leq \epsilon_{m} $, and $ \chi_{m}(\cdot) > 0 $ is increasing, and
	\[
	X^c_{m}(\delta_0(m)) \subset \Pi^c_{m} (U_{m,\gamma} (\epsilon_{m}) - m).
	\]
	That is, $ \Sigma $ is a \emph{$ C^{0,1} $ immersed submanifold} of $ X $ (cf. \cite{Pal66}) \emph{locally modeled on} $ X^{c}_{m} $, $ m \in \Sigma $.

	Take $ K \subset \Sigma $. Suppose $ \inf_{m \in K} \epsilon_{m} > \epsilon_1 > 0 $.

	\item ($ \{\Pi^{\kappa}_m\} $). There are constants $ L > 0, \widetilde{M} > 0 $ such that 
	\begin{enumerate}[(i)]
		\item $ \sup_{m \in K}|\Pi^{\kappa}_m| \leq \widetilde{M} $,
		\item $|\Pi^{\kappa}_{m_1} - \Pi^{\kappa}_{m_2}| \leq L|m_1 - m_2|$ for all $ m_1, m_2 \in U_{m,\gamma} (\epsilon_1) $, $\gamma \in \Lambda(m), m \in K $, $ \kappa = s, c, u $.
	\end{enumerate}

	\item (Almost uniformly differentiable at $ K $). $ \sup_{m \in K} \chi_{m}(\epsilon) \leq \chi(\epsilon) < 1 $ if $ 0 < \epsilon \leq \epsilon_1 $, where $ \chi(\cdot) $ is an increasing function.

	\item (Uniform size neighborhood at $ K $). There is a constant $ \delta_0 > 0 $ such that 
		\[
		\inf_{m \in K} \delta_0(m) > \delta_0.
		\]
\end{enumerate}
Usually, we say that $ ( \Sigma, K, \{ \Pi^c_m \} , \{U_{m,\gamma}(\epsilon)\} ) $ or $ \Sigma $ with $ K $ satisfies (H1)--(H4), and call $ \Sigma $ a \emph{uniformly $ C^{0,1} $ immersed submanifold} near $ K $; in addition, if $ \Sigma = K $, we often call $ \Sigma $ a \emph{uniformly $ C^{0,1} $ immersed submanifold} of $ X $.
We usually identify $ \widehat{\Sigma} $ with $ \Sigma $ when $ \Sigma $ is endowed with the immersed topology.

Some examples are the following:
\begin{enumerate}[$ \bullet $]
	\item Finite-dimensional cylinders (or $ \mathbb{T}^{n}, \mathbb{S}^n $) in $ X $;
	\item $ \Sigma $ is an open set of a complemented closed subspace of $ X $ and $ K \subset \Sigma $ with $ d(K, \partial \Sigma) > 0 $;
	\item Any $ C^1 $ compact embedding submanifold $ \Sigma $ of $ X $ and $ K \subset \Sigma $ with $ d(K, \partial \Sigma) > 0 $, where $ \partial \Sigma $ is the boundary of $ \Sigma $;
	\item $ \phi: \widehat{\Sigma} \to X $ is a $ C^1 $ immersion with $ 1 $-order tangency at self intersections (i.e., 
	\[
	D\phi(\widehat{m}) T_{\widehat{m}} \widehat{\Sigma} = D\phi(\widehat{m}') T_{\widehat{m}'} \widehat{\Sigma}
	\]
	if $ \phi(\widehat{m}) = \phi(\widehat{m}') $), where $ \widehat{\Sigma} $ is boundaryless and compact;
	\item Or more generally, $ \phi: \widehat{\Sigma} \to X $ is a $ C^1 $ \emph{leaf immersion} (see \cite[Section 6]{HPS77}) with $ \widehat{\Sigma} $ boundaryless.
	\item If $ K_i \subset \Sigma_i \subset X_i $ and $ \Sigma_i $ with $ K_i $ satisfies (H1)--(H4), $ i = 1,2 $, then $ \Sigma_1 \times \Sigma_2 $ with $ K_1 \times K_2 $ satisfies (H1)--(H4).
\end{enumerate}
See also \cite[Example 4.6]{Che18b} and \cite[Figure 2.1]{BLZ99}.

Since in the following results, the most useful information is near $ K $, without loss of generality, we take $ \widehat{K}_{\epsilon} $ and $ K_{\epsilon} $ for small $ \epsilon $ instead of $ \widehat{\Sigma} $ and $ \Sigma $, respectively.

\vspace{.5em}
\noindent{\emph{Convention}}.
Let $ \phi_1 : X^{h}_{\widehat{\Sigma}} \to X, (\widehat{m},x) \mapsto \phi(\widehat{m}) + x $; this map is usually not injective even when $ |x| $ is small, due to $ \Sigma $ being only immersed.
For brevity, for $ A_1, B_1 \subset X^{h}_{\widehat{\Sigma}} $ and a correspondence $ H: X \to X $,
\begin{enumerate}[$ \bullet $]
	\item for $ A_0 \subset X $ and $ B_1 \subset X^{h}_{\widehat{\Sigma}} $, $ A_0 \subset B_1 $ is understood as $ A_0 \subset \phi_1(B_1) $;
	\item $ A_1 \subset H^{-1}(B_1) $ means $ \phi_1(A_1) \subset H^{-1}(\phi_1(B_1)) $;
	\item for $ z = (\widehat{m}, x) \in X^{h}_{\widehat{\Sigma}} $, $ H(z) $ stands for $ H(\phi_1(z)) $.
\end{enumerate}
\vspace{.5em}

\begin{defi}[Almost uniform continuity]\label{def:almost}
	Let $ \mathfrak{M} $ be a metric space with metric $ d $.
	We say a function $ g : \Sigma \to \mathfrak{M} $ is \emph{$ \xi $-almost uniformly continuous} at $ K $ (in the immersed topology) if $ \mathfrak{B}_{K, g} \leq \xi $, where $ \mathfrak{B}_{K, g} $ is the \emph{amplitude} of $ g $ \emph{at} $ K \subset \Sigma $, defined by
	\begin{equation*}
	\mathfrak{B}_{K,g} \triangleq \lim_{\epsilon \to 0}\sup \{ d(g(m), g(m_0)): m \in U_{m_0,\gamma}(\epsilon) \cap K, m_0 \in K, \gamma \in \Lambda(m_0) \}.
	\end{equation*}
\end{defi}
Clearly, if $ g $ is $ 0 $-almost uniformly continuous at $ K $, then $ g $ is uniformly continuous at $ K $; but for $ \xi > 0 $, $ \xi $-almost uniform continuity does not even imply continuity. Note also that a function sufficiently close to a uniformly continuous function is almost uniformly continuous, but it may not be uniformly continuous (or even continuous).

(II) (\emph{Base map}). Assume $ u : K \to K \subset \Sigma $ is a $ C^0 $ map in the immersed topology, which means there exists a $ C^0 $ map $ \widehat{u}: \widehat{K} \to \widehat{K} \subset \widehat{\Sigma} $ such that $ \phi \circ \widehat{u} = u \circ \phi $.

(III) (\emph{Decomposition}). Let $ \widehat{\Pi}^\kappa_{m} $, $ \kappa = s, c, u $, be projections such that $ \widehat{\Pi}^s_{m} + \widehat{\Pi}^c_{m} + \widehat{\Pi}^u_{m} = \id $, $ m \in K $. Set $ \widehat{X}^{\kappa}_m = R(\widehat{\Pi}^{\kappa}_m) $ and
$ \widehat{X}^{\kappa_1 \kappa_2}_m = \widehat{X}^{\kappa_1}_m \oplus \widehat{X}^{\kappa_2}_m $, $ \kappa_1 \neq \kappa_2 $, $ \kappa, \kappa_1, \kappa_2 \in \{ s, c, u \} $.

(IV) (\emph{Correspondence}). Let $ H: X \to X $ be a correspondence.
Set $ \widehat{H}_{m} = H(\cdot + m) - u(m) $, i.e., $ \graph \widehat{H}_{m} = \graph H_{m} - (m, u(m)) \subset X \times X $ (see \autoref{sec:corr}).

For convenience, the following terminologies are used to describe partially normal hyperbolicity. 
If $\kappa_1, \kappa_2$ are one or more letters from $\{c, s, u\}$, then $\kappa_2 - \kappa_1$ denotes the characters remaining after removing from $\kappa_2$ those that are in $\kappa_1$; for example, $csu - cs$ is $u$, and $csu - s$ is $cu$.
\begin{defi}\label{defi:ABk}
	Let $ \kappa \in \{ cs, s \} $, $ \kappa_1 = csu - \kappa $.
	We say $ \widehat{H} \approx (\widehat{F}, \widehat{G}) $ satisfies the \emph{(A$ ' $) ($ \alpha_1, \lambda_1 $)} (resp. \emph{(A) ($ \alpha_1, \lambda_1 $)} or \emph{(A) ($ \alpha; \alpha_1, \lambda_1 $)}) \emph{condition in $ \kappa $-direction at $ K $} if there are constants $ r, r_1, r_2 > 0 $ such that for each $ m \in K $,
	\[
	\widehat{H}_{m} \sim (\widehat{F}_{m}, \widehat{G}_{m}): \widehat{X}^{\kappa}_{m}(r) \oplus \widehat{X}^{\kappa_1}_{m}(r_1) \to \widehat{X}^{\kappa}_{u(m)}(r_2) \oplus \widehat{X}^{\kappa_1}_{u(m)} (r),
	\]
	satisfies the (A$ ' $) $(\alpha_1(m)$, ${\lambda}_1(m))$ (resp. (A) $(\alpha_1(m)$, ${\lambda}_1(m))$ or (A) $(\alpha(m)$; $ \alpha_1(m) $, ${\lambda}_1(m))$) condition (see \autoref{defAB}).
	Similarly, \emph{(B$ ' $) ($ \beta_1, \lambda_2 $)} (resp. \emph{(B) ($ \beta_1, \lambda_2 $)} or \emph{(B) ($ \beta; \beta_1, \lambda_2 $)}) \emph{condition in $ \kappa $-direction at $ K $} can be defined.
\end{defi}

To clarify the dependence of constants on parameters, we adopt the following notations:
\begin{enumerate}[$ \bullet $]
	\item $ f = O_{\epsilon}(1) $ if $ f \to 0 $ as $ \epsilon \to 0 $;

	\item $ f = o(\epsilon) $ if $ f/\epsilon = O_{\epsilon}(1) $;

	\item $ f = O(\epsilon) $ if $ |f/\epsilon| < \infty $ as $ \epsilon \to 0 $;

	\item unless otherwise specified, the norm on $ X_1 \times X_2 \cong X_1 \oplus X_2 $ is taken as the maximum norm, i.e., $ |(x_1, x_2)| = \max \{ |x_1|, |x_2| \} $.
\end{enumerate}

\section{Main results I: dichotomy case} \label{subsec:main}

\begin{enumerate}[({A}1)]
	\item (Immersed submanifold) Let $ \Sigma $ with $ K $ satisfy (H1)--(H4).

	\item There is a small $ \xi_1 > 0 $ such that $ u: K \to X $ (in (II)) is $ \xi_1 $-almost uniformly continuous at $ K $ (in the immersed topology); see \autoref{def:almost}.

	\item (Approximately partially $ cs $-normal hyperbolicity) Assume the correspondence $ H: X \to X $ (from (IV)) satisfies the following conditions.

	(a) ((A$ ' $)(B) condition) Suppose $ \widehat{H} \approx (\widehat{F}^{cs}, \widehat{G}^{cs}) $ satisfies either ($ \bullet 1 $) the (A$ ' $)($ \alpha $, $ \lambda_{u} $) (B)($ \beta; \beta', \lambda_{cs} $) condition \emph{or} ($ \bullet 2 $) the (A)($ \alpha $, $ \lambda_{u} $) (B)($ \beta; \beta', \lambda_{cs} $) condition in $ cs $-direction at $ K $ (see \autoref{defi:ABk}). There exists a constant $ \varsigma_0 \geq 2 $ such that:

	(i) (Angle condition) 
	\[
	\sup_{m\in K} \alpha(m) \beta'(u(m)) < 1/(2\varsigma_0), \quad \inf_{m\in K} \{\beta(m) - \varsigma_0 \beta'(u(m))\} > 0;
	\]

	(ii) (Spectral condition) $ \sup_{m\in K} \lambda_{u}(m) \vartheta(m) < 1 $, where
	\begin{equation}\label{equ:v0}
	\begin{cases}
	\vartheta(m) = (1 - \varsigma_0\alpha(m) \beta'(u(m)))^{-1}, & \text{($ \bullet 1 $) case},  \\
	\vartheta(m) = 1, & \text{($ \bullet 2 $) case}.
	\end{cases}
	\end{equation}

	(iii) The functions in the (A$ ' $)(B) condition are \emph{bounded} and $ \xi_1 $-almost uniformly continuous at $ K $ (in the immersed topology); see \autoref{def:almost}.

	(a$ ' $) ((A$ ' $)(B) condition) Assume (a) (i)--(iii) hold with $ \varsigma_0 \geq 1 $; in addition,
	\begin{enumerate}[(iv$ ' $)]
		\item there is a small $ \gamma_0 > 0 $ such that for all $ m \in K $,
		\[
		|\widehat{G}^{cs}_{m}(x^{cs}, 0)| \leq \gamma_0|x^{cs}| + |\widehat{G}^{cs}_{m}(0, 0)|.
		\]
	\end{enumerate}

	(b) There are small $ \eta > 0 $ and $ \xi_2 > 0 $ such that
	\begin{enumerate}[(i)]
		\item (Approximation)
		\[
		\sup_{m \in K} |\widehat{F}^{cs}_m(0,0)| \leq \eta, \quad
		\sup_{m \in K} |\widehat{G}^{cs}_m(0,0)| \leq \eta,
		\]
		\item (Compatibility) $ \sup_{m \in K}|\widehat{\Pi}^\kappa_m - \Pi^\kappa_m| \leq \xi_2 $, for $ \kappa = s, c, u $.
	\end{enumerate}

\end{enumerate}

\begin{thmI}[Existence of center-stable manifold]\label{thm:I}
	There are $ 0 < c_{*} < 1 $ and small positive constants $ \xi_{1,*}, \xi_{2,*}, \epsilon_{*} $, with $ \chi_{0,*}, \gamma_{0,*} = O_{\epsilon_{*}}(1) $ and $ \eta_* = o(\epsilon_{*}) $, such that if either of the following cases holds:
	\begin{enumerate}[$ \bullet $]
		\item Case (1): (A1), (A2), (A3) (a) (b) are satisfied with the constants obeying $ \xi_{i} \leq \xi_{i,*} $ ($ i = 1,2 $), $ \chi(\epsilon) \leq \chi_{0,*} $, $ \epsilon \leq \epsilon_{*} $, and $ \eta \leq \eta_* $;

		\item Case (2): (A1), (A2), (A3) (a$ ' $) (b) are satisfied with the constants obeying $ \xi_{i} \leq \xi_{i,*} $ ($ i = 1,2 $), $ \chi(\epsilon) \leq \chi_{0,*} $, $ \epsilon \leq \epsilon_{*} $, $ \eta \leq \eta_* $, and $ \gamma_0 \leq \gamma_{0,*} $,
	\end{enumerate}
	then there exist $ \varepsilon = O(\epsilon_*) $ and $ \sigma, \varrho = o(\epsilon_*) $ with $ \varepsilon_0 = c_{*} \sigma $ such that the following statements hold.

	\begin{enumerate}[(1)]
		\item \label{censta} In $ X^s_{\widehat{\Sigma}}(\sigma) \oplus X^u_{\widehat{\Sigma}} (\varrho) $, there is a set $ W^{cs}_{loc}(K) $, called a \emph{local center-stable manifold} of $ K $ (which may not be unique), satisfying the following properties:
		\begin{enumerate}[(i)]
			\item \label{it:i11} (Partial characterization) Any $ (\varepsilon_{0}, \varepsilon, \varepsilon_{0}, \varrho) $-type forward orbit $ \{z_{k}\}_{k \geq 0} $ of $ H $ near $ K $ (see \autoref{defi:orbit}) lies entirely in $ W^{cs}_{loc}(K) $, i.e., $ \{z_{k}\}_{k \geq 0} \subset W^{cs}_{loc}(K) $.
			\item \label{it:i12} (Local invariance) Let $ \Omega = W^{cs}_{loc}(K) \cap \{X^s_{\widehat{K}_{\varepsilon_0}}(\varepsilon_0) \oplus X^u_{\widehat{K}_{\varepsilon_0}} (\varrho)\} $. Then 
			\[
			\Omega \subset H^{-1}(W^{cs}_{loc}(K)),
			\]
			and $ \Omega $ is an open subset of $ W^{cs}_{loc}(K) $; moreover, if $ \eta = 0 $, then $ K \subset \Omega $.
			\item (Representation) \label{it:h0} The manifold $ W^{cs}_{loc}(K) $ admits a representation as the graph of a Lipschitz map near $ K $. That is, there is a map $ h_0 $ such that $ h_0(\widehat{m}, \cdot): X^s_{\phi(\widehat{m})} (\sigma) \to X^u_{\phi(\widehat{m})}(\varrho) $ for $ \widehat{m} \in \widehat{\Sigma} $, and
			\[
			W^{cs}_{loc}(K) = \graph h_0 \triangleq \{ (\widehat{m}, x^s, h_0(\widehat{m}, x^s)): x^s \in X^{s}_{\phi(\widehat{m})}(\sigma), \widehat{m} \in \widehat{\Sigma} \}.
			\]
			In addition, there is a function $ \mu(\cdot) $ with $ \mu(m) = (1+\chi_{*}) \beta'(m) + \chi_{*} $, where $ \chi_{*} = O_{\epsilon_{*}}(1) $,
			such that for every $ m \in K $, 
			\begin{equation}\label{equ:minlip}
			|\Pi^u_{m}( h_0(\widehat{m}_1, x^s_1) - h_0(\widehat{m}_2, x^s_2)  )| \leq \mu(m) \max \{| \Pi^c_{m}(  \phi(\widehat{m}_1) - \phi(\widehat{m}_2)  ) |, | \Pi^s_{m}( x^s_1 - x^s_2 ) |\},
			\end{equation}
			where $ \widehat{m}_i \in \widehat{U}_{\widehat{m}}(\varepsilon_0) $, $ x^{s}_i \in X^s_{\phi(\widehat{m}_i)} (\varepsilon_0) $ for $ i = 1,2 $, and $ \widehat{m} \in \phi^{-1}(m) $.

		\end{enumerate}

		\item \label{map} The correspondence $ H $ induces a map on $ \Omega $: for any $ z_0 = (\widehat{m}, x^s, x^u) \in \Omega $, there is a unique $ z_1 = (\widehat{m}_1, x^s_1, x^u_1) \triangleq \phi(\widehat{m}_1) + x^s_1 + x^u_1 \in W^{cs}_{loc}(K) $ such that $ z_1 \in H(z_0) \triangleq H(\phi(\widehat{m}) + x^s + x^u) $, where $ \widehat{m}_1 \in \widehat{U}_{\widehat{u}(\widehat{m}_0)} (\varepsilon) $, $ x^\kappa_i \in X^{\kappa}_{\phi(\widehat{m}_i)} $ for $ \widehat{m}_i \in \widehat{\Sigma} $, $ i = 0,1 $, $ \kappa = s, u $, and $ \widehat{m}_0 \in \widehat{K} $ is chosen such that $ \widehat{m} \in \widehat{U}_{\widehat{m}_0} (\varepsilon_0) $.
	\end{enumerate}
\end{thmI}

In some cases, the manifold $ W^{cs}(K) $ can be uniquely characterized; see \autoref{rmk:inflowing}. To establish the smoothness results, we introduce the following assumptions.

\begin{enumerate}[(A4)]
	\item (Smoothness conditions) 
	\begin{enumerate}[(i)]
		\item For every $ m \in K $, assume that $ \widehat{F}^{cs}_m(\cdot) $ and $ \widehat{G}^{cs}_m(\cdot) $ are $ C^1 $.
		\item (Spectral gap condition) $ \sup_{m\in K} \lambda_{cs}(m) \lambda_u(m) \vartheta(m) < 1 $.
		\item (Smoothness of the ambient space) Assume that $ \Sigma \in C^1 $ and $ m \mapsto \Pi^{\kappa}_{m} $ is $ C^1 $ (in the immersed topology) for $ \kappa = s, c, u $.
		\item (Smoothness of $ cs $-spaces) Assume there exists a \emph{$ C^{1} $ and $ C_1 $-Lipschitz bump function} in $ X^s_{\widehat{\Sigma}} $ near $ K $, where $ C_1 \geq 1 $ (see \autoref{def:C1Lbump} and \autoref{cor:spaces} for examples of spaces admitting such bump functions: e.g., (a) $ X $ is a Hilbert space; (b) $ (X^{cs}_{m})^* $, $ m \in \Sigma $, are separable; in particular, (b1) $ X^* $ is separable or (b2) $ (X^{cs}_{m})^* $, $ m \in \Sigma $, are finite-dimensional; in case (b2) we can take $ C_1 = 1 $, and in case (b1) $ C_1 = 3 $).
	\end{enumerate}
\end{enumerate}

\begin{thmI}[Smoothness of center-stable manifold]\label{thm:smooth}
	Assume the hypotheses of \autoref{thm:I} hold, with either case (1) and $ \varsigma_0 \geq C_1 + 1 $ (in (A3) (a)), or case (2) with $ \gamma_0 $ sufficiently small (in (A3) (a$ ' $) (iv$ ' $)). If, in addition, (A4) is satisfied, then the center-stable manifold $ W^{cs}_{loc}(K) $ in \autoref{thm:I} can be chosen to be a $ C^1 $ immersed submanifold of $ X $.

	If only (A4) (i) (ii) are assumed, and further $ \eta = 0 $ with $ \widehat{X}^{cs} $ invariant under $ DH $ (i.e., $ D_{x^{cs}}\widehat{G}^{cs}_{m}( 0, 0, 0 ) = 0 $ for all $ m \in K $), then $ T_{\widehat{m}} W^{cs}_{loc}(K) = \widehat{X}^{cs}_{\phi(\widehat{m})} $ (see \autoref{def:tangent}) for all $ \widehat{m} \in \widehat{K} $.
\end{thmI}

We note that, to obtain regularity results, the constants $ \xi_1, \xi_2, \eta $, and $ \chi(\epsilon) $ (together with $ \gamma_0 $ in case (2)) may need to be smaller than those required for the existence results; moreover, the constants $ \varepsilon_0, \varepsilon, \sigma, \varrho $ may also be smaller (i.e., $ \epsilon_{*} $ may be further reduced).

\begin{rmk}
	\begin{enumerate}[(a)]
		\item In \autoref{thm:I}, if $ H $ additionally satisfies the \emph{strong $ s $-contraction} condition (see assumption ($ \star\star $) in \autoref{sub:limited}), then $ (\varepsilon_{0}, \varepsilon, \varepsilon_{0}, \varrho) $-type (in item \eqref{it:i11}) can be taken as $ (\varepsilon_{0}, \varepsilon, \sigma, \varrho) $-type, and $ \Omega = W^{cs}_{loc}(K) \cap \{X^s_{\widehat{K}_{\varepsilon_0}}(\sigma) \oplus X^u_{\widehat{K}_{\varepsilon_0}} (\varrho)\} $ (in item \eqref{it:i12}). Moreover, the Lipschitz estimate \eqref{equ:minlip} in item \eqref{it:h0} remains valid for $ x^{s}_i \in X^s_{\phi(\widehat{m}_i)} (\sigma) $.

		\item The smoothness result in \autoref{thm:smooth} continues to hold when (A4)(iv) is replaced by the following weaker condition:
		\begin{enumerate}
			\item[(iv$ ' $)] Assume there exists a \emph{$ C^{1} $ and $ C_1 $-Lipschitz bump function} in $ \widehat{\Sigma} $ near $ K $ with $ C_1 \geq 1 $ (see \autoref{def:C1Lbump} and \autoref{exa:case2}) and that $ H $ satisfies the \emph{strong $ s $-contraction} (see assumption ($ \star\star $) in \autoref{sub:limited}).
		\end{enumerate}

		\item The $ C^1 $ smoothness condition on the projection mappings $ m \mapsto \Pi^{\kappa}_{m} $ ($ \kappa = s, c, u $) in (A4)(iii) is sometimes redundant. In many cases, one can construct a $ C^{1} \cap C^{0,1} $ approximation of these maps; see, e.g., \autoref{cor:appLip}.
	\end{enumerate}
\end{rmk}

For higher order smoothness of $ W^{cs}_{loc}(K) $, see \autoref{lem:C1a}, \autoref{lem:udiff} and \autoref{lem:high}.

\section{Main results II: invariant case} \label{sub:invariant}

The assumption (A3) (b) (i) in \autoref{thm:I} implies that $ K $ is an ``\emph{approximately invariant set}'' of $ H $. We now examine the special case of exact invariance, i.e, $ \eta = 0 $.

\begin{thmI}\label{thm:invariant}
	Assume all conditions in \autoref{thm:I} hold, but with the following modifications: $ K = \Sigma $, $ \eta = 0 $, and the spectral condition in (A3) (a) (ii) is replaced by the stronger gap condition in (A4)(ii) (where $ \vartheta(\cdot) $ is defined as in \eqref{equ:v0}).
	
	Then there exist constants $ \varepsilon = O(\epsilon_*) $ and $ \varepsilon_0, \sigma, \varrho = o(\epsilon_*) $ with $ \varepsilon_0 < \sigma $, and a set $ W^{cs}_{loc}(\Sigma) \subset X^s_{\widehat{\Sigma}}(\sigma) \oplus X^u_{\widehat{\Sigma}} (\varrho) $, called a \emph{local center-stable manifold} of $ \Sigma $, which satisfies the graph representation in \autoref{thm:I} \eqref{it:h0} and contains $ \Sigma $ (i.e., $ \Sigma \subset W^{cs}_{loc}(\Sigma) $). Moreover, the following statements hold:
		\begin{enumerate}[(1)]
			\item (Local invariance) The manifold $ W^{cs}_{loc}(\Sigma) $ is locally invariant under $ H $, i.e., for $ \Omega = W^{cs}_{loc}(\Sigma) \cap \{X^s_{\widehat{\Sigma}}(\varepsilon_0) \oplus X^u_{\widehat{\Sigma}} (\varrho)\} $, we have $ \Omega \subset H^{-1}(W^{cs}_{loc}(\Sigma)) $, and $ H: \Omega \to W^{cs}_{loc}(\Sigma) $ induces a map, analogous to the construction in \autoref{thm:I} \eqref{map}.

			\item \label{it:invPartial} (Partial characterization) Let $ \{z_{k} = (\widehat{m}_{k}, x^{s}_{k}, x^{u}_{k})\}_{k \geq 0} \subset X^s_{\widehat{\Sigma}}(\varepsilon_0) \oplus X^u_{\widehat{\Sigma}} (\varrho) $ be a forward orbit satisfying:
			\begin{enumerate}[$ \bullet $]
				\item $ z_{k} \in H(z_{k-1}) $ with $ \widehat{m}_{k} \in \widehat{U}_{\widehat{u}(\widehat{m}_{k-1})} (\varepsilon) $ for all $ k \geq 0 $; and
				\item either $ |x^{u}_{k}| \leq \widetilde{\beta}_{0}(\widehat{m}_{k-1}) |x^{s}_{k}| $ for all $ k \geq 0 $, or 
				\[
				\sup_{k}\{\varepsilon_{s}(\widehat{m}_0) \varepsilon_{s}(\widehat{m}_1) \cdots \varepsilon_{s}(\widehat{m}_{k-1})\}^{-1} (|x^{s}_{k}| + |x^{u}_{k}|) < \infty.
				\]
			\end{enumerate}
			Then $ \{z_{k}\}_{k \geq 0} \subset W^{cs}_{loc}(\Sigma) $ and $ |x^{s}_{k+1}| \leq (\lambda_{cs}(\phi(\widehat{m}_{k})) + \hat{\chi}) |x^{s}_{k}| $. Here, the functions $ \widetilde{\beta}_{0}(\cdot), \varepsilon_{s}(\cdot): \widehat{\Sigma} \to \mathbb{R}_+ $ satisfy, for all $ \widehat{m} \in \widehat{\Sigma} $,
			\[
			\beta'(u(\phi(\widehat{m}))) + \hat{\chi} < \widetilde{\beta}_{0}(\widehat{m}) < \beta(\phi(\widehat{m})) - \hat{\chi}, \quad
			0 < \varepsilon_{s}(\widehat{m}) < \lambda^{-1}_{u} (\phi(\widehat{m})) \vartheta(\phi(\widehat{m})) - \hat{\chi},
			\]
			where $ \hat{\chi} > 0 $ is a small constant depending on $ \xi_1 $ and $ \epsilon_{*} $.

			\item (Differentiability) If (A4)(i) holds, then $ W^{cs}_{loc}(\Sigma) $ is differentiable at every $ \widehat{m} \in \widehat{\Sigma} $ in the sense of Whitney  (see \autoref{def:tangent}) with the tangent map $ \widehat{\Sigma} \to \mathbb{G}(X): \widehat{m} \mapsto T_{\widehat{m}} W^{cs}_{loc}(\Sigma) $ continuous. 
			
			Furthermore, if $ \widehat{X}^{cs} $ is invariant under $ DH $ (i.e., $ D_{x^{cs}}\widehat{G}^{cs}_{m}( 0, 0, 0 ) = 0 $ for all $ m \in \Sigma $), then $ T_{\widehat{m}} W^{cs}_{loc}(\Sigma) = \widehat{X}^{cs}_{\phi(\widehat{m})} $ (see \autoref{def:tangent}) for all $ \widehat{m} \in \widehat{\Sigma} $.

			\item (Smoothness) Under the same conditions as in \autoref{thm:smooth}, the local center-stable manifold $ W^{cs}_{loc}(\Sigma) $ can be chosen to be a $ C^1 $ immersed submanifold of $ X $.
		\end{enumerate}
\end{thmI}

For a discussion of (A4)(iv) in the case $ \Sigma = K $, see also \autoref{exa:case1}.

When $ \Sigma = \{ m_0 \} $, \autoref{thm:invariant} essentially generalizes the classical theory of \emph{local} invariant manifolds near an equilibrium. Depending on $ \lambda_{cs} $ and $ \lambda_{u} $ (whether they are greater than or less than $ 1 $), this construction yields center-stable or pseudo-stable (weak-stable) manifolds. We do not know whether a similar local result holds when $ K $ is a proper subset of $ \Sigma $.

\section{Main results III: trichotomy case} \label{sub:tri}

\begin{enumerate}[(B1)]
	\item (Immersed submanifold) Assume $ \Sigma $ with $ K $ satisfies (H1)--(H4).

	\item (Base map) Assume $ u $ and $ \widehat{u} $ (in (II)) are invertible. In addition, there is a small $ \xi_1 > 0 $ such that $ u $ and $ u^{-1}: K \to X $ are $ \xi_1 $-almost uniformly continuous at $ K $ (in the immersed topology); see \autoref{def:almost}.

	\item (Approximately partially normal hyperbolicity) Assume the correspondence $ H: X \to X $ (in (IV)) satisfies the following conditions.

	(a) ((A)(B) condition) Let $ \kappa_1 = cs $, $ \kappa_2 = u $, $ \kappa = cs $, or $ \kappa_1 = s $, $ \kappa_2 = cu $, $ \kappa = cu $. Suppose $ \widehat{H} \approx (\widehat{F}^{\kappa}, \widehat{G}^{\kappa}) $ satisfies the (A)($ {\alpha}_{\kappa_2} $; ${\alpha}_{\kappa_2}' $, $ {\lambda}_{\kappa_2} $) (B)($ {\beta}_{\kappa_1} $; $ {\beta}_{\kappa_1}' $, $ {\lambda}_{\kappa_1} $) condition in $ \kappa_1 $-direction at $ K $ (see \autoref{defi:ABk}). Moreover, there exists a constant $ \varsigma_0 \geq 2 $ such that

	(i) (Angle condition) 
	\begin{gather*}
	\sup_m \alpha'_{\kappa_2}(m) \beta'_{\kappa_1}(u(m)) < 1/(2\varsigma_0), \quad \inf_m\{ \alpha_{cu}(u(m)) - \varsigma_0\alpha'_{cu}(m) \} > 0 \\
	\inf_m\{ \beta_{cs}(m) - \varsigma_0\beta'_{cs}(u(m)) \} > 0, \quad \sup_m \alpha_{cu}(m) \beta_{cs}(m) < 1;
	\end{gather*}

	(ii) (Spectral condition) $ \sup_m \lambda_{s}(m) < 1 $ and $ \sup_m \lambda_{u}(m) < 1 $;

	(iii) the functions in the (A)(B) condition are \emph{bounded} and $ \xi_1 $-almost uniformly continuous at $ K $ (in the immersed topology); see \autoref{def:almost}.

	(a$ ' $) ((A)(B) condition) Assume (a) (i)--(iii) hold with $ \varsigma_0 \geq 1 $; in addition,

	(iv$ ' $) there is a small $ \gamma_0 > 0 $ such that for all $ m \in K $,
	\[
	|\widehat{G}^{cs}_{m}(x^{cs}, 0)| \leq \gamma_0|x^{cs}| + |\widehat{G}^{cs}_{m}(0, 0)|, \quad |\widehat{F}^{cu}_{m}(0, x^{cu})| \leq \gamma_0|x^{cu}| + |\widehat{F}^{cu}_{m}(0, 0)|.
	\]

	(b) There are small $ \eta > 0 $ and $ \xi_2 > 0 $ such that
	\begin{enumerate}[(i)]
		\item (Approximation) For $ \kappa = cs, cu $,
		\begin{gather*}
			\sup_{m \in K} |\widehat{F}^{\kappa}_m(0,0)| \leq \eta, \quad
			\sup_{m \in K} |\widehat{G}^{\kappa}_m(0,0)| \leq \eta,
		\end{gather*}
		\item (Compatibility) $ \sup_{m \in K}|\widehat{\Pi}^\kappa_m - \Pi^\kappa_m| \leq \xi_2 $ for $ \kappa = s, c, u $.
	\end{enumerate}

	\item (Smoothness conditions) 
	\begin{enumerate}[(i)]
		\item For every $ m \in K $, assume that $ \widehat{F}^{\kappa}_m(\cdot) $ and $ \widehat{G}^{\kappa}_m(\cdot) $ are $ C^1 $, where $ \kappa = cs, cu $.

		\item (Spectral gap condition) 
		\[
		\sup_{m \in K} \lambda_{cs}(m) \lambda_{u}(m) < 1 \quad \text{and} \quad \sup_{m \in K} \lambda_{cu}(m) \lambda_{s}(m) < 1.
		\]

		\item (Smoothness of the ambient space) Assume $ \Sigma \in C^1 $ and $ m \mapsto \Pi^{\kappa}_{m} $ is $ C^1 $ (in the immersed topology) for $ \kappa = s, c, u $.

		\item (Smoothness of $ cs, cu $-spaces) Assume there exist \emph{$ C^{1} $ and $ C_1 $-Lipschitz bump functions} in $ X^s_{\widehat{\Sigma}} $ and in $ X^u_{\widehat{\Sigma}} $ near $ K $, respectively, where $ C_1 \geq 1 $ (see \autoref{def:C1Lbump} and \autoref{cor:spaces}).

		\item (Smoothness of $ c $-spaces) There exists a $ C^{1} $ and $ C_1 $-Lipschitz bump function in $ \widehat{\Sigma} $ near $ K $ (see \autoref{def:C1Lbump} and \autoref{exa:case2}).
	\end{enumerate}
\end{enumerate}

\begin{thmI}[Trichotomy case]\label{thm:tri0}
	There are a constant $ 0 < c_{*} < 1 $ and small positive constants $ \xi_{1,*}, \xi_{2,*}, \epsilon_{*} $, with $ \chi_{0,*}, \gamma_{0,*} = O_{\epsilon_{*}}(1) $ and $ \eta_* = o(\epsilon_{*}) $, such that if either of the following cases holds:
	\begin{enumerate}[$ \bullet $]
		\item Case (1): (B1), (B2), (B3) (a) (b) are satisfied with the constants obeying $ \xi_{i} \leq \xi_{i,*} $ ($ i = 1,2 $), $ \chi(\epsilon) \leq \chi_{0,*} $, $ \epsilon \leq \epsilon_{*} $, and $ \eta \leq \eta_* $;

		\item Case (2): (B1), (B2), (B3) (a$ '$) (b) are satisfied with the constants obeying $ \xi_{i} \leq \xi_{i,*} $ ($ i = 1,2 $), $ \chi(\epsilon) \leq \chi_{0,*} $, $ \epsilon \leq \epsilon_{*} $, $ \eta \leq \eta_* $, and $ \gamma_0 \leq \gamma_{0,*} $,
	\end{enumerate}
	then there exist $ \varepsilon = O(\epsilon_*) $ and $ \sigma, \varrho = o(\epsilon_*) $ with $ \varepsilon_0 = c_{*}\sigma $ such that the following statements hold.

	\begin{enumerate}[(1)]
		\item \label{it:tri1} (Existence of invariant manifolds) There are $ W^{cs}_{loc}(K) \subset X^s_{\widehat{\Sigma}}(\sigma) \oplus X^u_{\widehat{\Sigma}} (\varrho) $, $ W^{cu}_{loc}(K) \subset X^s_{\widehat{\Sigma}}(\varrho) \oplus X^u_{\widehat{\Sigma}} (\sigma) $, and $ \Sigma^c = W^{cs}_{loc}(K) \cap W^{cu}_{loc}(K) $, called a \emph{local center-stable manifold}, a \emph{local center-unstable manifold}, and a \emph{local center manifold} of $ K $, respectively. These manifolds satisfy the following properties:
		\begin{enumerate}[(i)]
			\item \label{it:tri011} (Partial characterization) Any $ (\varepsilon_{0}, \varepsilon, \varepsilon_{0}, \varrho) $-type forward (resp. backward) orbit $ \{z_{k}\}_{k \geq 0} $ (resp. $ \{z_{k}\}_{k \leq 0} $) of $ H $ near $ K $ (see \autoref{defi:orbit}) lies entirely in $ W^{cs}_{loc}(K) $ (resp. $ W^{cu}_{loc}(K) $). Moreover, if $ \{z_{k}\}_{k \in \mathbb{Z}} $ is a $ (\varepsilon_{0}, \varepsilon, \varepsilon_{0}, \varrho) $-type orbit of $ H $ near $ K $, then $ \{z_{k}\}_{k \in \mathbb{Z}} \subset \Sigma^c $.

			\item (Local invariance) The manifolds $ W^{cs}_{loc}(K) $, $ W^{cu}_{loc}(K) $, and $ \Sigma^c $ are locally positively invariant, locally negatively invariant, and locally invariant under $ H $, respectively. That is, for
			\[
			\Omega_{c\kappa} = W^{c\kappa}_{loc}(K) \cap \{X^{\kappa}_{\widehat{K}_{\varepsilon_0}}(\varepsilon_0) \oplus X^{su - \kappa}_{\widehat{K}_{\varepsilon_0}} (\varrho)\}, \quad \kappa = s, u,
			\]
			we have $ \Omega_{cs} \subset H^{-1}(W^{cs}_{loc}(K)) $ and $ \Omega_{cu} \subset H(W^{cu}_{loc}(K)) $. In particular, $ \Omega_{c} = \Omega_{cs} \cap \Omega_{cu} $ is an open subset of $ \Sigma^c $ and satisfies $ \Omega_{c} \subset H^{\pm 1} (\Sigma^c) $.

			\item (Graph representation) The manifolds $ W^{cs}_{loc}(K) $, $ W^{cu}_{loc}(K) $, and $ \Sigma^c $ admit representations as graphs of Lipschitz maps near $ K $. That is, there are maps $ h^{\kappa}_0 $, $ \kappa = cs, cu, c $, such that for $ \widehat{m} \in \widehat{\Sigma} $,
			\begin{gather*}
			h^{cs}_0(\widehat{m}, \cdot): X^s_{\phi(\widehat{m})} (\sigma) \to X^u_{\phi(\widehat{m})}(\varrho), \\
			h^{cu}_0(\widehat{m}, \cdot): X^u_{\phi(\widehat{m})} (\sigma) \to X^s_{\phi(\widehat{m})}(\varrho), \\
			h^{c}_0 (\widehat{m}) \in X^s_{\phi(\widehat{m})}(\sigma) \oplus X^u_{\phi(\widehat{m})}(\sigma),
			\end{gather*}
			with $ W^{cs}_{loc}(K) = \graph h^{cs}_0 $, $ W^{cu}_{loc}(K) = \graph h^{cu}_0 $, and $ \Sigma^{c} = \graph h^{c}_0 $. Furthermore, there are functions $ \mu_{\kappa}(\cdot) $, $ \kappa = cs, cu, c $, satisfying 
			\[
			\mu_{cs}(m) = (1+\chi_{*}) \beta'_{cs}(m) + \chi_{*}, \mu_{cu}(m) = (1+\chi_{*}) \alpha'_{cu}(m) + \chi_{*}, \mu_{c} = \max\{ \mu_{cs}, \mu_{cu} \},
			\]
			where $ \chi_{*} = O_{\epsilon_{*}}(1) $. For every $ m \in K $ and $ \kappa = s, u $, the following Lipschitz estimates hold:
			\begin{multline*}
				|\Pi^{su-\kappa}_{m}( h^{c\kappa}_0(\widehat{m}_1, x^\kappa_1) - h^{c\kappa}_0(\widehat{m}_2, x^\kappa_2)  )| \\
				\leq \mu_{c\kappa}(m) \max \{| \Pi^c_{m}(  \phi(\widehat{m}_1) - \phi(\widehat{m}_2)  ) |, | \Pi^\kappa_{m}( x^\kappa_1 - x^\kappa_2 ) |\},
			\end{multline*}
			\begin{multline*}
			\max\{|\Pi^u_{m}( h^{c}_0(\widehat{m}_1) - h^{c}_0(\widehat{m}_2)  )|, |\Pi^s_{m}( h^{c}_0(\widehat{m}_1) - h^{c}_0(\widehat{m}_2) )|\} \\
			 \leq \mu_{c}(m) | \Pi^c_{m}(  \phi(\widehat{m}_1) - \phi(\widehat{m}_2)  ) |,
			\end{multline*}
			where $ \widehat{m}_i \in \widehat{U}_{\widehat{m}}(\varepsilon_0) $, $ x^{\kappa}_i \in X^\kappa_{\phi(\widehat{m}_i)} (\varepsilon_0) $ for $ i = 1,2 $, and $ \widehat{m} \in \phi^{-1}(m) $.

			\item (Strong contraction/expansion case) If, in addition, $ H $ satisfies the \emph{strong $ s $-contraction} and \emph{strong $ u $-expansion} (see assumption ($ \star\star $) in \autoref{sub:limited}; e.g., when $ \sup_{m} \alpha'_{cu} (m) $ and $ \sup_{m} \beta'_{cs}(m) $ are sufficiently small), then $ (\varepsilon_{0}, \varepsilon, \varepsilon_{0}, \varrho) $-type can be $ (\varepsilon_{0}, \varepsilon, \sigma, \varrho) $-type, and $ \Omega_{c\kappa} = W^{c\kappa}_{loc}(K) \cap \{X^{\kappa}_{\widehat{K}_{\varepsilon_0}}(\sigma) \oplus X^{su - \kappa}_{\widehat{K}_{\varepsilon_0}} (\varrho)\} $ for $ \kappa = s, u $.
		\end{enumerate}

		\item \label{it:Tmaps} (Induced dynamics) The correspondence $ H $ induces  maps: a map on $ \Omega_{cs} $, an invertible map on $ \Omega_{c} $, and a map via $ H^{-1} $ on $ \Omega_{cu} $. (See \autoref{thm:I} \eqref{map} for the meaning.)

		\item \label{it:tri0smooth} (Smoothness) Suppose one of the following conditions holds:
		\begin{enumerate}[(i)]
			\item Under case (1) with $ \varsigma_0 \geq C_1 + 1 $ (in (B3) (a)) or case (2) with $ \varsigma_0 \geq 1 $ and $ \gamma_0 $ (in (B3) (a$ '$)) further reduced, (B4) (i)--(iv) are satisfied;

			\item Under case (1) with $ \varsigma_0 \geq C_1 + 1 $ (in (B3) (a)) or case (2) with $ \varsigma_0 \geq 1 $ and $ \gamma_0 $ (in (B3) (a$ '$)) further reduced, (B4) (i)--(iii) and (iv$ '$) hold, and $ H $ satisfies the \emph{strong $ s $-contraction} and \emph{strong $ u $-expansion} (see assumption ($ \star\star $) in \autoref{sub:limited}; e.g., when $ \sup_{m} \alpha'_{cu} (m) $ and $ \sup_{m} \beta'_{cs}(m) $ are sufficiently small).
		\end{enumerate}
		Then  $ W^{cs}_{loc}(K) $, $ W^{cu}_{loc}(K) $ and $ \Sigma^c $ can be chosen to be $ C^1 $ immersed submanifolds of $ X $.

		\item \label{it:tri04} (Tangent space characterization) If (B4) (i) (ii) are assumed, and further $ \eta = 0 $ with
		\begin{equation*}
		D_{x^{cs}}\widehat{G}^{cs}_{m}( 0, 0, 0 ) = 0, \quad D_{x^{cu}}\widehat{F}^{cu}_{m}( 0, 0, 0 ) = 0, \quad \forall m \in K,
		\end{equation*}
		which implies that $ \widehat{X}^{cs} $ and $ \widehat{X}^{cu} $ are invariant under $ DH $, then $ K \subset \Sigma^c $ and the tangent spaces satisfy
		\begin{equation}\label{equ:tangent}
		T_{\widehat{m}} W^{cs}_{loc}(K) = \widehat{X}^{cs}_{\phi(\widehat{m})}, \quad T_{\widehat{m}} W^{cu}_{loc}(K) = \widehat{X}^{cu}_{\phi(\widehat{m})}, \quad T_{\widehat{m}}\Sigma^c = \widehat{X}^{c}_{\phi(\widehat{m})}, \quad \forall \widehat{m} \in \widehat{K},
		\end{equation}
		(see \autoref{def:tangent} for the meaning).
	\end{enumerate}
\end{thmI}

\begin{rmk}\label{rmk:mapmeaning}
	\begin{enumerate}[(a)]
		\item \autoref{thm:tri0} demonstrates that the set $ \Sigma $, which is ``approximately invariant'' and ``approximately partially normally hyperbolic at $ K $ with respect to $ H $'', can be replaced by a genuine locally invariant $ C^{0,1} $ immersed manifold $ \Sigma^c $ that is ``normally hyperbolic'' with respect to $ H $ (in the sense of \cite{Che18b}). Furthermore, under the conditions of \autoref{thm:tri0} \eqref{it:tri04}, the manifold $ \Sigma^c $ is differentiable at $ K $ in the sense of Whitney.

		\item From \autoref{thm:tri0} \eqref{it:tri011}, we observe that if $ \eta = 0 $, then the \emph{strong-stable lamination} (respectively, the \emph{strong-unstable lamination}) of $ K $ for $ H $ is contained in $ W^{cs}_{loc}(K) $ (respectively, $ W^{cu}_{loc}(K) $). For a comprehensive discussion on the existence and regularity of strong (un)stable laminations, we refer to \cite{HPS77} or \cite[Section 7.2]{Che18a}, where general results applicable to this context can be found.

		\item \label{it:maps} In \autoref{thm:tri0} \eqref{it:Tmaps}, we define $ \mathcal{H}^{cs}_{0} = H|_{\Omega_{cs}}: \Omega_{cs} \to W^{cs}_{loc}(K) $ and $ \mathcal{H}^{-cu}_{0} = H^{-1}|_{\Omega_{cu}} : \Omega_{cu} \to W^{cu}_{loc}(K) $. To clarify the invertibility of $ \mathcal{H}^{c}_{0} $, we note that $ \mathcal{H}^{c}_{0} = \mathcal{H}^{cs}_{0}|_{\Omega_{c}} : \Omega_{c} \to \Sigma^c $ and $ \mathcal{H}^{-c}_{0} = \mathcal{H}^{-cu}_{0}|_{\Omega_{c}} : \Omega_{c} \to \Sigma^c $. The composition $ \mathcal{H}^{-c}_{0} \circ \mathcal{H}^{c}_{0} (x) = \mathcal{H}^{c}_{0} \circ \mathcal{H}^{-c}_{0} (x) $ holds whenever $ \mathcal{H}^{c}_{0} (x) \in \Omega_{c} $ and $ \mathcal{H}^{-c}_{0} (x) \in \Omega_{c} $, thereby establishing the meaning of $ (\mathcal{H}^{c}_{0})^{-1} = \mathcal{H}^{-c}_{0} $.

		\item Results analogous to those in \autoref{thm:invariant} also hold for the trichotomy case. For details, see \autoref{app:im}.

	\end{enumerate}
\end{rmk}

\section{Corollaries I: approximately invariant sets} \label{sub:sets}

A special case where (A1) in \autoref{subsec:main} holds is the following.

\begin{enumerate}[({Aa}1)]
\item Let $ K \subset X $ be \emph{compact} and $ \{ \Pi^{\kappa}_{m}: m \in K, \kappa = s, c, u \} $ a family of projections such that $ \Pi^{s}_{m} + \Pi^{c}_{m} + \Pi^{u}_{m} = \id $, $ m \in K $. Set $ \Pi^{h}_{m} = \id - \Pi^{c}_m $ and $ X^{\kappa}_{m} \triangleq R(\Pi^{\kappa}_{m}) $, $ \kappa = c, s, u, h $, also $ X^{\kappa_1\kappa_2}_{m} = X^{\kappa_1}_{m} \oplus X^{\kappa_1}_{m} $, where $ \kappa_1 \neq \kappa_2 \in \{ s,c, u \} $. Suppose

(i) $ m \mapsto \Pi^{\kappa}_{m} $ is continuous, $ \kappa = c, s, u $, and $ T_{m} K \subset X^{c}_{m} $ (see \autoref{def:tangent}), $ m \in K $; and

(ii) for each $ m \in K $, the pair $ (X^c_{m}, X^{h}_{m}) $ has property (*) (see \autoref{def:p*} and \autoref{exa:p*}; for examples (a) $ X $ is a Hilbert space with $ X^c_{m} $ separable for each $ m \in K $, or (b) $ X^c_{m} $ is finite-dimensional for each $ m \in K $).
\end{enumerate}

\begin{thm}\label{thm:geo}
If (Aa1) holds, then for any small $ \varepsilon > 0 $, there are a $ C^1 $ submanifold $ \Sigma $ of $ X $ and projections $ \{ \widetilde{\Pi}^{\kappa}_{m}: m \in \Sigma, \kappa = s, c, u \} $ such that
\begin{enumerate}[(i)]
\item $ K \subset \Sigma $ and $ T_{m} \Sigma = X^{c}_{m} $, $ m \in K $;

\item $ \sup_{m \in K}|\widetilde{\Pi}^{\kappa}_{m} - \Pi^{\kappa}_{m}| \leq \varepsilon $, and $ \Sigma \to L(X): m \mapsto \widetilde{\Pi}^{\kappa}_{m} $ is $ C^1 $, $ \kappa = s, c, u $;

\item $ (\Sigma, K, \{ \widetilde{\Pi}^c_{m} \}, \{ \Sigma \cap \mathbb{B}_{\epsilon}(m) \}) $ satisfies (H1)--(H4) in \autoref{sub:setup}.
\end{enumerate}
\end{thm}
For a proof of \autoref{thm:geo}, see \autoref{sub:geo}.
In particular, if, in addition, (A2), (A3) (a$ ' $) (b) hold, then \autoref{thm:I} holds (for case (2)), i.e., there is a local center-stable manifold of $ K $, which is a $ C^{0,1} $ submanifold of $ X $ and is locally modeled on $ X^{cs}_{m} $, $ m \in K $. For the regularity result, if we further assume (A4) (i) (ii) hold and $ X^s_{m} = 0 $ for $ m \in K $, then the local center-stable manifold of $ K $ can be taken as $ C^1 $.

Now we give a detailed statement of the trichotomy case under (Aa1).
For $ k \in \mathbb{N} \cup \{0\} $, $ r \in [0,1] $ and for $ h_{m} : Z'_m \to Z''_m $ where $ Z'_m, Z''_m $ are open sets of Banach spaces, we say $ h_{m}(\cdot) \in C^{k,r} $ \emph{uniformly for} $ m \in K $ if $ \sup_{m \in K} \hol_{k,r} h_{m} (\cdot) < \infty $, where if $ f \in C^{k, r} $ ($ C^{k,0} = C^{k} $), $ \hol_{k,r} f $ is defined by, \label{page:uniformdiff}
\[
\hol_{k,r} f = \begin{cases}
\max\{\lip D^{i} f(\cdot): i = 1,2, \ldots, k-1\} = \hol_{k-1,1} f, & r = 0,\\
\max\{ \hol_{r} D^{k} f(\cdot), \lip D^{i} f(\cdot): i = 1,2, \ldots, k-1 \}, & r \in (0, 1].
\end{cases}
\]

\begin{corI}\label{cor:compact}
	Case (1). Assume (Aa1), (B2), (B3) (a) (b) hold, with $ \xi_1 $, $ \xi_2 $, and $ \eta $ small.

	Case (2). Assume (Aa1), (B2), (B3) (a$ ' $) (b) hold, with $ \xi_1 $, $ \xi_2 $, $ \eta $, and $ \gamma_0 $ small.

	Then the following statements hold.
	\begin{enumerate}[(1)]
		\item In a neighborhood of $ K $, there are three $ C^{0,1} $ submanifolds: $ W^{cs}_{loc}(K) $, $ W^{cu}_{loc}(K) $, and $ \Sigma^c $, called a \emph{local center-stable manifold}, a \emph{local center-unstable manifold}, and a \emph{local center manifold} of $ K $, respectively. These manifolds are locally modeled on $ X^{cs}_{m} $, $ X^{cu}_m $, and $ X^{c}_{m} $ (for $ m \in K $), respectively. They exhibit the following invariance properties: $ W^{cs}_{loc}(K) $ is locally positively invariant, $ W^{cu}_{loc}(K) $ is locally negatively invariant, and $ \Sigma^c $ is locally invariant under $ H $. That is, $ H^{-1}(W^{cs}_{loc}(K)) $ and $ H(W^{cu}_{loc}(K)) $ contain open subsets $ \Omega_{cs} \subset W^{cs}_{loc}(K) $ and $ \Omega_{cu} \subset W^{cu}_{loc}(K) $, respectively. Moreover, $ \Omega_{c} = \Omega_{cs} \cap \Omega_{cu} \subset H^{\pm 1} (\Sigma^c) $ is an open subset of $ \Sigma^c $. If $ \eta = 0 $, then $ K \subset \Omega_c $.

		\item Furthermore, the conclusion of \autoref{thm:tri0} \eqref{it:tri04} holds.

		\item  
		Suppose $ H $ satisfies the \emph{strong $ s $-contraction} and \emph{strong $ u $-expansion} (see assumption ($ \star\star $) in \autoref{sub:limited}; e.g., $ \sup_{m} \alpha'_{cu} (m) $ and $ \sup_{m} \beta'_{cs}(m) $ are sufficiently small).

		\begin{enumerate}[(i)]
			\item Assume (B4) (i) (ii) hold. Then $ W^{cs}_{loc}(K) $, $ W^{cu}_{loc}(K) $, and $ \Sigma^c $ can be chosen to be $ C^1 $ manifolds (possibly in a smaller neighborhood of $ K $).

			\item \label{it:ht} Let $ k \in \mathbb{N} $ and $ r \in [0,1] $. Assume the following conditions:
			\begin{enumerate}[(a)]
				\item $ \widehat{F}^{\kappa}_m(\cdot) $ and $ \widehat{G}^{\kappa}_m(\cdot) $ are uniformly $ C^{k,r} $ for $ m \in K $, where $ \kappa = cs, cu $;
				\item $ \sup_{m \in K} \lambda^{k + r}_{cs}(m) \lambda_{u}(m) < 1 $ and $ \sup_{m \in K} \lambda^{k + r}_{cu}(m) \lambda_{s}(m) < 1 $; and
				\item For each $ m \in K $, $ X^c_{m} $ is finite-dimensional.
			\end{enumerate}
			
		\end{enumerate}
		Then $ W^{cs}_{loc}(K) $, $ W^{cu}_{loc}(K) $, and $ \Sigma^c $ can be chosen to be (uniformly) $ C^{k, r} $ manifolds (provided that $ \gamma_0 $ can be further reduced in case (2), and $ \varsigma_0 $ (in (B3) (a)) can be larger (depending only on $ \{X^c_{m}\} $) in case (1)).
	\end{enumerate}
\end{corI}

\autoref{cor:compact} also provides a revisited version of the theorem for normally hyperbolic invariant compact manifolds \cite{HPS77, Fen72}, which may have boundaries in the infinite-dimensional setting; see also \cite{Jon95}. In the dichotomy case, it further provides relevant conclusions about normally hyperbolic positively or negatively invariant compact manifolds that lack the strongly inflowing or overflowing property with respect to the dynamics. 

\section{Corollaries II: $ C^1 $ maps case} \label{sub:maps}

A very special case where (A3) (or (B3)) holds is when $ H $ (in (IV) of \autoref{sub:setup}) is a $ C^1 $ map, as described below. This is a highly classical way to describe the hyperbolicity of $ K $ with respect to $ H $, which loses some generality (see also \cite{Yan09}).

(IV$ ' $) Suppose $ H $ is a map from $ X $ to $ X $ that is $ C^1 $ in $ \mathbb{B}_{\epsilon'} (K) $ for some small $ \epsilon' > 0 $. Assume there is a small $ \eta_* > 0 $ such that
\[
\sup_{m \in K}|H(m) - u(m)| \leq \eta_*,
\]
where $ u: K \to K $. Set
\[
\mathfrak{A} (\epsilon') \triangleq \sup \{ |DH(m_1) - DH(m_2)|: m_1, m_2 \in \mathbb{B}_{\epsilon'} (K), |m_1 - m_2| \leq \epsilon' \}.
\]
Under the decomposition (III) in \autoref{sub:setup}, consider the following two cases.
\begin{enumerate}[(a)]
	\item \emph{(Dichotomy)} For each $ m \in K $, assume $ \widehat{\Pi}^u_{u(m)}DH(m): \widehat{X}^u_{m} \to \widehat{X}^u_{u(m)} $ is invertible, and write
	\begin{enumerate}[(i$ _{\text{d}} $)]
		\item $ \|\widehat{\Pi}^{cs}_{u(m)}DH(m)|_{\widehat{X}^{cs}_m}\| = \lambda'_{cs}(m) $, $ \| (\widehat{\Pi}^u_{u(m)}DH(m)|_{\widehat{X}^u_m})^{-1} \| = \lambda'_u(m) $,
		\item $ \| \widehat{\Pi}^{\kappa_1}_{u(m)} DH(m)\widehat{\Pi}^{\kappa_2}_m \| \leq \xi_0 $, $ \kappa_1 \neq \kappa_2 $, $ \kappa_1, \kappa_2 \in \{cs,u\} $.
		\item Suppose $ \sup_{m}\lambda'_{cs}(m)\lambda'_{u}(m) < 1 $, $ \sup_{m}\lambda'_{u}(m) < 1 $, $ \sup_{m}\lambda'_{cs}(m) < \infty $.
	\end{enumerate}

	\item[(a$ ' $)] \emph{(Trichotomy)} For each $ m \in K $, assume $ \widehat{\Pi}^{\kappa}_{u(m)}DH(m): \widehat{X}^{\kappa}_{m} \to \widehat{X}^{\kappa}_{u(m)} $ is invertible, $ \kappa \in \{ cu, u \} $, and write
	\begin{enumerate}[(i$ _{\text{t}} $)]
		\item $ \|\widehat{\Pi}^{\kappa_1}_{u(m)}DH(m)|_{\widehat{X}^{\kappa_1}_m}\| = \lambda'_{\kappa_1}(m) $, $ \kappa_1 \in \{ cs, s \} $, $ \| (\widehat{\Pi}^{\kappa_2}_{u(m)}DH(m)|_{\widehat{X}^{\kappa_2}_m})^{-1} \| = \lambda'_{\kappa_2}(m) $, $ \kappa_2 \in \{ cu, u \} $.
		\item $ \| \widehat{\Pi}^{\kappa_1}_{u(m)} DH(m)\widehat{\Pi}^{\kappa_2}_m \| \leq \xi_0 $, $ \kappa_1 \neq \kappa_2 $, $ \kappa_1, \kappa_2 \in \{s, c, u\} $.
		\item Suppose $ \sup_{m}\lambda'_{\kappa}(m)\lambda'_{\kappa_1}(m) < 1 $, $ \sup_{m}\lambda'_{\kappa_1}(m) < 1 $, $ \sup_{m}\lambda'_{\kappa}(m) < \infty $, $ \kappa \in \{cs, cu\} $, $ \kappa_1 = csu - \kappa $.
	\end{enumerate}
\end{enumerate}

\begin{lem}\label{lem:mapAB}
	Under (IV$ ' $) with $ \sup_{m \in K} |\widehat{\Pi}^{\kappa}_{m} | < \infty $, $ \kappa = s, c, u $, we have the following.
	\begin{enumerate}[(a)]
		\item Let (IV$ ' $) (a) hold. If $ \eta_*, \mathfrak{A}(\epsilon'), \xi_0 $ are small, then, except for (A3) (a) (iii), we have that (A3) (a$ ' $) (b) (i) are satisfied with $ \eta = K_3\eta_* $ for some constant $ K_3 > 0 $, $ \varsigma_0 = 1 $ and $ \gamma_0 \to 0 $ as $ \mathfrak{A}(\epsilon'), \xi_0 \to 0 $. In addition,
		\begin{enumerate}[(i)]
			\item if $ \sup_{m}\|\widehat{\Pi}^{s}_{u(m)}DH(m)|_{\widehat{X}^{s}_m}\| < 1 $, and $ \| \widehat{\Pi}^{\kappa_1}_{u(m)} DH(m)\widehat{\Pi}^{\kappa_2}_m \| \leq \xi_0 $, for $ \kappa_1 \neq \kappa_2 $, $ \kappa_1, \kappa_2 \in \{s, c\} $, then $ H $ has the strong $ s $-contraction (see assumption ($ \star\star $) in \autoref{sub:limited});
			\item under (I) (in \autoref{sub:setup}) and (A2), if $ m \mapsto \widehat{\Pi}^{cs}_{m} $ is $ \xi_* $-almost uniformly continuous at $ K $ (in the immersed topology) (see \autoref{def:almost}), then (A3) (a) (iii) holds with $ \xi_1 $ replaced by $ C_* (\xi_1\xi_* + \xi_* + \mathfrak{A} (\epsilon')) $ for some $ C_* > 1 $.
		\end{enumerate}

		\item Let (IV$ ' $) (a$ ' $) hold. If $ \eta_*, \mathfrak{A}(\epsilon'), \xi_0 $ are small, then (B3) (a$ ' $) (b) (i) hold with $ \eta = K_3\eta_* $ for some constant $ K_3 > 0 $ and $ \gamma_0 \to 0 $ as $ \mathfrak{A}(\epsilon'), \xi_0 \to 0 $; in addition, $ H $ has the strong $ s $-contraction and strong $ u $-expansion (see assumption ($ \star\star $) in \autoref{sub:limited}). Moreover, under (I) (in \autoref{sub:setup}) and (B2), if $ m \mapsto \widehat{\Pi}^{\kappa_1}_{m} $ ($ \kappa_1 = s, c, u $) are $ \xi_* $-almost uniformly continuous at $ K $ (in the immersed topology) (see \autoref{def:almost}), then (B3) (a) (iii) holds with $ \xi_1 $ replaced by $ C_* (\xi_1\xi_* + \xi_* + \mathfrak{A} (\epsilon')) $ for some $ C_* > 1 $.
	\end{enumerate}
\end{lem}

\begin{proof}
	See \autoref{app:mapAB}.
\end{proof}

Using \autoref{lem:mapAB}, we have the following consequences for the trichotomy case; the dichotomy case is left to the readers.

\begin{corI}\label{cor:mapT}
	Under (Aa1) (in \autoref{sub:sets}), let $ u: K \to K $ be $ C^0 $ and $ H $ as in (IV$ ' $), and assume (IV$ ' $) (a$ ' $) holds and $ \sup_{m \in K}|\widehat{\Pi}^\kappa_m - \Pi^\kappa_m| \leq \xi_2 $, $ \kappa = s, c, u $.
	If $ \eta_*, \xi_0, \xi_2 $ are small, then for a small $ \epsilon' > 0 $, in $ \mathbb{B}_{\epsilon'}(K) $, there are three $ C^{0,1} \cap C^1 $ submanifolds $ W^{cs}_{loc}(K) $, $ W^{cu}_{loc}(K) $, $ \Sigma^c $ called a \emph{local center-stable manifold}, a \emph{local center-unstable manifold}, and a \emph{local center manifold} of $ K $, respectively, which are locally modeled on $ X^{cs}_{m} $, $ X^{cu}_m $, $ X^{c}_{m} $, $ m \in K $, respectively, and are locally positively invariant, locally negatively invariant, locally invariant under $ H $, respectively. That is, $ H^{-1}(W^{cs}_{loc}(K)) $ and $ H(W^{cu}_{loc}(K)) $ contain open sets $ \Omega_{cs} $ and $ \Omega_{cu} $ in $ W^{cs}_{loc}(K) $ and $ W^{cu}_{loc}(K) $, respectively; in addition $ \Omega_{c} = \Omega_{cs} \cap \Omega_{cu} \subset H^{\pm 1} (\Sigma^c) $ is open in $ \Sigma^c $; if $ \eta_* = 0 $, then $ K \subset \Omega_c $.
	
	Moreover, the following hold.
	\begin{enumerate}[(1)]
		\item If $ \eta_* = 0, \xi_0 = 0 $, then $ T_{m}W^{\kappa}_{loc}(K) = \widehat{X}^{\kappa}_{m} $, $ \kappa = cs, cu $, $ T_{m} \Sigma^c = \widehat{X}^{c}_{m} $, $ m \in K $.
		
		\item \label{it:mapT2} In addition, let $ k, r \in \mathbb{N} $ and $ 1 < r \leq k $, and suppose
		\begin{enumerate}[(i)]
			\item $ H $ is $ C^{k} $ in $ \mathbb{B}_{\epsilon'} (K) $;
			\item $ \sup_{m \in K} (\lambda'_{cs}(m))^{r}\lambda'_{u}(m) < 1 $ and $ \sup_{m \in K} (\lambda'_{cu}(m))^{r} \lambda'_{s}(m) < 1 $; 
			\item for each $ m \in K $, $ X^c_{m} $ is finite-dimensional.
		\end{enumerate}
		
	\end{enumerate}
	
	Then $ W^{cs}_{loc}(K) $, $ W^{cu}_{loc}(K) $, $ \Sigma^c $ can be chosen such that they are (uniformly) $ C^{r} \cap C^{r-1,1} $ (if $ \eta_*, \xi_0 $ are possibly further reduced).
\end{corI}

In \autoref{cor:mapT} \eqref{it:mapT2}, $ r $ can be non-integer. Note also that $ \mathfrak{A}(\epsilon') \to 0 $ as $ \epsilon' \to 0 $ since $ K $ is compact. See also \autoref{sub:dyntan} for a dynamical characterization of (Aa1) (i) when $ K $ is invariant under $ H $, originally due to Bonatti and Crovisier \cite{BC16}.

\begin{corI}\label{cor:mapG}
	Under (B1) and (B2), let $ H $ be as in (IV$ ' $), and assume (IV$ ' $) (a$ ' $) holds and $ \sup_{m \in K}|\widehat{\Pi}^\kappa_m - \Pi^\kappa_m| \leq \xi_2 $, $ \kappa = s, c, u $. If $ \mathfrak{A}(\epsilon'), \eta_*, \xi_0, \xi_1, \xi_2 $ are small and $ \chi(\epsilon) $ is small when $ \epsilon $ is small, then the conclusions (1) (2) in \autoref{thm:tri0} all hold with $ \mu_{cs}(m) \equiv \mu_* \to 0 $ as $ \mathfrak{A}(\epsilon'), \xi_0 \to 0 $.

	In addition, if (B4) (iii) (iv$ ' $) hold, then $ W^{cs}_{loc}(K) $, $ W^{cu}_{loc}(K) $, and $ \Sigma^c $ can be chosen such that they are $ C^1 $ immersed submanifolds of $ X $; if $ \eta_* = 0, \xi_0 = 0 $, then \eqref{equ:tangent} holds.
\end{corI}

In general, if $ K $ is not compact, then $ H $ in \autoref{cor:mapG} may not satisfy $ \mathfrak{A}(\epsilon') \to 0 $ as $ \epsilon' \to 0 $.

\begin{rmk}[Precise asymptotics of constants]
	Consider the choice of constants in \autoref{thm:I}; similarly for \autoref{thm:invariant}, \autoref{thm:tri0} as well as \autoref{app:im}. In ``Choice of constants'' of \autoref{sub:preparation}, the reader can find more details.
	
	Let $ f \sim_{\epsilon} g $ denote $ \limsup_{\epsilon \to 0}|f/g| < \infty $ and $ \limsup_{\epsilon \to 0}|g/f| < \infty $, where $ f, g $ are functions of $ \epsilon $.
	
	\begin{enumerate}[$ \bullet $]
		\item There are $ 0 < c_* < 1 $ and small $ \xi_{1,*}, \xi_{2,*} > 0 $ such that the following statements are true.
		
		Case (1). For any $ \hat{\chi}_0 = O_{\hat{\epsilon}} (1) $, there are $ \epsilon_{*} > 0 $ and $ \hat{\eta}_0 = o(\hat{\chi}_0\hat{\epsilon}) $ such that if (A1), A(2), (A3) (a) (b) hold with $ \chi(\hat{\epsilon}) \leq \hat{\chi}_0(\hat{\epsilon}) $, $ \eta \leq \hat{\eta}_0(\hat{\epsilon}) $, $ \xi_{i} \leq \xi_{i,*} $, then there are $ \varepsilon = O(\hat{\epsilon}) \leq \epsilon_{*} $, $ \sigma,\varrho = O(\hat{\chi}_0\hat{\epsilon}) $, $ \varepsilon_{0} = c_{*} \sigma $ with $ \sigma \sim_{\hat{\epsilon}} \varrho $ and $ \hat{\eta}_0 / \varrho = O_{\hat{\epsilon}} (1) $, such that \autoref{thm:I} holds.
		
		Case (2). For any $ \hat{\chi}_0, \hat{\gamma}_0 = O_{\hat{\epsilon}} (1) $, there are $ \epsilon_{*} > 0 $ and $ \hat{\eta}_0 = o(\hat{\chi}_0\hat{\epsilon}) $ such that if (A1), A(2), (A3)(a$ '$)(b) hold with $ \chi(\hat{\epsilon}) \leq \hat{\chi}_0(\hat{\epsilon}) $, $ \gamma_0 \leq \hat{\gamma}_0(\hat{\epsilon}) $, $ \eta \leq \hat{\eta}_0(\hat{\epsilon}) $, $ \xi_{i} \leq \xi_{i,*} $, then there are $ \varepsilon = O(\hat{\epsilon}) \leq \epsilon_{*} $, $ \sigma,\varrho = o(\hat{\epsilon}) $, $ \varepsilon_{0} = c_{*} \sigma $ with $ \varrho / \sigma = O_{\hat{\epsilon}} (1) $, $ \hat{\gamma}_0\sigma / \varrho = O_{\hat{\epsilon}} (1) $, $ \hat{\eta}_0 / \varrho = O_{\hat{\epsilon}} (1) $, such that \autoref{thm:I} holds.
	\end{enumerate}
	The choice of small $ \xi_{1,*}, \xi_{2,*}, \epsilon_{*} $ (and thus $ \chi(\epsilon) $, $ \eta $, $ \gamma_0 $) depends only on (A3) (a) and $ \widetilde{M}, L, \delta_0 $ (in (H2) (H4)). Let $ \gamma' = \sup_m \alpha(m) \beta'(u(m)) $, $ c'(\varsigma_0) = \inf_m \{\beta(m) - \varsigma_0\beta'(u(m))\} $, $ \lambda' = \sup_m \lambda_{u}(m) \vartheta(m) $. Then $ \xi_{1,*}, \xi_{2,*}, \epsilon_{*} $ decrease with $ \gamma' $, $ 1/c'(\varsigma_0) $, $ \lambda' $, $ \varsigma_0 $, $ 1/\delta_0 $, $ \widetilde{M} $, $ L $, $ \sup_{m} \alpha(m) $, $ \sup_{m} \beta(m) $, $ \sup_{m} \lambda_{cs}(m) $, $ \sup_{m} \lambda_{u} (m) $; and $ c_{*} $ depends only on $ \sup_{m} \lambda_{cs}(m) $ and $ \varsigma_0 $.
\end{rmk}

\begin{rmk}
	Since we only assume that $ \Sigma $ is an immersed submanifold, the local center-(un)stable and center manifolds constructed in the above results are generally immersed. However, if $ \Sigma $ is actually an embedding, i.e., in (H1) (see \autoref{sub:setup}), we have $ U_{m,\gamma}(\epsilon_{m}) = \Sigma \cap \mathbb{B}_{\epsilon_{m}}(m) $ for all $ m \in \Sigma $ and $ \gamma \in \Lambda(m) $ (in particular, $ \phi: \widehat{\Sigma} \to \Sigma $ is injective), then the local center-(un)stable and center manifolds obtained above can be made embedding by restricting them to a sufficiently small neighborhood of $ K $. For example, consider the tubular neighborhood $ X^s_{\widehat{K}_{\varepsilon}}(\rho) \oplus X^u_{\widehat{K}_{\varepsilon}} (\rho) $ of $ K_{\varepsilon} \subset \Sigma $ in $ X $ with $ \rho = \min\{ \sigma,\varrho \} $; in this case, the restriction of the manifolds to this neighborhood yields embedding submanifolds, as ensured by our construction in \autoref{sec:submanifold}.
\end{rmk}

\begin{rmk}[Verification of geometric assumptions (A1) (B1)]\label{rmk:geo}
	Although the uniform geometric (or more precisely, the metric) structure assumption on $ \Sigma $ near $ K $ (i.e., (A1) in \autoref{subsec:main} or (B1) in \autoref{sub:tri}) can often be trivially satisfied in physical applications, unlike the normal hyperbolicity condition (see \cite{Che18b}), in concrete applications one typically only knows the set $ K $ together with the uniform trichotomy assumption on $ K $ and the associated splittings $ X = X^{s}_{m} \oplus X^{c}_{m} \oplus X^{u}_{m} $, $ m \in K $; the manifold $ \Sigma $ itself is generally unknown a priori. Roughly speaking, the partially normal hyperbolicity of $ K $ implies that $ T_{m} K \subset X^{c}_{m} $ for all $ m \in K $, and this condition should be used to construct $ \Sigma $.
	\begin{enumerate}[$ \bullet $]
		\item As demonstrated in \autoref{thm:geo}, when $ K $ is a compact set and the pair $ (X^c_{m}, X^{h}_{m}) $ possesses a good approximation property (i.e., \emph{property (*)} as defined in \autoref{def:p*}), one can employ the Whitney extension theorem to obtain the desired manifold $ \Sigma $. If the Whitney extension theorem is not applicable, then the existence of such a manifold $ \Sigma $ cannot be expected. Since property (*) has not been extensively studied to date (though see also \autoref{exa:p*}), its direct verification presents a challenging task, particularly when $ X $ is not a Hilbert space and $ X^{c}_{m} $ is infinite-dimensional.

		\item For the invariant case (i.e., \autoref{thm:invariant} and \autoref{app:im}), the explicit existence of $ \Sigma $ is not required. Instead, we assume that $ K $ itself is a uniform manifold (i.e., it satisfies (H1)--(H4) in \autoref{sub:setup} with $ K = \Sigma $). In practical applications, the maps $ m \mapsto T_{m}K $ and $ m \mapsto X^{c}_{m} $ are typically well understood. What is needed is the information about $ m \mapsto X^{c_1}_{m} $ such that $ T_{m}K \oplus X^{c_1}_{m} = X^{c}_{m} $ (see, e.g., \autoref{sec:application}). When $ X $ is a Hilbert space, one can choose $ X^{c_1}_{m} = (T_{m}K)^{\perp} \cap X^{c}_{m} $, where $ (T_{m}K)^{\perp} $ denotes the orthogonal complement of $ T_{m}K $; typically, the Lipschitz continuity of $ m \mapsto X^{c_1}_{m} $ inherits from that of $ m \mapsto T_{m}K $ and $ m \mapsto X^{c}_{m} $. When $ K $ is a $ C^1 $ compact submanifold of $ X $ without boundary, by virtue of \autoref{lem:selection} and \autoref{cor:C1app}, there is no difficulty in applying \autoref{thm:invariant} and \autoref{app:im}, regardless of whether $ X^{c}_{m} $ is finite-dimensional or not.

		\item A situation that avoids the use of the Whitney extension theorem occurs when $ K $ is a uniform manifold (i.e., it satisfies (H1)--(H4) in \autoref{sub:setup} with $ K = \Sigma $), which is similar to the previous case. In this case, a natural construction of $ \Sigma $ utilizes the \emph{tubular neighborhood} of $ K $ in $ X $ (see \autoref{sec:submanifold}). So we need to find $ \{X^{c_1}_{m}\}_{m \in K} $ satisfying $ T_{m}K \oplus X^{c_1}_{m} = X^{c}_{m} $ with $ m \mapsto X^{c_1}_{m} $ being Lipschitz continuous. Then we can take $ \Sigma = X^{c_1}_{K} (\sigma) $ for sufficiently small $ \sigma > 0 $ (see \autoref{sec:submanifold} for details). To ensure the smoothness of $ X^{c_1}_{K} (\sigma) $, a smooth and Lipschitz approximation of $ m \mapsto T_{m}K, X^{c_1}_{m} $ is required, for which the results in \autoref{app:general} are applicable.
	\end{enumerate}
	In infinite-dimensional settings, we do no know whether there exist any results addressing invariant manifolds near $ K $ when $ K $ is merely a compact set; in the existing literature, the case where $ K $ is a uniform manifold has been comprehensively studied, as mentioned in \autoref{rmk:PDEs}.
\end{rmk}

\chapter{Preliminaries}\label{sec:facts}

\section{Grassmann manifolds: review}\label{sub:Grassmann}

Throughout this section, $ X $ is assumed to be a Banach space. We briefly present some basic facts about the Grassmann manifold of $ X $.

Write $ X_1 \oplus X_2 $ if $ X_1, X_2 $ are closed linear subspaces of $ X $ with $ X_1 \cap X_2 = \{0\} $. If $ X_1 \oplus X_2 = X $, then we say $ X_i $ is complemented in $ X $ for $ i = 1, 2 $, and $ X_1 $ is a complemented space of $ X_2 $.
Set
\begin{align*}
\mathbb{K}(X) & = \{ X_1: X_1 \text{ is a closed linear subspace of } X \}, \\
\mathbb{G}(X) & = \{ X_1 \in \mathbb{K}(X): \exists X_2 \in \mathbb{K}(X) \text{ such that } X = X_1 \oplus X_2 \}, \\
\overline{\Pi}(X) & = \{ \Pi: \Pi \in L(X) \text{ is a projection, i.e., } \Pi^2 = \Pi \}.
\end{align*}
For $ X_1, X_2 \in \mathbb{K}(X) $, let
\[
\Pi_{X_2}(X_1): X_1 \oplus X_2 \to X_1
\]
denote the projection with $ R(\Pi_{X_2}(X_1)) = X_1 $ and $ \ker (\Pi_{X_2}(X_1)) = X_2 $.
Set
\begin{gather*}
d(X_1, X_2) = \sup_{x \in \mathbb{S}_{X_1}} d(x, \mathbb{S}_{X_2}) \quad (\leq 2), \\
\delta(X_1, X_2) = \sup_{x \in \mathbb{S}_{X_1}} d(x, X_2), \\
\alpha(X_1, X_2) = \inf_{x \in \mathbb{S}_{X_1}} d(x, X_2) = \inf_{\substack{x \in \mathbb{S}_{X_1} \\ x_2 \in X_2}} |x - x_2| \quad (\leq 1), \\
\widehat{d}(X_1, X_2) = \max \{ d(X_1, X_2), d(X_2, X_1) \}, \\
\widehat{\delta}(X_1, X_2) = \max \{ \delta(X_1, X_2), \delta(X_2, X_1) \},
\end{gather*}
where $ \mathbb{S}_{X_1} = \{ x \in X_1: |x| = 1 \} $. Obviously,
\begin{gather*}
\alpha(X_1, X_2) \leq \delta(X_1, X_2) \leq d(X_1, X_2) \leq 2 \delta(X_1, X_2),\\
\alpha(X_2, X_1) \leq 2 \alpha(X_1, X_2), \quad \alpha(X_1, X_2)^{-1} \leq 1 + \alpha(X_2, X_1)^{-1}, \\
\widehat{\delta}(X_1, X_2) \leq \widehat{d}(X_1, X_2) \leq 2 \widehat{\delta}(X_1, X_2).
\end{gather*}

Note that $ \widehat{d}(X_1, X_2) = d_{H}(\mathbb{S}_{X_1}, \mathbb{S}_{X_2}) $ is the Hausdorff metric of $ \mathbb{S}_{X_1}, \mathbb{S}_{X_2} $. So $ (\mathbb{K}(X), \widehat{d}) $ is a complete metric space.
In general, $ \widehat{\delta} $ is not a metric on $ \mathbb{K}(X) $, but it is commonly used to characterize the convergence in $ \mathbb{K}(X) $. The following two lemmas \ref{lem:gram1} and \ref{lem:gram2} are well known; see also the Appendix of \cite{BY17}.

\begin{lem}\label{lem:gram1}
	\begin{enumerate} [(1)]
		\item (About $ \alpha $) $ X_1 \oplus X_2 $ is closed if and only if $ \alpha(X_1, X_2) > 0 $; in this case, $ \alpha(X_1, X_2) = |\Pi_{X_2}(X_1)|^{-1} $.
		\item (About $ \delta $) If $ Y \subset X_1 $ is a closed subspace, then $ \delta(X_1, Y) < 1 $ if and only if $ X_1 = Y $.
		\item \label{gramc} (About $ d $) $ d(X_2, \overline{X_1 + X_3}) \leq d(X_2, X_3) $, and if $ X_1 \oplus X_2 $ is closed, then
		\[
		\delta(X_1 \oplus X_2, \overline{X_1 + X_3}) \leq |\Pi_{X_1}(X_2)| \delta(X_2, \overline{X_1 + X_3}) \leq |\Pi_{X_1}(X_2)| d(X_2, X_3).
		\]
		\item \label{gramd} $ \alpha(X_1, X_2) \leq d(X_3, X_1) + \alpha(X_3, X_2) $, in particular,
		\[
		|\alpha(X_1, X_2) - \alpha(X_3, X_2)| \leq \widehat{d}(X_3, X_1),
		\]
		i.e., $ \alpha(\cdot, X_2): \mathbb{K}(X) \to \mathbb{R}_+ $ is Lipschitz with Lipschitz constant less than $ 1 $.
		\item \label{grame} If $ X_2 \oplus X_3 $ is closed and $ X_1 \subset X_2 \oplus X_3 $, then
		\[
		|\Pi_{X_3}(X_2)|_{X_1}| \triangleq \sup_{x_1 \in \mathbb{S}_{X_1} } | \Pi_{X_3}(X_2) x_1 | \leq |\Pi_{X_3}(X_2)| \delta(X_1, X_3) \leq |\Pi_{X_3}(X_2)| d(X_1, X_3).
		\]
		\item \label{gramf} If $ \Pi_1, \Pi_2 \in \overline{\Pi}(X) $, then $ \widehat{d}(R(\Pi_1), R(\Pi_2)) \leq 2 |\Pi_1 - \Pi_2| $.
	\end{enumerate}
\end{lem}
\begin{proof}
	See \autoref{app:GramBasic}.
\end{proof}

\begin{lem}[Perturbation of complemented subspaces]\label{lem:gram2}
	If $ X_1 \oplus X_2 = X $, $ X_3 \in \mathbb{K}(X) $ and
	\[
	\widehat{d}(X_3, X_1) < \alpha(X_1, X_2) = |\Pi_{X_2}(X_1)|^{-1},
	\]
	then
	\begin{enumerate}[(1)]
		\item $ X_3 \oplus X_2 = X $, so $ \mathbb{G}(X) $ is open in $ \mathbb{K}(X) $;
		\item \label{gram2b} $ |\Pi_{X_2}(X_3)| \leq \frac{|\Pi_{X_2}(X_1)|}{1 - |\Pi_{X_2}(X_1)|d(X_3, X_1)} $ ($ \Leftrightarrow \alpha(X_1, X_2) \leq d(X_3, X_1) + \alpha(X_3, X_2) $);
		\item \label{gram2c} $ \Pi_{X_2}(X_1): X_3 \to X_1 $ is invertible and $ (\Pi_{X_2}(X_1)|_{X_3})^{-1} = \Pi_{X_2}(X_3)|_{X_1} $.
		Moreover, 
		\begin{align*}
			|\Pi_{X_2}(X_3)|_{X_1}| & \leq |\Pi_{X_2}(X_3)\Pi_{X_2}(X_1)| = |\Pi_{X_2}(X_1) - \Pi_{X_2}(X_3)| \\
			& \leq C(X_1, X_2, X_3) \delta(X_1, X_3),
		\end{align*}
		where
		\[
		C(X_1, X_2, X_3) = |\Pi_{X_2}(X_1)| \left(1 + \frac{|\Pi_{X_2}(X_1)|}{1 - |\Pi_{X_2}(X_1)|d(X_3, X_1)}\right),
		\]
		and in particular, $ \Pi_{X_2}(\cdot) $ is locally Lipschitz in $ \mathbb{G}(X) $.
	\end{enumerate}
\end{lem}
\begin{proof}
	See \autoref{app:GramBasic}.
\end{proof}

Let
\[
\Omega(X) = \{ (X_1, X_2) \in \mathbb{K}(X)^2: \alpha(X_1, X_2) > 0 \},
\]
and define
\begin{gather*}
\varGamma: \Omega(X) \to \mathbb{K}(X),~ (X_1, X_2) \mapsto X_1 \oplus X_2, \\
\Upsilon: \varGamma^{-1}(X) \to \overline{\Pi}(X),~ (X_1, X_2) \mapsto \Pi_{X_2}(X_1).
\end{gather*}

The following lemma was also proved in \cite[Lemma 5.3]{LL10} for a special case.
\begin{lem}
	\begin{enumerate}[(1)]
		\item $ \varGamma $ is continuous; in particular, $ \varGamma^{-1}(X) \subset \mathbb{G}(X)^2 $ is closed in $ \mathbb{K}(X)^2 $.

		\item $ \Upsilon $ is a homeomorphism. In fact, $ \Upsilon^{-1} $ is Lipschitz and $ \Upsilon $ is locally Lipschitz. Therefore, $ \overline{\Pi}(X) $ can be regarded as a closed subset of $ \mathbb{K}(X)^2 $.
	\end{enumerate}
\end{lem}
\begin{proof}
	(1). Let $ X_{n} \to X_0 $ and $ Y_{n} \to Y_0 $ in $ \mathbb{G}(X) $ with $ \alpha(X_0, Y_0) > 0 $. Then $ \alpha(X_n, Y_n) \geq \varepsilon_0 $, $ \alpha(X_n, X_0) \geq \varepsilon_0 $ and $ \alpha(Y_n, Y_0) \geq \varepsilon_0 $ for all large $ n $ and some $ \varepsilon_0 > 0 $. For example, by \autoref{lem:gram1} \eqref{gramd},
	\begin{align*}
	\alpha(X_0, Y_0) \leq & \widehat{d}(X_n, X_0) + \alpha(X_n, Y_0) \leq \widehat{d}(X_n, X_0) + 2\alpha(Y_0, X_n) \\
	\leq & \widehat{d}(X_n, X_0) + 2\widehat{d}(Y_n, Y_0) + 2\alpha(Y_n, X_n),
	\end{align*}
	i.e., $ \alpha(Y_n, X_n) \geq \frac{1}{2}(\alpha(X_0, Y_0) - \widehat{d}(X_n, X_0) - 2\widehat{d}(Y_n, Y_0)) \geq \varepsilon_0 > 0 $ for all large $ n $. By \autoref{lem:gram1} \eqref{gramc}, one gets
	\begin{align*}
	\widehat{d}(X_n \oplus Y_n, X_0 \oplus Y_0) \leq & \widehat{d}(X_n \oplus Y_n, X_n \oplus Y_0) + \widehat{d}(X_n \oplus Y_0, X_0 \oplus Y_0) \\
	\leq & 2(\widehat{\delta}(X_n \oplus Y_n, X_n \oplus Y_0) + \widehat{\delta}(X_n \oplus Y_0, X_0 \oplus Y_0)) \\
	\leq & 2\varepsilon^{-1}_0(\widehat{d}(Y_n, Y_0)+\widehat{d}(X_n, X_0)) \to 0.
	\end{align*}

	(2). That $ \Upsilon $ is locally Lipschitz follows from \autoref{lem:gram2} \eqref{gram2c}. Also, by \autoref{lem:gram1} \eqref{gramf}, we know $ \Upsilon^{-1} $ is Lipschitz.
\end{proof}

Take $ X_i \in \mathbb{G}(X) $, $ X = X_1 \oplus X_2 $. Consider
\[
U_{X_1, X_2} = \{ X'_1: X'_1 \oplus X_2 = X \} \cong \{ \Pi_{X_2}(X'_1): X'_1 \oplus X_2 = X \},
\]
a neighborhood of $ X_1 $, and define a local chart $ \varphi_{X_1, X_2}: U_{X_1, X_2} \to L(X_1, X_2) $ as
\begin{equation}\label{equ:localGrass}
\varphi_{X_1, X_2}(X'_1) = \Pi_{X_1}(X_2) \Pi_{X_2}(X'_1)|_{X_1}: X_1 \to X_2.
\end{equation}
Note that
\[
\varphi_{X_1, X_2}^{-1} (f) = \graph f \triangleq \{ x_1 + f(x_1): x_1 \in X_1 \},
\]
where $ f \in L(X_1, X_2) $, and
\[
\Pi_{X_2}(X'_1) = \mymatrix{\id}{\varphi_{X_1, X_2}(X'_1)}{0}{0}: X_1 \oplus X_2 \to X_1 \oplus X_2.
\]
Now $ \{ ( U_{X_1, X_2}, \varphi_{X_1, X_2} ) \}_{X_1 \oplus X_2 = X} $ gives a $ C^{\infty} $ differential structure on $ \mathbb{G}(X) $ (see \cite[Example 3.1.8 G]{AMR88}) which is locally modeled on $ L(X_1, X_2) (\subset L(X)) $. Moreover, we have
\begin{lem}\label{lem:granlip}
	The topology induced by the differential structure of $ \mathbb{G}(X) $ coincides with the metric topology induced by the metric $ \widehat{d} $. In fact, for $ X_1 \oplus X_2 = X $, $ \varphi_{X_1, X_2} $ is bi-Lipschitz and
	\begin{enumerate}[(1)]
		\item if $ \varepsilon < \alpha(X_1, X_2) / 2 $, then $ \varphi^{-1}_{X_1, X_2}: \mathbb{B}_{\varepsilon} (0) (\subset L(X_1, X_2)) \to \mathbb{B}_{\rho} (X_1) (\subset U_{X_1, X_2}) $, where $ \rho = \varepsilon / (1 - \varepsilon) \leq 2\varepsilon $; $ \lip \varphi^{-1}_{X_1, X_2}|_{\mathbb{B}_{\varepsilon}(0)} \leq 1/(1 - \varepsilon) $;
		\item \label{granlip} if $ \rho < \alpha(X_1, X_2) / 4 $, then $ \varphi_{X_1, X_2}: \mathbb{B}_\rho (X_1) (\subset U_{X_1, X_2}) \to \mathbb{B}_{\varepsilon}(0) (\subset L(X_1, X_2)) $, where $ \varepsilon = 4 \alpha(X_1, X_2)^{-2} \rho $; $ \lip \varphi_{X_1, X_2} |_{\mathbb{B}_{\rho}(X_1)} \leq  12 \alpha(X_1, X_2)^{-2} $ if $ \rho < \alpha(X_1, X_2)^2 / 8 $.
	\end{enumerate}
\end{lem}
\begin{proof}
	(1). First note that if $ f \in L(X_1, X_2) $ and $ |f| < 1 $, then $ \widehat{d}(\graph f, X_1) \leq \frac{|f|}{1-|f|} $. (Obviously, $ \delta (X_1, \graph f) \leq |f| $, and $ d (x_1 + f(x_1), X_1) \leq |f||x_1| $ if $ |x_1+f(x_1)| \leq 1 $, which yields $ |f||x_1| \leq \frac{|f|}{1-|f|} $.) Thus, if $ f \in \mathbb{B}_{\varepsilon} (0) $ (so $ |f| < \alpha(X_1, X_2) / 2 $ ($ \leq 1/2 $)), then $ \widehat{d}(\graph f, X_1) < \alpha(X_1, X_2) $ and $ \varphi^{-1}_{X_1, X_2} (f) \in \mathbb{B}_{\rho} (X_1) $. Similarly, if $ f_1, f_2 \in \mathbb{B}_{\varepsilon}(0) $, then $ \widehat{d}(\graph f_1, \graph f_2) \leq |f_1 - f_2|/(1 -\varepsilon) $, i.e., $ \lip \varphi^{-1}_{X_1, X_2}|_{\mathbb{B}_{\varepsilon}(0)} \leq 1/(1 - \varepsilon) $.

	(2). Take $ X_3 \in \mathbb{B}_\rho (X_1) $. Then by \autoref{lem:gram1} \eqref{grame} and \autoref{lem:gram2} \eqref{gram2c}, we have
	\begin{align*}
	& |\varphi_{X_1, X_2}(X_3)| \leq |\Pi_{X_1}(X_2)|_{X_3}||\Pi_{X_2}(X_3)|_{X_1}| \\
	\leq & |\Pi_{X_1}(X_2)| \widehat{d}(X_3, X_1) |(\Pi_{X_2}(X_3) - \Pi_{X_2}(X_1) + \id)\Pi_{X_2}(X_1)| \\
	\leq & |\Pi_{X_1}(X_2)| ((1 + 2|\Pi_{X_2}(X_1)|) \widehat{d}(X_3, X_1) + 1 ) |\Pi_{X_2}(X_1)| \widehat{d}(X_3, X_1) \\
	\leq & 4 \alpha(X_1, X_2)^{-2} \rho.
	\end{align*}
	The Lipschitz continuity of $ \varphi_{X_1, X_2} $ can be deduced from the (local) Lipschitz continuity of $ \Pi_{X_2}(\cdot) $; for example, if $ \rho < \alpha(X_1, X_2)^2 / 8 $ (so $ \varepsilon \leq 1/2 $ and $ 2\rho(1+\varepsilon) < 1/2 $), then for $ X_3', X_3'' \in \mathbb{B}_\rho (X_1) $, we have
	\[
	|\Pi_{X_2}(X'_3)| \leq 1 + |\varphi_{X_1, X_2}(X'_3)| \leq 1 + \varepsilon,
	\]
	and
	\begin{align*}
	&~	|\varphi_{X_1, X_2}(X_3') - \varphi_{X_1, X_2}(X_3'')| \\
	 \leq &~  |\Pi_{X_1}(X_2)||\Pi_{X_2}(X_1)| |\Pi_{X_2}(X_3') - \Pi_{X_2}(X_3'')| \\
	\leq &~  2\alpha(X_1, X_2)^{-2}|\Pi_{X_2}(X'_3)|\left(1+\frac{|\Pi_{X_2}(X'_3)|}{1 - |\Pi_{X_2}(X'_3)|\widehat{d}(X_3', X_3'')}\right)\widehat{d}(X_3', X_3'') \\
	\leq &~  2 (1 + \varepsilon) (1 + 2(1+\varepsilon)) \alpha(X_1, X_2)^{-2} \widehat{d}(X_3', X_3''),
	\end{align*}
	i.e., $ \lip \varphi_{X_1, X_2} |_{\mathbb{B}_{\rho}(X_1)} \leq  12 \alpha(X_1, X_2)^{-2} $. The proof is complete.
\end{proof}

Using $ \Upsilon^{-1} $, one can naturally endow $ \overline{\Pi}(X) $ with a $ C^{\infty} $ differential structure induced by $ \mathbb{G}(X)^2 $.
\begin{lem}\label{lem:proj}
	$ \overline{\Pi}(X) $ is a smooth embedded submanifold of $ L(X) $ without boundary, locally modeled on $ L(X_1, X_2) \times L(X_2, X_1) $, where $ X_1 \oplus X_2 = X $.
\end{lem}
\begin{proof}
	It suffices to show that $ \Upsilon: \varGamma^{-1}(X) \to \overline{\Pi}(X), (X_1, X_2) \mapsto \Pi_{X_2}(X_1) $, is $ C^{\infty} $. Consider the local representation of $ \Upsilon $. Let $ X_0 \oplus \tilde{X}_0 = X $ and take a sufficiently small $ \varepsilon > 0 $ (for example, $ \varepsilon < \min \{\alpha(X_0, \tilde{X}_0), \alpha(\tilde{X}_0, X_0)\} / 4 $). Now the local representation of $ \Upsilon $ can be taken as
	\begin{align}
	\widetilde{\Upsilon}: \quad & \mathbb{B}_{\varepsilon}(0) \times \mathbb{B}_{\varepsilon}(0) \subset L(X_0,\tilde{X}_0) \times L(\tilde{X}_0,X_0) \to \varGamma^{-1}(X) \to L(X), \notag\\
	&  (f_1, f_2) \mapsto (\varphi^{-1}_{X_0, \tilde{X}_0}(f_1) \triangleq X_1, \varphi^{-1}_{\tilde{X}_0, X_0}(f_2) \triangleq X_2) \mapsto \Pi_{X_2}(X_1). \label{equ:Pilocal}
	\end{align}
	Define
	\[
	F(f_1, f_2)(x_0, \tilde{x}_0) \triangleq f_1(x_0) + f_2(\tilde{x}_0) : X_0 \oplus \tilde{X}_0 \to X_0 \oplus \tilde{X}_0.
	\]
	Note that $ F: (f_1, f_2) \mapsto F(f_1, f_2) $ is linear and so $ C^{\infty} $, and
	\[
	A: \mathbb{B}_1(0) \subset L(X) \to L(X), B \mapsto (\id + B)^{-1},
	\]
	is $ C^{\infty} $. This shows that
	\[
	A \circ F: \mathbb{B}_{\varepsilon}(0) \times \mathbb{B}_{\varepsilon}(0) \subset L(X_0,\tilde{X}_0) \times L(\tilde{X}_0,X_0) \to L(X),
	\]
	is $ C^{\infty} $ if $ \varepsilon $ is small. Notice that
	\[
	\Pi_{X_2}(X_1) = \Pi_{\tilde{X}_0}(X_0) A \circ F (f_1, f_2) + f_1 \circ \Pi_{\tilde{X}_0}(X_0) A \circ F (f_1, f_2),
	\]
	yielding that $ \widetilde{\Upsilon} $ is $ C^{\infty} $. The proof is complete.
\end{proof}

\section{Uniform submanifolds in Banach spaces: review} \label{sec:submanifold}

Throughout this section, we make the following assumption.

\vspace{.5em}
\noindent{Assumptions}.
Let $ \Sigma $ and $ K $ be as in \autoref{sub:setup} (I) and assume that (H1)--(H4) in \autoref{sub:setup} hold.
\vspace{.5em}

In the following, let us review some important geometric properties of $ \Sigma $ near $ K $, which are taken from \cite[Section 4.3]{Che18b} (see also  \cite[Section 3]{BLZ08} and \cite[Section 3]{BLZ99} under the unnecessary condition $ \Sigma \in C^1 $); note also that (H4) is equivalent to \cite[Definition 2.1 (2)]{BLZ08}.

Recall the notations in \eqref{equ:notationM}: $ \widehat{K}_{\epsilon} $ and $ {X}^{\kappa}_{\widehat{K}'} (r) $ ($ \kappa = s, c, u, h $). We also use the following notations: For $ m_0 \in K $, $ \gamma \in \Lambda(m_0) $, let
\begin{align*}
m_0 + X^c_{m_0}(\epsilon) \oplus X^h_{m_0}(\varrho) & = \{ m_0 + x^c_0 + x^h_0: x^c_0 \in X^c_{m_0}(\epsilon), x^h_0 \in X^h_{m_0}(\varrho) \}, 
\end{align*}
and define $ \chi_{m_0,\gamma} $ by
\begin{gather*}
	\chi_{m_0,\gamma}: U_{m_0,\gamma}(\epsilon_1) \to X^c_{m_0}: m' \mapsto \Pi^c_{m}(m' - m_0),
\end{gather*}
the inverse of $ \chi_{m_0,\gamma} $ by
\[
\omega_{m_0,\gamma}(\cdot) = \chi_{m_0,\gamma}^{-1} : X^c_{m_0}(\delta_0) \to U_{m_0,\gamma}(\epsilon_1),
\]
and $ \chi^h_{m_0,\gamma} $ by
\[
\chi^h_{m_0,\gamma}: X^c_{m_0}(\delta_0) \to X^h_{m_0}: x^c_0 \mapsto \Pi^h_{m_0} (\omega_{m_0,\gamma}(x^c_0) - m_0).
\]
Now for all $ m_1, m_2 \in U_{m_0,\gamma}(\epsilon_1) $, by (H1),
\[
\frac{1}{1 + \chi(\epsilon_1)} |\chi_{m_0,\gamma}(m_1) - \chi_{m_0,\gamma}(m_2)| \leq |m_1 - m_2| \leq \frac{1}{1 - \chi(\epsilon_1)} |\chi_{m_0,\gamma}(m_1) - \chi_{m_0,\gamma}(m_2)|.
\]
Therefore, 
\[
X^c_{m_0}(c_1(\epsilon_1)\varepsilon) \xrightarrow{\omega_{m_0,\gamma}} U_{m_0,\gamma}(\varepsilon) \xrightarrow{\chi_{m_0,\gamma}} X^c_{m_0}(c_2(\epsilon_1) \varepsilon),
\]
where
\[
\varepsilon < \min \{ \delta_0 / c_1(\epsilon_1), \epsilon_1 \},\quad c_1(\epsilon_1) = 1 - \chi(\epsilon_1), \quad c_2(\epsilon_1) = 1 + \chi(\epsilon_1).
\]
For $ \kappa = s, u $, define a local chart of $ \mathbb{G}(X) $ at $ X^\kappa_{m_0} $ by
\begin{equation*}
	(\varphi^{c,\kappa}_{m_0} (\overline{m}), \varphi^\kappa_{m_0} (\overline{m}))
	= \Pi_{X^\kappa_{m_0}} (X^{csu-\kappa}_{m_0}) \Pi_{X^{csu-\kappa}_{m_0}} (X^\kappa_{\overline{m}}) |_{X^\kappa_{m_0}} :
	X^\kappa_{m_0} \to X^{c}_{m_0} \oplus X^{su-\kappa}_{m_0};
\end{equation*}
see, e.g., \cite[p. 145]{AMR88}. Take $ \epsilon_* > 0 $ such that
\[
\epsilon_* < \min\left\{ \frac{1}{8\widetilde{M}L}, \epsilon_1, \frac{\delta_0}{1 - \chi(\epsilon_1)} \right\}.
\]
Let $ c_1 = c_1(\epsilon_*) = 1 - \chi(\epsilon_*) $, and $ c_2 = c_2(\epsilon_*) = 1 + \chi(\epsilon_*) $.
Note that $ \alpha(X^{c}_m, X^{\kappa}_m) \geq \widetilde{M}^{-1} $ if $ m \in K $, $ \kappa = s, u, h $.
By (H2), for every $ m_1, m_2 \in U_{m_0,\gamma} (\epsilon_*) $, we have
\[
|\Pi^{\kappa}_{m_1} - \Pi^{\kappa}_{m_2}| \leq L|m_1 - m_2| \leq 2L\epsilon_*.
\]
So $ X^\kappa_{\overline{m}} \in \mathbb{B}_{(2\widetilde{M})^{-1}} (X^\kappa_{m_0}) $, $ \kappa = s, c, u $, if $ \overline{m} \in U_{m_0,\gamma} (\epsilon_*) $. 
Now by \autoref{lem:gram1} \eqref{grame}, for $ \kappa = s, u $,
\begin{align*}
	X^h_{m_0} (e_1\varrho) \xrightarrow{ \Pi_{X^c_{m_0}}(X^h_{\overline{m}}) }  X^h_{\overline{m}} (\varrho) \xrightarrow{ \Pi_{X^c_{m_0}}(X^h_{m_0}) } X^h_{m_0} (e_2\varrho), \\
	X^{\kappa}_{m_0} (e_1\varrho) \xrightarrow{ \Pi_{X^{scu-\kappa}_{m_0}}(X^{\kappa}_{\overline{m}}) }  X^{\kappa}_{\overline{m}} (\varrho) \xrightarrow{ \Pi_{X^{scu-\kappa}_{m_0}}(X^{\kappa}_{m_0}) } X^{\kappa}_{m_0} (e_2\varrho),
\end{align*}
where $ e_1 = e_1(\epsilon_*) = (1+4\widetilde{M}L\epsilon_*)^{-1} $, $ e_2 = e_2(\epsilon_*) = 1 + \widetilde{M}L\epsilon_* $. 

The tubular neighborhood of $ K $ can be constructed through the following map:
\begin{align}\label{equ:tub}
 X^c_{m_0}(c_1\varepsilon) \times X^s_{m_0}(e_1\varrho) \times X^u_{m_0}(e_1\sigma) &~ \to X^c_{m_0} \times X^s_{m_0} \times X^u_{m_0}, \notag \\
\Phi_{m_0,\gamma} :  (x^c_0, y^s, y^u) & ~ \mapsto (\widehat{x}^c, \widehat{x}^s, \widehat{x}^u),
\end{align}
where $ m_0 \in K $, $ \gamma \in \phi^{-1}(m_0) $, and
\begin{equation*}
\begin{cases}
\widehat{x}^c = x^c_0 + \varphi^{c,s}_{m_0} (\overline{m}) y^s + \varphi^{c,u}_{m_0} (\overline{m}) y^u,  \quad\overline{m} = \omega_{m_0,\gamma} (x^c_0),\\
\widehat{x}^\kappa = \chi^{\kappa}_{m_0,\gamma} (x^c_0) + \varphi^{su-\kappa}_{m_0} (\overline{m}) y^{su-\kappa} + y^{\kappa}, \quad \kappa = s, u,
\end{cases}
\end{equation*}
that is, $ \overline{m} + \overline{x}^s + \overline{x}^u = m_0 + \widehat{x}^c + \widehat{x}^s + \widehat{x}^u $, and
\begin{equation}\label{equ:bundle}
\begin{cases}
\overline{m} = m_0 + x^c_0 + \chi^{h}_{m_0,\gamma} (x^c_0) \in U_{m_0,\gamma} (\varepsilon), \\
\overline{x}^s = \Pi_{X^{cu}_{m_0}} (X^s_{\overline{m}}) y^s \in X^s_{\overline{m}} (\varrho), \\
\overline{x}^u = \Pi_{X^{cs}_{m_0}} (X^{u}_{\overline{m}}) y^u \in X^u_{\overline{m}} (\sigma).
\end{cases}
\end{equation}
From \autoref{lem:granlip} \eqref{granlip}, (H2) and \autoref{lem:gram1} \eqref{gramf}, we see that $ \varphi^{c,\kappa}_{m_0} (\cdot) $ and $ \varphi^\kappa_{m_0} (\cdot) $ are Lipschitz with Lipschitz constants bounded by a fixed constant. Using this fact and (H1), it is not hard to see that 
\begin{equation}\label{equ:lipT}
	\sup_{m_0 \in K, \gamma} \lip (\Phi_{m_0,\gamma} - I) \to 0 ~\text{as}~ \varepsilon, \chi(\varepsilon), \sigma, \rho \to 0,
\end{equation}
(see also \cite[Section 4.3]{Che18b} or \cite[Lemma 3.6]{BLZ08}). We should also mention that if $ \Sigma $ is $ C^1 $ with $ U_{m_0,\gamma} \to L(X), m' \mapsto \Pi^{\kappa}_{m'} $ being $ C^1 $ ($ \kappa = s, c, u $), then $ \Phi_{m_0,\gamma}(\cdot) \in C^1 $. In particular, we obtain

\begin{lem}\label{lem:tub}
	Assume $ \epsilon_{*}, \chi(\epsilon_{*}) $ are sufficiently small. Then there is a small $ \chi_{*} < 1/4 $ depending on $ \epsilon_{*}, \chi(\epsilon_{*}) $ such that if $ 0 < \varepsilon \leq \epsilon_{*} $ and $ \chi_{*} \varepsilon \leq \varrho, \sigma \leq \varepsilon $, and $ \widehat{m}_0 \in \widehat{K}, m_0 = \phi(\widehat{m}_0) $, then
	\begin{multline*}
		m_0 + X^c_{m_0}(c_1\varepsilon) \oplus X^s_{m_0}(e_1\varrho) \times X^u_{m_0}(e_1\sigma) \subset X^s_{\widehat{U}_{\widehat{m}_0}(\varepsilon)} (\sigma) \oplus X^u_{\widehat{U}_{\widehat{m}_0}(\varepsilon)} (\varrho) \\ \subset m_0 + X^c_{m_0}(c_2\varepsilon) \oplus X^s_{m_0}(e_2\varrho) \times X^u_{m_0}(e_2\sigma),
	\end{multline*}
	where $ c_{i}, e_{i} \to 1 $ and $ \chi_{*} \to 0 $ as $ \epsilon_{*}, \chi(\epsilon_{*}) \to 0 $.
\end{lem}

We call $ X^s_{\widehat{U}_{\widehat{m}_0}(\varepsilon)} (\sigma) \oplus X^u_{\widehat{U}_{\widehat{m}_0}(\varepsilon)} (\varrho) $ a local (immersed) tubular neighborhood of $ \Sigma $ in $ X $ near $ m_0 \in K $, and $ X^s_{\widehat{K}_{\varepsilon}} (\sigma) \oplus X^u_{\widehat{K}_{\varepsilon}} (\varrho) $ an (immersed) tubular neighborhood of $ \Sigma $ in $ X $ near $ K $.

Let $ \overline{m}_i \in U_{m_0,\gamma} (\varepsilon) $, $ \overline{x}^s_i \in X^s_{\overline{m}_i} (\sigma) $, $ \overline{x}^u_i \in X^h_{\overline{m}_i} (\varrho) $, or $ \widehat{x}^c_i \in X^c_{m_0} (c_1\varepsilon) $, $ \widehat{x}^s_i \in X^s_{m_0} (e_1\sigma) $, $ \widehat{x}^u_i \in X^u_{m_0} (e_1\varrho) $, and consider
\begin{equation}\label{equ:pre}
\begin{split}
\overline{m}_i + \overline{x}^s_i + \overline{x}^u_i & = m_0 + x^c_i + \chi^h_{m_0,\gamma} (x^c_i) + \sum_{\kappa = s,u}\{\varphi^{c,\kappa}_{m_0}(\overline{m}_i) y^\kappa_i + \varphi^{\kappa}_{m_0}(\overline{m}_i) y^\kappa_i + y^\kappa_i\},\\
& = m_0 + \widehat{x}^c_i + \widehat{x}^s_i + \widehat{x}^u_i,
\end{split}
\end{equation}
where $ \overline{m}_i \in U_{m_0,\gamma} $, $ \overline{x}^\kappa_i \in X^h_{\overline{m}_i} $, $ x^c_i \in X^c_{m_0} $, $ y^\kappa_i \in X^\kappa_{m_0} $, $ \kappa = s, u $, $ \widehat{x}^{\kappa_1}_i \in X^{\kappa_1}_{m_0} $, $ \kappa_1 = s, c, u $, $ i = 1,2 $.
Then we have $ \Phi_{m_0,\gamma} (x^c_i,y^s_i, y^u_i) = (\widehat{x}^c_i, \widehat{x}^s_i, \widehat{x}^u_i) $. Here, note that from \eqref{equ:bundle}, we have 
\[
\Pi^\kappa_{m_0} (\overline{x}^\kappa_1 - \overline{x}^\kappa_2) = y^{\kappa}_1 - y^{\kappa}_2,~\kappa = s, u,~\text{and}~\Pi^c_{m_0} (\overline{m}_1 - \overline{m}_2) = x^c_1 - x^c_2.
\]

The following notation $ \{ s, c, u \} - \kappa $ means the collection obtained by deleting the letter $ \kappa $ from $ \{ s, c, u \} $; for instance, if $ \kappa = su~ (= h) $, then $ \{ s, c, u \} - \kappa = \{ s \} $. 

\begin{lem}\label{lem:lip2}
	For sufficiently small $ \epsilon_{*,*} > 0 $, if $ \max\{\epsilon_{*}, \chi(\epsilon_{*})\} \leq \epsilon_{*,*} $, then there is a small $ \chi_{*} < 1 / 4 $ depending on $ \epsilon_{*}, \chi(\epsilon_{*}) $ such that $ \chi_{*} \to 0 $ as $ \epsilon_{*}, \chi(\epsilon_{*}) \to 0 $ and the following hold. Let $ \varepsilon \leq \epsilon_{*} $, $ \chi_{*} \varepsilon \leq \sigma, \varrho \leq \varepsilon $ and assume \eqref{equ:pre} holds.

	\begin{enumerate}[(1)]
		\item For $ \kappa = c, s, u, h $, the following estimates hold:
		\begin{equation}\label{equ:estimates}
		\left\{
		\begin{split}
		|\widehat{x}^\kappa_1 - \widehat{x}^\kappa_2| & \leq (1 + \chi_{*}) |x^\kappa_1 - x^\kappa_2| + \chi_{*} \sum_{\kappa' \in \{ s, c, u \} - \kappa } |x^{\kappa'}_1 - x^{\kappa'}_2|, \\
		|{x}^\kappa_1 - {x}^\kappa_2| & \leq (1 + \chi_{*}) |\widehat{x}^\kappa_1 - \widehat{x}^\kappa_2| + \chi_{*} \sum_{\kappa' \in \{ s, c, u \} - \kappa } |\widehat{x}^{\kappa'}_1 - \widehat{x}^{\kappa'}_2|,
		\end{split}
		\right.
		\end{equation}
		where $ \Pi^\kappa_{m_0} (\overline{x}^\kappa_1 - \overline{x}^\kappa_2) \triangleq x^\kappa_1 - x^{\kappa}_2 $, $ \kappa = s, u, h (= su) $.
		Moreover, for $ \widehat{m}_0 \in \widehat{K} $, $ m_0 = \phi(\widehat{m}_0) $,
		\begin{multline} \label{equ:nnn}
		m_0 + X^c_{m_0}(c_1\varepsilon) \oplus X^s_{m_0}(e_1\sigma) \oplus X^u_{m_0}(e_1 \varrho) \subset X^s_{\widehat{U}_{\widehat{m}_0}(\varepsilon)}(\sigma) \oplus X^u_{\widehat{U}_{\widehat{m}_0}(\varepsilon)}(\varrho) \\
		\subset m_0 + X^c_{m_0}(c_2\varepsilon) \oplus X^s_{m_0}(e_2\sigma) \oplus X^u_{m_0}(e_2 \varrho).
		\end{multline}
		The constants $ c_i, e_i $, $ i = 1,2 $, are given as in \autoref{lem:tub}.
		\item Moreover, given $ \mu_* > 0 $ and $ \mu \leq \mu_* $, the constant $ \chi_{*} $ can be chosen small enough (depending on $ \mu_* $) such that $ \mu_* \chi_{*} < 1/4 $ and the following hold. For $ \kappa = s, u $,
		\begin{multline*}
		|\Pi^{\kappa}_{m_0} (\overline{x}^{\kappa}_1 - \overline{x}^{\kappa}_2)| \leq \mu \max\{ |\Pi^c_{m_0} (\overline{m}_1 - \overline{m}_2) |, |\Pi^{su - \kappa}_{m_0} (\overline{x}^{su - \kappa}_1 - \overline{x}^{su - \kappa}_2)| \} \\
		\Rightarrow |\widehat{x}^{\kappa}_1 - \widehat{x}^{\kappa}_2| \leq \mu_1 \max\{|\widehat{x}^c_1 - \widehat{x}^c_2|, |\widehat{x}^{su - \kappa}_1 - \widehat{x}^{su - \kappa}_2| \},
		\end{multline*}
		and
		\begin{multline*}
		|\widehat{x}^{\kappa}_1 - \widehat{x}^{\kappa}_2| \leq \mu \max\{|\widehat{x}^c_1 - \widehat{x}^c_2|, |\widehat{x}^{su - \kappa}_1 - \widehat{x}^{su - \kappa}_2| \}  \\
		\Rightarrow |\Pi^{\kappa}_{m_0} (\overline{x}^{\kappa}_1 - \overline{x}^{\kappa}_2)| \leq \mu_1 \max\{ |\Pi^c_{m_0} (\overline{m}_1 - \overline{m}_2) |, |\Pi^{su - \kappa}_{m_0} (\overline{x}^{su - \kappa}_1 - \overline{x}^{su - \kappa}_2)| \},
		\end{multline*}
		where $ \mu_1 = (1 + \chi_*) \mu + \chi_{*} $.
	\end{enumerate}
\end{lem}
\begin{proof}
	Note that \eqref{equ:estimates} is a restatement of \eqref{equ:lipT}. (2) is a direct consequence of \eqref{equ:estimates} by simple computations.
\end{proof}

We now describe the graph of a bundle map $ h: X^s_{\widehat{K}_{\varepsilon}} (\sigma) \to X^u_{\widehat{K}_{\varepsilon}} (\varrho) $ over $ \id $.

\begin{defi}\label{def:lip}
	Let $ \widehat{m}_0 \in \widehat{K} $ and $ m_0 = \phi(\widehat{m}_0) $.
	A set $ \mathcal{G}_{\widehat{m}_0} \subset X^s_{\widehat{U}_{\widehat{m}_0}(\varepsilon)} (\sigma) \oplus X^u_{\widehat{U}_{\widehat{m}_0}(\varepsilon)} (\varrho) $ is said to be \emph{$ \mu $-Lip in $ u $-direction} if for any $ (\widehat{m}_i, \overline{x}^s_i) \in X^s_{\widehat{U}_{\widehat{m}_0}(\varepsilon)} (\sigma) $, there is $ \overline{x}^u_i \in X^{u}_{\phi(\widehat{m}_i)} $ such that $ (\widehat{m}_i, \overline{x}^s_i, \overline{x}^u_i) \in \mathcal{G}_{\widehat{m}_0} $, $ i = 1,2 $, and they satisfy
	\[
	|\Pi^u_{m_0}(\overline{x}^u_1 - \overline{x}^u_2)| \leq \mu  \max \{|\Pi^c_{m_0}(\phi(\widehat{m}_1) - \phi(\widehat{m}_2))|, |\Pi^s_{m_0}(\overline{x}^s_1 - \overline{x}^s_2)|\} .
	\]
	A set $ \mathcal{G} \subset X^s_{\widehat{K}_{\varepsilon}} (\sigma) \oplus X^u_{\widehat{K}_{\varepsilon}} (\varrho) $ is said to be \emph{$ \mu $-Lip in $ u $-direction} near $ \widehat{K} $, where $ \mu: \widehat{K} \to \mathbb{R}_+ $ (or $ \mu: K \to \mathbb{R}_+ $), if every $ \mathcal{G}_{\widehat{m}_0} \triangleq \mathcal{G} \cap X^s_{\widehat{U}_{\widehat{m}_0}(\varepsilon)} (\sigma) \oplus X^u_{\widehat{U}_{\widehat{m}_0}(\varepsilon)} (\varrho) $ is $ \mu(\widehat{m}_0) $-Lip in $ u $-direction for $ \widehat{m}_0 \in \widehat{K} $.
\end{defi}

For example, by \autoref{lem:lip2}, the set $ X^s_{\widehat{K}_{\epsilon_*}} (\sigma) $ is $ \chi_{*} $-Lip in $ u $-direction if $ \chi_{*}\epsilon_* \leq \sigma \leq \epsilon_{*} $.
A similar notion of \emph{$ \mu $-Lip in $ s $-direction} can be defined.

\begin{cor}\label{lem:represent2}
	Let $ \epsilon_{*}, \chi(\epsilon_{*}), \chi_{*}, \mu_* > 0 $ be given as in \autoref{lem:lip2}. Take $ 0 < \varepsilon \leq \epsilon_{*} $, $ \chi_{*} \varepsilon \leq \sigma, \varrho \leq \varepsilon $.
	Let $ \mathcal{G} \subset X^s_{\widehat{K}_{\varepsilon}} (\sigma) \oplus X^u_{\widehat{K}_{\varepsilon}} (\varrho) $ be a $ \mu $-Lip in $ u $-direction set with $ \sup_{\widehat{m}_0 \in \widehat{K}}\mu(\widehat{m}_0) \leq \mu_* $.
	\begin{enumerate}[(1)]
		\item $ \mathcal{G} $ is the graph of a bundle map $ h: X^s_{\widehat{K}_{\varepsilon}} (\sigma) \to X^u_{ \widehat{K}_{\varepsilon} } (\varrho), (\widehat{m}, \overline{x}^s) \mapsto h(\widehat{m}, \overline{x}^s) \in X^u_{\phi(\widehat{m})}(\varrho) $ over $ \id $, i.e.,
		\[
		\left\{ (\widehat{m}, \overline{x}^s, h(\widehat{m}, \overline{x}^s)) \triangleq \phi(\widehat{m}) + \overline{x}^s + h(\widehat{m}, \overline{x}^s): (\widehat{m}, \overline{x}^s) \in X^{s}_{\widehat{K}_{\varepsilon}} (\sigma) \right\} = \mathcal{G}.
		\]

		\item Let $ \widehat{m}_0 \in \widehat{K} $, $ m_0 = \phi(\widehat{m}_0) $. Then there is a unique function $ f_{\widehat{m}_0}: X^{c}_{m_0}(c_1\varepsilon) \oplus X^{s}_{m_0}(e_1\sigma) \to X^{u}_{m_0}(e_2\varrho) $ such that
		\begin{multline*}
		\graph f_{\widehat{m}_0} |_{X^{c}_{m_0}(c_1\varepsilon) \oplus X^{s}_{m_0}(e_1\sigma)} \\
		\triangleq \{ m_0 + {x}^c + {x}^s + f_{\widehat{m}_0} ({x}^c, {x}^s): ({x}^c, {x}^s) \in X^{c}_{m_0}(c_1\varepsilon) \oplus X^{s}_{m_0}(e_1\sigma) \} \subset \mathcal{G}_{\widehat{m}_0}.
		\end{multline*}
		Moreover, $ f_{\widehat{m}_0} $ is Lipschitz with Lipschitz constant less than $ \mu_1(\widehat{m}_0) $, where $ \mu_1 = (1 + \chi_*) \mu + \chi_{*} $.
		The constants $ c_i, e_i $, $ i = 1,2 $, are given as in \autoref{lem:tub}.
	\end{enumerate}
\end{cor}
\begin{proof}
	(1) The $ \mu $-Lip property of $ \mathcal{G} $ implies that for any $ (\widehat{m}, \overline{x}^s) \in X^s_{\widehat{K}_{\varepsilon}} (\sigma) $, there is only \emph{one} $ \overline{x}^u \in X^{u}_{\phi(\widehat{m})} $ such that $ (\widehat{m}, \overline{x}^s, \overline{x}^u) \in \mathcal{G} $.
	
	(2) Take $ (\widehat{m}, \overline{x}^s, \overline{x}^u) \in \mathcal{G}_{\widehat{m}_0} $. By \autoref{lem:tub}, we have $ \phi(\widehat{m}) + \overline{x}^s + \overline{x}^u = \phi(\widehat{m}_0) + \widehat{x}^c + \widehat{x}^s + \widehat{x}^u $, where $ \widehat{x}^c \in X^c_{m_0} (c_2\varepsilon) $, $ \widehat{x}^s \in X^s_{m_0} (e_2\sigma) $, $ \widehat{x}^u \in X^u_{m_0} (e_2\rho) $. Due to \autoref{lem:lip2}, we further see that $ \widehat{x}^u $ is uniquely determined by $ (\widehat{x}^c, \widehat{x}^s) $, and thus let $ f_{\widehat{m}_0} (\widehat{x}^c, \widehat{x}^s) = \widehat{x}^u $.
\end{proof}
\begin{cor}\label{lem:represent3}
	Let $ \epsilon_{*}, \chi(\epsilon_{*}), \chi_{*}, \mu_* > 0 $ be given as in \autoref{lem:lip2}. Take $ \varepsilon \leq \epsilon_{*} $, $ \chi_{*} \varepsilon \leq \sigma, \varrho \leq \varepsilon $, and $ \sup_{\widehat{m}_0 \in \widehat{K}}\mu(\widehat{m}_0) \leq \mu_* $. Let $ \widehat{m}_0 \in \widehat{K} $, $ m_0 = \phi(\widehat{m}_0) $. Assume there is a Lipschitz function 
	\[
	f_{\widehat{m}_0}: X^{c}_{m_0}(c_2\varepsilon) \oplus X^{s}_{m_0}(e_2\sigma) \to X^{u}_{m_0}(e_2\varrho)
	\]
	with $ \lip f_{\widehat{m}_0} \leq \mu_1(\widehat{m}_0) $, where $ \mu_1 = (1 + \chi_*) \mu + \chi_{*} $ and the constants $ c_i, e_i $, $ i = 1,2 $, are given as in \autoref{lem:tub}. Then for any $ (\widehat{m}, \overline{x}^s) \in X^s_{\widehat{U}_{\widehat{m}_0}(\varepsilon)} (\sigma) $, there is a unique $ \overline{x}^u \in X^{u}_{\phi(\widehat{m})} (\varrho) $ such that $ (\widehat{m}, \overline{x}^s, \overline{x}^u) \in \graph f_{\widehat{m}_0} |_{X^{c}_{m_0}(c_2\varepsilon) \oplus X^{s}_{m_0}(e_2\sigma)} $.
\end{cor}
\begin{proof}
	The uniqueness of $ \overline{x}^u $ follows from the Lipschitz continuity of $ f_{\widehat{m}_0} $ and \autoref{lem:lip2} (2). Write $ \Phi^{-1}_{m_0,\gamma} = (\Phi^{-c}_{m_0,\gamma}, \Phi^{-s}_{m_0,\gamma}, \Phi^{-u}_{m_0,\gamma}) $. Also, the map 
	\[ 
	( \Phi^{-c}_{m_0,\gamma}(\cdot, \cdot, f_{\widehat{m}_0}(\cdot,\cdot)), \Phi^{-s}_{m_0,\gamma}(\cdot, \cdot, f_{\widehat{m}_0}(\cdot,\cdot)) ): X^c_{m_0}(c_2\varepsilon) \times X^s_{m_0}(e_2\sigma) \to X^c_{m_0}(c_2\varepsilon) \times X^s_{m_0}(e_2\sigma) 
	\] 
	is bi-Lipschitz. This gives the existence of $ \overline{x}^u $.
\end{proof}

\section{Grassmann manifolds: continued}\label{sub:granTwo}

Let us compute the tangent bundle and the normal bundle of $ \overline{\Pi}(X) $ in $ L(X) $. Let $ X_0 \oplus \tilde{X}_0 = X $ and $ m_0 = \Pi_{\tilde{X}_0} (X_0) \in \overline{\Pi}(X) $. Define
\[
\Pi_{m_0}: L(X) \to L(X), ~ L \mapsto \Pi_{\tilde{X}_0} (X_0) L \Pi_{X_0} (\tilde{X}_0) + \Pi_{X_0} (\tilde{X}_0) L \Pi_{\tilde{X}_0} (X_0).
\]
It is straightforward to verify that $ \Pi_{m_0} $ is a projection and $ m_0 \mapsto \Pi_{m_0} $ is smooth.

\begin{lem}\label{lem:Pitangent}
	$ R(\Pi_{m_0}) $ is the tangent space of $ \overline{\Pi}(X) $ at $ m_0 $, i.e.,
	\[
	\frac{|m - m_0 - \Pi_{m_0}(m - m_0)|}{|m - m_0|} \to 0 \quad \text{as} \quad m \in \overline{\Pi}(X) \text{ and } m \to m_0.
	\]
\end{lem}
\begin{proof}
	Consider the local representation at $ m_0 $; see the proof of \autoref{lem:proj} where the notations are also used here. We can write
	\[
	m = \Pi_{\tilde{X}_0}(X_0) (\id + f_1 + f_2)^{-1} + f_1 \circ \Pi_{\tilde{X}_0}(X_0) (\id + f_1 + f_2)^{-1},
	\]
	where $ (f_1, f_2) \in \mathbb{B}_{\varepsilon}(0) \times \mathbb{B}_{\varepsilon}(0) \subset L(X_0,\tilde{X}_0) \times L(\tilde{X}_0,X_0) $. Then
	\begin{multline*}
	m - m_0 - \Pi_{m_0}(m - m_0) = \Pi_{\tilde{X}_0}(X_0) (\id + f_1 + f_2)^{-1} - \Pi_{\tilde{X}_0}(X_0) (\id + f_2)^{-1} \\
	+ f_1(\Pi_{\tilde{X}_0}(X_0) (\id + f_1 + f_2)^{-1} - \Pi_{\tilde{X}_0}(X_0) (\id + f_1)^{-1}),
	\end{multline*}
	which yields
	\[
	m - m_0 - \Pi_{m_0}(m - m_0) = O(\varepsilon^2) + f_1 \circ f_2 + f_1(O(\varepsilon^2)).
	\]
	This gives the proof.
\end{proof}

If $ \mathcal{K} \subset \overline{\Pi}(X) $ is a bounded set, then $ (\overline{\Pi}(X), \mathcal{K}, \{ \Pi_{m_0} \}, \{\overline{\Pi}(X) \cap \mathbb{B}_{m_0}(\varepsilon)\}) $ satisfies (H1)--(H4) in \autoref{sub:setup} for small $ \varepsilon $ depending only on $ \sup_{\Pi \in \mathcal{K}} |\Pi| $. Indeed, let $ \sup_{\Pi \in \mathcal{K}} |\Pi| \leq C_0 < \infty $. Then for any $ \Pi \in \mathcal{K} $, we have
\[
\min \{\alpha(R(\Pi), R(\id - \Pi)), \alpha(R(\id - \Pi), R(\Pi))\} = \min \{|\Pi|^{-1}, |\id - \Pi|^{-1}\}  \geq (1 + C_0)^{-1} > 0.
\]
From \autoref{lem:Pitangent}, (H1) is satisfied. (H4) holds by \autoref{lem:granlip} and \eqref{equ:Pilocal}. From the computation in \autoref{lem:Pitangent} and the fact that for $ (f_1, f_2) \in L(X_0,\tilde{X}_0) \times L(\tilde{X}_0,X_0) $,
\[
|f_1 + f_2| \leq |f_1| + |f_2| \leq (1 + |\Pi_{\tilde{X}_0}(X_0)|) |f_1 + f_2|,
\]
we see that (H3) is satisfied. Finally, by the definition of $ \Pi_{m} $, together with \autoref{lem:granlip} and \eqref{equ:Pilocal}, (H2) is fulfilled.

In particular, by \autoref{lem:tub} and \autoref{lem:lip2}, we obtain the following.
\begin{lem}\label{lem:Piret}
	If $ \mathcal{K} \subset \overline{\Pi}(X) $ is a bounded set, then there are $ \varepsilon, \varrho > 0 $ and a smooth Lipschitz map $ r: \mathbb{B}_{\varrho}(\mathcal{K}) \to O_{\varepsilon}(\mathcal{K}) \triangleq \mathbb{B}_{\varepsilon}(\mathcal{K}) \cap \overline{\Pi}(X) $ such that $ r|_{O_{\varepsilon}(\mathcal{K})} = \id $, where $ \mathbb{B}_{\varrho}(\mathcal{K}) = \{ L \in L(X): d(L, \mathcal{K}) < \varrho \} $. The natural embedding of $ O_{\varepsilon}(\mathcal{K}) $ into $ L(X) $ is $ C^{\infty} \cap C^{0,1} $.
\end{lem}

Such a map $ r $ is often called a \emph{retraction}.

To summarize, we have the following facts:
\begin{enumerate}[(1)]
	\item $ \mathbb{K}(X) $ is a complete metric space with respect to the metric $ \widehat{d} $.
	\item $ \mathbb{G}(X) $ is a $ C^{\infty} $ paracompact Banach manifold locally modeled on $ L(X_1, X_2) $ (where $ X_1 \oplus X_2 = X $) with the metric $ \widehat{d} $. Moreover, $ \mathbb{G}(X) $ is an open subset of $ \mathbb{K}(X) $. The boundary $ \partial \mathbb{G}(X) = \emptyset $ if and only if $ X $ admits a Hilbert inner product (a well known result due to Lindenstrauss and Tzafriri).
	\item $ \overline{\Pi}(X) $ is a $ C^{\infty} $ closed submanifold of $ L(X) $; in addition, for any bounded set $ \mathcal{K} \subset \overline{\Pi}(X) $, there is a smooth and Lipschitz retraction that maps a neighborhood of $ \mathcal{K} $ in $ L(X) $ into a neighborhood of $ \mathcal{K} $ in $ \overline{\Pi}(X) $.
\end{enumerate}

Finally, let us consider the continuous choice of complemented spaces; for the (strong) measurability version in separable Banach spaces, see e.g. \cite[Chapter 7]{LL10}. In general, one cannot expect such a choice to be uniformly continuous or Lipschitz; however, we seek conditions under which the following statement holds: if $ m \mapsto X_{m} $ is Lipschitz, then there is a Lipschitz selection of complemented spaces $ X^{h}_{m} $ such that $ X_{m} \oplus X^{h}_{m} = X $. Note that the following lemma becomes trivial when $ Y_{m} $ admits a Hilbert inner product $ \langle\cdot,\cdot\rangle_{m}: Y_{m} \times Y_{m} \to \mathbb{R} $ (or $ \mathbb{C} $) with $ m \mapsto \langle\cdot,\cdot\rangle_{m} $ continuous (or Lipschitz for the Lipschitz choice case), for example, when $ X $ is a Hilbert space or $ Y_{m} $ is finite-dimensional; in such cases, the condition that $ X_{m} $ is finite-dimensional is not required.
\begin{lem}\label{lem:selection}
	Let $ \mathcal{N} $ be a paracompact (Hausdorff) topological space. Assume $ m \mapsto X_{m}, Y_{m}: \mathcal{N} \to \mathbb{G}(X) $ are continuous, with the dimension of $ X_{m} $ equal to $ n $ and $ X_{m} \subset Y_{m} $ for all $ m \in \mathcal{N} $. Then there is a continuous map $ m \mapsto X^{h}_{m}: \mathcal{N} \to \mathbb{G}(X) $ such that $ X_{m} \oplus X^{h}_{m} = Y_{m} $ with $ |\Pi_{X^{h}_{m}} (X_{m})| \leq \beta $, where $ \beta $ is a constant depending only on $ n $.
\end{lem}
\begin{proof}
	We do not aim to find the optimal constant $ \beta $, but note that $ \beta > n $ suffices for our purposes. The following fact is straightforward:

	If $ X_1 \oplus X_2 = X $ and $ X_1 \subset Y \subset X $, then $ Y = X_1 \oplus (X_2 \cap Y) $ with $ \Pi_{X_2 \cap Y} (X_1) = \Pi_{X_2} (X_1)|_{Y} $.

	The following result is due to N. J. Kalton (see \cite[Proposition 2.4 (ii)]{Kal08}) whose proof is also presented here for the convenience of the readers.
	\begin{slem}
		Given $ \gamma > 1 $, there is a continuous map $ \Theta: \mathbb{S}_{X} \to L(X) $ such that $ \Theta(x): X \to \mathrm{span} \{ x \} $ is a projection onto $ \mathrm{span} \{ x \} $ with $ |\Theta(x)| \leq \gamma $ and $ \Theta(\alpha x) = \Theta(x) $ where $ |\alpha| = 1 $.
	\end{slem}
	\begin{proof}
		This is a standard application of the Michael continuous selection (see \cite{Mic56}). Fix $ \gamma > 1 $. For each $ x \in \mathbb{S}_{X} $, define
		\[
		\widehat{\Phi}(x) = \{ \Pi: \Pi \text{ is a projection onto } \mathrm{span} \{ x \} \text{ with } |\Pi| \leq \gamma \}.
		\]
		It is rapidly seen that $ \widehat{\Phi}(x) \neq \emptyset $ is closed and convex. Let us show $ x \mapsto \widehat{\Phi}(x) $ is lower-semicontinuous. Given an open set $ U \subset L(X) $ such that $ \widehat{\Phi}(x) \cap U \neq \emptyset $ and a sufficiently small $ \varepsilon > 0 $, it suffices to show that if $ |y| = 1 $ and $ |y - x| < \varepsilon $, then $ \widehat{\Phi}(y) \cap U \neq \emptyset $. Take $ x^* \in X^{*} $ such that $ x^{*}(x) = 1 $ and $ |x^*| = 1 $, and write $ P = \Pi_{\ker x^*} (\mathrm{span}\{x\}) $; note that $ |P| = 1 $. Take $ \widetilde{P} \in \widehat{\Phi}(x) \cap U $. Then $ P_1 = a \widetilde{P} + (1 - a) P \in \widehat{\Phi}(x) \cap U $ and $ |P_1| < \gamma $ if $ a \in (0, 1) $ is sufficiently close to $ 1 $. Define $ Lz = x^{*}(z) (x - y) $ for $ z \in X $. Then $ \id - L $ is invertible and maps $ \mathrm{span}\{ x \} $ onto $ \mathrm{span}\{ y \} $. Set $ \Pi = (\id - L) P_1 (\id - L)^{-1} $. This yields $ \Pi \in \widehat{\Phi}(y) \cap U $ for sufficiently small $ \varepsilon $.

		Therefore, by the Michael continuous selection (see \cite{Mic56}), there is a continuous map $ \Theta_0: \mathbb{S}_{X} \to L(X) $ such that $ \Theta_0(x) \in \widehat{\Phi}(x) $ for each $ x \in \mathbb{S}_{X} $. Define
		\[
		\Theta(x) = \begin{cases}
		(\Theta_0(x) + \Theta_0(-x)) / 2, & \text{$ X $ is real}, \\
		(2\pi)^{-1}\int_{0}^{2\pi}\Theta_0(e^{i\theta}x) ~\mathrm{d} \theta, & \text{$ X $ is complex}.
		\end{cases}
		\]
		The map $ \Theta $ satisfies the desired properties.
	\end{proof}

	If $ n = 1 $, then $ \Theta(\mathbb{S}_{X} \cap X_{m}) $ is well defined since $ \Theta(\alpha x) = \Theta(x) $ for $ |\alpha| = 1 $, and hence $ m \mapsto \Theta(\mathbb{S}_{X} \cap X_{m}) $ is continuous. So we can take $ X^{h}_{m} = \ker (\Theta(\mathbb{S}_{X} \cap X_{m})) \cap Y_{m} = (\id - \Theta(\mathbb{S}_{X} \cap X_{m})) Y_{m} $, which yields a continuous map $ m \mapsto X^{h}_{m} $ such that $ X_{m} \oplus X^{h}_{m} = Y_{m} $ and $ \Pi_{X^{h}_{m}}(X_{m}) \leq \gamma $.

	Consider the general case $ n $. Fix $ \gamma > \gamma_1 > 1 $. Define
	\[
	\Phi(m) = \{ \Pi: \Pi \text{ is a projection of $ X $ onto $ X_{m} $ with } |\Pi| \leq n\gamma \}.
	\]
	$ \Phi(m) $ is closed, convex, and nonempty. 
	
	Next, we show that $ m \mapsto \Phi(m) $ is lower-semicontinuous. Take any open set $ U \subset L(X) $ such that $ \Phi(m_0) \cap U \neq \emptyset $, and let $ P_0 \in \Phi(m_0) \cap U $.

	(1) In a neighborhood $ O_{m_0} \subset \mathcal{N} $ of $ m_0 $, there is a continuous map $ m \mapsto e_{1}(m) \in X_{m} $ with $ |e_{1}(m)| = 1 $. By the case $ n = 1 $, we have $ \mathrm{span}\{e_{1}(m)\} \oplus \widehat{X}^1_{m} = X $ with $ m \mapsto \widehat{X}^1_{m} $  continuous and $ |\Pi_{\widehat{X}^1_{m}} (\mathrm{span}\{e_{1}(m)\})| \leq \gamma_1 $ for $ m \in O_{m_0} $.

	(2) Similarly, for $ \widehat{X}^1_{m} \cap X_{m} $, if $ O_{m_0} $ is ``small'', there is a continuous map $ m \mapsto e_{2}(m) \in \widehat{X}^1_{m} \cap X_{m} $ with $ |e_{2}(m)| = 1 $, and $ \mathrm{span}\{e_{2}(m)\} \oplus \widehat{X}^2_{m} = \widehat{X}^1_{m} $, where $ m \mapsto \widehat{X}^2_{m} $ is continuous and $ |\Pi_{\widehat{X}^2_{m}} (\mathrm{span}\{e_{2}(m)\})| \leq \gamma_1 $ for $ m \in O_{m_0} $.

	(3) Proceeding inductively, we obtain a continuous map $ m \mapsto e_{n}(m) \in \widehat{X}^{n-1}_{m} \cap X_{m} $ with $ |e_{n}(m)| = 1 $, and $ \mathrm{span}\{e_{n}(m)\} \oplus \widehat{X}^{n}_{m} = \widehat{X}^{n-1}_{m} $, where $ m \mapsto \widehat{X}^{n}_{m} $ is continuous and $ |\Pi_{\widehat{X}^{n}_{m}} (\mathrm{span}\{e_{n}(m)\})| \leq \gamma_1 $ for $ m \in O_{m_0} $.

	Thus, we obtain $ X_{m} = \mathrm{span} \{ e_1(m), e_2(m), \ldots, e_{n}(m) \} $, and
	\[
	X = \mathrm{span}\{e_{1}(m)\} \oplus \widehat{X}^1_{m} = \mathrm{span}\{e_{1}(m)\} \oplus \mathrm{span}\{e_{2}(m)\} \oplus \widehat{X}^2_{m} = \cdots = X_{m} \oplus \widehat{X}^n_{m}.
	\]
	Moreover, $ |\Pi_{\widehat{X}^n_{m}} (X_{m})| \leq n\gamma_1 $.
	Note that $ \Pi_{\ker P_{0}} (X_{m}) \to P_{0} $ as $ m \to m_0 $. Hence, there is $ \epsilon_{m} > 0 $ such that $ |\Pi_{\ker P_{0}} (X_{m})| \leq n\gamma + \epsilon_{m} $ with $ \epsilon_{m} \to 0 $ as $ m \to m_0 $. Set
	\[
	P_{m} = \epsilon^{1/2}_{m} \Pi_{\widehat{X}^n_{m}} (X_{m}) + (1 - \epsilon^{1/2}_{m}) \Pi_{\ker P_{0}} (X_{m}).
	\]
	Then $ |P_{m}| \leq \epsilon^{1/2}_{m} n\gamma_1 + (1 - \epsilon^{1/2}_{m}) (n\gamma + \epsilon_{m}) \leq n\gamma $ for $ m $ close to $ m_0 $, due to $ \gamma - \gamma_1 > 0 $ and $ P_{m} \to P_{m_0} $. That is, $ P_{m} \in \Phi(m) \cap U $ if $ m $ is close to $ m_0 $.

	Therefore, by the Michael continuous selection (see \cite{Mic56}), we get $ \Pi_{m} \in \Phi(m) $ with $ m \mapsto \Pi_{m} $ continuous. Define $ X^{h}_{m} = \ker \Pi_{m} \cap Y_{m} $ for $ m \in \mathcal{N} $ and complete the proof.
\end{proof}

\chapter{Existence of a center-stable manifold: proof of \autoref{thm:I}} \label{sec:existence}

\section{Preparation for proofs}\label{sub:preparation}

We begin with some preliminary observations. The following fact shows that the property of $H$ remains preserved under small perturbations.
\begin{slem}\label{lem:qq}
	Suppose $ \widehat{H} \approx (\widehat{F}^{cs}, \widehat{G}^{cs}) $ satisfies the (A$ '$)($ \alpha $, $ \lambda_{u} $) (B)($ \beta; \beta', \lambda_{cs} $) condition in $cs$-direction at $K$ (see \autoref{defi:ABk}).
	Then there is a small $ \epsilon'_2 > 0 $ such that for all $ m \in K $, if $ |m' - u(m)| < \epsilon'_2 $ and $ \widetilde{\Pi}_1^\kappa, \widetilde{\Pi}_2^\kappa \in \overline{\Pi}(X) $ satisfy $ |\widetilde{\Pi}^{\kappa}_1 - \widehat{\Pi}^{\kappa}_{m}| \leq \epsilon'_2 $, $ |\widetilde{\Pi}^{\kappa}_2 - \widehat{\Pi}^{\kappa}_{u(m)}| \leq \epsilon'_2 $ for $ \kappa = cs, u $, then for $ r_* < r / 4 $,
	\[
	H(m + \cdot) - m' \sim (\widehat{F}'_{m, m'}, \widehat{G}'_{m, m'}): \widetilde{X}^{cs}_1(r_*) \oplus \widetilde{X}^{u}_1(r_1) \to \widetilde{X}^{cs}_{2}(r_2) \oplus \widetilde{X}^{u}_{2}(r_*)
	\]
	satisfies the (A$ '$)($ \widetilde{\alpha}(m), \widetilde{\lambda}_u(m) $) (B)($ \widetilde{\beta}(m); \widetilde{\beta}'(m), \widetilde{\lambda}_{cs}(m) $) condition, where
	\begin{enumerate}[(i)]
		\item $ \widetilde{X}^{\kappa}_j = R(\widetilde{\Pi}^{\kappa}_j) $ for $ \kappa = cs, u $, $ j = 1,2 $;
		\item $ \widetilde{\alpha}(m) = (1+\epsilon'_0)\alpha(m) + \epsilon'_0 $, $ \widetilde{\beta}'(m) = (1+\epsilon'_0)\beta'(m) + \epsilon'_0 $, $ \widetilde{\lambda}_\kappa(m) = (1+\epsilon'_0)\lambda_\kappa(m) + \epsilon'_0 $ for $ \kappa = cs, u $, and $ \widetilde{\beta}(m) = \frac{1-\epsilon'_0\beta_0}{1+\epsilon'_0}\beta(m) $ for some $ 0 < \epsilon'_{0} < \beta_0^{-1} $ where $ \beta_0 = \sup_{m} \beta(m) $;
		\item If, in addition, $ |\widehat{F}^{cs}_{m}(0, 0)| \leq \eta $, $ |\widehat{G}^{cs}_{m}(0,0)| \leq \eta $, then $ |\widehat{F}'_{m, m'}(0,0)| \leq (\epsilon'_0 + 1)\eta $, $ |\widehat{G}'_{m, m'}(0,0)| \leq (\epsilon'_0 + 1)\eta $;
		\item Moreover, $ \epsilon'_{0} \to 0 $ as $ \epsilon'_2 \to 0 $.
	\end{enumerate}
\end{slem}

\begin{proof}
	By \autoref{lem:gram2}, $ \widetilde{X}^{\kappa}_1, \widetilde{X}^{\kappa}_2 $ are close to $  \widehat{X}^{\kappa}_{m}, \widehat{X}^{\kappa}_{u(m)} $ as $\epsilon'_2 \to 0$ (provided $ \epsilon'_2 < 1/(2\widetilde{M}) $ where $ \widetilde{M} $ is the constant from (H2) (i)). Write
	\[
	u(m) = m' + x_0 + y_0,
	\]
	where $x_0 \in \widetilde{X}^{cs}_2$, $y_0 \in \widetilde{X}^{u}_2$. For $ \widetilde{x}_1 \in \widetilde{X}^{cs}_1(r_*)$ and $ \widetilde{y}_2 \in \widetilde{X}^{u}_2(r_*) $, set
	\begin{gather*}
		x_1 = \widehat{\Pi}^{cs}_m \widetilde{x}_1, \quad y_2 = \widehat{\Pi}^{u}_{u(m)} (\widetilde{y}_2 - y_0),\\
		\widetilde{x}_2 = \widetilde{\Pi}^{cs}_{2} (\widehat{F}^{cs}_m(x_1, y_2) + y_2) + x_0, \quad
		\widetilde{y}_1 = \widetilde{\Pi}^{u}_{1} (\widehat{G}^{cs}_m(x_1, y_2) + x_1).
	\end{gather*}
	Define $ \widehat{F}'_{m, m'}(\widetilde{x}_1, \widetilde{y}_2) = \widetilde{x}_2 $ and $ \widehat{G}'_{m, m'}(\widetilde{x}_1, \widetilde{y}_2) = \widetilde{y}_1 $. Then $ H(m + \cdot) - m' \sim (\widehat{F}'_{m, m'}, \widehat{G}'_{m, m'}) $. As $ \widetilde{X}^{\kappa}_1, \widetilde{X}^{\kappa}_2 $ are close to $  \widehat{X}^{\kappa}_{m}, \widehat{X}^{\kappa}_{u(m)} $ and $|x_0|, |y_0|$ are small when $|m' - u(m)|$ is small, there is $ \epsilon'_0 > 0 $ such that:
	\begin{gather*}
		(1 + \epsilon'_0)^{-1} |\widetilde{x}_1| \leq |x_1| \leq (1 + \epsilon'_0) |\widetilde{x}_1|, \quad
		(1 + \epsilon'_0)^{-1} |\widetilde{y}_2| \leq |y_2| \leq (1 + \epsilon'_0) |\widetilde{y}_2|, \\
		|\widetilde{x}_2| \leq (1 + \epsilon'_0) |x_2| + \epsilon'_0 |y_2|, \quad
		|\widetilde{y}_1| \leq (1 + \epsilon'_0) |y_1| + \epsilon'_0 |x_1|.
	\end{gather*}
	The conclusion now follows from the above estimates.
\end{proof}

Hereafter, when referring to assumption (A3) (a) in \autoref{subsec:main}, we always consider the ($ \bullet 1 $) case, i.e.,
\begin{center}
	$ \widehat{H} \approx (\widehat{F}^{cs}, \widehat{G}^{cs}) $ satisfies the (A$ '$)($ \alpha $, $ \lambda_{u} $) (B)($ \beta; \beta', \lambda_{cs} $) condition in $cs$-direction at $K$.
\end{center}

In this chapter, we consider the following two distinct cases.

\vspace{.5em}
\noindent{Assumptions}.
Case (1): Assume (A1), (A2), (A3) (a) (i) (iii) (excluding (A2) (a) (ii)) and (A3) (b) in \autoref{subsec:main} hold.

Case (2): Assume (A1), (A2), (A3) (a) (i) (iii) with $ \varsigma_0 \geq 1 $ and (A3) (b) in \autoref{subsec:main} hold; in addition, there are small $ \gamma_0 > 0 $ and $ 0 < \gamma^{*}_{u} < 1 $ such that for all $ m \in K $ and $ (x^{cs}, x^{u}) \in \widehat{X}^{cs}_{m}(r) \times \widehat{X}^{u}_{u(m)}(r) $,
\begin{equation}\label{equ:est}
|\widehat{G}^{cs}_{m}(x^{cs}, x^{u})| \leq \gamma_0|x^{cs}| + \gamma^{*}_{u}|x^{u}| + \eta.
\end{equation}
Note that under (A3) (a$ '$), the inequality \eqref{equ:est} holds (with $ \gamma^{*}_{u} = \sup_{m} \lambda_{u} (m) < 1 $).
\vspace{.5em}

In what follows, we fix a function (depending on $ \epsilon_{*} $)
\begin{equation}\label{equ:o(1)}
\epsilon_{0,*} = \epsilon_{0,*}(\epsilon_{*}) ~\text{such that}~ \epsilon_{0,*} \to 0 ~\text{as}~ \epsilon_{*} \to 0,
\end{equation}
i.e., $ \epsilon_{0,*} = O_{\epsilon_{*}}(1) $, e.g., $ \epsilon_{0,*} = \epsilon_{*} $.

\begin{asparaenum}[({Observation} I).]
	\item \label{OI} Let $ \epsilon'_2 > 0 $ and $ \widetilde{\alpha}, \widetilde{\beta}, \widetilde{\beta}', \widetilde{\lambda}_{u}, \widetilde{\lambda}_{cs} $ be as in \autoref{lem:qq} such that if $ \alpha, \beta, \beta', \lambda_{u}, \lambda_{cs} $ satisfy (A3) (a) (or (A3) (a$ '$)), then so do $ \widetilde{\alpha}, \widetilde{\beta}, \widetilde{\beta}', \widetilde{\lambda}_{u}, \widetilde{\lambda}_{cs} $.
	Let
	\[
	\hat{\beta}'' = \inf_{m \in K} \widetilde{\beta}'(m) > 0, ~\text{and}~ r_0 = r_* / 8  ~(\text{where $ r_* $ is given in \autoref{lem:qq}}).
	\]

	\item \label{OII} By \autoref{lem:qq}, we may assume without loss of generality that $ \xi_2 = 0 $ (in (A3) (b) (ii)), i.e., $ \widehat{\Pi}^\kappa_m = \Pi^\kappa_m $ for $ m \in K $ and $ \kappa = s, c, u $.

	\item \label{OIII} In the following, $ \epsilon_2 $ will be taken even smaller; first require that $ \epsilon_2 $ satisfies (provided $ \chi(\epsilon_2) $ is small)
	\[
	L\epsilon_2 \leq \epsilon'_2 < 1, (\widetilde{C}' + 1) \chi(\epsilon_{2}) < 1, ~\text{where}~  \widetilde{C}' = \sup_{m_0 \in K}\{ \widetilde{\alpha}(m_0) + \widetilde{\lambda}_{cs}(m_0) \}.
	\]

	\item \label{OIV} In case (1), since $ \inf_{m \in K} \{ \beta(m) - \varsigma_0\beta'(u(m)) \} > 0 $ (from (A3) (a) (i)), there is $ \varsigma > 2 $ close to $ 2 $ such that $ \inf_{m \in K}\{ \beta(m) - \varsigma \beta'(u(m)) \} > 0 $.
	Similarly, in case (2), we can let the following \eqref{equ:pp} holds with $ \varsigma > 1 $ but close to $ 1 $; \emph{we fix $ \varsigma > 1 $ when considering case (2)}. In (A3) (a) (iii), assume $ \xi_1 $ (and $ \epsilon_2 $) is small such that
	\begin{equation}\label{equ:pp}
	\inf_{\widehat{m} \in \widehat{K}} \left\{ \widetilde{\beta}(\phi(\widehat{m})) -  \varsigma\max\{ \widetilde{\beta}'(\phi(\widehat{m}')): \widehat{m}' \in \widehat{U}_{\widehat{u}(\widehat{m})}(\epsilon_2) \cap \widehat{K} \}  \right\} > 0,
	\end{equation}
	and for all $ \widehat{m} \in \widehat{K} $,
	\begin{equation}\label{equ:ppb}
	\max\left\{ \widetilde{\beta}'(\phi(\widehat{m}')): \widehat{m}' \in \widehat{U}_{\widehat{m}}(\epsilon_2) \cap \widehat{K} \right\} < \hat{\varsigma} \widetilde{\beta}'(\phi(\widehat{m})),
	\end{equation}
	where $ 1 < \hat{\varsigma} < \frac{\varsigma-1}{\varsigma' - 1} $ and $ \varsigma' = \frac{\varsigma + 2}{2} $ in case (1) and $ \varsigma' = \frac{\varsigma + 1}{2} $ in case (2). Note that \eqref{equ:ppb} can hold due to $ \hat{\varsigma} > 1 $ and $ \inf_{m \in K} \widetilde{\beta}'(m) > 0 $ (in Observation \eqref{OI}).

	\item \label{OV} Let $ \epsilon_{*}, \chi(\epsilon_{*}) $ be sufficiently small such that \autoref{lem:lip2} and \autoref{lem:represent2} hold. Let $ \chi_{*} < 1/16 $ and take $ c_i, e_i $ ($ i = 1,2 $) as in \autoref{lem:lip2}. Note that $ \chi_{*} \to 0 $ and $ c_i, e_i \to 1 $ as $ \epsilon_{*}, \chi(\epsilon_{*}) \to 0 $; for example, $ e_1 = c_1 \leq 1 - 1/(2^{10}) $ and $ e_2 = c_2 = e^{-1}_1 $. In what follows, $ \epsilon_{*}, \chi(\epsilon_{*}) $ will be taken even smaller; first require $ \epsilon_{*} < \min\{r_0, \epsilon_2\} / 8 $. In (A2) (ii), assume $ \xi_1 $ is small such that
	\[
	\sup \{ |u(m) - u(m_0)|: m \in U_{m_0,\gamma}(2\epsilon_*) \cap K, m_0 \in K, \gamma \in \Lambda(m_0) \} \leq 2\xi_1 < \epsilon_2 / 2.
	\]

	\item \label{OVI} Let $ \hat{\alpha}, \hat{\beta} > 0 $ such that
	\[
	\sup_m \widetilde{\alpha}(m) \leq \hat{\alpha} < \infty, \quad \sup_m \widetilde{\beta}(m) \leq \hat{\beta} < \infty.
	\]
	Take functions $ \mu, \mu_1: \widehat{K} \to \mathbb{R}_+ $ as
	\[
	\mu_1 = (1 + \chi_*) \mu + \chi_{*}, \quad \mu = (1 + \chi_*) \varsigma\widetilde{\beta}' \circ \phi + \chi_{*}.
	\]
	Let $ \epsilon_{*}, \chi(\epsilon_{*}) $ be further reduced such that $ \chi_{*} $ is small enough to satisfy the following:
	\begin{enumerate}[(i)]
		\item (due to \eqref{equ:pp})
		\[
		\varsigma\widetilde{\beta}'(\phi(\widehat{m}')) <  \mu(\widehat{m}') < \mu_1(\widehat{m}') < \widetilde{\beta}(\phi(\widehat{m})) < \beta(\phi(\widehat{m})), \widehat{m} \in \widehat{K}, \widehat{m}' \in \widehat{U}_{\widehat{m}}(\epsilon_2) \cap \widehat{K};
		\]
		\item $ \chi_{*} \hat{\beta} < 1 / 8 $ and $ \chi_{*} < \hat{\beta}'' $.
	\end{enumerate}

	\item \label{OVII} Set
	\[
	\gamma \triangleq \sup_{\widehat{m} \in \widehat{K}} \alpha(\phi(\widehat{m})) \mu_1(\widehat{u}(\widehat{m})), \quad \overline{\lambda}_u \triangleq \sup_{\widehat{m} \in \widehat{K}}\frac{\lambda_{u}(\phi(\widehat{m}))}{1 - \alpha(\phi(\widehat{m})) \mu_1(\widehat{u}(\widehat{m}))}.
	\]
	As $ \sup_{m \in K} \alpha(m) \beta'(u(m)) < 1/(2\varsigma_0) $ (in (A3) (a) (i)), we can ensure that $ \gamma < 1 / 2 $ (provided $ \chi_{*}, \xi_1 $ are sufficiently small and $ \varsigma $ is close to $ 2 $ in case (1) and close to $ 1 $ in case (2)); moreover, if (A3) (a) (ii) holds, then we can assume $ \overline{\lambda}_u < 1 $.
\end{asparaenum}

To simplify notation, for $ \kappa_1 \neq \kappa_2 \in \{s, c, u\} $, set
\[
X^{\kappa_1\kappa_2}_{\widehat{K}'} (\sigma, \varrho) \triangleq X^{\kappa_1}_{\widehat{K}'} (\sigma) \oplus X^{\kappa_2}_{\widehat{K}'} (\varrho), X^{\kappa_1\kappa_2}_{\widehat{K}'} (\sigma) = X^{\kappa_1\kappa_2}_{\widehat{K}'} (\sigma, \sigma), \widehat{K}' \subset \widehat{\Sigma},
\]
and
\[
\overline{X^u_{\widehat{\Sigma}} (\varrho_*)} = \{ (\widehat{m}, x^u): x^u \in X^u_{\phi(\widehat{m})}, |x^u| \leq \varrho_*, \widehat{m} \in \widehat{\Sigma} \}.
\]

\vspace{.5em}
\noindent{\textbf{Choice of constants.}}
Let $ \hat{\chi}_{*} $ be any function depending on $ \epsilon_{*}, \chi(\epsilon_{*}) $ such that
\[
\chi_{*} \leq \hat{\chi}_{*}, \quad \text{and} \quad \hat{\chi}_{*} = \hat{\chi}_{*}(\epsilon_{*}, \chi(\epsilon_{*})) \to 0 \quad \text{as} \quad \epsilon_{*}, \chi(\epsilon_{*}) \to 0.
\]
Let us take $ \sigma_{*}, \varrho_{*}, \eta_0 = o(\epsilon_{*}) $; in addition, in case (2), require that $ \varrho_{*} / \sigma_{*} \to 0 $ and $ \gamma_0 \to 0 $ such that $ (\gamma_0\sigma_{*}) / \varrho_{*} \to 0 $. For example,
\[
\begin{cases}
\eta_0 = \hat{\chi}_{*}\epsilon_{*} / 16 \leq \epsilon_{*} / 8 \leq r_0 / 64, \\
\varrho_* = e^{-1}_1( \eta_0 + \max\{2\hat{\beta} , 1\} \hat{\chi}_{*}\epsilon_{*}) \leq \epsilon_{*} / 16;
\end{cases}
\]
and
\[
\sigma_* = \begin{cases}
2\hat{\chi}_{*}\epsilon_{*}, &\text{in case (1)}, \\
2\hat{\chi}^{1/2}_{*}\epsilon_{*}, &\text{in case (2)}, \\
\end{cases} 
\]
(so that $ \sigma_* \leq \min\{\epsilon_{*}, \min\{\hat{\beta}^{-1}, 1\} (r_0 - \eta_0), \epsilon_2 \} / 8 $), with $ \gamma_0 \leq \hat{\chi}^{1.1/2}_{*} $ in case (2) where $ \gamma_0 $ is defined in \eqref{equ:est}.

Note that $ \chi_{*} \epsilon_{*} \leq \sigma_*, \varrho_* \leq \sigma_* + \varrho_* \leq  \epsilon_{*} / 4 $. The primary motivation for choosing these constants is to ensure that \eqref{equ:nnn} and the subsequent inequalities \eqref{equ:mmm} and \eqref{equ:mm11} hold. In what follows, \emph{$ \epsilon_{*}, \chi(\epsilon_{*}) $ (and also $ \eta \leq \epsilon_{0, *} \eta_0 $) will be taken even smaller.} 

For the reader's convenience, we list below other constants that will be used in the proofs:

\begin{enumerate}[$ \bullet $]
	\item $ K_2 $ (see \eqref{equ:KKK}) and $ K'_1 $ (see \eqref{equ:K00}), which depend on the constant $ K_1 > 1 $;

	\item $ \sigma_{0} = e_1 \sigma_{*} $ (see \eqref{equ:sss}), $ \sigma^c_{*} $ (see \eqref{equ:csss}), and $ \sigma^1_{*} $ (see \eqref{equ:s111});

	\item $ \eta_1, \eta_2 $ in the definition of the bump function $ \Psi $ (see \eqref{equ:cutoff} and \autoref{lem:belong});

	\item $ \widehat{\lambda}_{u} < 1 $ in \autoref{lem:contractive};

	\item $ \varepsilon_{0} = \min\{e^{-1}_2,c^{-1}_2\} e_0 \eta_2 $, where $ e_0 = \min\{ e_1, c_1 \} $ (in \autoref{lem:existence});

	\item $ \varpi_1^* > 1 $: constants such that $ \varpi_1^* \to 1 $ as $ \epsilon_{*}, \chi(\epsilon_{*}) \to 0 $;

	\item $ \varpi^*_0 = \varpi_1^* - 1 > 0 $: constants such that $ \varpi^*_0 \to 0 $ as $ \epsilon_{*}, \chi(\epsilon_{*}) \to 0 $;

	\item $ \widetilde{C} $: universal constants independent of (small) $ \epsilon_{*}, \chi(\epsilon_{*}), \eta $.
\end{enumerate}

\section{Construction of graph transform} \label{sub:graph}

In the following, we focus on two distinct cases: the unlimited case (see \autoref{sub:unlimited}) and the limited case (see Sections \ref{sub:limited}--\ref{sub:limited*}), both of which are essential for our subsequent analysis.

\subsection{Unlimited case}\label{sub:unlimited}
For any $ K_1 > 1 $, define
\begin{equation}\label{equ:KKK}
K_2 = K_2(K_1) \triangleq \frac{\hat{\alpha} K_1 + 1}{1 - 2\gamma}.
\end{equation}
Let $ \eta $ be small such that e.g. case (1) (taking $ \gamma^{*}_{u} = 0 $),
\begin{equation}\label{equ:small000}
(1 - \gamma^{*}_{u})^{-1}(\hat{\beta} + 1)(K_1 + K_2 + \overline{\lambda}_u ( \hat{\beta} + K_1 ) + 1) \eta \leq \eta_0 / 2,
\end{equation}
which ensures that the following inequalities \eqref{**1}, \eqref{**2}, and \eqref{**3} are satisfied, and
\begin{equation}\label{equ:sss}
\sigma_{0} = e_1\sigma_{*} < c_1\epsilon_{*} < \min\{ \hat{\beta}^{-1} (r_0 - \eta_0) / 2, r_0/2 \}.
\end{equation}
Recall the definition of $ \mu $-Lip in $ u $-direction near $ \widehat{K} $ (see \autoref{def:lip}). Define
\begin{multline}\label{equ:space0}
\varSigma_{\mu, K_1, \epsilon_{*}, \sigma_{*}, \varrho_* } = \{ h: X^s_{\widehat{\Sigma}} (\sigma_*) \to \overline{X^u_{\widehat{\Sigma}} (\varrho_*)} \text{ is a bundle map over } \id: \\
\sup_{\widehat{m} \in \widehat{K}} |h(\widehat{m}, 0)| \leq K_1\eta,
 \graph h \cap X^{su}_{\widehat{K}_{\epsilon_*}} (\sigma_*, \varrho_*) \text{ is $ \mu $-Lip in $ u $-direction near $ \widehat{K} $}\}.
\end{multline}

Take $ h \in \varSigma_{\mu, K_1, \epsilon_{*}, \sigma_{*}, \varrho_* } $. For any $ \widehat{m}_0 \in \widehat{K} $ with $ m_0 = \phi(\widehat{m}_0) $, let $ f_{\widehat{m}_0} $ be the local representation of $ \graph h \cap X^{su}_{\widehat{K}_{\epsilon_*}} (\sigma_*, \varrho_*) $ at $ \widehat{m}_0 $. By \autoref{lem:represent2},
\[
f_{\widehat{m}_0}: X^{c}_{m_0}(c_1\epsilon_*) \oplus X^{s}_{m_0}(e_1\sigma_*) \to X^{u}_{m_0}(e_2\varrho_*),
\]
with $ \lip f_{\widehat{m}_0}(\cdot) \leq \mu_1(\widehat{m}_0) $, such that
\begin{equation}\label{equ:loc00}
\graph f_{\widehat{m}_0}|_{X^{c}_{m_0}(c_1\epsilon_{*}) \oplus X^{s}_{m_0}(e_1\sigma_*)} \subset \graph h \cap X^{su}_{\widehat{U}_{\widehat{m}_0}(\epsilon_{*})} (\sigma_*, \varrho_*).
\end{equation}
Since $ |h(\widehat{m}_0, 0)| \leq K_1\eta $, we have $ |f_{\widehat{m}_0}(0)| \leq K_1\eta $.
Define
\[
f^1_{\widehat{m}_0}(z) = f_{\widehat{m}_0} (r_{\sigma_0}(z)),
\]
where $ r_{\sigma_0}(\cdot) $ is the radial retraction (see \eqref{equ:radial}). By the choice of $ \sigma_{0} $, $ f^1_{\widehat{m}_0}(X^{cs}_{{m}_0}) \subset X^{u}_{{m}_0}(r_0) $.

Take $ \widehat{m} \in \widehat{K} $ and set $ \widehat{u}(\widehat{m}) = \widehat{m}_1 $, $ \phi(\widehat{m}_1) = m_1 $ ($ \in K $ as $ u(K) \subset K $), and $ \phi(\widehat{m}) = m $, as considered until before \autoref{lem:first2}.
Consider the following fixed point equation:
\begin{equation}\label{equ:fixed}
\widehat{F}^{cs}_{m}( x^{cs}, f^1_{\widehat{m}_1}(z) ) = z, \quad z \in X^{cs}_{m_1},
\end{equation}
where $ x^{cs} \in X^{cs}_{m} (r_0) $. Because $ \sup_{\widehat{m} \in \widehat{K}} \alpha(\phi(\widehat{m})) \mu_1(\widehat{u}(\widehat{m})) = \gamma < 1 / 2 $ (in Observation \eqref{OVII}), there is a unique point
\[
x_{\widehat{m}} (x^{cs}) = (x^c_{\widehat{m}} (x^{cs}), x^s_{\widehat{m}} (x^{cs}) ) \in X^{cs}_{m_1},
\]
satisfying equation \eqref{equ:fixed}. Since $ | f^{1}_{\widehat{m}_1} ( 0, 0 ) | = |f_{\widehat{m}_1} ( 0, 0 )| \leq K_1\eta $ and $ | \widehat{F}^{cs}_m( 0, 0 ) | \leq \eta $, we have
\[
|f^{1}_{\widehat{m}_1} (x_{\widehat{m}}(0))| \leq |f^{1}_{\widehat{m}_1} (x_{\widehat{m}}(0)) - f^{1}_{\widehat{m}_1} (0) | + |f^{1}_{\widehat{m}_1} (0)| \leq 2\mu_1(\widehat{m}_1) |x_{\widehat{m}}(0)| + K_1\eta,
\]
and so
\begin{align*}
|x_{\widehat{m}}(0)| & \leq |\widehat{F}^{cs}_{m}(0,f^1_{\widehat{m}_1}(x_{\widehat{m}}(0))) - \widehat{F}^{cs}_{m}(0,0)| + |\widehat{F}^{cs}_{m}(0,0)| \\
& \leq \alpha(m)|f^1_{\widehat{m}_1}(x_{\widehat{m}}(0))| + \eta \\
& \leq 2\alpha(m) \mu_1(\widehat{m}_1) |x_{\widehat{m}}(0)| + \alpha(m) K_1\eta + \eta,
\end{align*}
yielding
\[ \label{**1} \tag{**1}
\begin{cases}
|x_{\widehat{m}}(0)| \leq \frac{\alpha(m)K_1\eta + \eta}{1 - 2\gamma} \leq K_2 \eta < \sigma_0, \\
|f^{1}_{\widehat{m}_1} (x_{\widehat{m}}(0))| = |f_{\widehat{m}_1} (x_{\widehat{m}}(0))| \leq \mu_1(\widehat{m}_1) |x_{\widehat{m}}(0)| + K_1\eta \leq \hat{\beta} K_2\eta + K_1\eta < r_0.
\end{cases}
\]

Next, let us show
\begin{lem}\label{lem:first1}
	Let $ \sigma^c_{*} > 0 $ such that $ \max\{\sup_{m}\{ \lambda_{cs}(m) \}, 1\} \sigma^c_{*} + K_2 \eta < \sigma_{0} $. Then 
	\[
	x_{\widehat{m}}(X^{cs}_{m}(\sigma^c_{*})) \subset X^{cs}_{m_1}(\sigma_{0}), \quad \text{and} \quad \lip x_{\widehat{m}}(\cdot)|_{X^{cs}_{m}(\sigma^c_{*})} \leq \lambda_{cs}(m).
	\]
\end{lem}
\begin{proof}
	Set $ \overline{\lambda}_{cs} = \sup_{m} \lambda_{cs}(m) $. By the construction of $ x_{\widehat{m}}(\cdot) $, we have $ \lip x_{\widehat{m}}(\cdot)|_{X^{cs}_{m}(r_0)} \leq \frac{\overline{\lambda}_{cs}}{1 - 2\gamma} $.
	If $ r' > 0 $ is such that $ x_{\widehat{m}}(X^{cs}_{m}(r')) \subset X^{cs}_{m_1}(\sigma_{0}) $, then
	\[
	\widehat{F}^{cs}_{m} ( x^{cs}, f_{\widehat{m}_1} ( x_{\widehat{m}} (x^{cs}) ) ) = x_{\widehat{m}} (x^{cs}).
	\]
	By the (B) condition, $ \lip x_{\widehat{m}}(\cdot,\cdot)|_{X^{cs}_{m}(r')} \leq \lambda_{cs}(m) $. Let
	\[
	\sigma_1 = \sup \{ \sigma \leq \sigma^c_{*}: \lip x_{\widehat{m}}(\cdot)|_{X^{cs}_{m}(\sigma)} \leq \lambda_{cs}(m) \}.
	\]
	Note that $ \sigma_1 > 0 $. If $ \sigma_1 < \sigma^c_{*} $, then $  \lip x_{\widehat{m}}(\cdot)|_{\overline{X^{cs}_{m}(\sigma_1)}} \leq \lambda_{cs}(m) $ and $ \overline{\lambda}_{cs}\sigma_1 + K_2\eta < \sigma_0 $. So we can choose a small $ \varepsilon > 0 $ such that $ \overline{\lambda}_{cs}\sigma_1 + K_2\eta + \frac{\overline{\lambda}_{cs}}{1 - 2\gamma} \varepsilon < \sigma_0 $, which implies that $ x_{\widehat{m}}(X^{cs}_{m}(\sigma_1 + \varepsilon)) \subset X^{cs}_{m_1}(\sigma_{0}) $ and thus $ \lip x_{\widehat{m}}(\cdot)|_{X^{cs}_{m}(\sigma_1 + \varepsilon)} \leq \lambda_{cs}(m) $, contradicting the definition of $ \sigma_1 $. Therefore, $ \sigma_1 = \sigma^c_{*} $. The proof is complete.
\end{proof}

Let
\begin{equation}\label{equ:csss}
\sigma^c_{*} = (\max\{\sup_{m}\{ \lambda_{cs}(m) \}, 1\})^{-1} (\sigma_{0} - \eta_0) / 2.
\end{equation}
Consider the following equation:
\begin{equation}\label{equ:local}
\begin{cases}
\widehat{F}^{cs}_{m} ( x^{cs}, f_{\widehat{u}(\widehat{m})} ( x_{\widehat{m}} (x^{cs}) ) ) = x_{\widehat{m}} (x^{cs}), \\
\widehat{G}^{cs}_{m} ( x^{cs}, f_{\widehat{u}(\widehat{m})} ( x_{\widehat{m}} (x^{cs}) ) ) \triangleq \widetilde{f}_{\widehat{m}} (x^{cs}),
\end{cases}
\quad x^{cs} \in X^{cs}_{m}(\sigma^c_{*}),
\end{equation}
where $ m = \phi(\widehat{m}), \widehat{m} \in \widehat{K} $.
\begin{lem} \label{lem:first2}
	We have
	$ |\widetilde{f}_{\widehat{m}} (0)| \leq K'_1 \eta $, where
	\begin{equation}\label{equ:K00}
	K'_1 = \overline{\lambda}_u ( \hat{\beta} + K_1 ) + 1,
	\end{equation}
	and $ \lip \widetilde{f}_{\widehat{m}}(\cdot)|_{X^{cs}_{m}(\sigma^c_{*})} \leq \beta'(m) $. In particular, if $ \overline{\lambda}_u < 1 $ and $ K_1 = \frac{ \overline{\lambda}_u\hat{\beta} + 1 }{1 - \overline{\lambda}_u} $, then we can take $ K'_1 = K_1 $.
\end{lem}
\begin{proof}
	Since $ | f_{\widehat{m}_1} ( 0, 0 ) | \leq K_1\eta $, $ | \widehat{F}^{cs}_m( 0, 0 ) | \leq \eta $, $ | \widehat{G}^{cs}_m( 0, 0 ) | \leq \eta $, and
	\begin{gather*}
	\begin{cases}
	|x_{\widehat{m}}(0)| \leq |\widehat{F}^{cs}_{m}(0,f_{\widehat{m}_1}(x_{\widehat{m}}(0))) - \widehat{F}^{cs}_{m}(0,0)| + |\widehat{F}^{cs}_{m}(0,0)|
	\leq \alpha(m)|f_{\widehat{m}_1}(x_{\widehat{m}}(0))| + \eta, \\
	|f_{\widehat{m}_1} (x_{\widehat{m}}(0))| \leq |f_{\widehat{m}_1} (x_{\widehat{m}}(0)) - f_{\widehat{m}_1} (0) | + |f_{\widehat{m}_1} (0)| \leq \mu_1(\widehat{m}_1) |x_{\widehat{m}}(0)| + K_1\eta,
	\end{cases}
	\\
	\Rightarrow |f_{\widehat{m}_1} (x_{\widehat{m}}(0))| \leq \frac{\mu_1(\widehat{m}_1)\eta + K_1\eta}{1 - \alpha(m)\mu_1(\widehat{m}_1)},
	\end{gather*}
	it follows that
	\begin{align*}
	| \widetilde{f}_{\widehat{m}} (0)  | & \leq |\widehat{G}^{cs}_{m} ( 0, f_{\widehat{m}_1} ( x_{\widehat{m}} (0) ) ) - \widehat{G}^{cs}_{m} ( 0, 0 ) | + |\widehat{G}^{cs}_{m} ( 0, 0 )| \\
	& \leq \lambda_{u}(m)|f_{\widehat{m}_1} (x_{\widehat{m}} (0))| + \eta
	\leq \lambda_u(m) \frac{\mu_1(\widehat{m}_1)\eta + K_1\eta}{1 - \alpha(m)\mu_1(\widehat{m}_1)} + \eta \\
	& \leq \overline{\lambda}_u ( \hat{\beta} \eta + K_1\eta ) + \eta = K'_1\eta.
	\end{align*}
	Since $ \lip f_{\widehat{m}_1} (\cdot) \leq \mu_1(\widehat{m}_1) < \beta(m) $, by the (B) condition, we have $ \lip \widetilde{f}_{\widehat{m}}(\cdot)|_{X^{cs}_{m}(\sigma^c_{*})} \leq \beta'(m) $. 
	The proof is complete.
\end{proof}

\begin{lem}\label{lem:estR}
	In both cases, we have $ \widetilde{f}_{\widehat{m}} (X^{cs}_{m}(\sigma^c_{*})) \subset X^{u}_{m}(e_1\varrho_{*}) $.
\end{lem}
\begin{proof}
	In case (1), by the choice of $ \varrho_{*} $, one gets
	\[\tag{**2} \label{**2}
	|\widetilde{f}_{\widehat{m}} (x^{cs})| \leq \hat{\beta} |x^{cs}| + |\widetilde{f}_{\widehat{m}} (0)| \leq \hat{\beta}\sigma^c_{*} + K'_1\eta \leq e_1\varrho_{*},~ x^{cs} \in X^{cs}_{m}(\sigma^c_{*}).
	\]
	
	In case (2), from \eqref{equ:est} and the choice of $ \gamma_0 $, we see, for $ x^{cs} \in X^{cs}_{m}(\sigma^c_{*}) $,
	\begin{align*}
	|\widetilde{f}_{\widehat{m}} (x^{cs})| & = | \widehat{G}^{cs}_{m} ( x^{cs}, f_{\widehat{u}(\widehat{m})} ( x_{\widehat{m}} (x^{cs}) ) ) | \\
	& \leq \gamma_0 |x^{cs}| + \eta + \gamma^{*}_u |f_{\widehat{u}(\widehat{m})} ( x_{\widehat{m}} (x^{cs}) )| \\
	& \leq \gamma_0 |x^{cs}| + \eta + \gamma^{*}_u e_2\varrho_{*} \\
	& \leq e_1\varrho_{*}, \label{**3} \tag{**3}
	\end{align*}
	where the last inequality is a consequence of $ \gamma_0 \sigma^{c}_{*} / \varrho_* \to 0 $, $ \gamma^{*}_{u} < 1 $, and $ e_1, e_2 \to 1 $. The proof is complete.
\end{proof}

Take $ \widehat{m}_i \in \widehat{K} $, $ m_i = \phi(\widehat{m}_i) $, and $ x^{cs}_i \in X^{cs}_{m_i}(\sigma^c_{*}) $, $ i = 1,2 $.
By \autoref{lem:lip2}, write
\[
m_i + x^{cs}_i + \widetilde{f}_{\widehat{m}_i} (x^{cs}_i) = \overline{m}_i + \overline{x}^s_i + \overline{x}^{u}_i,
\]
where $ \widehat{\overline{m}}_i \in \widehat{U}_{\widehat{m}_i} (\sigma^c_{*}+\varrho_{*}) \subset \widehat{U}_{\widehat{m}_i} (\epsilon_{*}) $, $ \overline{m}_i = \phi(\widehat{\overline{m}}_i) $, $ \overline{x}^s_i \in X^{s}_{\overline{m}_i} (e^{-1}_1\sigma^c_{*}) \subset X^{s}_{\overline{m}_i}(\sigma_{*}) $ and $ \overline{x}^u_i \in X^{u}_{\overline{m}_i} (\varrho_{*}) $.

\begin{lem}\label{lem:locLip}
	Let $ \overline{m}_i, \overline{x}^s_i, \overline{x}^{u}_i $ be as given above. Let $ \widehat{m}_0 \in \widehat{K} $, $ m_0 = \phi(\widehat{m}_0) $, and
	\[
	\overline{m}_i + \overline{x}^s_i + \overline{x}^{u}_i = m_0 + \widehat{x}^s_i + \widehat{x}^c_i + \widehat{x}^{u}_i, \quad i = 1,2,
	\]
	where $ \widehat{x}^{\kappa}_i \in X^{\kappa}_{m_0} $, $ \kappa = s, c, u $, $ i = 1,2 $. If $ \widehat{\overline{m}}_1, \widehat{\overline{m}}_2 \in \widehat{U}_{\widehat{m}_0}(\epsilon_{*}) $, then
	\begin{equation}\label{equ:lip}
	|\widehat{x}^u_1 - \widehat{x}^u_2| \leq \widetilde{\beta}'(m_0) \max \{|\widehat{x}^c_1 - \widehat{x}^c_2|, |\widehat{x}^s_1 - \widehat{x}^s_2|\}.
	\end{equation}
\end{lem}
\begin{proof}
	By \autoref{lem:lip2} (1), we have $ |\widehat{x}^{u}_i| \leq e_2\varrho_{*} < r_0/2 $, $ |\widehat{x}^{c}_i| \leq c_2\epsilon_{*} < r_0/2 $, $ |\widehat{x}^{s}_i| \leq e_2\sigma_{*} < r_0/2 $, $ i = 1,2 $. It suffices to consider the case
	\begin{equation}\label{equ:case}
	\widetilde{\beta}'(m_0) \max \{|\widehat{x}^c_1 - \widehat{x}^c_2|, |\widehat{x}^s_1 - \widehat{x}^s_2|\} \leq |\widehat{x}^u_1 | + | \widehat{x}^u_2| \leq 2 e_2\varrho_*,
	\end{equation}
	as the inequality \eqref{equ:lip} is already true otherwise.
	Write
	\begin{align*}
	& u(m_1) + x_{\widehat{m}_1} (x^{cs}_1) + {f}_{\widehat{u}(\widehat{m}_1)} ( x_{\widehat{m}_1} (x^{cs}_1) ) \\
	  = ~& u(m_2) + x_{\widehat{m}_2} (x^{cs}_2) + {f}_{\widehat{u}(\widehat{m}_2)} ( x_{\widehat{m}_2} (x^{cs}_2) ) + \widetilde{x}^c_1 + \widetilde{x}^s_1 + \widetilde{x}^u_1,
	\end{align*}
	where $ \widetilde{x}^\kappa_1 \in X^\kappa_{u(m_2)} $, $ \kappa = s, c, u $.

	Since $ \widehat{\overline{m}}_2 \in \widehat{U}_{\widehat{m}_2} (\epsilon_{*}) $ and $ \widehat{m}_2 \in \widehat{U}_{\widehat{m}_0}(\epsilon_{*}) $ ($ \Rightarrow \widehat{\overline{m}}_2 \in \widehat{U}_{\widehat{m}_0}(2\epsilon_{*}) $), by Observation \eqref{OV}, one gets
	\[
	|u(m_0) - u(m_2)| \leq 2 \xi_1  < \epsilon_2 / 2,
	\]
	and so $ |u(m_1) - u(m_2)| \leq 4 \xi_1 < \epsilon_2 $. Hence, we also have $ \widetilde{x}^\kappa_1 \in X^\kappa_{u(m_2)} (r_0 / 2) $, $ \kappa = s, c, u $; also note that
	\[
	\max\{  |x_{\widehat{m}_i} (x^{cs}_i)|, |{f}_{\widehat{u}(\widehat{m}_i)} ( x_{\widehat{m}_i} (x^{cs}_i) )|: i = 1,2 \} < r_0 /2.
	\]
	Thus, $ H(m_0 + \cdot) - u({m}_2) \sim (\widetilde{F}^{cs}, \widetilde{G}^{cs}) $ satisfies the (A$ ' $) $( \widetilde{\alpha}(m_0), \widetilde{\lambda}_u(m_0) )$ (B) ($ \widetilde{\beta}(m_0) $; $ \widetilde{\beta}'(m_0) $, $ \widetilde{\lambda}_{cs}(m_0) $) condition by Observation \eqref{OI}, and
	\begin{equation}\label{equ:H00}
	\begin{cases}
	u(m_2) + x_{\widehat{m}_2} (x^{cs}_2) + {f}_{\widehat{u}(\widehat{m}_2)} ( x_{\widehat{m}_2} (x^{cs}_2) ) + \widetilde{x}^c_1 + \widetilde{x}^s_1 + \widetilde{x}^u_1 \in H(m_0 + \widehat{x}^s_1 + \widehat{x}^c_1 + \widehat{x}^{u}_1),\\
	u(m_2) + x_{\widehat{m}_2} (x^{cs}_2) + {f}_{\widehat{u}(\widehat{m}_2)} ( x_{\widehat{m}_2} (x^{cs}_2) ) \in H(m_0 + \widehat{x}^s_2 + \widehat{x}^c_2 + \widehat{x}^{u}_2),
	\end{cases}
	\end{equation}
	i.e.,
	\[
	\begin{cases}
	\widetilde{F}^{cs}(\widehat{x}^s_1 + \widehat{x}^c_1, {f}_{\widehat{u}(\widehat{m}_2)} ( x_{\widehat{m}_2} (x^{cs}_2) ) + \widetilde{x}^u_1) = x_{\widehat{m}_2} (x^{cs}_2) + \widetilde{x}^c_1 + \widetilde{x}^s_1, \\
	\widetilde{G}^{cs}(\widehat{x}^s_1 + \widehat{x}^c_1, {f}_{\widehat{u}(\widehat{m}_2)} ( x_{\widehat{m}_2} (x^{cs}_2) ) + \widetilde{x}^u_1) = \widehat{x}^{u}_1,\\
	\widetilde{F}^{cs}(\widehat{x}^s_2 + \widehat{x}^c_2, {f}_{\widehat{u}(\widehat{m}_2)} ( x_{\widehat{m}_2} (x^{cs}_2)) )= x_{\widehat{m}_2} (x^{cs}_2), \\
	\widetilde{G}^{cs}(\widehat{x}^s_2 + \widehat{x}^c_2, {f}_{\widehat{u}(\widehat{m}_2)} ( x_{\widehat{m}_2} (x^{cs}_2)) )= \widehat{x}^{u}_2.
	\end{cases}
	\]

	Due to the Lipschitz continuity of $ \widetilde{F}^{cs} $, we get
	\begin{equation}\label{equ:00}
	|\widetilde{x}^c_1| + |\widetilde{x}^s_1| \leq \widetilde{C}' (|\widehat{x}^{c}_1 - \widehat{x}^{c}_2| + |\widehat{x}^{s}_1 - \widehat{x}^{s}_2| + |\widetilde{x}^u_1|),
	\end{equation}
	where $ \lip \widetilde{F}^{cs} \leq \widetilde{C}' = \sup_{m_0 \in K}\{   \widetilde{\alpha}(m_0) + \widetilde{\lambda}_{cs}(m_0) \} $.
	Compute
	\begin{align}
	|\widetilde{x}^u_1| = & |\Pi^u_{u({m}_2)} (u({m}_1) - u({m}_2)) + \widetilde{f}_{\widehat{u}(\widehat{m}_2)} (x_{\widehat{m}_2} (x^{cs}_2)) \notag\\
	& + \Pi^u_{u({m}_2)} ( x_{\widehat{m}_1} (x^{cs}_1) + \widetilde{f}_{\widehat{u}(\widehat{m}_1)} ( x_{\widehat{m}_1} (x^{cs}_1) ) ) | \notag\\
	\leq & \chi(\epsilon_2) |u({m}_1) - u({m}_2)| + e_2\varrho_* + L\epsilon_2 \sigma_* + (L\epsilon_2 + 1) e_2\varrho_*  \notag\\
	\leq & \chi(\epsilon_2) |u({m}_1) - u({m}_2)| + 4\varrho_* +\sigma_*. \label{equ:11}
	\end{align}
	(Note that $ L\epsilon_2 \leq 1 $ and $ \sigma_{0} \leq \sigma_{*} $.)
	Thus,
	\begin{align*}
	|u(m_1) - u(m_2)| \leq & \sum_{\kappa}|\widetilde{x}^{\kappa}_1| + \sum_{i} (|x_{\widehat{m}_i} (x^{cs}_i)| + | \widetilde{f}_{\widehat{u}(\widehat{m}_i)} (x_{\widehat{m}_i} (x^{cs}_i)) |) \\
	\leq & 2 (e_2\varrho_* + \sigma_*) + |\widetilde{x}^u_1| + |\widetilde{x}^c_1| + |\widetilde{x}^s_1| \\
	\leq & 4 (\varrho_* + \sigma_*) + (\widetilde{C}' + 1)|\widetilde{x}^u_1| + 4\widetilde{C}' e_2\varrho_* / \hat{\beta}'' \\
	\leq & \widetilde{C}_1 (\varrho_* + \sigma_*) + (\widetilde{C}' + 1) \chi(\epsilon_2) |u(\overline{m}_1) - u(\overline{m}_2)|,
	\end{align*}
	where \eqref{equ:case}, \eqref{equ:00} and \eqref{equ:11} are used, $ \widetilde{C}_1 = 8 (\widetilde{C}' + 1) + 4\widetilde{C}'e_2/\hat{\beta}''  $ is a constant, and $ \hat{\beta}'' > 0 $ is defined in Observation \eqref{OI}. It follows that
	\begin{equation}\label{equ:mmm}
	|u({m}_1) - u({m}_2)| \leq \frac{\widetilde{C}_1 (\varrho_* + \sigma_* )}{1 - (\widetilde{C}' + 1) \chi(\epsilon_{2})} < \epsilon_* / 2,
	\end{equation}
	if $ \epsilon_{*}, \chi(\epsilon_{*}) $ (and hence $ \eta $) are small; here Observation \eqref{OIII} is used and $ (\varrho_* + \sigma_*) / \epsilon_{*} \to 0 $.

	By \eqref{equ:loc00}, we write
	\[
	u({m}_i) + x_{\widehat{m}_i} (x^{cs}_i) + {f}_{\widehat{u}(\widehat{m}_i)} (x_{\widehat{m}_i} (x^{cs}_i))
	= \overline{m}'_i + {\overline{x}'}^{s}_i + h(\widehat{\overline{m}}'_i, {\overline{x}'}^{s}_i),
	\]
	where $ \widehat{\overline{m}}'_i \in \widehat{U}_{\widehat{u}(\widehat{m}_i)} (\epsilon_*) $, $ {\overline{x}'}^{s}_i \in X^s_{\overline{m}'_i} (\sigma_*) $ and $ \phi(\widehat{\overline{m}}'_i) = \overline{m}'_i $.
	Moreover, by \autoref{lem:lip2} (1), we have
	\[
	|\overline{m}'_i - u(m_i) | \leq \sigma_{*} + \varrho_{*}.
	\]
	So
	\[
	|\overline{m}'_1 - u(m_2)| \leq |\overline{m}'_1 - u(m_1)| + |u(m_1) - u(m_2)| \leq  \sigma_* + \varrho_* + \epsilon_* / 2 < \epsilon_*,
	\]
	and $ \widehat{\overline{m}}'_i \in \widehat{U}_{\widehat{u}(\widehat{m}_2)} (\epsilon_*) $, $ i = 1,2 $. Observing that $ \widehat{u}(\widehat{m}_2) \in \widehat{K} $ and $ h \in \varSigma_{\mu, K_1, \epsilon_{*}, \sigma_{*}, \varrho_* } $, we know
	\[
	|\widetilde{x}^u_1| \leq \mu_1(\widehat{u}(\widehat{m}_2)) \max \{|\widetilde{x}^s_1|, |\widetilde{x}^c_1|\}.
	\]
	Thus, by the (B) condition for $ H(m_0 + \cdot) - u(m_2) $ (i.e., \eqref{equ:H00} and Observation \eqref{OI}), we finally get \eqref{equ:lip}; here note also that $ \mu_1(\widehat{u}(\widehat{m}_2)) < \widetilde{\beta}(m_0) $ by Observation \eqref{OVI}. The proof is complete.
\end{proof}

By \autoref{lem:represent3}, for every $ \widehat{m}_0 \in \widehat{K} $ and every $ (\widehat{m}, \overline{x}^s) \in X^s_{ \widehat{U}_{\widehat{m}_0} (\sigma^1_{*}) }(\sigma^1_{*}) $, where
\begin{equation}\label{equ:s111}
\sigma^1_{*} = \min\{c^{-1}_2, e^{-1}_2\}\sigma^c_{*},
\end{equation}
there is a unique $ \overline{x}^{u} \in X^{u}_{m}(\varrho_{*}) $, where $ m = \phi(\widehat{m}) $, such that $ (\widehat{m}, \overline{x}^s, \overline{x}^u) \in \graph \widetilde{f}_{\widehat{m}_0} $. Furthermore, by \autoref{lem:locLip}, for every $ \widehat{m} \in \widehat{K}_{\sigma^1_{*}} $ and $ \overline{x}^s \in X^s_{m} (\sigma^1_{*}) $, where $ m = \phi(\widehat{m}) $, there is a unique $ \overline{x}^{u} \in X^{u}_{m}(\varrho_{*}) $ such that
\[
(\widehat{m}, \overline{x}^s, \overline{x}^u) \in \bigcup_{\widehat{m}_0 \in \widehat{K}}\graph \widetilde{f}_{\widehat{m}_0} \subset X^{h}_{\widehat{\Sigma}}.
\]
Define $ \widetilde{h} $ by
\[
\widetilde{h}: X^{s}_{\widehat{K}_{\sigma^1_{*}}} (\sigma^1_{*}) \to X^{u}_{\widehat{K}_{\sigma^1_{*}}} (\varrho_{*}),~ (\widehat{m}, \overline{x}^s) \mapsto \overline{x}^u.
\]
We denote this construction process by
\begin{equation}\label{equ:graph}
\widetilde{\varGamma}: \varSigma_{\mu, K_1, \epsilon_{*}, \sigma_{*}, \varrho_* } \to \varSigma_{\mu, K'_1, \sigma^1_{*}, \sigma^1_{*}, \varrho_* }, \quad h \mapsto \widetilde{h}.
\end{equation}

Since $ \widetilde{h} $ is not defined on all of $ X^s_{\widehat{\Sigma}} (\sigma_*) $, a truncation is needed. The construction can be as follows. Take $ \ell \in C^{\infty}(\mathbb{R}_+, [0,1]) $ such that
\begin{equation}\label{equ:ct}
\ell(t) = \begin{cases}
1, & t \leq \eta_2, \\
0, & t \geq \eta_1,
\end{cases}
\quad \text{and} \quad \lip\ell \leq \frac{1}{\eta_1 - \eta_2},
\end{equation}
where $ 0 < \eta_2 < \eta_1 < \sigma^1_{*} $ ($ < \sigma_{*} $). Let $ 0 < \eta_1 < \sigma^1_{*} $ but close to $ \sigma^1_{*} $ (e.g., $ \eta_1 = (1 - \epsilon_{*}) \sigma^1_{*} $). Set\footnote{In general, this does not cause confusion with the same symbol defined in \autoref{sub:Grassmann} (i.e., the metric in $ \mathbb{G}(X) $).}
\begin{equation}\label{equ:dP}
\widehat{d}(\widehat{m}, \widehat{K})
= \begin{cases}
\inf \{ |\phi(\widehat{m}) - \phi(\widehat{m}_0)|: \widehat{m}_0 \in \widehat{K}, \widehat{m} \in \widehat{U}_{\widehat{m}_0} \}, & \text{if this set is not empty},\\
\eta_1, & \text{otherwise}.
\end{cases}
\end{equation}
Obviously, if $ \widehat{m}_1, \widehat{m}_2 \in \widehat{U}_{\widehat{m}_0}(\epsilon_{*}) $ and $ \widehat{m}_0 \in \widehat{K} $, then
\[
|\widehat{d}(\widehat{m}_1, \widehat{K}) - \widehat{d}(\widehat{m}_2, \widehat{K}) | \leq |\phi(\widehat{m}_1) - \phi(\widehat{m}_2)| \leq \varpi^*_1|\Pi^c_{\phi(\widehat{m}_0)}(\phi(\widehat{m}_1) - \phi(\widehat{m}_2))|,
\]
where $ \varpi^*_1 \to 1 $ ($ \varpi^*_1 > 1 $) as $ \epsilon_{*}, \chi(\epsilon_*),\eta \to 0 $.
Let
\begin{equation}\label{equ:cutoff}
\Psi: X^s_{\widehat{\Sigma}} (\sigma_*) \to \mathbb{R}_+, \quad (\widehat{m}, \overline{x}^s) \mapsto \ell(\max\{ \widehat{d}(\widehat{m}, \widehat{K}), |\overline{x}^s| \}),
\end{equation}
and define
\[
\widehat{h}: X^s_{\widehat{\Sigma}} (\sigma_*) \to X^u_{\widehat{\Sigma}} (\varrho_{*}), \quad (\widehat{m}, \overline{x}^s) \mapsto \Psi(\widehat{m}, \overline{x}^s) \widetilde{h}(\widehat{m}, \overline{x}^s).
\]
Note that $ \widehat{h} $ is well defined, since
\[
\widehat{h}(\widehat{m}, \overline{x}^s) = \begin{cases}
\widetilde{h}(\widehat{m}, \overline{x}^s), & (\widehat{m}, \overline{x}^s) \in X^s_{\widehat{K}_{\eta_2}} (\eta_2), \\
0, & (\widehat{m}, \overline{x}^s) \in X^s_{\widehat{\Sigma}} (\sigma_*) \setminus X^s_{\widehat{K}_{\eta_1}} (\eta_1).
\end{cases}
\]

Note that in case (1), $ \varsigma > 2 $ is taken such that $ \inf_{m \in K}\{ \beta(m) - \varsigma \beta'(u(m)) \} > 0 $ (in Observation \eqref{OIV}). The function $ \epsilon_{0,*} $ is defined in \eqref{equ:o(1)}. Since $ \epsilon_{0,*} \to 0 $ as $ \epsilon_{*} \to 0 $, we know if $ \eta \leq \epsilon_{0,*} \eta_0 $, then \eqref{equ:small000} can always hold (when $ \epsilon_{*} $ is small). Note that $ \eta_1 $ is close to $ \sigma^1_{*} $ (e.g., $ \eta_1 = (1 - \epsilon_{*}) \sigma^1_{*} $).
\begin{lem}\label{lem:belong}
	In case (1), let $ \eta_2 = \frac{\varsigma' - 2}{\varsigma' - 1} \eta_1 $ and $ \eta \leq \epsilon_{0,*} \eta_0 $ (e.g., $ \eta \leq \epsilon_{*} \eta_0 $) where $ \varsigma' = (\varsigma + 2) / 2 $. Then $ \widehat{h} \in \varSigma_{\mu, K'_1, \epsilon_{*}, \sigma_{*}, \varrho_* } $. 

	In case (2), let $ \eta_2 = (1/2) \eta_1 $ and $ \eta \leq \epsilon_{0,*} \eta_0 $. Then $ \widehat{h} \in \varSigma_{\mu, K'_1, \epsilon_{*}, \sigma_{*}, \varrho_* } $.
\end{lem}
\begin{proof}
	In case (2), let $ \varsigma' = (\varsigma + 1) / 2 $.
	By \autoref{lem:estR}, all we need to show is that $ \graph \widehat{h} \cap X^{su}_{\widehat{K}_{\epsilon_*}} (\sigma_*, \varrho_*) $ is $ \mu $-Lip in $ u $-direction near $ \widehat{K} $. Let $ \widehat{m}_0 \in \widehat{K} $, $ m_0 = \phi(\widehat{m}_0) $, and $ (\widehat{m}_i, \overline{x}^s_i) \in X^s_{\widehat{U}_{\widehat{m}_0}(\epsilon_*)} (\sigma_*) $, $ i = 1,2 $; write $ {m}_{i} = \phi(\widehat{m}_i) $. Without loss of generality, assume $ (\widehat{m}_1, \overline{x}^s_1) \in X^s_{\widehat{K}_{\eta_1}} (\eta_1) $ and so further assume there is $ \widehat{m}'_0 \in \widehat{K} $ such that $ \widehat{m}_1 \in \widehat{U}_{\widehat{m}'_0}(\eta_1) $; write $ m'_0 = \phi(\widehat{m}'_0) $. Now, 
	\begin{align*}
	|\Pi^{u}_{m_0} \widetilde{h}(\widehat{m}_1, \overline{x}^s_1)| & \leq |(\Pi^{u}_{m_0} - \Pi^{u}_{m'_0}) \widetilde{h}(\widehat{m}_1, \overline{x}^s_1)| + |\Pi^{u}_{m'_0} \widetilde{h}(\widehat{m}_1, \overline{x}^s_1)| \\
	& \leq L|m_0 - m'_0| \varrho_{*} + |\Pi^{u}_{m'_0} \widetilde{h}(\widehat{m}_1, \overline{x}^s_1)| \\
	& \leq \begin{cases}
	2L \epsilon_{*} \varrho_{*} + \varpi^*_1 (\widetilde{\beta}'(m_0)\eta_1 + K_1\eta), & \text{case (1)},\\
	2L \epsilon_{*} \varrho_{*} + \varrho_{*}, & \text{case (2)},
	\end{cases}
	\end{align*}
	where in case (1), \autoref{lem:locLip} is used, and in case (2), \autoref{lem:estR} is used; here $ \varpi^*_1 \to 1 $ as $ \epsilon_{*}, \chi(\epsilon_*) \to 0 $. 

	Consider
	\begin{align}
	&~ |\Pi^{u}_{m_0}(\widehat{h}(\widehat{m}_1, \overline{x}^s_1) - \widehat{h}(\widehat{m}_2, \overline{x}^s_2))| = |\Pi^{u}_{m_0}\{ \Psi(\widehat{m}_1, \overline{x}^s_1)\widetilde{h}(\widehat{m}_1, \overline{x}^s_1) - \Psi(\widehat{m}_2, \overline{x}^s_2)\widetilde{h}(\widehat{m}_2, \overline{x}^s_2)\}| \notag\\
	\leq &~ |\{ \Psi(\widehat{m}_1, \overline{x}^s_1) - \Psi(\widehat{m}_2, \overline{x}^s_2) \} \Pi^{u}_{m_0} \widetilde{h}(\widehat{m}_1, \overline{x}^s_1) | + | \Psi(\widehat{m}_2, \overline{x}^s_2) \Pi^{u}_{m_0}\{ \widetilde{h}({m}_1, \overline{x}^s_1) - \widetilde{h}({m}_2, \overline{x}^s_2) \}| \notag\\
	\leq &~ \frac{\varpi^*_1}{\eta_1 - \eta_2} \max\{ |\Pi^c_{m_0} ({m}_1 - {m}_2) |, |\Pi^s_{m_0} (\overline{x}^s_1 - \overline{x}^s_2)| \} |\Pi^{u}_{m_0} \widetilde{h}(\widehat{m}_1, \overline{x}^s_1)| \notag\\
	& \quad + |\Pi^{u}_{m_0}\{ \widetilde{h}(\widehat{m}_1, \overline{x}^s_1) - \widetilde{h}(\widehat{m}_2, \overline{x}^s_2) \}| \notag\\
	\leq &~ P_* \max\{ |\Pi^c_{m_0} ({m}_1 - {m}_2) |, |\Pi^s_{m_0} (\overline{x}^s_1 - \overline{x}^s_2)| \} \notag \\
	\leq &~ \varpi^*_1 \cdot \varsigma'\widetilde{\beta}'(m_0) \max\{ |\Pi^c_{m_0} ({m}_1 - {m}_2) |, |\Pi^s_{m_0} (\overline{x}^s_1 - \overline{x}^s_2)| \} \label{equ:mm11}\\
	\leq &~ \mu(\widehat{m}_0) \max\{ |\Pi^c_{m_0} ({m}_1 - {m}_2) |, |\Pi^s_{m_0} (\overline{x}^s_1 - \overline{x}^s_2)| \}, \label{equ:mm22}
	\end{align}
	where
	\[
	P_* = \begin{cases}
	\frac{\varpi^*_1 (2L \epsilon_{*} \varrho_{*} + \widetilde{\beta}'(m_0)\eta_1 + K_1\eta) }{\eta_1 - \eta_2}  + \widetilde{\beta}'(m_0)(1+\chi_{*}) + \chi_{*}, & \text{case (1)},\\
	\frac{2\varpi^*_1 (2L \epsilon_{*} \varrho_{*} + \varrho_{*})}{\eta_1} + \widetilde{\beta}'(m_0)(1+\chi_{*}) + \chi_{*}, & \text{case (2)}.
	\end{cases}
	\]
	In case (2), \eqref{equ:mm11} holds due to $ \varrho_{*} / \eta_1 \to 0 $ (by the choice of $ \varrho_{*}, \sigma_{*} $ such that $ \varrho_{*} / \sigma_{*} \to 0 $); note also that $ \frac{\epsilon_{*} \varrho_{*}}{\eta_1 - \eta_2} \to 0 $ and $ \frac{\eta}{\eta_1 - \eta_2} \to 0 $ as $ \epsilon_{*}, \chi(\epsilon_*) \to 0 $. \eqref{equ:mm22} holds since $ \varsigma' < \varsigma $.
	This gives the proof.
\end{proof}

We define the \emph{graph transform} to be
\[
\varGamma: \varSigma_{\mu, K_1, \epsilon_{*}, \sigma_{*}, \varrho_* } \to \varSigma_{\mu, K'_1, \epsilon_{*}, \sigma_{*}, \varrho_* }, ~ h \mapsto \widehat{h}.
\]

\subsection{Limited case I: $ s $-contraction and $ u $-expansion}\label{sub:limited}
In order not to truncate the $ s $-direction, in addition, assume $ H $ satisfies the following \emph{strong $ s $-contraction assumption} (see also \cite{Che18b}); others are the same as in \autoref{sub:preparation}. Recall $ \widehat{H}_{m} = H(m + \cdot) - u(m) $.
\begin{enumerate}[($ \star\star $)]
	\item If $ (\hat{x}^{c}_i, \hat{x}^s_i, \hat{x}^{u}_i) \times (\tilde{x}^{c}_i, \tilde{x}^s_i, \tilde{x}^{u}_{i}) \in \graph \widehat{H}_{m} \cap \{ \{\widehat{X}^{cs}_{m}(r) \oplus \widehat{X}^{u}_{m}(r_{1}) \} \times \{ \widehat{X}^{cs}_{u(m)}(r_{2}) \oplus \widehat{X}^{u}_{u(m)} (r)\} \} $, $ i = 1,2 $, $ m \in K $, and $ |\tilde{x}^{u}_{1} - \tilde{x}^{u}_{2}| \leq B (|\hat{x}^c_1 - \hat{x}^c_2| + |\hat{x}^s_1 - \hat{x}^s_2|) $, then
	\[
	|\tilde{x}^{s}_{1} - \tilde{x}^{s}_{2}| \leq \lambda^*_s  (|\hat{x}^c_1 - \hat{x}^c_2| + |\hat{x}^s_1 - \hat{x}^s_2|),
	\]
	where $ B > \sup_{m \in K}\lambda_{cs}(m) \beta(m) $ is some constant and $ \lambda^*_s < 1 $. In fact, the case $ \hat{x}^{c}_1 = \hat{x}^{c}_2 $ is enough.
\end{enumerate}
Similarly, one can consider the \emph{strong $ u $-expansion assumption} (i.e., the dual of $ \widehat{H}_{m}: \widehat{X}^{s}_{m}(r) \oplus \widehat{X}^{cu}_{m}(r_1) \to \widehat{X}^{s}_{u(m)}(r_2) \oplus \widehat{X}^{cu}_{u(m)} (r) $ (see \autoref{defi:dual}) satisfies assumption ($ \star\star $)).

\begin{rmk}\label{rmk:conexp}
	A special case where $ H $ satisfies the strong $ s $-contraction assumption is the following.
	\begin{enumerate}[$\bullet$]
		\item If $ (\hat{x}^{c}_i, \hat{x}^s_i, \hat{x}^{u}_i) $, $ (\tilde{x}^{c}_i, \tilde{x}^s_i, \tilde{x}^{u}_{i}) $ are given in assumption ($ \star\star $), then
		\[
		|\tilde{x}^s_1 - \tilde{x}^s_2| \leq \zeta(|\tilde{x}^c_1 - \tilde{x}^c_2| + |\tilde{x}^u_1 - \tilde{x}^u_2|) + \lambda^*_s |\hat{x}^s_1 - \hat{x}^s_2|,
		\]
		where $ \zeta > 0 $ is sufficiently small and $ \lambda^*_s < 1 $.
	\end{enumerate}
	For example, $ \widehat{H}_{m} \sim (\widehat{F}^s_m, \widehat{G}^s_m) $ where
	\[
	\widehat{F}^s_m: \widehat{X}^{s}_{m}(r_{1}) \times \widehat{X}^{cu}_{u(m)} (r'_{2}) \to \widehat{X}^{s}_{u(m)}(r'_{1}), ~
	\widehat{G}^s_m: \widehat{X}^{s}_{m}(r_{1}) \times \widehat{X}^{cu}_{u(m)} (r'_{2}) \to \widehat{X}^{cu}_{m}(r_{2}),
	\]
	and $ \sup_{x^s}\lip\widehat{F}^s_m(x^s, \cdot) $ is sufficiently small and $ \sup_{x^{cu}}\lip\widehat{F}^s_m(\cdot, x^{cu}) < 1 $.
	\begin{proof}[Proof of \autoref{rmk:conexp}]
		Write $ \tilde{x}^{\kappa_1} = \tilde{x}^{\kappa_1}_1 - \tilde{x}^{\kappa_1}_2 $, $ \hat{x}^{\kappa_1} = \hat{x}^{\kappa_1}_1 - \hat{x}^{\kappa_1}_2 $, $ \kappa_1 = s, c, u $.
		Let $ |\tilde{x}^u| \leq B (|\hat{x}^c| + |\hat{x}^s|) $. By the above assumption, $ |\tilde{x}^s| \leq \lambda^*_s |\hat{x}^s| + \zeta (|\tilde{x}^c| + |\tilde{x}^u|) $, and by (A3) (a), $ |\tilde{x}^c| \leq \lambda_{cs}(m) (|\hat{x}^c| + |\hat{x}^s|) + \alpha (m) |\tilde{x}^{u}| $, which yields
		\[
		|\tilde{x}^s| \leq \lambda^*_s |\hat{x}^s| + \{\zeta ( \lambda_{cs}(m) + \alpha(m)B + B )\} (|\hat{x}^c| + |\hat{x}^s|).
		\]
		If $ \zeta $ is small, then $ \lambda^*_s + \zeta ( \sup_{m}\lambda_{cs}(m) + \sup_{m}\alpha(m)B + B ) < 1 $. The proof is complete.
	\end{proof}
\end{rmk}

We can assume $ \lambda^*_s e_2 < 1 $ (as $ e_2 \to 1 $ when $ \epsilon_{*}, \chi(\epsilon_{*}) \to 0 $).
Similar to the proof of \autoref{lem:first1}, for the unique point $ x_{\widehat{m}} (x^{cs}) = (x^c_{\widehat{m}} (x^{cs}), x^s_{\widehat{m}} (x^{cs}) ) \in X^{cs}_{m_1} $ of \eqref{equ:fixed}, we have
\begin{lem}\label{lem:ss}
	Under ($ \star\star $), let $ \sigma^c_{*} > 0 $ such that $ \max\{\sup_{m}\{ \lambda_{cs}(m) \}, 1\} \sigma^c_{*} + K_2 \eta < (1 - \lambda^*_s e_2) \sigma_{0} $, then
	\[
	\lip x^s_{\widehat{m}}(x^c, \cdot)|_{X^{s}_{m}(e_2\sigma_*)} \leq \lambda^*_s, ~ x^c \in X^{c}_{m}(\sigma^c_{*}), ~\lip x_{\widehat{m}}(\cdot)|_{X^{c}_{m}(\sigma^c_{*}) \oplus X^{s}_{m}(e_2\sigma_*)} \leq \lambda_{cs}(m),
	\]
	and $ x_{\widehat{m}}(X^{c}_{m}(\sigma^c_{*}) \oplus X^{s}_{m}(e_2\sigma_*)) \subset X^{cs}_{m_1}(\sigma_{0}) $.
\end{lem}
\begin{proof}
	If $ r' > 0 $ is such that $ x_{\widehat{m}}(X^{cs}_{m}(r')) \subset X^{cs}_{m_1}(\sigma_{0}) $, then by ($ \star\star $), for $ x^c \in X^{c}_{m}(r') $, we have $ \lip x^s_{\widehat{m}}(x^c, \cdot)|_{X^{s}_{m}(r')} \leq \lambda^*_s $. So arguing as in the proof of \autoref{lem:first1}, we see for $ x^c \in X^{c}_{m}(\sigma^c_{*}) $, $ \lip x^s_{\widehat{m}}(x^c, \cdot)|_{X^{s}_{m}(e_2\sigma_*)} \leq \lambda^*_s $ and then $ x_{\widehat{m}}(X^{c}_{m}(\sigma^c_{*}) \oplus X^{s}_{m}(e_2\sigma_*)) \subset X^{cs}_{m_1}(\sigma_{0}) $. The proof is complete.
\end{proof}

Under the above assumption ($ \star\star $), using the second equation of \eqref{equ:local}, we can define $ \widetilde{f}_{\widehat{m}} (x^{cs}) $ for all $ x^{cs} = (x^c, x^s) \in X^{c}_{m}(\sigma^c_{*}) \oplus X^{s}_{m}(e_2\sigma_{*}) $, where $ \sigma^c_{*} $ is now taken as
\[
\sigma^c_{*} = (\max\{\sup_{m}\{ \lambda_{cs}(m) \}, 1\})^{-1} ((1 - \lambda^*_s e_2)\sigma_{0} - \eta_0) / 2;
\]
here, $ \eta_0 $ is chosen to be even smaller, e.g., $ \eta_0 = (1 - \lambda^*_s e_2)\chi_{*}\epsilon_{*} / 16 $.
In this case, we only need to truncate the $ c $-direction. More precisely, the definition of $ \Psi $ (see \eqref{equ:cutoff}) is now replaced by
\[
\Psi: \widehat{\Sigma} \to \mathbb{R}_+, \quad \widehat{m} \mapsto \ell( \widehat{d}(\widehat{m}, \widehat{K}) ),
\]
where $ \ell(\cdot) $ is given by \eqref{equ:ct}. The graph transform $ \varGamma $ also satisfies
\begin{equation}\label{equ:graphss}
\varGamma: \varSigma_{\mu, K_1, \epsilon_{*}, \sigma_{*}, \varrho_* } \to \varSigma_{\mu, K'_1, \epsilon_{*}, \sigma_{*}, \varrho_* }, \quad h \mapsto \widehat{h} \triangleq \Psi \cdot \widetilde{\varGamma}(h).
\end{equation}

\subsection{Limited case II: strictly inflowing}\label{sub:limited0}

We will add an assumption (called the \emph{strictly inflowing assumption}) on $ H $ such that in the construction of the graph transform, the truncation is not needed. (In many situations, it is not so easy to verify this condition.)
\begin{enumerate}[ ($ \bullet\bullet $) ]
	\item There exists a positive constant $ c < \min\{ c_1, e_1, c^{-1}_2, e^{-1}_2 \} $ such that for all $ m \in K $,
	\[
	\begin{cases}
	\widehat{F}^{cs}_{m}(X^{c}_{m} (c^{-1}\epsilon_{*}) \oplus X^{s}_{m}(c^{-1}\sigma_{*}), X^{u}_{u(m)}(c^{-1}\varrho_{*})) \subset X^{c}_{u(m)} (c\epsilon_{*}) \oplus X^{s}_{u(m)}(c\sigma_{*}), \\
	\widehat{G}^{cs}_{m}(X^{c}_{m} (c^{-1}\epsilon_{*}) \oplus X^{s}_{m}(c^{-1}\sigma_{*}), X^{u}_{u(m)}(c^{-1}\varrho_{*})) \subset X^{u}_{u(m)}(c\varrho_{*}).
	\end{cases}
	\]
\end{enumerate}

Under the above assumption ($ \bullet\bullet $) and the Assumptions Case (1) (in \autoref{sub:preparation}) with $ \varsigma_0 \geq 1 $ (in (A3) (a)), the construction of the graph transform is much simpler. The angle condition $ \sup_{m\in K} \alpha(m) \beta'(u(m)) < 1 / 2 $ in (A3) (a) (i) can be replaced by $ \sup_{m\in K} \alpha(m) \beta'(u(m)) < 1 $.
\begin{enumerate}[(1)]
	\item Note that $ \sigma_{0} = e_1\sigma_{*} > c\sigma_{*} $.

	\item By assumption ($ \bullet\bullet $), for the following fixed point equation,
	\[
	\widehat{F}^{cs}_{m}( x^{cs}, f_{\widehat{m}_1}(x_{\widehat{m}} (x^{cs})) ) = x_{\widehat{m}} (x^{cs}), \quad x^{cs} = (x^c, x^s) \in X^{c}_{m}(c^{-1}\epsilon_{*}) \oplus X^{s}_{m}(c^{-1}\sigma_{*}),
	\]
	we have $ x_{\widehat{m}} (x^{cs}) = (x^c_{\widehat{m}} (x^{cs}), x^s_{\widehat{m}} (x^{cs}) ) \in X^{c}_{m_1} (c\epsilon_{*}) \oplus X^{s}_{m_1}(c\sigma_{*}) $ (\autoref{lem:first1} is not needed anymore); here note that as $ h \in \varSigma_{\mu, K_1, \epsilon_{*}, \sigma_{*}, \varrho_* } $, one has $ f_{\widehat{m}_1} (X^{c}_{u(m)}(c\epsilon_{*}) \oplus X^{s}_{u(m)}(c\sigma_{*})) \subset X^{u}_{u(m)}(c^{-1}\varrho_{*}) $.
	\item Again, using the second equation of \eqref{equ:local}, we can define $ \widetilde{f}_{\widehat{m}} (x^{cs}) $ for all $ x^{cs} = (x^c, x^s) \in X^{c}_{m}(c^{-1}\epsilon_{*}) \oplus X^{s}_{m}(c^{-1}\sigma_{*}) $.

	\item \autoref{lem:first2} also holds. More precisely,
	$ |\widetilde{f}_{\widehat{m}} (0)| \leq K'_1 \eta $, where $ K'_1 = \overline{\lambda}_u ( \hat{\beta} + K_1 ) + 1 $, and
	\[
	\lip \widetilde{f}_{\widehat{m}}(\cdot)|_{X^{c}_{m}(c^{-1}\epsilon_{*}) \oplus X^{s}_{m}(c^{-1}\sigma_{*})} \leq \beta'(m), \quad \widetilde{f}_{\widehat{m}} (X^{c}_{m}(c^{-1}\epsilon_{*}) \oplus X^{s}_{m}(c^{-1}\sigma_{*})) \subset X^{u}_{m}(c\varrho_{*}),
	\]
	where the above second formula is a consequence of assumption ($ \bullet \bullet $) on $ \widehat{G}^{cs}_{m}(\cdot) $ and $ h \in \varSigma_{\mu, K_1, \epsilon_{*}, \sigma_{*}, \varrho_* } $.

	\item \label{it:limited5} Similar to \autoref{lem:locLip}, for $ \widehat{m}_j \in \widehat{K} $, $ m_j = \phi(\widehat{m}_j) $, $ j = 0,1,2 $, and $ x^{cs}_i \in X^{c}_{m_i}(c^{-1}\epsilon_{*}) \oplus X^{s}_{m_i}(c^{-1}\sigma_{*}) $, $ i = 1,2 $, write
	\[
	m_i + x^{cs}_i + \widetilde{f}_{\widehat{m}_i} (x^{cs}_i) = \overline{m}_i + \overline{x}^s_i + \overline{x}^{u}_i = m_0 + \widehat{x}^s_i + \widehat{x}^c_i + \widehat{x}^{u}_i, ~i = 1,2,
	\]
	where $ \widehat{x}^{\kappa} \in X^{\kappa}_{m_0} $, $ \kappa = s, c, u $, $ \widehat{\overline{m}}_i \in \widehat{U}_{\widehat{m}_i} (c_{1}^{-1}c^{-1}\epsilon_{*}) $, $ \overline{m}_i = \phi(\widehat{\overline{m}}_i) $, $ \overline{x}^s_i \in X^{s}_{\overline{m}_i} (e^{-1}_1c^{-1}\sigma_{*}) $ and $ \overline{x}^u_i \in X^{u}_{\overline{m}_i} (\varrho_{*}) $. If $ \widehat{\overline{m}}_1, \widehat{\overline{m}}_2 \in \widehat{U}_{\widehat{m}_0}(\epsilon_{*}) $, then \eqref{equ:lip} also holds.

	\item \label{it:limited6} From \eqref{it:limited5} and \autoref{lem:represent3}, for every $ \widehat{m} \in \widehat{K}_{\epsilon_{*}} $ and $ \overline{x}^s \in X^s_{m} (\sigma_{*}) $, where $ m = \phi(\widehat{m}) $, there is a unique $ \overline{x}^{u} \in X^{u}_{m}(\varrho_{*}) $ such that
	\[
	(\widehat{m}, \overline{x}^s, \overline{x}^u) \in \bigcup_{\widehat{m}_0 \in \widehat{K}}\graph \widetilde{f}_{\widehat{m}_0} \subset X^{h}_{\widehat{\Sigma}};
	\]
	so we can define $ \widetilde{h}(\widehat{m}, \overline{x}^s) = \overline{x}^u $ and the \emph{graph transform} as
	\[
	\widetilde{\varGamma}: \varSigma_{\mu, K_1, \epsilon_{*}, \sigma_{*}, \varrho_* } \to \varSigma_{\mu, K'_1, \epsilon_{*}, \sigma_{*}, \varrho_* }, h \mapsto \widetilde{h}.
	\]
	In this case, without loss of generality, assume $ \Sigma = K_{\epsilon_{*}} $.
\end{enumerate}

\subsection{Limited case III: parameter-dependent correspondences}\label{sub:limited*}
We now consider the following Assumption (CS) for parameter-dependent correspondences $ \{ H^{\delta} \} $, which serves as a preparation for differential equations.

\vspace{.5em}
\noindent{Assumption (CS)}.
Assume (A1), (A2) (in \autoref{subsec:main}) hold.
Write
\[
\widehat{H}^{\delta}_{m} \triangleq \widehat{H}^{\delta}(m+\cdot) - u(m) \sim (\widehat{F}^{cs, \delta}_{m}, \widehat{G}^{cs, \delta}_{m}) : \widehat{X}^{cs}_{m}(r) \oplus \widehat{X}^{u}_{m}(r_1) \to \widehat{X}^{cs}_{u(m)}(r_2) \oplus \widehat{X}^{u}_{u(m)} (r).
\]
Let (A3) (b) (in \autoref{subsec:main}) hold for $ \widehat{F}^{cs, \delta}_{m}, \widehat{G}^{cs, \delta}_{m} $ instead of $ \widehat{F}^{cs}_{m}, \widehat{G}^{cs}_{m} $. Take a small $ \epsilon_{00} > 0 $, a constant $ 0 < c < 1 $, a large constant $ C_{\#} > 1 $, and a sufficiently small $ \zeta_{00} > 0 $. For each $ m \in K $, suppose
\[
\widehat{H}^{\delta}_{m} \sim (\widehat{F}^{cs, \delta}_{m}, \widehat{G}^{cs, \delta}_{m}) : \widehat{X}^{cs}_{m}(C_{\#}\epsilon_{00}) \oplus \widehat{X}^{u}_{m}(r_1) \to \widehat{X}^{cs}_{u(m)}(r_2) \oplus \widehat{X}^{u}_{u(m)} (\zeta_{00}C_{\#}\epsilon_{00} + \eta)
\]
satisfies the (A$ ' $)($ \alpha(m) $, $ \lambda_{u}(m) $) (B)($ \beta(m); \beta'(m), \lambda_{cs}(m) $) condition. Also assume the functions $ \alpha(\cdot), \beta(\cdot), \beta'(\cdot) $ and $ \lambda_{u}(\cdot), \lambda_{cs}(\cdot) $ satisfy (A3) (a) (i) (iii) with only $ \varsigma_0 \geq 1 $ (excluding (A3) (a) (ii)) in \autoref{subsec:main}. In addition, assume the following conditions (i)--(iii):
\begin{enumerate}[(i)]
\item \emph{(Very strong $ s $-contraction)} For
\begin{multline*}
	(\hat{x}^{c}_i, \hat{x}^s_i, \hat{x}^{u}_i) \times (\tilde{x}^{c}_i, \tilde{x}^s_i, \tilde{x}^{u}_{i}) \\ \in \graph \widehat{H}^{\delta}_{m} \cap \{\{\widehat{X}^{cs}_{m}(C_{\#}\epsilon_{00}) \oplus \widehat{X}^{u}_{m}(r_{1})\} \times \{\widehat{X}^{cs}_{u(m)}(r_{2}) \oplus \widehat{X}^{u}_{u(m)} (\zeta_{00}C_{\#}\epsilon_{00} + \eta)\}\},
\end{multline*}
$ i = 1,2 $, $ m \in K $, if $ |\tilde{x}^{u}_{1} - \tilde{x}^{u}_{2}| \leq B (|\hat{x}^c_1 - \hat{x}^c_2| + |\hat{x}^s_1 - \hat{x}^s_2|) $, then
\[
|\tilde{x}^{s}_{1} - \tilde{x}^{s}_{2}| \leq \zeta_{00} |\hat{x}^c_1 - \hat{x}^c_2| + \lambda^*_s |\hat{x}^s_1 - \hat{x}^s_2|,
\]
where $ B > \sup_{m \in K}\lambda_{cs}(m) \beta(m) $ is some constant and $ \lambda^*_s < 1 $.

\item \emph{($ c $-direction strictly inflowing)} For all $ \delta $,
\[
\widehat{\Pi}^{c}_{u(m)}\widehat{F}^{cs, \delta}_{m}(\widehat{X}^{c}_{m} (c^{-1}\epsilon_{00}) \oplus \widehat{X}^{s}_{m}(c^{-1}\epsilon_{00}), \widehat{X}^{u}_{u(m)}(c^{-1}\epsilon_{00})) \subset \widehat{X}^{c}_{u(m)} (c\epsilon_{00}).
\]

\item There is a positive constant $ \gamma^*_{u} < 1 $ such that for all $ m \in K $ and $ (x^{cs}, x^{u}) \in \widehat{X}^{cs}_{m} (C_{\#}\epsilon_{00}) \times \widehat{X}^{u}_{u(m)} (\zeta_{00}C_{\#}\epsilon_{00} + \eta) $,
\[
|\widehat{G}^{cs, \delta}_{m}(x^{cs}, x^{u})| \leq \zeta_{00} |x^{cs}| + \gamma^*_{u} |x^{u}| + \eta.
\]
\end{enumerate}
\vspace{.5em}

We emphasize that the (A$ ' $) (B) condition for $ \widehat{H}^{\delta}_{m} $ is defined on the domain
\[
\{\widehat{X}^{cs}_{m}(C_{\#}\epsilon_{00}) \oplus \widehat{X}^{u}_{m}(r_1)\} \times \{\widehat{X}^{cs}_{u(m)}(r_2) \oplus \widehat{X}^{u}_{u(m)} (\zeta_{00}C_{\#}\epsilon_{00} + \eta)\}
\]
(not on $ \{\widehat{X}^{cs}_{m}(r) \oplus \widehat{X}^{u}_{m}(r_1)\} \times \{\widehat{X}^{cs}_{u(m)}(r_2) \oplus \widehat{X}^{u}_{u(m)} (r)\} $).

Under the above assumptions, the truncation is not needed, and the construction of the graph transform becomes significantly simpler.

\begin{enumerate}[(1)]
 	\item To determine how small the constants $ \epsilon_{00}, \zeta_{00} $ can be chosen, let $ \epsilon_{00} = \epsilon_{*} $ and consider $ \zeta_{00} $ as a function of $ \epsilon_{*} $ such that $ \zeta_{00} \to 0 $ as $ \epsilon_{*} \to 0 $.
 	Take $ \hat{\chi}_{*} = \max \{ \chi_{*}, 16c^{-1}(1 - \lambda^*_s)^{-1}\zeta_{00}, 16c^{-1}(1 - \gamma^*_{u})^{-1}\zeta_{00} \} $ (in ``Choice of constants'' in \autoref{sub:preparation}). Choose the constant $ C_{\#} $ large such that $ C_{\#} > 2^{20}c^{-1}\max\{ (1 - \lambda^*_s)^{-1}, (1 - \gamma^*_{u})^{-1} \} $ to ensure that the construction makes sense in the domain
	\[
	\widehat{X}^{cs}_{m}(C_{\#}\epsilon_{00}) \oplus \widehat{X}^{u}_{m}(r_1) \times \widehat{X}^{cs}_{u(m)}(r_2) \oplus \widehat{X}^{u}_{u(m)} (\zeta_{00}C_{\#}\epsilon_{00} + \eta).
	\]
	A sufficient condition is $ \zeta_{00} \leq \frac{c}{16}\min\{ 1 - \lambda^*_s, 1 - \gamma^*_{u} \} \chi_{*} $ (which allows us to take $ \hat{\chi}_{*} = \chi_{*} $). All constants from \autoref{sub:preparation} and \autoref{sub:unlimited} will be used here. Assuming $ \epsilon_{*} $ is small such that $ c < \min\{ c_1, e_1, c^{-1}_2, e^{-1}_2 \} $, we may take $ \Sigma = K_{\epsilon_{*}} $ without loss of generality.

	\item Consider $ h \in \varSigma_{\mu, K_1, \epsilon_{*}, \sigma_{*}, \varrho_* } $ and let $ f_{\widehat{m}_0} $ denote the local representation of $ \graph h \cap X^{su}_{\widehat{K}_{\epsilon_*}} (\sigma_*, \varrho_*) $ at $ \widehat{m}_0 \in \widehat{K} $. Note that $ f_{\widehat{m}_0} (X^{c}_{m_0}(c_1\epsilon_{*}) \oplus X^{s}_{m_0}(e_1\sigma_{*})) \subset X^{u}_{m_0}(e_2\varrho_{*}) $, where $ m_0 = \phi(\widehat{m}_0) $. We now replicate nearly all steps in \autoref{sub:unlimited} but replacing $ (\widehat{F}^{cs}_{m}, \widehat{G}^{cs}_{m}) $ by $ (\widehat{F}^{cs, \delta}_{m}, \widehat{G}^{cs, \delta}_{m}) $.

	\item Let $ f^1_{\widehat{m}_0}(x^c, x^s) = f_{\widehat{m}_0} (x^c, r_{e_1\sigma_{*}}(x^s)) $. Consider the following fixed point equation (see also \eqref{equ:fixed})
	\[
	\widehat{F}^{cs, \delta}_{m}( x^{cs}, f^1_{\widehat{m}_1}(z) ) = z, \quad z \in X^{c}_{m_1}(c_1\epsilon_{*}) \oplus X^{u}_{m_1}.
	\]
	Let $ x_{\widehat{m}} (x^{cs}) = (x^c_{\widehat{m}} (x^{cs}), x^s_{\widehat{m}} (x^{cs}) ) \in X^{c}_{m_1}(c_1\epsilon_{*}) \oplus X^{u}_{m_1} $ be the unique fixed point satisfying this equation. Then
	\[
	\widehat{\Pi}^{c}_{m_1}\widehat{F}^{cs, \delta}_{m}(x^{cs}, f^1_{\widehat{m}_1}(x_{\widehat{m}} (x^{cs}))) = x^c_{\widehat{m}} (x^{cs}), \quad
	\widehat{\Pi}^{s}_{m_1}\widehat{F}^{cs, \delta}_{m}(x^{cs}, f^1_{\widehat{m}_1}(x_{\widehat{m}} (x^{cs}))) = x^s_{\widehat{m}} (x^{cs}).
	\]

	\item Under (CS) (i)(ii) with $ \eta \leq \epsilon_{0, *} \eta_0 $ small, analogous to \autoref{lem:ss}, we have
	\[
	\lip x^s_{\widehat{m}}(\cdot, \cdot)|_{X^{c}_{m}(c_2\epsilon_{*}) \oplus X^{s}_{m}(e_2\sigma_*)} \leq \lambda^*_s, \quad \lip x_{\widehat{m}}(\cdot)|_{X^{c}_{m}(c_2\epsilon_{*}) \oplus X^{s}_{m}(e_2\sigma_*)} \leq \lambda_{cs}(m),
	\]
	and $ x_{\widehat{m}}(X^{c}_{m}(c_2\epsilon_{*}) \oplus X^{s}_{m}(e_2\sigma_*)) \subset X^{c}_{m_1}(c_1\epsilon_{*}) \oplus X^{s}_{m_1}(e_1\sigma_{*}) $.
	\begin{proof}
		If $ r' > 0 $ is such that $ x_{\widehat{m}}(X^{cs}_{m}(r')) \subset X^{c}_{m_1}(c_1\epsilon_{*}) \oplus X^{s}_{m_1}(e_1\sigma_{*}) $, then by (CS) (i), for any $ (x^c_i, x^s_{i}) \in X^{cs}_{m}(r') $, one gets
		\[
		|x^s_{\widehat{m}}(x^c_1, x^s_1) - x^s_{\widehat{m}}(x^c_2, x^s_2)| \leq \zeta_{00}|x^c_1 - x^c_2| + \lambda^*_{s} |x^s_1 - x^s_2|.
		\]
		Applying the same reasoning as in \autoref{lem:first1}, for
		\[
		\epsilon_{1,1} = \sup \left\{ \epsilon \leq \epsilon_{*}: x_{\widehat{m}}(X^{c}_{m}(c_2\epsilon) \oplus X^{s}_{m}(e_2\sigma_*)) \subset X^{c}_{m_1}(c_1\epsilon_{*}) \oplus X^{s}_{m_1}(e_1\sigma_{*}) \right\},
		\]
		we have $ \epsilon_{1,1} = \epsilon_{*} $, which implies that $ \lip x_{\widehat{m}}(\cdot)|_{X^{c}_{m}(c_2\epsilon_{*}) \oplus X^{s}_{m}(e_2\sigma_*)} \leq \lambda_{cs}(m) $.
	\end{proof}

	\item Under (CS) (iii) with $ \eta \leq \epsilon_{0,*} \eta_0 $ small, for $ x^{cs} \in X^{cs}_{m} (c_2\epsilon_{*}) $ and $ x^u \in X^{u}_{u(m)} (e_2\varrho_{*}) $,
	\begin{align*}
	|\widehat{G}^{cs, \delta}_{m}(x^{cs}, x^u)| & \leq \gamma^*_{u} |x^u| + \zeta_{00} |x^{cs}| + |\widehat{G}^{cs, \delta}_{m}(0, 0)| < e_1 \varrho_{*}.
	\end{align*}
	Now we can define $ \widetilde{f}_{\widehat{m}} (x^{cs}) $ for all $ x^{cs} = (x^c, x^s) \in X^{c}_{m}(c_2\epsilon_{*}) \oplus X^{s}_{m}(e_2\sigma_{*}) $ by (see also \eqref{equ:local})
	\[
	\widehat{G}^{cs, \delta}_{m} ( x^{cs}, f_{\widehat{u}(\widehat{m})} ( x_{\widehat{m}} (x^{cs}) ) ) \triangleq \widetilde{f}_{\widehat{m}} (x^{cs});
	\]
	moreover, $ |\widetilde{f}_{\widehat{m}} (x^{cs})| \leq e_1 \varrho_{*} $ and $ \lip \widetilde{f}_{\widehat{m}}(\cdot) \leq \beta'(m) $.

	\item Following the approach in \autoref{sub:limited0} \eqref{it:limited5} \eqref{it:limited6}, for sufficiently small $ \delta > 0 $, the \emph{graph transform} $ \widetilde{\varGamma}^{\delta} $ (defined for $ H^{\delta} $) satisfies
	\[
	\widetilde{\varGamma}^{\delta}: \varSigma_{\mu, K_1, \epsilon_{*}, \sigma_{*}, \varrho_* } \to \varSigma_{\mu, K'_1, \epsilon_{*}, \sigma_{*}, \varrho_* }, \quad h \mapsto \widetilde{h}.
	\]

	\item For any $ \delta_1, \delta_2 $, if $ H^{\delta_1} \circ H^{\delta_2} = H^{\delta_2} \circ H^{\delta_1} $, then $ \widetilde{\varGamma}^{\delta_1} \circ \widetilde{\varGamma}^{\delta_2} = \widetilde{\varGamma}^{\delta_2} \circ \widetilde{\varGamma}^{\delta_1} $ on $ X^s_{\widehat{K}_{\epsilon_{*}}} (\sigma_{*}) $.
\end{enumerate}

\section{Existence of center-stable manifold and proof of \autoref{thm:I}}

In this section, we further assume (A3) (a) (ii), which implies $ \overline{\lambda}_{u} < 1 $ (as stated in Observation \eqref{OVII}). So we take
\[
K'_1 = K_1 = \frac{ \overline{\lambda}_u\hat{\beta} + 1 }{1 - \overline{\lambda}_u}.
\] Let $ 0 < \eta_1 < \sigma^1_{*} $ close to $ \sigma^1_{*} $ (e.g., $ \eta_1 = (1 - \epsilon_{*}) \sigma^1_{*} $), and $ \eta \leq \epsilon_{0,*} \eta_0 $.
\begin{enumerate}[$ \bullet $]
	\item In case (1), take $ \eta_2 = \frac{\varsigma' - 2}{\varsigma' - 1} \eta_1 $, where $ \varsigma' = (\varsigma + 2) / 2 $; and

	\item in case (2), take $ \eta_2 = (1/2) \eta_1 $.
\end{enumerate}
Let
\begin{multline}\label{equ:space1}
\varSigma_{lip, \mu, K_1} = \{ h: X^s_{\widehat{\Sigma}} (\sigma_*) \to \overline{X^u_{\widehat{\Sigma}} (\varrho_*)} ~\text{is a bundle map over}~ \id: h|_{X^s_{\widehat{\Sigma}} (\sigma_*) \setminus X^s_{\widehat{K}_{\eta_1}} (\eta_1)} = 0,
\\ \sup_{\widehat{m} \in \widehat{K}} |h(\widehat{m}, 0)| \leq K_1\eta,
\graph h \cap X^{su}_{\widehat{K}_{\epsilon_*}} (\sigma_*, \varrho_*) ~\text{is $ \mu $-Lip in $ u $-direction near $ \widehat{K} $}\},
\end{multline}
and define a metric on $ \varSigma_{lip, \mu, K_1} $ by
\[
d_1(h_1, h_2) = \sup\{ |h_1(\widehat{m},\overline{x}^s) - h_2(\widehat{m},\overline{x}^s)|: (\widehat{m},\overline{x}^s) \in X^s_{\widehat{\Sigma}} (\sigma_*) \}.
\]

\begin{lem}
	$ \varSigma_{lip, \mu, K_1} \neq \emptyset $ and $ (\varSigma_{lip, \mu, K_1}, d_1) $ is a complete metric space.
\end{lem}
\begin{proof}
	Note that $ X^s_{\widehat{\Sigma}} (\sigma_*) \in \varSigma_{lip, \mu, K_1} $ (by \autoref{lem:lip2} and the choice of $ \mu $ in Observation \eqref{OVI}), and so $\varSigma_{lip, \mu, K_1} \neq \emptyset$. The rest is obvious.
\end{proof}

By \autoref{lem:first2} and \autoref{lem:belong}, the graph transform $ \varGamma $ maps $ \varSigma_{lip, \mu, K_1} $ into $ \varSigma_{lip, \mu, K_1} $ if $ \epsilon_{*} $, $ \chi(\epsilon_{*}) $ are sufficiently small.
In this context, we will show $ \lip \varGamma < 1 $.

Take $ h^i \in \varSigma_{lip, \mu, K_1} $ and set $ \widehat{h}^i = \varGamma(h^i) $, $ \widetilde{h}^i = \widetilde{\varGamma}(h^i) $ (see \eqref{equ:graph}), $ i = 1,2 $. Since
\[
\sup_{(\widehat{m}, \overline{x}^s) \in X^s_{\widehat{\Sigma}} (\sigma_*)}|\widehat{h}^1(\widehat{m}, \overline{x}^s) - \widehat{h}^2(\widehat{m}, \overline{x}^s)| \leq \sup_{(\widehat{m}, \overline{x}^s) \in X^s_{\widehat{K}_{\eta_1}}(\eta_1)}|\widetilde{h}^1(\widehat{m}, \overline{x}^s) - \widetilde{h}^2(\widehat{m}, \overline{x}^s)|,
\]
we only need to show the following.
\begin{lem}\label{lem:contractive}
	There is a positive constant $ \widehat{\lambda}_{u} < 1 $ such that
	\[
	\sup_{(\widehat{m}, \overline{x}^s) \in X^s_{\widehat{K}_{\eta_1}}(\eta_1)}|\widetilde{h}^1(\widehat{m}, \overline{x}^s) - \widetilde{h}^2(\widehat{m}, \overline{x}^s)| \leq \widehat{\lambda}_{u} \sup_{(\widehat{m}, \overline{x}^s) \in X^s_{\widehat{\Sigma}} (\sigma_*)}|{h}^1(\widehat{m}, \overline{x}^s) - {h}^2(\widehat{m}, \overline{x}^s)|.
	\]
\end{lem}

\begin{proof}
Take $ \widehat{m} \in \widehat{K}_{\eta_1} \subset \widehat{K}_{\epsilon_{*}} $ and $ \overline{x}^s \in X^s_{m}(\eta_1) \subset X^s_{m}(\sigma_{*}) $, where $ m = \phi(\widehat{m}) $.
Take $ \widehat{m}_0 \in \widehat{K} $ such that $ \widehat{m} \in \widehat{U}_{\widehat{m}_0}(\eta_1) $. Set $ m_0 = \phi(\widehat{m}_0) $, $ \widehat{m}_1 = \widehat{u}(\widehat{m}_0) $, and $ m_1 = \phi(\widehat{m}_1) $ ($ \in K $).
Let $ f^i_{\widehat{m}_1}, \widetilde{f}^i_{\widehat{m}_0} $ be the local representations of $ \graph h^{i} \cap X^{su}_{\widehat{K}_{\epsilon_*}} (\sigma_*, \varrho_*) $ and $ \graph \widetilde{h}^{i} \cap X^{su}_{\widehat{K}_{\epsilon_*}} (\sigma_*, \varrho_*) $ at $ \widehat{m}_1 $ and $ \widehat{m}_0 $, respectively. By the construction of $ \widetilde{h}^{i} $, we have
\begin{equation}\label{equ:abc}
m + \overline{x}^{s} + \widetilde{h}^{i}(\widehat{m}, \overline{x}^{s}) = m_0 + x^{cs}_i + \widetilde{f}^i_{\widehat{m}_0}(x^{cs}_i), \quad i = 1,2,
\end{equation}
where $ x^{cs}_i \in X^{cs}_{m_0} (\sigma^{c}_{*}) \subset X^{cs}_{m_0} (\sigma_{*}) $ and $ |\widetilde{f}^i_{\widehat{m}_0}(x^{cs}_i)| \leq e_1\varrho_{*} $. Furthermore,
\begin{equation}\label{equ:zz}
m_1 + x^{i}_{\widehat{m}_0}(x^{cs}_i) + f^{i}_{\widehat{m}_1}(x^{i}_{\widehat{m}_0}(x^{cs}_i)) \in H(m_0 + x^{cs}_i + \widetilde{f}^i_{\widehat{m}_0}(x^{cs}_i)),
\end{equation}
i.e.,
\[
\begin{cases}
\widehat{F}^{cs}_{m_0} ( x^{cs}_i, f^i_{\widehat{m}_1} ( \widetilde{x}^i ) ) = x^i_{\widehat{m}_0} (x^{cs}_i), \\
\widehat{G}^{cs}_{m_0} ( x^{cs}_i, f^i_{\widehat{m}_1} ( \widetilde{x}^i ) ) = \widetilde{f}^i_{\widehat{m}_0}(x^{cs}_i),
\end{cases}
\]
where $ \widetilde{x}^i = x^{i}_{\widehat{m}_0}(x^{cs}_i) \in X^{cs}_{m_1}(e_1\sigma_{*}) \subset X^{cs}_{m_1}(\sigma_{*}) $; also $ |f^i_{\widehat{m}_1} ( \widetilde{x}^i ) | \leq e_2\varrho_{*} $.

In the following, we use $ \varpi^*_0 $, which may vary from line to line, to denote positive constants that approach $ 0 $ as $ \epsilon_{*}, \chi(\epsilon_{*}) \to 0 $.
By \eqref{equ:abc}, we have (see also \autoref{lem:lip2})
\[
|x^{cs}_1 - x^{cs}_2| \leq \varpi^*_0 |\widetilde{f}^1_{\widehat{m}_0}(x^{cs}_1) - \widetilde{f}^2_{\widehat{m}_0}(x^{cs}_2)|.
\]
By the (A$ ' $) condition ((A3) (a)), we get
\begin{align*}
| \widetilde{f}^1_{\widehat{m}_0}(x^{cs}_1) - \widetilde{f}^2_{\widehat{m}_0}(x^{cs}_1) | & \leq \lambda_{u}(m_0)|f^1_{\widehat{m}_1} ( \widetilde{x}^1 ) - f^2_{\widehat{m}_1} ( x^{2}_{\widehat{m}_0}(x^{cs}_1) )|, \\
|\widetilde{x}^1 - x^{2}_{\widehat{m}_0}(x^{cs}_1)| & \leq \alpha(m_0) |f^1_{\widehat{m}_1} ( \widetilde{x}^1 ) - f^2_{\widehat{m}_1} ( x^{2}_{\widehat{m}_0}(x^{cs}_1) )|,
\end{align*}
and so
\begin{align*}
|f^1_{\widehat{m}_1} ( \widetilde{x}^1 ) - f^2_{\widehat{m}_1} ( x^{2}_{\widehat{m}_0}(x^{cs}_1) )| & \leq |f^1_{\widehat{m}_1} ( \widetilde{x}^1 ) - f^2_{\widehat{m}_1} ( \widetilde{x}^1 )| + |f^2_{\widehat{m}_1} ( \widetilde{x}^1 ) - f^2_{\widehat{m}_1} ( x^{2}_{\widehat{m}_0}(x^{cs}_1) )|\\
& \leq |f^1_{\widehat{m}_1} ( \widetilde{x}^1 ) - f^2_{\widehat{m}_1} ( \widetilde{x}^1 )| + \mu_1(\widehat{m}_1) |\widetilde{x}^1 - x^{2}_{\widehat{m}_0}(x^{cs}_1)| \\
& \leq |f^1_{\widehat{m}_1} ( \widetilde{x}^1 ) - f^2_{\widehat{m}_1} ( \widetilde{x}^1 )| \\
& \quad + \alpha(m_0) \mu_1(\widehat{m}_1) |f^1_{\widehat{m}_1} ( \widetilde{x}^1 ) - f^2_{\widehat{m}_1} ( x^{2}_{\widehat{m}_0}(x^{cs}_1) )|,
\end{align*}
which yields
\[
|f^1_{\widehat{m}_1} ( \widetilde{x}^1 ) - f^2_{\widehat{m}_1} ( x^{2}_{\widehat{m}_0}(x^{cs}_1) )| \leq \frac{1}{1 - \alpha(m_0) \mu_1(\widehat{m}_1)} |f^1_{\widehat{m}_1} ( \widetilde{x}^1 ) - f^2_{\widehat{m}_1} ( \widetilde{x}^1 )|.
\]
Thus,
\begin{align*}
&~| \widetilde{f}^1_{\widehat{m}_0}(x^{cs}_1) - \widetilde{f}^2_{\widehat{m}_0}(x^{cs}_2) | \\
 \leq &~ |\widetilde{f}^1_{\widehat{m}_0}(x^{cs}_1) - \widetilde{f}^2_{\widehat{m}_0}(x^{cs}_1)| + |\widetilde{f}^2_{\widehat{m}_0}(x^{cs}_1) - \widetilde{f}^2_{\widehat{m}_0}(x^{cs}_2)| \\
 \leq &~ \frac{\lambda_{u}(m_0)}{1 - \alpha(m_0) \mu_1(\widehat{m}_1)} |f^1_{\widehat{m}_1} ( \widetilde{x}^1 ) - f^2_{\widehat{m}_1} ( \widetilde{x}^1 )| + \mu_1(\widehat{m}_1) |x^{cs}_1 - x^{cs}_2|\\
 \leq &~ \frac{\lambda_{u}(m_0)}{1 - \alpha(m_0) \mu_1(\widehat{m}_1)} |f^1_{\widehat{m}_1} ( \widetilde{x}^1 ) - f^2_{\widehat{m}_1} ( \widetilde{x}^1 )| + \mu_1(\widehat{m}_1) \varpi^*_0 |\widetilde{f}^1_{\widehat{m}_0}(x^{cs}_1) - \widetilde{f}^2_{\widehat{m}_0}(x^{cs}_2)|,
\end{align*}
and
\begin{equation}\label{equ:estM}
| \widetilde{f}^1_{\widehat{m}_0}(x^{cs}_1) - \widetilde{f}^2_{\widehat{m}_0}(x^{cs}_2) | \leq \frac{(1+\varpi^*_0)\lambda_{u}(m_0)}{1 - \alpha(m_0) \mu_1(\widehat{m}_1)} |f^1_{\widehat{m}_1} ( \widetilde{x}^1 ) - f^2_{\widehat{m}_1} ( \widetilde{x}^1 )|.
\end{equation}

By \eqref{equ:abc}, we further obtain (see \eqref{equ:estimates})
\begin{equation}\label{equ:estM0}
|\widetilde{h}^{1}(\widehat{m}, \overline{x}^{s}) - \widetilde{h}^{2}(\widehat{m}, \overline{x}^{s})| \leq (1 + \varpi^*_0) |\widetilde{f}^1_{\widehat{m}_0}(x^{cs}_1) - \widetilde{f}^2_{\widehat{m}_0}(x^{cs}_2)|.
\end{equation}
Rewrite
\[
\begin{cases}
m_1 + \widetilde{x}^1 + f^{1}_{\widehat{m}_1}(\widetilde{x}^1) = \overline{m}_1 + \overline{x}^{s}_1 + h^{1}(\widehat{\overline{m}}_1, \overline{x}^{s}_1), \\
m_1 + \widetilde{x}^1 + f^{2}_{\widehat{m}_1}(\widetilde{x}^1) = \overline{m}_2 + \overline{x}^{s}_2 + h^{2}(\widehat{\overline{m}}_2, \overline{x}^{s}_2),
\end{cases}
\]
where $ \widehat{\overline{m}}_i \in \widehat{U}_{\widehat{m}_1}(\epsilon_{*}) $, $ \overline{m}_i = \phi(\widehat{\overline{m}}_i) $, and $ \overline{x}^s_i \in X^{s}_{\overline{m}_i} (\sigma_{*}) $.
Compute
\begin{equation}\label{equ:estM1}
\begin{split}
& |f^1_{\widehat{m}_1} ( \widetilde{x}^1 ) - f^2_{\widehat{m}_1} ( \widetilde{x}^1 )| \\
\leq & | \Pi^u_{m_1} ( h^1 (\widehat{\overline{m}}_1, \overline{x}^s_1) - h^2 (\widehat{\overline{m}}_2, \overline{x}^s_2) ) | + |\Pi^u_{m_1} ( (\overline{m}_1 + \overline{x}^s_1) - (\overline{m}_2 + \overline{x}^s_2) )| \\
\leq & | \Pi^u_{m_1} ( h^1 (\widehat{\overline{m}}_1, \overline{x}^s_1) - h^2 (\widehat{\overline{m}}_1, \overline{x}^s_1) ) | + |\Pi^u_{m_1} ( h^2 (\widehat{\overline{m}}_1, \overline{x}^s_1) - h^2 (\widehat{\overline{m}}_2, \overline{x}^s_2) )| + \square_2 \\
\leq & (1 + \varpi^*_0) |h^1 (\widehat{\overline{m}}_1, \overline{x}^s_1) - h^2 (\widehat{\overline{m}}_1, \overline{x}^s_1)| + \square_1 + \square_2,
\end{split}
\end{equation}
where
\[
\square_1  \triangleq |\Pi^u_{m_1} ( h^2 (\widehat{\overline{m}}_1, \overline{x}^s_1) - h^2 (\widehat{\overline{m}}_2, \overline{x}^s_2) )| \quad \text{and} \quad
\square_2 \triangleq |\Pi^u_{m_1} ( (\overline{m}_1 + \overline{x}^s_1) - (\overline{m}_2 + \overline{x}^s_2) )|.
\]
As $ \graph h^2 $ is $ \mu $-Lip in $ u $-direction (i.e., $ h^2 \in \varSigma_{lip, \mu, K_1} $) and by \eqref{equ:estimates}, one gets
\begin{equation}\label{equ:estM2}
\begin{split}
\square_1 & \triangleq |\Pi^u_{m_1} ( h^2 (\widehat{\overline{m}}_1, \overline{x}^s_1) - h^2 (\widehat{\overline{m}}_2, \overline{x}^s_2) )| \leq \mu(\widehat{m}_1) \max\{ |\Pi^c_{m_1}(\overline{m}_1 - \overline{m}_2)|, |\Pi^s_{m_1}(\overline{x}^s_1 - \overline{x}^s_2)| \} \\
& \leq \mu_1(\widehat{m}_1) \varpi^*_0 |f^1_{\widehat{m}_1} ( \widetilde{x}^1 ) - f^2_{\widehat{m}_1} ( \widetilde{x}^2 )|;
\end{split}
\end{equation}
furthermore, since $ X^s_{\widehat{\Sigma}} (\sigma_*) $ is $ \chi_{*} $-Lip in $ u $-direction, by \eqref{equ:estimates}, one obtains
\begin{equation}\label{equ:estM3}
\begin{split}
\square_2 & \triangleq |\Pi^u_{m_1} ( (\overline{m}_1 + \overline{x}^s_1) - (\overline{m}_2 + \overline{x}^s_2) )|  \leq \varpi^*_0 \max\{ |\Pi^c_{m_1}(\overline{m}_1 - \overline{m}_2)|, |\Pi^s_{m_1}(\overline{x}^s_1 - \overline{x}^s_2)| \} \\
& \leq \varpi^*_0 |f^1_{\widehat{m}_1} ( \widetilde{x}^1 ) - f^2_{\widehat{m}_1} ( \widetilde{x}^2 )|.
\end{split}
\end{equation}
Combining the above inequalities \eqref{equ:estM}--\eqref{equ:estM3}, we finally get
\[
|\widetilde{h}^1(\widehat{m}, \overline{x}^s) - \widetilde{h}^2(\widehat{m}, \overline{x}^s)| \leq \widehat{\lambda}_{u} |{h}^1(\widehat{\overline{m}}_1, \overline{x}^s_1) - {h}^2(\widehat{\overline{m}}_1, \overline{x}^s_1)|,
\]
where $ \widehat{\lambda}_{u} \triangleq (1+\varpi^*_0)\sup_{\widehat{m}_0 \in \widehat{K}}\frac{\lambda_{u}(\phi(\widehat{m}_0))}{1 - \alpha(\phi(\widehat{m}_0)) \mu_1(\widehat{u}(\widehat{m}_0))} = (1+\varpi^*_0) \overline{\lambda}_{u} < 1 $. The proof is complete.
\end{proof}

By \autoref{lem:contractive}, we have $ \lip \varGamma < 1 $, and so we prove the following result.
\begin{lem}[Existence of the center-stable manifold]\label{lem:existence}
	There is a unique $ h_0 \in \varSigma_{lip, \mu, K_1} $ such that $ \varGamma h_0 = h_0 $. In particular, if $ f^{0}_{\widehat{m}} $ is the \emph{local representation} of $ \graph h_0 \cap X^{su}_{\widehat{K}_{\epsilon_*}} (\sigma_*, \varrho_*) $ at $ \widehat{m} \in \widehat{K} $, then
	\begin{equation}\label{equ:main}
	\begin{cases}
	\widehat{F}^{cs}_{m_0} ( x^{cs}, f^{0}_{\widehat{m}_1} ( x_{\widehat{m}} (x^{cs}) ) ) = x_{\widehat{m}} (x^{cs}), \\
	\widehat{G}^{cs}_{m_0} ( x^{cs}, f^{0}_{\widehat{m}_1} ( x_{\widehat{m}} (x^{cs}) ) ) = f^0_{\widehat{m}_0}(x^{cs}),
	\end{cases}
	x^{cs} \in X^{cs}_{m_0}(e_0 \eta_2),
	\end{equation}
	where $ \widehat{m}_0 \in \widehat{K} $, $ \widehat{m}_1 = \widehat{u}(\widehat{m}_0) $, $ m_i = \phi(\widehat{m}_i) $, $ i = 0,1 $, and $ e_0 = \min\{e_1, c_1\} $. Set $ \varepsilon_{0} = \min\{e^{-1}_2,c^{-1}_2\} e_0 \eta_2 $,
	\[
	W^{cs}_{loc}(K) = \graph h_0 \triangleq \{ (\widehat{m}, x^s, h_0(\widehat{m}, x^s)): (\widehat{m}, x^s) \in X^s_{\widehat{\Sigma}}(\sigma_{*}) \} \subset X^s_{\widehat{\Sigma}}(\sigma_{*}) \oplus X^u_{\widehat{\Sigma}} (\varrho_{*}),
	\]
	and
	\begin{multline*}
	\Omega = \graph h_0|_{X^{s}_{\widehat{K}_{\varepsilon_{0}}} (\varepsilon_{0})} \\
	\triangleq \{ (\widehat{m}, x^s, h_0(\widehat{m}, x^s)): (\widehat{m}, x^s) \in X^{s}_{\widehat{K}_{\varepsilon_{0}}} (\varepsilon_{0}) \}  
	\subset  \bigcup_{\widehat{m}_0 \in \widehat{K}}\graph f^0_{\widehat{m}_0}|_{X^{cs}_{m_0}(e_0 \eta_2)} \\
	\triangleq \{ \phi(\widehat{m}_0) + x^{cs} + f^0_{\widehat{m}_0}(x^{cs}): x^{cs} \in X^{cs}_{\phi(\widehat{m}_0)}(e_0 \eta_2), \widehat{m}_0 \in \widehat{K} \}.
	\end{multline*}
	Then $ \Omega $ is an open subset of $ W^{cs}_{loc}(K) $ and $ \Omega \subset H^{-1} (W^{cs}_{loc}(K)) $.
\end{lem}

Note that, in general, the existence of $ h_0 $ depends on the choice of $ \Psi $.

\begin{lem}\label{lem:unique}
	$ H : \Omega \to W^{cs}_{loc}(K) $ induces a map, i.e., \autoref{thm:I} \eqref{map} holds.
\end{lem}
\begin{proof}
	Let $ z_0 = (\widehat{m}, x^s, x^u) \in \Omega $. By the construction of $ \Omega $, there is a point $ \widehat{m}_0 \in \widehat{K} $ such that $ \widehat{m} \in \widehat{U}_{\widehat{m}_0} (\varepsilon_{0}) $. Now $ \phi(\widehat{m}) + x^s + x^u = \phi(\widehat{m}_0) + x^{cs} + f^0_{\widehat{m}_0}(x^{cs}) $, $ x^{cs} \in X^{cs}_{\phi(\widehat{m}_0)}(e_0 \eta_2) $. Write
	\[
	u(\phi(\widehat{m}_0)) + x_{\widehat{m}_0} (x^{cs}) + f^{0}_{\widehat{u}(\widehat{m}_0)} ( x_{\widehat{m}_0} (x^{cs}) ) = \phi(\widehat{m}'_1) + \overline{x}^s_1 + \overline{x}^u_1 \triangleq z_1 \in W^{cs}_{loc}(K),
	\]
	where $ \widehat{m}'_1 \in \widehat{U}_{\widehat{u}(\widehat{m}_0)}(\epsilon_{*}) $, and $ \overline{x}^\kappa_1 \in X^{\kappa}_{\phi(\widehat{m}'_1)} $, $ \kappa = s, u $. By \eqref{equ:main}, we have $ z_1 \in H(z_0) $.

	Let us show $ z_1 $ is the unique point in the sense that if $ z_2 = (\widehat{m}'_2, \overline{x}^s_2, \overline{x}^u_2) \in W^{cs}_{loc}(K) $, $ \widehat{m}'_2 \in \widehat{U}_{\widehat{u}(\widehat{m}_0)}(\epsilon_{*}) $, and $ z_2 \in H(z_0) $, then $ z_2 = z_1 $. Here note that $ z_2 = u(\phi(\widehat{m}_0)) + \tilde{x}^{cs}_1 + f^{0}_{\widehat{u}(\widehat{m}_0)} ( \tilde{x}^{cs}_1 ) $ for some $ \tilde{x}^{cs}_1 \in X^{cs}_{u(\phi(\widehat{m}_0))} (r_0) $. From
	\[
	\begin{cases}
	x_{\widehat{m}_0} (x^{cs}) + f^{0}_{\widehat{u}(\widehat{m}_0)} ( x_{\widehat{m}_0} (x^{cs}) ) \in \widehat{H}_{\phi(\widehat{m}_0)} (x^{cs} + f^0_{\widehat{m}_0}(x^{cs})),\\
	\tilde{x}^{cs}_1 + f^{0}_{\widehat{u}(\widehat{m}_0)} ( \tilde{x}^{cs}_1 ) \in \widehat{H}_{\phi(\widehat{m}_0)} (x^{cs} + f^0_{\widehat{m}_0}(x^{cs})),
	\end{cases}
	\]
	and the (B) condition for 
	\[
	\widehat{H}_{\phi(\widehat{m}_0)}: \widehat{X}^{cs}_{\phi(\widehat{m}_0)}(r_{0}) \oplus \widehat{X}^{u}_{\phi(\widehat{m}_0)}(r_{0}) \to \widehat{X}^{cs}_{u(\phi(\widehat{m}_0))}(r_{0}) \oplus \widehat{X}^{u}_{u(\phi(\widehat{m}_0))} (r_{0})
	\]
	(i.e., (A3) (a)), we see $ x_{\widehat{m}_0} (x^{cs}) = \tilde{x}^{cs}_1 $. The proof is complete.
\end{proof}

To characterize $ W^{cs}_{loc}(K) $, we introduce a special class of orbits in a neighborhood of $ K $; this is necessary even when $ H $ is a map but not Lipschitz.
\begin{defi}[Orbit]\label{defi:orbit}
	Let $ \{z_{k} = (\widehat{m}_{k}, \overline{x}^s_{k}, \overline{x}^{u}_{k})\}_{k \geq 0} \subset X^{s}_{\widehat{K}_{\varepsilon_0}} (\sigma) \oplus X^{u}_{\widehat{K}_{\varepsilon_0}} (\varrho) $. We say $ \{z_{k}\}_{k \geq 0} $ is a \emph{$ (\varepsilon_{0}, \varepsilon_{1}, \sigma, \varrho) $-type forward orbit of $ H $ near $ K $} if $ z_{k+1} \in H(z_{k}) $ and there is a sequence $ \{ \widehat{m}^0_{k} \}_{k \geq 0} \subset \widehat{K} $ such that $ \widehat{m}_{k} \in \widehat{U}_{\widehat{m}^0_{k}}(\varepsilon_{0}) $ and $ \widehat{m}_{k+1} \in \widehat{U}_{\widehat{u}(\widehat{m}^0_{k})}(\varepsilon_{1}) $ for all $ k \in \mathbb{N} $. A similar notion of \emph{$ (\varepsilon_{0}, \varepsilon_{1}, \sigma, \varrho) $-type backward orbit} of $ H $ near $ K $ can be defined if $ u: K \to K $ is invertible; $ \{ z_{k} \}_{k \in \mathbb{Z}} $ is called a \emph{$ (\varepsilon_{0}, \varepsilon_{1}, \sigma, \varrho) $-type orbit} of $ H $ near $ K $ if $ \{ z_{k} \}_{k \geq 0} $ is a $ (\varepsilon_{0}, \varepsilon_{1}, \sigma, \varrho) $-type forward orbit and $ \{ z_{k} \}_{k \leq 0} $ is a $ (\varepsilon_{0}, \varepsilon_{1}, \sigma, \varrho) $-type backward orbit.
\end{defi}

\begin{figure}[!htp]
	\centering
	\includegraphics[height=0.3\linewidth]{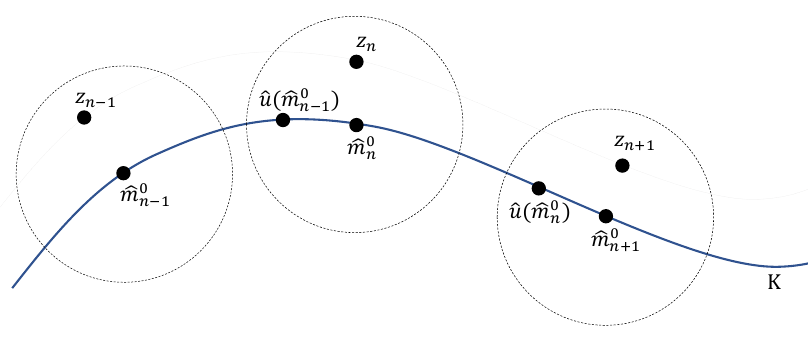}
	\caption{$ (\varepsilon_{0}, \varepsilon_{1}, \sigma, \varrho) $-type forward orbit of $ H $ near $ K $}
	\label{fig:doc1}
\end{figure}

\begin{lem}\label{lem:partial}
	Any $ (\varepsilon_{0}, \epsilon_{*}, \varepsilon_{0}, \varrho_*) $-type forward orbit $ \{z_{k}\}_{k \geq 0} $ of $ H $ near $ K $ belongs to $ W^{cs}_{loc}(K) $, i.e., $ \{z_{k}\}_{k \geq 0} \subset W^{cs}_{loc}(K) $. In particular, if $ \eta = 0 $, then $ K \subset \Omega $.
\end{lem}
\begin{proof}
	Employing the proof of \autoref{lem:contractive}, we establish the following fact:
	If 
	\begin{enumerate}[(i)]
		\item $ \widetilde{h} = \widetilde{\varGamma}(h) $, where $ h \in \varSigma_{lip, \mu, K_1} $,
		\item $ \phi(\widehat{m}'_{1}) + \overline{x}^s_{1} + \overline{x}^{u}_{1} \in H(\phi(\widehat{m}'_{0}) + \overline{x}^s_{0} + \overline{x}^{u}_{0}) $, with $ (\widehat{m}'_{i}, \overline{x}^s_{i}, \overline{x}^u_{i}) \in X^{su}_{\widehat{K}_{\epsilon_*}} (\sigma_*, \varrho_*) $ for $ i = 0,1 $, $ \overline{x}^s_{0} \in X^{s}_{\phi(\widehat{m}'_{0})} (\varepsilon_{0}) $,
		\item $ \widehat{m}'_{0} \in \widehat{U}_{\widehat{m}^0_0}(\varepsilon_{0}) $ and $ \widehat{m}'_{1} \in \widehat{U}_{\widehat{u}(\widehat{m}^0_{0})}(\epsilon_{*}) $, where $ \widehat{m}^0_0 \in \widehat{K} $,
	\end{enumerate}
	then
	\[
	|\widetilde{h}(\widehat{m}'_{0}, \overline{x}^s_0) - \overline{x}^{u}_{0}| \leq \widehat{\lambda}_{u} |h(\widehat{m}'_1, \overline{x}^s_1) - \overline{x}^{u}_{1}|.
	\]
	\begin{proof}[Proof of the fact]
		Replace $ \widetilde{h}^2 $ and $ h^2 $ with $ \widetilde{h} $ and $ h $, respectively, in the argument of \autoref{lem:contractive}. The proof proceeds as follows:
		
		By (ii) (iii), we can write
		\[
		\phi(\widehat{m}'_{0}) + \overline{x}^s_{0} + \overline{x}^{u}_{0} = \phi(\widehat{m}^0_0) + \widetilde{x}^{cs}_1 + \widetilde{x}^{u}_1, \quad
		\phi(\widehat{m}'_{1}) + \overline{x}^s_{1} + \overline{x}^{u}_{1} = \phi(\widehat{u}(\widehat{m}^0_0)) + {x}^{cs}_1 + {x}^{u}_1,
		\]
		where $ \widetilde{x}^{\kappa}_1 \in X^{\kappa}_{\phi(\widehat{m}^0_0)} $ and $ {x}^{\kappa}_1 \in X^{\kappa}_{\phi(\widehat{u}(\widehat{m}^0_0))} $, $ \kappa = cs, u $. Also, let
		\[
		\phi(\widehat{m}'_{0}) + \overline{x}^s_{0} + \widetilde{h}(\widehat{m}'_{0}, \overline{x}^s_{0}) = \phi(\widehat{m}^0_0) + \widetilde{x}^{cs}_2 + \widetilde{x}^{u}_2,
		\]
		where $ \widetilde{x}^{\kappa}_2 \in X^{\kappa}_{\phi(\widehat{u}(\widehat{m}^0_0))} $, $ \kappa = cs, u $, and let $ f_{\widehat{m}^0_0} $ and $ \widetilde{f}_{\widehat{u}(\widehat{m}^0_0)} $ be the local representations of $ \graph h \cap X^{su}_{\widehat{K}_{\epsilon_*}} (\sigma_*, \varrho_*) $ and $ \graph \widetilde{h} \cap X^{su}_{\widehat{K}_{\epsilon_*}} (\sigma_*, \varrho_*) $ at $ \widehat{m}_0 $ and $ \widehat{u}(\widehat{m}^0_0) $, respectively. Note that $ \widetilde{x}^{u}_2 = \widetilde{f}_{\widehat{m}^0_0} (\widetilde{x}^{cs}_2) $.

		Following the estimates in \autoref{lem:contractive}, we obtain:
		\begin{enumerate}[(1)]
			\item $ |\widetilde{x}^u_1 - \widetilde{x}^u_2| \leq \frac{\varpi^*_1\lambda_{u}(\phi(\widehat{m}^0_0))}{1 - \alpha(\phi(\widehat{m}^0_0)) \mu_1(\phi(\widehat{u}(\widehat{m}^0_0)))} |x^u_1 - \widetilde{f}_{\widehat{u}(\widehat{m}^0_0)}(x^{cs}_1)| $,

			\item $ |\overline{x}^{u}_{0} - \widetilde{h}(\widehat{m}'_{0}, \overline{x}^s_0)| \leq \varpi^*_1 |\widetilde{x}^u_1 - \widetilde{x}^u_2| $, and

			\item $ |x^u_1 - {f}_{\widehat{u}(\widehat{m}^0_0)}(x^{cs}_1)| \leq \varpi^*_1 |\overline{x}^{u}_{1} - h(\widehat{m}'_1, \overline{x}^s_1)| $.
		\end{enumerate}
		This gives the proof of the fact.
	\end{proof}

	Write $ z_{k} = (\widehat{m}_{k}, \overline{x}^s_{k}, \overline{x}^{u}_{k}) $. Note that by the choice of $ z_{k} $, we have $ \widetilde{h}(\widehat{m}_{k}, \overline{x}^s_k) = {h}(\widehat{m}_{k}, \overline{x}^s_k) $. Then by the above fact, we see
	\[
	|{h}(\widehat{m}_{0}, \overline{x}^s_0) - \overline{x}^{u}_{0}| \leq (\widehat{\lambda}_{u})^{n} |h(\widehat{m}_n, \overline{x}^s_n) - \overline{x}^{u}_{n}| \leq (\widehat{\lambda}_{u})^{n} \cdot 2\varrho_{*} \to 0,
	\]
	as $ n \to \infty $. The proof of \autoref{lem:partial} is complete.
\end{proof}

\begin{rmk}[Limited case]\label{rmk:inflowing}
	\begin{enumerate}[(a)]
		\item \label{it:i1} ($ s $-contraction) Under assumption ($ \star\star $) in \autoref{sub:limited} with (A1)--(A3), the graph transform $ \varGamma $ given by \eqref{equ:graphss} can be defined on $ \varSigma_{0, \mu, K_1} $, where
		\begin{multline*}
		\varSigma_{0, \mu, K_1} = \left\{ h: X^s_{\widehat{\Sigma}} (\sigma_*) \to \overline{X^u_{\widehat{\Sigma}} (\varrho_*)} ~\text{is a bundle map over}~ \id: \right. \\ 
		\left. h|_{X^s_{\widehat{\Sigma}} (\sigma_*) \setminus X^s_{\widehat{K}_{\eta_1}} (\sigma_*)} = 0, \right.
		 \sup_{\widehat{m} \in \widehat{K}} |h(\widehat{m}, 0)| \leq K_1\eta, \\
		\left. \graph h \cap X^{su}_{\widehat{K}_{\epsilon_*}} (\sigma_*, \varrho_*) ~\text{is $ \mu $-Lip in $ u $-direction near $ \widehat{K} $} \right\}.
		\end{multline*}
		We have a unique $ h_0 \in \varSigma_{0, \mu, K_1} $ such that $ \varGamma h_0 = h_0 $. Write $ W^{cs}_{loc}(K) = \graph h_0 \cap X^{su}_{\widehat{K}_{\epsilon_*}} (\sigma_*, \varrho_*) $. In addition, we have the following partial characterization of $ W^{cs}_{loc}(K) $: $ z \in W^{cs}_{loc}(K) $ if there is a $ (\varepsilon_{0}, \epsilon_{*}, \sigma_{*}, \varrho_*) $-type forward orbit $ \{z_{k}\}_{k \geq 0} $ of $ H $ near $ K $ such that $ z_0 = z $. Set
		\[
		\Omega = \graph h_0|_{X^{s}_{\widehat{K}_{\varepsilon_{0}}} (\sigma_*)}.
		\]
		Then we have $ \Omega \subset H^{-1}(W^{cs}_{loc}(K)) $.

		\item \label{it:i2} (Strictly inflowing) Under assumption ($ \bullet\bullet $) in \autoref{sub:limited0} and (A1)--(A3) with only $ \varsigma_0 \geq 1 $ (in (A3) (a)), if $ \epsilon_{*}, \chi(\epsilon_{*}), \eta $ are small, then the graph transform $ \widetilde{\varGamma} $ (in \autoref{sub:limited0}) satisfies $ \widetilde{\varGamma} \varSigma_{\mu, K_1, \epsilon_{*}, \sigma_{*}, \varrho_* } \subset \varSigma_{\mu, K_1, \epsilon_{*}, \sigma_{*}, \varrho_* } $ and $ \lip\widetilde{\varGamma} < 1 $. In this case, we have a unique $ h_0 \in \varSigma_{\mu, K_1, \epsilon_{*}, \sigma_{*}, \varrho_* } $ such that $ \widetilde{\varGamma}h_0 = h_0 $. That is, for
		\[
		W^{cs}_{loc}(K) = \graph h_0 \cap X^{su}_{\widehat{K}_{\epsilon_*}} (\sigma_*, \varrho_*),
		\]
		we have
		\begin{enumerate}[(i)]
			\item $ W^{cs}_{loc}(K) \subset H^{-1} (W^{cs}_{loc}(K)) $;
			\item $ H $ in $ W^{cs}_{loc}(K) $ induces a (Lipschitz) map;
			\item (Characterization of $ W^{cs}_{loc}(K) $) $ z \in W^{cs}_{loc}(K) $ if and only if there is a $ (\epsilon_{*}, \epsilon_{*}, \sigma_{*}, \varrho_*) $-type forward orbit $ \{z_{k}\}_{k \geq 0} $ of $ H $ near $ K $ such that $ z_0 = z $.
		\end{enumerate}
		The last property means that the center-stable manifold of $ K $ is unique.

		\item \label{it:i3} (Parameter-dependent correspondences) Under Assumption (CS) in \autoref{sub:limited*} with (A3) (a) (ii) in \autoref{subsec:main}, if $ \epsilon = \epsilon_{*}, \chi(\epsilon_{*}), \eta $ are small, then for each $ \delta $, we have a unique $ h^{\delta}_0 \in \varSigma_{\mu, K_1, \epsilon_{*}, \sigma_{*}, \varrho_* } $ such that $ \widetilde{\varGamma}^{\delta}h^{\delta}_0 = h^{\delta}_0 $. In particular,, the properties (i)--(iii) in (b) hold for $ W^{cs, \delta}_{loc}(K) = \graph h^{\delta}_0 \cap X^{su}_{\widehat{K}_{\epsilon_*}} (\sigma_*, \varrho_*) $ and $ H^{\delta} $ instead of $ W^{cs}_{loc}(K) $ and $ H $, respectively. Moreover, if $ H^{\delta_1} \circ H^{\delta_2} = H^{\delta_2} \circ H^{\delta_1} $, then $ W^{cs, \delta_1}_{loc}(K) = W^{cs, \delta_2}_{loc}(K) $ (i.e., $ h^{\delta_1}_0 = h^{\delta_2}_0 $).

	\end{enumerate}
\end{rmk}

\begin{proof}[Proof of \autoref{thm:I}]
	We only consider the case where the correspondence $ \widehat{H} \approx (\widehat{F}^{cs}, \widehat{G}^{cs}) $ satisfies the (A$ ' $)($ \alpha $, $ \lambda_{u} $) (B)($ \beta; \beta', \lambda_{cs} $) condition in $ cs $-direction at $ K $. For a proof of the ($ \bullet 2 $) case that $ \widehat{H} \approx (\widehat{F}^{cs}, \widehat{G}^{cs}) $ satisfies the (A)($ \alpha $, $ \lambda_{u} $) (B)($ \beta; \beta', \lambda_{cs} $) condition, see \autoref{sec:tri}.
	The result now follows from \autoref{lem:existence}, \autoref{lem:locLip}, \autoref{lem:unique} and \autoref{lem:partial}.
\end{proof}

\chapter{Smoothness of a center-stable manifold: proof of \autoref{thm:smooth}}\label{sec:smooth}

\section{Lipschitz and smooth bump function} \label{sub:LipC1bm}

Since the bump function $ \Psi $ defined in \eqref{equ:cutoff} is generally not smooth, $ \widehat{h} = \varGamma(h) $ is also not smooth even if $ h $ is smooth. Suppose $ \Psi $ satisfies the following conditions:
\begin{enumerate}[(a)]
	\item $ \Psi \in C^{1}(X^s_{\widehat{\Sigma}} (\sigma_*), [0,1]) $,
	\[
	\Psi(\widehat{m}, \overline{x}^s) = \begin{cases}
	1, & \quad (\widehat{m}, \overline{x}^s) \in X^s_{\widehat{K}_{\eta_2}} (\eta_2), \\
	0, &\quad (\widehat{m}, \overline{x}^s) \in X^s_{\widehat{\Sigma}} (\sigma_*) \setminus X^s_{\widehat{K}_{\eta_1}} (\eta_1);
	\end{cases}
	\]
	\item $ \Psi $ is Lipschitz in the following sense: if $ \widehat{m}_0 \in \widehat{K} $, $ m_0 = \phi(\widehat{m}_0) $ and $ (\widehat{m}_i, \overline{x}^s_i) \in X^s_{\widehat{U}_{\widehat{m}_0}(\epsilon_*)} (\sigma_*) $, $ i = 1,2 $, then
	\[
	|\Psi(\widehat{m}_1, \overline{x}^s_1) - \Psi(\widehat{m}_2, \overline{x}^s_2)| \leq \frac{\varpi^*_1 C_1}{\eta_1 - \eta_2} \max\{ |\Pi^c_{m_0} (\phi(\widehat{m}_1) - \phi(\widehat{m}_2) )|, |\Pi^s_{m_0} (\overline{x}^s_1 - \overline{x}^s_2)| \},
	\]
	where $ C_1 \geq 1 $ (independent of $ \epsilon_{*}, \chi_{*} $) and $ \varpi^*_1 \to 1 $ as $ \epsilon_{*}, \chi_{*} \to 0 $.
\end{enumerate}

\begin{defi}[$ C^{1} \cap C^{0,1} $ bump function]\label{def:C1Lbump}
	Such a $ \Psi $ is called a \emph{$ C^{1} $ and $ C_1 $-Lipschitz bump function} in $ X^s_{\widehat{\Sigma}} $ near $ K $. If $ X^s_{m} = \{0\} $ for all $ m \in \Sigma $, then $ \Psi $ is also called a \emph{$ C^{1} $ and $ C_1 $-Lipschitz bump function} in $ \widehat{\Sigma} $ near $ K $. If the constant $ C_1 $ is not emphasized, we may simply call it a \emph{$ C^{1} \cap C^{0,1} $ bump function}.
\end{defi}
See \autoref{sub:bump} and particularly \autoref{cor:spaces} for some spaces where such $ \Psi $ exists.

\vspace{.5em}
\noindent{Assumptions}.
From now until before \autoref{sub:bump}, we make the following two assumptions.

Case (1). Assume (A1), (A2), (A3) (a) (b) and (A4) (i)--(iii) hold with $ \varsigma_0 \geq C_1 + 1 $.

Case (2). Assume (A1), (A2), (A3) (a$ ' $) (b) and (A4) (i)--(iii) hold with $ \varsigma_0 \geq 1 $. 
\vspace{.5em}

We need all the constants listed in \autoref{sub:preparation}. 

\begin{enumerate}[$\bullet$]
	\item In case (1), let $ \varsigma_1 > 2 $ such that
	\[
	\inf_{m \in K}\{ \beta(m) - (C_1(\varsigma_1 - 1) + 1) \beta'(u(m)) \} > 0,
	\]
	and take
	\[
	\varsigma = C_1(\varsigma_1 - 1) + 1, \quad \eta_2 = \frac{\varsigma'_1 - 2}{\varsigma'_1 - 1} \eta_1, \quad \varsigma'_1 = (\varsigma_1 + 2)/2.
	\]
	We use this new value $ \varsigma = C_1(\varsigma_1 - 1) + 1 $ instead of the original $ \varsigma $ in Observations \eqref{OI}--\eqref{OVII} in \autoref{sub:preparation}. Set
	\begin{equation}\label{equ:qqq}
	\vartheta_1(m) = ( 1 - (C_1(\varsigma_1 - 1) + 1)\alpha(m)\beta'(u(m)) )^{-1}.
	\end{equation}
	As (A4) (ii) holds, we can further assume (if $ \varsigma_1 $ is close to $ 2 $)
	\[
	\sup_{m\in K} \lambda_{cs}(m) \lambda_u(m) \vartheta_1(m) < 1 \quad \text{and} \quad \sup_{\widehat{m}_0 \in \widehat{K}} \frac{ \lambda_{cs}(\phi(\widehat{m}_0)) \lambda_{u}(\phi(\widehat{m}_0)) }{ 1 - \alpha(\phi(\widehat{m}_0)) \mu_1(\widehat{u}(\widehat{m}_0)) } < 1.
	\]
	\item In case (2), let $ \varsigma'_1 = 3 $, $ \eta_2 = (1/2) \eta_1 $. Furthermore, Observations \eqref{OI}--\eqref{OVII} in \autoref{sub:preparation} also hold with $ \varsigma > 1 $ but close to $ 1 $. 
	From (A4) (ii), we can assume (if $ \varsigma $ is close to $ 1 $ and $ \epsilon_{*}, \chi(\epsilon_{*}) $ are small)
	\[
	\sup_{\widehat{m}_0 \in \widehat{K}} \frac{ \lambda_{cs}(\phi(\widehat{m}_0)) \lambda_{u}(\phi(\widehat{m}_0)) }{ 1 - \alpha(\phi(\widehat{m}_0)) \mu_1(\widehat{u}(\widehat{m}_0)) } < 1.
	\]
	\item Let $ \eta $ satisfy $ \eta \leq \epsilon_{0,*} \eta_0 $ (where $ \epsilon_{0,*} = O_{\epsilon_{*}}(1) $ defined in \eqref{equ:o(1)}) and $ 0 < \eta_1 < \sigma^1_{*} $ close to $ \sigma^1_{*} $ (e.g., $ \eta_1 = (1 - \epsilon_{*}) \sigma^1_{*} $).
\end{enumerate}

Note that under (A4) (iii), $ X^s_{\widehat{K}_{\epsilon_{*}}}(\sigma_{*}) $ is a $ C^1 $ manifold. We write the local representations of this $ C^1 $ manifold as
\begin{equation}\label{equ:cslocal}
m + \overline{x}^s = m_0 + x' + g^{cs}_{\widehat{m}_0}(x'), \quad x' \in X^{c}_{m_0}(c_2\epsilon_{*}) \oplus X^{s}_{m_0}(e_2\sigma_{*}),
\end{equation}
where $ g^{cs}_{\widehat{m}_0}: X^{c}_{m_0}(c_2\epsilon_{*}) \oplus X^{s}_{m_0}(e_2\sigma_{*}) \to X^{u}_{m_0} $ is $ C^1 $, $ \widehat{m}_0 \in \widehat{K} $, and $ m_0 = \phi(\widehat{m}_0) $. Without loss of generality, we assume $ X^s_{\widehat{K}_{\epsilon_{*}}}(\sigma_{*}) \subset \bigcup_{\widehat{m}_0 \in \widehat{K}}\graph g^{cs}_{\widehat{m}_0} $.

Let $ h_0 $ be the center-stable manifold obtained in \autoref{lem:existence} (using the above map $ \Psi $ instead of the original one defined in \eqref{equ:cutoff}), and let $ f^0_{\widehat{m}_0} $ be the \emph{local representation} of $ \graph h_0 \cap X^{su}_{\widehat{K}_{\epsilon_*}} (\sigma_*, \varrho_*) $ at $ \widehat{m}_0 \in \widehat{K} $. Also, write $ W^{cs}_{loc}(K) = \graph h_0 $. The map $ f^0 $ can be naturally considered as a bundle map from $ X^{cs}_{\widehat{K}} (\sigma_{0}) $ to $ X^{u}_{\widehat{K}} (\varrho_{*}) $ over id, i.e.,
\begin{equation}\label{equ:f00}
f^0: X^{cs}_{\widehat{K}} (\sigma_{0}) \to X^{u}_{\widehat{K}} (\varrho_{*}), \quad (\widehat{m}_0, x) \mapsto (\widehat{m}_0, f^0_{\widehat{m}_0}(x)).
\end{equation}
Let $ \widetilde{h}_0 = \widetilde{\varGamma}(h_0) $ (see \eqref{equ:graph}). Its corresponding local representations are denoted by $ \widetilde{f}^0_{\widehat{m}_0} $, $ \widehat{m}_0 \in \widehat{K} $; That is, for $ \widehat{m} \in \widehat{U}_{\widehat{m}_0} (\sigma^1_{*}) $ and $ \overline{x}^s \in X^s_{m} (\sigma^1_{*}) $, where $ m = \phi(\widehat{m}) $ and $ m_0 = \phi(\widehat{m}_0) $, we have
\[
m + \overline{x}^s + \widetilde{h}_0(\widehat{m}, \overline{x}^s) = m_0 + \widetilde{x}^{cs} + \widetilde{f}^0_{\widehat{m}_0}(\widetilde{x}^{cs}), \quad \text{where}~ \widetilde{x}^{cs} \in X^{cs}_{m_0}(\sigma^c_{*}).
\]

\section{Local representations of the bump function}

For brevity, write 
\[ 
\phi_{m_0,\gamma} = \phi|_{\widehat{U}_{\widehat{m}_0}}: \widehat{U}_{\widehat{m}_0} \to U_{m_0,\gamma},  
\]
where $ \phi(\widehat{m}_0) = m_0 $ and $ \gamma \in \Lambda(m_0) $.
Let $ \Phi_{m_0,\gamma} $ be defined as in \eqref{equ:tub}, and set
\begin{multline*}
\overline{\Phi}^{-1}_{m_0,\gamma} = (\overline{\Phi}^{-c}_{m_0,\gamma}, \overline{\Phi}^{-s}_{m_0,\gamma}, \overline{\Phi}^{-u}_{m_0,\gamma}): (\widehat{x}^c, \widehat{x}^s, \widehat{x}^u)  \stackrel{\Phi^{-1}_{m_0,\gamma}}{\mapsto} (x^c_0, y^s, y^u) \mapsto (\widehat{\overline{m}}, \overline{x}^s, \overline{x}^u), \\
X^c_{m_0}(c_2\epsilon_*) \times X^s_{m_0}(e_2\sigma_*) \times X^s_{m_0}(e_2\varrho_*) \to X^c_{m_0}(c_2\epsilon_*) \times X^s_{m_0}(e_2\sigma_*) \times X^s_{m_0}(e_2\varrho_*) \to X^{h}_{\widehat{\Sigma}},
\end{multline*}
where $ \widehat{\overline{m}} = \phi^{-1}_{m_0,\gamma}(\overline{m}) $, and $ \overline{m}, \overline{x}^s, \overline{x}^u $ are determined by \eqref{equ:bundle}.

Consider the following $ C^1 $ maps:
\[
\begin{cases}
F^{0}_{\widehat{m}_0} (\widehat{x}^{c}, \widehat{x}^{s}, \widehat{x}^{u}) = (1 - \Psi(\widehat{m}, \overline{x}^s)) \Pi^{cs}_{m_0}(\overline{m} + \overline{x}^s - m_0) + \Psi(\widehat{m}, \overline{x}^s) (\widehat{x}^{c} + \widehat{x}^{s}), \\
G^{0}_{\widehat{m}_0} (\widehat{x}^{c}, \widehat{x}^{s}, \widehat{x}^{u}) = (1 - \Psi(\widehat{m}, \overline{x}^s)) \Pi^{u}_{m_0}(\overline{m} + \overline{x}^s - m_0) + \Psi(\widehat{m}, \overline{x}^s) \widehat{x}^{u},
\end{cases}
\]
where
\[
\widehat{m} = \overline{\Phi}^{-c}_{m_0,\gamma}(\widehat{x}^{c}, \widehat{x}^{s}, \widehat{x}^{u}), \quad \overline{x}^s = \overline{\Phi}^{-s}_{m_0,\gamma}(\widehat{x}^{c}, \widehat{x}^{s}, \widehat{x}^{u}), \quad \overline{m} = \phi(\widehat{m}), \quad m_0 = \phi_{m_0,\gamma}(\widehat{m}_0).
\]
In other words, if we write
\[
\begin{cases}
\overline{m} + \overline{x}^s + \overline{x}^u = m_0 + \widehat{x}^c + \widehat{x}^s + \widehat{x}^u,\\
\overline{m} + \overline{x}^s + \Psi(\widehat{m}, \overline{x}^s) \overline{x}^u = m_0 + {x}^c + {x}^s + {x}^u,
\end{cases}
\]
where $ \widehat{m} = \phi^{-1}_{m_0,\gamma}(\overline{m}) $, $ \widehat{x}^\kappa, x^{\kappa} \in X^{\kappa}_{m_0} $ for $ \kappa = s, c, u $, and $ (\overline{m}, \overline{x}^s, \overline{x}^u) \in X^{su}_{\widehat{K}_{\epsilon_{*}}}(\sigma_{*}, \varrho_{*}) $, then
\[
x^c + x^s = F^{0}_{\widehat{m}_0} (\widehat{x}^{c}, \widehat{x}^{s}, \widehat{x}^{u}), \quad x^{u} = G^{0}_{\widehat{m}_0} (\widehat{x}^{c}, \widehat{x}^{s}, \widehat{x}^{u}).
\]

\begin{lem}\label{lem:est000}
	Let $ (\widehat{x}^c, \widehat{x}^s, \widehat{x}^u) \in X^c_{m_0}(c_2\epsilon_*) \times X^s_{m_0}(e_2\sigma_*) \times X^s_{m_0}(e_2\varrho_*) $. Then the following estimates hold for $ F^{0}_{\widehat{m}_0} $ and $ G^{0}_{\widehat{m}_0} $:
	\begin{enumerate}[(1)]
		\item $ \sup_{\widehat{m}_0 \in \widehat{K}}\lip (F^{0}_{\widehat{m}_0} - \Pi^{cs}_{\phi(\widehat{m}_0)}) \to 0 $ as $ \epsilon_{*} \to 0 $, and so 
		\[
		\lip F^{0}_{\widehat{m}_0}(\widehat{x}^{c}, \widehat{x}^{s}, \cdot) \leq \varpi^*_0, \quad \lip F^{0}_{\widehat{m}_0}(\cdot, \widehat{x}^{u}) \leq \varpi^*_1.
		\]

		\item $ \lip G^{0}_{\widehat{m}_0}(\cdot, \widehat{x}^{u}) \leq \varpi^*_1C_1(\varsigma'_1 - 1) / \eta_1|\widehat{x}^{u}| $,
		$ \lip G^{0}_{\widehat{m}_0}(\widehat{x}^{c}, \widehat{x}^{s}, \cdot) \leq \varpi^*_1 $.
	\end{enumerate}
\end{lem}

\begin{proof}
	The proof is straightforward. Observe that
	\begin{gather*}
	F^{0}_{\widehat{m}_0}(\widehat{x}^{c}, \widehat{x}^{s}, \widehat{x}^{u}) - \widehat{x}^{c} - \widehat{x}^{s} = - ( 1 - \Psi(\widehat{m}, \overline{x}^s) ) (   \chi^u_{m_0,\gamma} (x^c_0) + \varphi^{s}_{m_0}(\overline{m}) y^s  ), \\
	G^{0}_{\widehat{m}_0}(\widehat{x}^{c}, \widehat{x}^{s}, \widehat{x}^{u}) = ( 1 - \Psi(\widehat{m}, \overline{x}^s) ) (  \chi^u_{m_0,\gamma} (x^c_0) - \varphi^{s}_{m_0}(\overline{m}) y^s ) -  \Psi(\widehat{m}, \overline{x}^s) \widehat{x}^{u},
	\end{gather*}
	where 
	\[
	(x^c_0, y^s, y^u) = \Phi^{-1}_{m_0,\gamma} (\widehat{x}^{c}, \widehat{x}^{s}, \widehat{x}^{u}), \quad (\overline{m}, \overline{x}^s, \overline{x}^u) = \overline{\Phi}^{-1}_{m_0,\gamma}(\widehat{x}^{c}, \widehat{x}^{s}, \widehat{x}^{u}),
	\]
	and $ \chi^u_{m_0,\gamma} (\cdot), \varphi^{s}_{m_0}(\cdot) $ are defined in \autoref{sec:submanifold}.
	
	Now let $ (x^c_i, y^s_i, y^u_i) = \Phi^{-1}_{m_0,\gamma} (\widehat{x}^c_i, \widehat{x}^s_i, \widehat{x}^u_i) $, $ (\overline{m}_{i}, \overline{x}^s_{i}, \overline{x}^u_{i}) = \overline{\Phi}^{-1}_{m_0,\gamma}(\widehat{x}^{c}_{i}, \widehat{x}^{s}_{i}, \widehat{x}^{u}_{i}) $ for $ i = 1,2 $. Since $ \Psi|_{X^s_{\widehat{\Sigma}} (\sigma_*) \setminus X^s_{\widehat{K}_{\eta_1}} (\eta_1)} = 0 $, without loss of generality, assume $ x^c_1 \in X^{c}_{m_0} (c_2\eta_1) $ and $ y^s_1 \in X^{s}_{m_0}(e_2\eta_1) $. Then
	\begin{align*}
	&~ |(F^{0}_{\widehat{m}_0}(\widehat{x}^{c}_1, \widehat{x}^{s}_1, \widehat{x}^{u}_1) - \widehat{x}^{c}_1 - \widehat{x}^{s}_1) - (F^{0}_{\widehat{m}_0}(\widehat{x}^{c}_2, \widehat{x}^{s}_2, \widehat{x}^{u}_2) - \widehat{x}^{c}_2 - \widehat{x}^{s}_2)| \\
	\leq &~ \left(\varpi^*_1 C_1(\varsigma'_1 - 1) / \eta_1 \cdot \max\{ |\chi^u_{m_0,\gamma} (x^c_1) |, |\varphi^{s}_{m_0}(\overline{m}_1) y^s_1| \}  + \varpi^*_1 \chi_{*}\right) \max_{\kappa = s, c,u}\{|\widehat{x}^{\kappa}_1 - \widehat{x}^{\kappa}_2|\} \\
	\leq &~ \{\varpi^*_1 C_1(\varsigma'_1 - 1) \chi_{*} + \varpi^*_1 \chi_{*}\} \max_{\kappa = s, c,u}\{|\widehat{x}^{\kappa}_1 - \widehat{x}^{\kappa}_2|\},
	\end{align*}
	which also gives the estimates for $ \lip F^{0}_{\widehat{m}_0}(\widehat{x}^{c}, \widehat{x}^{s}, \cdot) $ and $ \lip F^{0}_{\widehat{m}_0}(\cdot, \widehat{x}^{u}) $.

	Furthermore, let $ \widehat{x}^u_1 = \widehat{x}^u_2 = \widehat{x}^u $. Then
	\begin{align*}
	&~ |G^{0}_{\widehat{m}_0}(\widehat{x}^{c}_1, \widehat{x}^{s}_1, \widehat{x}^{u}) - G^{0}_{\widehat{m}_0}(\widehat{x}^{c}_2, \widehat{x}^{s}_2, \widehat{x}^{u})| \\
	\leq &~ \left|( 1 - \Psi(\widehat{m}_1, \overline{x}^s_1) ) (  \chi^u_{m_0,\gamma} (x^c_1) - \varphi^{s}_{m_0}(\overline{m}_1) y^s_1 ) \right. \\
	&~ - \left. ( 1 - \Psi(\widehat{m}_2, \overline{x}^s_2) ) (  \chi^u_{m_0,\gamma} (x^c_2) - \varphi^{s}_{m_0}(\overline{m}_2) y^s_2 )\right| \\
	& \quad + \left| \Psi(\widehat{m}_1, \overline{x}^s_1) - \Psi(\widehat{m}_2, \overline{x}^s_2) \right| \cdot | \widehat{x}^{u} | \\
	\leq &~ \left|( \Psi(\widehat{m}_2, \overline{x}^s_2) - \Psi(\widehat{m}_1, \overline{x}^s_1) ) (  \chi^u_{m_0,\gamma} (x^c_1) - \varphi^{s}_{m_0}(\overline{m}_1) y^s_1 ) \right| \\
	& + \left| \chi^u_{m_0,\gamma} (x^c_1) - \chi^u_{m_0,\gamma} (x^c_2)  + \varphi^{s}_{m_0}(\overline{m}_1) y^s_1 -  \varphi^{s}_{m_0}(\overline{m}_2) y^s_2 \right|  \\
	& + \left| \Psi(\widehat{m}_1, \overline{x}^s_1) - \Psi(\widehat{m}_2, \overline{x}^s_2) \right| \cdot | \widehat{x}^{u} | \\
	\leq &~ \{ \varpi^*_0 + \varpi^*_1 C_1(\varsigma'_1 - 1) / \eta_1|\widehat{x}^{u}| \} \max_{\kappa = s, c}\{|\widehat{x}^{\kappa}_1 - \widehat{x}^{\kappa}_2|\}.
	\end{align*}
	Similar for $ \lip G^{0}_{\widehat{m}_0}(\widehat{x}^{c}, \widehat{x}^{s}, \cdot) $. The proof is complete.
\end{proof}

By the above lemma, the inverse $ (F^0_{\widehat{m}_0})^{-1} (\cdot, \cdot, x^{u}) $ exists. Set
\begin{equation}\label{equ:localcut-off}
F^1_{\widehat{m}_0} (\cdot, \cdot, x^{u}) = (F^0_{\widehat{m}_0})^{-1} (\cdot, \cdot, x^{u}), \quad G^1_{\widehat{m}_0} (\cdot, \cdot, {x}^{u})  = G^0_{\widehat{m}_0} (F^1_{\widehat{m}_0} (\cdot, \cdot, x^{u}), x^{u}).
\end{equation}

Note that by construction, $ \Psi(\widehat{m}, \overline{x}^s) \widetilde{h}_0(\widehat{m}, \overline{x}^s) = h_0(\widehat{m}, \overline{x}^s) $. For $ x^{cs} \in X^{cs}_{m_0}(e_0\eta'_1) $ (where $ e_0 = \min\{e_1, c_1\} $), we write
\[
m_0 + x^{cs} + f^0_{\widehat{m}_0} (x^{cs}) = m + \overline{x}^s + \Psi(\widehat{m}, \overline{x}^s) \widetilde{h}_0(\widehat{m}, \overline{x}^s),
\]
where $ m = \phi(\widehat{m}) $, and $ \eta'_1 $ is chosen close to $ \eta_1 $ such that
\[
\max\{ \eta_1, e^{-1}_0 \eta''_1 \}  < \eta'_1 < \sigma^1_{*} (< \sigma^c_{*}), \quad \text{with} ~ \eta''_1 =  \max\{ e_2, c_2 \} \eta_1.
\]
In particular, if $ |x^{cs}| > \eta''_1 $, then $ \Psi(\widehat{m}, \overline{x}^s) = 0 $.
Additionally, $ \widehat{m} \in \widehat{U}_{\widehat{m}_0}(\eta'_1) \subset \widehat{U}_{\widehat{m}_0}(\sigma^1_{*}) $ and $ \overline{x}^s \in X^{s}_{m} (\eta'_1) \subset X^{s}_{m}(\sigma^1_{*}) $. Thus, we can use $ F^1_{\widehat{m}_0} $ and $ G^1_{\widehat{m}_0} $ to relate $ f^0_{\widehat{m}_0} $ and $ \widetilde{f}^{0}_{\widehat{m}_0} $ as follows:
\begin{equation}\label{equ:extension}
\begin{cases}
F^1_{\widehat{m}_0} (x^{cs}, \widetilde{f}^0_{\widehat{m}_0}(\widetilde{x}^{cs})) = \widetilde{x}^{cs}, \\
G^1_{\widehat{m}_0} (x^{cs}, \widetilde{f}^0_{\widehat{m}_0}(\widetilde{x}^{cs})) = f^0_{\widehat{m}_0} (x^{cs}),
\end{cases}
\quad x^{cs} \in X^{cs}_{m_0}(e_0\eta'_1).
\end{equation}
Since $ \lip F^1_{\widehat{m}_0}(x^{cs}, \cdot) $ can be made small, we obtain $ \widetilde{x}^{cs} = \widetilde{x}_{\widehat{m}_0}(x^{cs}): X^{cs}_{m_0} (e_0\eta'_1) \to X^{cs}_{m_0} (\sigma^c_{*}) $. By the construction of $ \widetilde{h}_0 $, we have (see, e.g., \eqref{equ:local})
\begin{equation}\label{equ:local0}
\begin{cases}
\widehat{F}^{cs}_{m_0} ( x^{cs}, f^0_{\widehat{u}(\widehat{m}_0)} ( x_{\widehat{m}_0} (x^{cs}) ) ) = x_{\widehat{m}_0} (x^{cs}), \\
\widehat{G}^{cs}_{m_0} ( x^{cs}, f^0_{\widehat{u}(\widehat{m}_0)} ( x_{\widehat{m}_0} (x^{cs}) ) ) = \widetilde{f}^0_{\widehat{m}_0} (x^{cs}),
\end{cases}
\quad x^{cs} \in X^{cs}_{m_0}(\sigma^c_{*}).
\end{equation}

\section{Construction of the local pre-tangent bundle}\label{sub:pre}
Consider the following ``variant'' equations of the above equations \eqref{equ:extension} and \eqref{equ:local0}, i.e., for $x^{cs} \in X^{cs}_{m_0}(\sigma^c_{*})$
\begin{equation}\label{equ:vlocal}
\begin{cases}
D\widehat{F}^{cs}_{m_0} ( x^{cs}, f^0_{\widehat{u}(\widehat{m}_0)} ( x_{\widehat{m}_0} (x^{cs}) ) ) ( \id, K^{0}_{\widehat{u}(\widehat{m}_0)}(x_{\widehat{m}_0} (x^{cs})) R^{0}_{\widehat{m}_0} (x^{cs}) ) = R^{0}_{\widehat{m}_0} (x^{cs}), \\
D\widehat{G}^{cs}_{m_0} ( x^{cs}, f^0_{\widehat{u}(\widehat{m}_0)} ( x_{\widehat{m}_0} (x^{cs}) ) ) ( \id, K^{0}_{\widehat{u}(\widehat{m}_0)}(x_{\widehat{m}_0} (x^{cs})) R^{0}_{\widehat{m}_0} (x^{cs}) ) \triangleq \widetilde{K}^0_{\widehat{m}_0} (x^{cs}),
\end{cases}
\end{equation}
and for $x^{cs} \in X^{cs}_{m_0}(e_0\eta'_1)$,
\begin{equation}\label{equ:vex}
\begin{cases}
DF^1_{\widehat{m}_0} (x^{cs}, \widetilde{f}^0_{\widehat{m}_0}(\widetilde{x}_{\widehat{m}_0}(x^{cs}))) ( \id, \widetilde{K}^0_{\widehat{m}_0} (\widetilde{x}_{\widehat{m}_0}(x^{cs})) \widetilde{R}^{0}_{\widehat{m}_0}(x^{cs}) ) = \widetilde{R}^{0}_{\widehat{m}_0}(x^{cs}),\\
DG^1_{\widehat{m}_0} (x^{cs}, \widetilde{f}^0_{\widehat{m}_0}(\widetilde{x}_{\widehat{m}_0}(x^{cs}))) ( \id, \widetilde{K}^0_{\widehat{m}_0} (\widetilde{x}_{\widehat{m}_0}(x^{cs})) \widetilde{R}^{0}_{\widehat{m}_0}(x^{cs}) ) = K^0_{\widehat{m}_0} (x^{cs}),
\end{cases}
\end{equation}
where $ m_0 = \phi(\widehat{m}_0) $.

Define a metric space
\begin{multline*}
E_K = \{  K^0: \Upsilon^{cs} \to \Upsilon^{u} \text{ is a vector bundle map over } f^0: \\
|K^0_{\widehat{m}_0}(x)| \leq \mu_1(\widehat{m}_0), \forall x \in X^{cs}_{\phi(\widehat{m}_0)} (\sigma_0), K^0_{\widehat{m}_0}(\cdot) \text{ is continuous}, \widehat{m}_0 \in \widehat{K}  \},
\end{multline*}
with the metric
\[
d_{K} (K^0, K^{0}{'}) \triangleq \sup_{\widehat{m}_0 \in \widehat{K}} \sup_{x \in X^{cs}_{\phi(\widehat{m}_0)} (\sigma_0)} |K^0_{\widehat{m}_0}(x) - K^0_{\widehat{m}_0}{'}(x)|,
\]
where $ \Upsilon^{cs} $ (resp.\ $ \Upsilon^{u} $) is a vector bundle over $ X^{cs}_{\widehat{K}} (\sigma_{0}) $ (resp.\ $ X^{u}_{\widehat{K}} (\varrho_{*}) $) with fibers $ \Upsilon^{cs}_{(\widehat{m}_0, x)} = X^{cs}_{\phi(\widehat{m}_0)} $ (resp.\ $ \Upsilon^{u}_{(\widehat{m}_0, x)} = X^{u}_{\phi(\widehat{m}_0)} $), and $ f^0 $ is defined by \eqref{equ:f00}.
For $ K^0 \in E_{K} $, we write $ K^0_{\widehat{m}_0}(x) = K^{0}(\widehat{m}_0, x) $.

The following lemma is straightforward.
\begin{lem}
	The metric $ d_{K} $ is well defined and $ (E_{K}, d_{K}) \neq \emptyset $ is complete.
\end{lem}

Next, we employ the graph transform method to solve \eqref{equ:vlocal} and \eqref{equ:vex}, and then demonstrate that this solution is indeed the (fiber) derivative of $ f^0 $. The approach used here closely follows that in \cite{Che18a}. Note that $ (D\widehat{F}^{cs}_{m_0}( x^{cs}, x^u ), D\widehat{G}^{cs}_{m_0}( x^{cs}, x^u ) ) $ also satisfies (A$'$) ($ \alpha(m_0) $, $ \lambda_{u}(m_0) $) and (B) ($ \beta'(u(m_0)) $; $ \beta'(m_0) $, $ \lambda_{cs}(m_0) $) (see, e.g., \autoref{lem:c1}).

\emph{(Construction)}. Given $ K^0 \in E_{K} $, since $ \alpha(m_0)\mu_1(\widehat{m}_0) < 1 $, there is a unique 
\[
R^0_{\widehat{m}_0}(x^{cs}) \in L(X^{cs}_{m_0}, X^{cs}_{u(m_0)})
\]
satisfying the first equation in \eqref{equ:vlocal} with $ |R^0_{\widehat{m}_0}(x^{cs})| \leq \lambda_{cs}(m_0) $ by the (B) condition. Moreover, the map $ x^{cs} \mapsto R^0_{\widehat{m}_0}(x^{cs}) $ is continuous. Define $ \widetilde{K}^0_{\widehat{m}_0} (x^{cs}) $ as the second equation in \eqref{equ:vlocal}. By the (B) condition again, we have $ |\widetilde{K}^0_{\widehat{m}_0} (x^{cs})| \leq \beta'(m_0) $. Using $ \widetilde{K}^0_{\widehat{m}_0} (x^{cs}) $, there is a unique 
\[
\widetilde{R}^{0}_{\widehat{m}_0}(x^{cs}) \in L(X^{cs}_{m_0}, X^{cs}_{m_0})
\]
satisfying the first equation in \eqref{equ:vex} for $ x^{cs} \in X^{cs}_{m_0}(e_0\eta'_1) $, since 
\[
|D_2F^1_{\widehat{m}_0} (x^{cs}, \widetilde{f}^0_{\widehat{m}_0}(\widetilde{x}_{\widehat{m}_0}(x^{cs})))|
\]
can be made sufficiently small. Furthermore, the map $ x^{cs} \mapsto \widetilde{R}^{0}_{\widehat{m}_0}(x^{cs}) $ is continuous. We define $ \widehat{K}^{0}_{\widehat{m}_0}(x^{cs}) $ as the left-hand side of the second equation in \eqref{equ:vex}, which is continuous for $ x^{cs} \in X^{cs}_{m_0}(e_0\eta'_1) $, i.e.,
\[
\widehat{K}^{0}_{\widehat{m}_0}(x^{cs}) = DG^1_{\widehat{m}_0} (x^{cs}, \widetilde{f}^0_{\widehat{m}_0}(\widetilde{x}_{\widehat{m}_0}(x^{cs}))) ( \id, \widetilde{K}^0_{\widehat{m}_0} (\widetilde{x}_{\widehat{m}_0}(x^{cs})) \widetilde{R}^{0}_{\widehat{m}_0}(x^{cs}) ).
\]

\begin{lem}\label{lem:KK0}
	$ \widehat{K}^{0}_{\widehat{m}_0}(x^{cs}) $ can be continuously extended to all of $ X^{cs}_{m_0}(\sigma_{0}) $ by setting 
	\[
	\widehat{K}^{0}_{\widehat{m}_0}(x^{cs}) = Dg^{cs}_{\widehat{m}_0}(x^{cs}), \quad \text{for} \quad x^{cs} \in X^{cs}_{m_0}(\sigma_{0}) \setminus X^{cs}_{m_0}(e_0\eta'_1) ,
	\]
	where $ g^{cs}_{\widehat{m}_0} $ is defined by \eqref{equ:cslocal}. Moreover, $ |\widehat{K}^{0}_{\widehat{m}_0}(x^{cs})| \leq \mu_1(\widehat{m}_0) $, $ x^{cs} \in X^{cs}_{m_0}(\sigma_{0}) $.
\end{lem}
\begin{proof}
	The first statement holds because for $ \eta''_1 =  \max\{ e_2, c_2 \} \eta_1 $ (by the choice of $ \eta'_1 $ such that $ \eta''_1 < e_0 \eta'_1 $), we have $ f^0_{\widehat{m}_0} (x^{cs}) = g^0_{\widehat{m}_0} (x^{cs}) $ when $ |x^{cs}| > \eta''_1 $ (i.e., $ \Psi(\widehat{m}, \overline{x}^s) = 0 $). This implies that for any $ x^{u} $, $ G^0_{\widehat{m}_0} (x^{cs}, x^u) = g^0_{\widehat{m}_0} (x^{cs}) $. So, $ DG^0_{\widehat{m}_0} (x^{cs}, x^u) = (Dg^{cs}_{\widehat{m}_0}(x^{cs}), 0) $, yielding $ \widehat{K}^{0}_{\widehat{m}_0}(x^{cs}) = Dg^{cs}_{\widehat{m}_0}(x^{cs}) $ for $ |x^{cs}| > \eta''_1 $.

	Now consider the second statement.
	Let $ x^{u} = \widetilde{f}^0_{\widehat{m}_0}(\widetilde{x}_{\widehat{m}_0}(x^{cs})) $, $ \widetilde{x}^{cs} = \widetilde{x}_{\widehat{m}_0}(x^{cs}) $, and
	\[
	\tilde{x}^{cs}_2 = DF^1_{\widehat{m}_0} (x^{cs}, x^{u}) (\tilde{x}^{cs}_1, \tilde{x}^u_2), \quad \tilde{x}^{u}_1 = DG^1_{\widehat{m}_0} (x^{cs}, x^{u}) (\tilde{x}^{cs}_1, \tilde{x}^u_2), \quad x^{cs} \in X^{cs}_{m_0}(e_0\eta'_1).
	\]
	In case (1), since $ \lip \widetilde{f}^0_{\widehat{m}_0}(\cdot) \leq \beta'(m_0) $, we have $ |x^{u}| \leq K_1\eta + \beta'(m_0) \sigma^c_{*} $; in case (2), by \autoref{lem:estR}, we have $ |x^{u}| \leq e_1\varrho_{*} $.

	By \autoref{lem:est000}, if $ |\tilde{x}^u_2| \leq \beta'(m_0) |\tilde{x}^{cs}_2| $, then
	\[
	|\tilde{x}^{cs}_2| \leq \varpi^*_1|\tilde{x}^{cs}_1| + \varpi^*_0|\tilde{x}^u_2| \quad \Rightarrow \quad |\tilde{x}^{cs}_2| \leq \varpi^*_1|\tilde{x}^{cs}_1|,
	\]
	and
	\begin{align*}
	|\tilde{x}^{u}_1| & = |D_{x^{cs}} G^0_{\widehat{m}_0} (\widetilde{x}^{cs}, x^{u}) D_{x^{cs}} F^1_{\widehat{m}_0} (x^{cs}, x^{u}) \tilde{x}^{cs}_1 + D_{x^u} G^0_{\widehat{m}_0} (\widetilde{x}^{cs}, x^{u}) \tilde{x}^u_2| \\
	& \leq P_{1,*} |\tilde{x}^{cs}_1| \leq \mu_1(\widehat{m}_0) |\tilde{x}^{cs}_1|,
	\end{align*}
	provided that $ \epsilon_{*}, \chi(\epsilon_{*}) $ are small and $ \eta_1 $ is sufficiently close to $ \sigma^1_{*} $. The last inequality holds because $ \epsilon_{0,*}\eta_0 / \eta_1 \to 0 $, $ \varpi^*_1 \to 1 $ as $ \epsilon_{*}, \chi(\epsilon_{*}) \to 0 $, and $ \varsigma'_1 < \varsigma $ (in case (1)) or $ \varrho_{*} / \eta_1 \to 0 $ (in case (2)). Here,
	\[
	P_{1,*} = \begin{cases}
	\varpi^*_1K_1C_1(\varsigma'_1 - 1)\epsilon_{0,*}\eta_0 / \eta_1 + \varpi^*_1 (C_1(\varsigma'_1 - 1) + 1) \beta'(m_0), & \text{case (1)},\\
	\varpi^*_1(2 C_1e_1\varrho_{*} / \eta_1 + 1) \beta'(m_0), & \text{case (2)}.
	\end{cases}
	\]
	In particular, due to $ |\widetilde{K}^0_{\widehat{m}_0} (x^{cs})| \leq \beta'(m_0) $, one has $ |\widehat{K}^{0}_{\widehat{m}_0}(x^{cs})| \leq \mu_1(\widehat{m}_0) $ for all $ x^{cs} \in X^{cs}_{m_0}(e_0\eta'_1) $, and so for all $ x^{cs} \in X^{cs}_{m_0}(\sigma_{0}) $.
	The proof is complete.
\end{proof}

Define the \emph{graph transform} to be
\[
\varGamma_0: E_{K} \to E_{K}, \quad K^0 \mapsto \widehat{K}^0.
\]
\autoref{lem:KK0} shows that this graph transform is well defined.

\begin{lem}\label{lem:lipest}
	$ \lip \varGamma_0 \leq \varpi^*_1 \sup_{\widehat{m}_0 \in \widehat{K}} \frac{ \lambda_{cs}(\phi(\widehat{m}_0)) \lambda_{u}(\phi(\widehat{m}_0)) }{ 1 - \alpha(\phi(\widehat{m}_0)) \mu_1(\widehat{u}(\widehat{m}_0)) } < 1 $.
\end{lem}
\begin{proof}
	Let $ \widehat{K}^0 = \varGamma_0 K^0 $, $ \widehat{K}^1 = \varGamma_0 K^1 $, and $ m_0 = \phi(\widehat{m}_0) $ for $ \widehat{m}_0 \in \widehat{K} $. Denote by $ R^0 $, $ R^1 $ the operators associated with the construction of $ \widehat{K}^0 $, $ \widehat{K}^1 $ in \eqref{equ:vlocal}, respectively, and similarly $ \widetilde{R}^0 $, $ \widetilde{R}^1 $ in \eqref{equ:vex}. Since $ \widehat{K}^0_{\widehat{m}_0}(x^{cs}) = \widehat{K}^1_{\widehat{m}_0}(x^{cs}) $ for $ x^{cs} \in X^{cs}_{m_0}(\sigma_{0}) \setminus X^{cs}_{m_0}(e_0\eta'_1) $, we only need to consider $ x^{cs} \in X^{cs}_{m_0}(e_0\eta'_1) $. From \eqref{equ:vlocal} and the (A$ ' $) condition, for $ x^{cs} \in X^{cs}_{m_0}(\sigma^c_{*}) $, we see
	\begin{align*}
	&~ |R^0_{\widehat{m}_0}(x^{cs}) - R^1_{\widehat{m}_0}(x^{cs})| \\
	\leq &~ \alpha(m_0)| K^{0}_{\widehat{u}(\widehat{m}_0)}(x_{\widehat{m}_0} (x^{cs})) R^0_{\widehat{m}_0}(x^{cs}) - K^{1}_{\widehat{u}(\widehat{m}_0)}(x_{\widehat{m}_0} (x^{cs})) R^1_{\widehat{m}_0}(x^{cs}) |,
\end{align*}
	and so
	\begin{align*}
	&~ | K^{0}_{\widehat{u}(\widehat{m}_0)}(x_{\widehat{m}_0} (x^{cs})) R^0_{\widehat{m}_0}(x^{cs}) - K^{1}_{\widehat{u}(\widehat{m}_0)}(x_{\widehat{m}_0} (x^{cs})) R^1_{\widehat{m}_0}(x^{cs}) | \\
	\leq &~ | K^{0}_{\widehat{u}(\widehat{m}_0)}(x_{\widehat{m}_0} (x^{cs})) R^0_{\widehat{m}_0}(x^{cs}) - K^{1}_{\widehat{u}(\widehat{m}_0)}(x_{\widehat{m}_0} (x^{cs})) R^0_{\widehat{m}_0}(x^{cs}) | \\
	&~ + \mu_1(\widehat{u}(\widehat{m}_0)) |R^0_{\widehat{m}_0}(x^{cs}) - R^1_{\widehat{m}_0}(x^{cs})| \\
	\leq & ~ \frac{1}{1 - \alpha(m_0) \mu_1(\widehat{u}(\widehat{m}_0))} | K^{0}_{\widehat{u}(\widehat{m}_0)}(x_{\widehat{m}_0} (x^{cs})) R^0_{\widehat{m}_0}(x^{cs}) - K^{1}_{\widehat{u}(\widehat{m}_0)}(x_{\widehat{m}_0} (x^{cs})) R^0_{\widehat{m}_0}(x^{cs}) |,
	\end{align*}
	yielding
	\begin{align*}
	&~ |\widetilde{K}^0_{\widehat{m}_0}(x^{cs}) - \widetilde{K}^1_{\widehat{m}_0}(x^{cs})| \\
	\leq &~ \lambda_{u} (m_0) | K^{0}_{\widehat{u}(\widehat{m}_0)}(x_{\widehat{m}_0} (x^{cs})) R^0_{\widehat{m}_0}(x^{cs}) - K^{1}_{\widehat{u}(\widehat{m}_0)}(x_{\widehat{m}_0} (x^{cs})) R^1_{\widehat{m}_0}(x^{cs}) | \\
	\leq &~ \frac{\lambda_{u} (m_0)}{1 - \alpha(m_0) \mu_1(\widehat{u}(\widehat{m}_0))} | K^{0}_{\widehat{u}(\widehat{m}_0)}(x_{\widehat{m}_0} (x^{cs})) R^0_{\widehat{m}_0}(x^{cs}) - K^{1}_{\widehat{u}(\widehat{m}_0)}(x_{\widehat{m}_0} (x^{cs})) R^0_{\widehat{m}_0}(x^{cs}) | \\
	\leq &~ \frac{\lambda_{cs}(m_0)\lambda_{u} (m_0)}{1 - \alpha(m_0) \mu_1(\widehat{u}(\widehat{m}_0))} | K^{0}_{\widehat{u}(\widehat{m}_0)}(x_{\widehat{m}_0} (x^{cs})) - K^{1}_{\widehat{u}(\widehat{m}_0)}(x_{\widehat{m}_0} (x^{cs})) |.
	\end{align*}
	A similar argument applied to \eqref{equ:vex} shows that for $ x^{cs} \in X^{cs}_{m_0}(e_0\eta'_1) $,
	\[
	|\widehat{K}^0_{\widehat{m}_0}(x^{cs}) - \widehat{K}^1_{\widehat{m}_0}(x^{cs})| \leq \varpi^*_1 | \widetilde{K}^{0}_{\widehat{u}(\widehat{m}_0)}(\widetilde{x}_{\widehat{m}_0} (x^{cs})) - \widetilde{K}^{1}_{\widehat{u}(\widehat{m}_0)}(\widetilde{x}_{\widehat{m}_0} (x^{cs})) |.
	\]
	By \autoref{lem:est000}, $ (DF^1_{\widehat{m}_0} (x^{cs}, x^{u}), DG^1_{\widehat{m}_0} (x^{cs}, x^{u})) $ satisfies the (A) ($ \varpi^*_0{'} $; $ \varpi^*_0 $, $ \varpi^*_1 $) and (B) ($ \beta'(m_0) $; $ \mu_1(\widehat{m}_0) $, $ \varpi^*_1 $) condition, where $ \varpi^*_0{'} < \varpi^*_0 $, when $ |x^{u}| \leq K_1\eta + \beta'(m_0) \sigma^c_{*} $ in case (1) or $ |x^{u}| \leq e_1\varrho_{*} $ in case (2), with $ \eta_1 $ sufficiently close to $ \sigma^1_{*} $ and $ \epsilon_{*}, \chi(\epsilon_{*}) $ small. Therefore, for $ x^{cs} \in X^{cs}_{m_0}(e_0\eta'_1) $, we get
	\begin{multline*}
		|\widehat{K}^0_{\widehat{m}_0}(x^{cs}) - \widehat{K}^1_{\widehat{m}_0}(x^{cs})| \\
		\leq \varpi^*_1 \frac{\lambda_{cs}(m_0)\lambda_{u} (m_0)}{1 - \alpha(m_0) \mu_1(\widehat{u}(\widehat{m}_0))} | K^{0}_{\widehat{u}(\widehat{m}_0)}(x_{\widehat{m}_0} (\widetilde{x}_{\widehat{m}_0} (x^{cs}))) - K^{1}_{\widehat{u}(\widehat{m}_0)}(x_{\widehat{m}_0} (\widetilde{x}_{\widehat{m}_0} (x^{cs}))) |,
	\end{multline*}
	and consequently $ \lip \varGamma_0 \leq \varpi^*_1 \sup_{\widehat{m}_0 \in \widehat{K}} \frac{ \lambda_{cs}(\phi(\widehat{m}_0)) \lambda_{u}(\phi(\widehat{m}_0)) }{ 1 - \alpha(\phi(\widehat{m}_0)) \mu_1(\widehat{u}(\widehat{m}_0)) } < 1 $. The proof is complete.
\end{proof}

By \autoref{lem:lipest}, there is a unique $ K^0 \in E_{K} $ such that $ \varGamma_0 K^0 = K^0 $, i.e., \eqref{equ:vlocal} and \eqref{equ:vex} hold.

\section{$ C^1 $ smoothness of the center-stable manifold and proof of \autoref{thm:smooth}}\label{sub:C1smooth}
We now show

\begin{lem}\label{lem:C1smooth}
	$ f^0_{\widehat{m}_0}(\cdot) \in C^1 $ and $ D_{x^{cs}} f^0_{\widehat{m}_0}(x^{cs}) = K^0_{\widehat{m}_0}(x^{cs}) $ for each $ \widehat{m}_0 \in \widehat{K} $ and $ x^{cs} \in X^{cs}_{\phi(\widehat{m}_0)} (\sigma_0) $.
\end{lem}
\begin{proof}
	For $ x^{cs} \in X^{cs}_{m_0}(\sigma_{0}) \setminus X^{cs}_{m_0}(e_0\eta'_1) $, we have $ f^0_{\widehat{m}_0} (x^{cs}) = g^{cs}_{\widehat{m}_0}(x^{cs}) $, and so
	\[
	D_{x^{cs}} f^0_{\widehat{m}_0}(x^{cs}) = D_{x^{cs}} g^{cs}_{\widehat{m}_0} (x^{cs}) = K^0_{\widehat{m}_0}(x^{cs}).
	\]
	Thus, we only need to consider $ x^{cs} \in X^{cs}_{m_0}(e_0\eta'_1) $. Take $ \widehat{m}_0 \in \widehat{K} $, $ m_0 = \phi(\widehat{m}_0) $, and set
	\[
	\mathcal{Q}(\widehat{m}_0, x', x) = f^0_{\widehat{m}_0} (x') - f^{0}_{\widehat{m}_0} (x) - K^0_{\widehat{m}_0} (x) (x' - x), \quad x', x \in X^{cs}_{m_0}(\sigma_{0}).
	\]
	As noted previously, $ \limsup_{x' \to x}\frac{|\mathcal{Q}(\widehat{m}_0, x', x)|}{|x' - x|} = 0 $ if $ x \in X^{cs}_{m_0}(\sigma_{0}) \setminus X^{cs}_{m_0}(e_0\eta'_1) $. Also, note that
	\[
	\sup_{\widehat{m}_0 \in \widehat{K}} \sup_{x \in X^{cs}_{\phi(\widehat{m}_0)}(\sigma_{0})} \limsup_{x' \to x} \frac{|\mathcal{Q}(\widehat{m}_0, x', x)|}{|x' - x|} < \infty.
	\]
	\begin{slem}\label{slem:aa}
		For $ x \in X^{cs}_{m_0}(e_0\eta'_1) $,
		\[
		\limsup_{x' \to x} \frac{|\mathcal{Q}(\widehat{m}_0, x', x)|}{|x' - x|} \leq \varpi^*_1 \frac{\lambda_{cs}(m_0)\lambda_{u} (m_0)}{1 - \alpha(m_0) \mu_1(\widehat{u}(\widehat{m}_0))} \sup_{x \in X^{cs}_{u(m_0)}(\sigma_{0})} \limsup_{x' \to x}  \frac{|\mathcal{Q}(\widehat{u}(\widehat{m}_0), x', x)|}{|x' - x|}.
		\]
	\end{slem}
	\begin{proof}
		We use the notation $ |h_{\widehat{m}_0}(x', x)| \leq o(1) $ to indicate that $ \limsup_{x' \to x}|h_{\widehat{m}_0}(x', x)| = 0 $. Set
		\[
		x^{u} = f^0_{\widehat{u}(\widehat{m}_0)} ( x_{\widehat{m}_0} (x) ).
		\]
		For $ x', x \in X^{cs}_{m_0}(\sigma^{c}_*) $, from \eqref{equ:local0} and \eqref{equ:vlocal}, we have
		\begin{align*}
			&~ \left|x_{\widehat{m}_0} (x') - x_{\widehat{m}_0} (x) - R^{0}_{\widehat{m}_0} (x) (x' - x)\right| \\
			= &~\left|  \widehat{F}^{cs}_{m_0} ( x', f^0_{\widehat{u}(\widehat{m}_0)} ( x_{\widehat{m}_0} (x') ) ) -  \widehat{F}^{cs}_{m_0} ( x, f^0_{\widehat{u}(\widehat{m}_0)} ( x_{\widehat{m}_0} (x) ) ) \right. \\
			& \left. \quad - D\widehat{F}^{cs}_{m_0} ( x, x^{u} ) ( x' - x, K^{0}_{\widehat{u}(\widehat{m}_0)}(x_{\widehat{m}_0} (x)) R^{0}_{\widehat{m}_0} (x) (x' - x) ) \right| \\
			= &~ \left|  \widehat{F}^{cs}_{m_0} ( x', f^0_{\widehat{u}(\widehat{m}_0)} ( x_{\widehat{m}_0} (x') ) ) -  \widehat{F}^{cs}_{m_0} ( x, f^0_{\widehat{u}(\widehat{m}_0)} ( x_{\widehat{m}_0} (x) ) ) \right. \\
			&  \quad -  D\widehat{F}^{cs}_{m_0} ( x, x^{u} ) (  x' - x, f^0_{\widehat{u}(\widehat{m}_0)} ( x_{\widehat{m}_0} (x') ) - f^0_{\widehat{u}(\widehat{m}_0)} ( x_{\widehat{m}_0} (x) ) )  \\
			& ~ \left. +  D_{x^u}\widehat{F}^{cs}_{m_0} ( x, x^{u} )  \left\{ f^0_{\widehat{u}(\widehat{m}_0)} ( x_{\widehat{m}_0} (x') )  f^0_{\widehat{u}(\widehat{m}_0)} ( x_{\widehat{m}_0} (x) ) - \right. \right. \\
			& \quad ~ \left. \left. - K^{0}_{\widehat{u}(\widehat{m}_0)}(x_{\widehat{m}_0} (x)) R^{0}_{\widehat{m}_0} (x) (x' - x)  \right\}   \right| \\
			\leq &~ o(1) |x' - x| + \alpha(m_0) \left|f^0_{\widehat{u}(\widehat{m}_0)} ( x_{\widehat{m}_0} (x') ) - f^0_{\widehat{u}(\widehat{m}_0)} ( x_{\widehat{m}_0} (x) ) \right. \\
			& \quad \left. - K^{0}_{\widehat{u}(\widehat{m}_0)}(x_{\widehat{m}_0} (x)) R^{0}_{\widehat{m}_0} (x) (x' - x)\right| \\
			\leq &~ o(1) |x' - x| + \alpha(m_0) \left|  K^{0}_{\widehat{u}(\widehat{m}_0)}(x_{\widehat{m}_0} (x)) \{x_{\widehat{m}_0} (x') - x_{\widehat{m}_0} (x') R^{0}_{\widehat{m}_0} (x) (x' - x)\}  \right| \\
			& \quad + \alpha(m_0) \left|  f^0_{\widehat{u}(\widehat{m}_0)} ( x_{\widehat{m}_0} (x') ) - f^0_{\widehat{u}(\widehat{m}_0)} ( x_{\widehat{m}_0} (x) ) - K^{0}_{\widehat{u}(\widehat{m}_0)}(x_{\widehat{m}_0} (x)) (x_{\widehat{m}_0} (x') - x_{\widehat{m}_0} (x'))  \right| \\
			\leq &~ o(1) |x' - x| + \alpha(m_0) \mu_1(\widehat{u}(\widehat{m}_0)) \left|x_{\widehat{m}_0} (x') - x_{\widehat{m}_0} (x) - R^{0}_{\widehat{m}_0} (x) (x' - x)\right| \\
			& \quad + \alpha(m_0) \left|\mathcal{Q}(\widehat{u}(\widehat{m}_0), x_{\widehat{m}_0} (x'), x_{\widehat{m}_0} (x))\right| \\
			\leq &~ o(1)|x' - x| + \frac{\alpha(m_0)}{1 - \alpha(m_0) \mu_1(\widehat{u}(\widehat{m}_0))} \left|\mathcal{Q}(\widehat{u}(\widehat{m}_0), x_{\widehat{m}_0} (x'), x_{\widehat{m}_0} (x))\right|.
		\end{align*}
		Similarly,
		\begin{align*}
			&~ \left|\widetilde{f}^0_{\widehat{m}_0} (x') - \widetilde{f}^0_{\widehat{m}_0} (x) - \widetilde{K}^0_{\widehat{m}_0} (x^{cs}) (x' - x)\right| \\
			= &~ \left|  \widehat{G}^{cs}_{m_0} ( x', f^0_{\widehat{u}(\widehat{m}_0)} ( x_{\widehat{m}_0} (x') ) ) -  \widehat{G}^{cs}_{m_0} ( x, f^0_{\widehat{u}(\widehat{m}_0)} ( x_{\widehat{m}_0} (x) ) ) \right. \\
			& \quad -  D\widehat{G}^{cs}_{m_0} ( x, x^{u} ) (  x' - x, f^0_{\widehat{u}(\widehat{m}_0)} ( x_{\widehat{m}_0} (x') ) - f^0_{\widehat{u}(\widehat{m}_0)} ( x_{\widehat{m}_0} (x) ) )  \\
			&~ \left. +  D_{x^u}\widehat{G}^{cs}_{m_0} ( x, x^{u} )  \left\{ f^0_{\widehat{u}(\widehat{m}_0)} ( x_{\widehat{m}_0} (x') ) - f^0_{\widehat{u}(\widehat{m}_0)} ( x_{\widehat{m}_0} (x) ) \right. \right. \\
			&\quad \left. \left. - K^{0}_{\widehat{u}(\widehat{m}_0)}(x_{\widehat{m}_0} (x)) R^{0}_{\widehat{m}_0} (x) (x' - x)  \right\}   \right| \\
			\leq &~ o(1)|x' - x| + \lambda_{u}(m_0) \mu_1(\widehat{u}(\widehat{m}_0)) \left|x_{\widehat{m}_0} (x') - x_{\widehat{m}_0} (x) - R^{0}_{\widehat{m}_0} (x) (x' - x)\right| \\
			& \quad + \lambda_{u}(m_0) \left|\mathcal{Q}(\widehat{u}(\widehat{m}_0), x_{\widehat{m}_0} (x'), x_{\widehat{m}_0} (x))\right|\\
			\leq &~ o(1)|x' - x| + \frac{\lambda_{u}(m_0)}{1 - \alpha(m_0)\mu_1(\widehat{u}(\widehat{m}_0))} \left|\mathcal{Q}(\widehat{u}(\widehat{m}_0), x_{\widehat{m}_0} (x'), x_{\widehat{m}_0} (x))\right|.
		\end{align*}
		Furthermore, the same argument shows that for $ x', x \in X^{cs}_{m_0}(e_0\eta'_1) $, from \eqref{equ:extension} and \eqref{equ:vex}, we get
		\begin{multline*}
			|\widetilde{x}_{\widehat{m}_0}(x') - \widetilde{x}_{\widehat{m}_0}(x) - \widetilde{R}^{0}_{\widehat{m}_0} (x) (x' - x)| \leq o(1)|x' - x| \\
			+ \varpi^*_0 |  \widetilde{f}^0_{\widehat{m}_0} ( \widetilde{x}_{\widehat{m}_0} (x') ) - \widetilde{f}^0_{\widehat{m}_0} ( \widetilde{x}_{\widehat{m}_0} (x) ) - \widetilde{K}^{0}_{\widehat{m}_0}(\widetilde{x}_{\widehat{m}_0} (x)) (\widetilde{x}_{\widehat{m}_0} (x') - \widetilde{x}_{\widehat{m}_0} (x'))  |,
		\end{multline*}
		and so
		\begin{multline*}
			|\mathcal{Q}(\widehat{m}_0, x', x)| = |f^0_{\widehat{m}_0} (x') - f^{0}_{\widehat{m}_0} (x') - K^0_{\widehat{m}_0} (x) (x' - x)| \leq o(1)|x' - x| \\
			+ \varpi^*_1 |  \widetilde{f}^0_{\widehat{m}_0} ( \widetilde{x}_{\widehat{m}_0} (x') ) - \widetilde{f}^0_{\widehat{m}_0} ( \widetilde{x}_{\widehat{m}_0} (x) ) - \widetilde{K}^{0}_{\widehat{m}_0}(\widetilde{x}_{\widehat{m}_0} (x)) (\widetilde{x}_{\widehat{m}_0} (x') - \widetilde{x}_{\widehat{m}_0} (x'))  |.
		\end{multline*}
		Therefore, for $ x', x \in X^{cs}_{m_0}(e_0\eta'_1) $, we obtain
		\begin{multline*}
			|\mathcal{Q}(\widehat{m}_0, x', x)| \\
			\leq o(1)|x' - x| + \varpi^*_1 \frac{\lambda_{u}(m_0)}{1 - \alpha(m_0)\mu_1(\widehat{u}(\widehat{m}_0))} |\mathcal{Q}(\widehat{u}(\widehat{m}_0), x_{\widehat{m}_0} (\widetilde{x}_{\widehat{m}_0} (x')), x_{\widehat{m}_0} (\widetilde{x}_{\widehat{m}_0} (x)))|.
		\end{multline*}
		Now for $ x \in X^{cs}_{m_0}(e_0\eta'_1) $, we see
		\begin{align*}
			\limsup_{x' \to x} & \frac{|\mathcal{Q}(\widehat{m}_0, x', x)|}{|x' - x|}
			\leq \varpi^*_1 \frac{\lambda_{u}(m_0)}{1 - \alpha(m_0)\mu_1(\widehat{u}(\widehat{m}_0))} \\
			& \cdot \limsup_{x' \to x} \frac{|\mathcal{Q}(\widehat{u}(\widehat{m}_0), x_{\widehat{m}_0} (\widetilde{x}_{\widehat{m}_0} (x')), x_{\widehat{m}_0} (\widetilde{x}_{\widehat{m}_0} (x)))|}{|x_{\widehat{m}_0} (\widetilde{x}_{\widehat{m}_0} (x')) - x_{\widehat{m}_0} (\widetilde{x}_{\widehat{m}_0} (x))|} \frac{|x_{\widehat{m}_0} (\widetilde{x}_{\widehat{m}_0} (x')) - x_{\widehat{m}_0} (\widetilde{x}_{\widehat{m}_0} (x))|}{|x' - x|} \\
			\leq & \varpi^*_1 \frac{\lambda_{cs}(m_0)\lambda_{u}(m_0)}{1 - \alpha(m_0)\mu_1(\widehat{u}(\widehat{m}_0))}  \limsup_{x' \to x} \frac{|\mathcal{Q}(\widehat{u}(\widehat{m}_0), x_{\widehat{m}_0} (\widetilde{x}_{\widehat{m}_0} (x')), x_{\widehat{m}_0} (\widetilde{x}_{\widehat{m}_0} (x)))|}{|x_{\widehat{m}_0} (\widetilde{x}_{\widehat{m}_0} (x')) - x_{\widehat{m}_0} (\widetilde{x}_{\widehat{m}_0} (x))|} \\
			\leq & \varpi^*_1 \frac{\lambda_{cs}(m_0)\lambda_{u}(m_0)}{1 - \alpha(m_0)\mu_1(\widehat{u}(\widehat{m}_0))} \sup_{x \in X^{cs}_{u(m_0)}(\sigma_{0})} \limsup_{x' \to x}  \frac{|\mathcal{Q}(\widehat{u}(\widehat{m}_0), x', x)|}{|x' - x|}.
		\end{align*}
		The proof is complete.
	\end{proof}
	Let $ \theta_1 = \varpi^*_1 \sup_{\widehat{m}_0 \in \widehat{K}} \frac{\lambda_{cs}(\phi(\widehat{m}_0)) \lambda_{u}(\phi(\widehat{m}_0)) }{ 1 - \alpha(\phi(\widehat{m}_0)) \mu_1(\widehat{u}(\widehat{m}_0)) } < 1 $. Then by \autoref{slem:aa}, we have
	\[
	\sup_{ \widehat{m}_0 \in \widehat{K} } \sup_{x \in X^{cs}_{m_0}(\sigma_{0})} \limsup_{x' \to x} \frac{|\mathcal{Q}(\widehat{m}_0, x', x)|}{|x' - x|}
	\leq \theta_1 \sup_{ \widehat{m}_0 \in \widehat{K} } \sup_{x \in X^{cs}_{m_0}(\sigma_{0})} \limsup_{x' \to x} \frac{|\mathcal{Q}(\widehat{m}_0, x', x)|}{|x' - x|} < \infty,
	\]
	which shows that $ \limsup_{x' \to x}\frac{|\mathcal{Q}(\widehat{m}_0, x', x)|}{|x' - x|} = 0 $, i.e., $ D f^0_{\widehat{m}_0}(x) = K^0_{\widehat{m}_0}(x) $. As $ x \mapsto K^0_{\widehat{m}_0}(x) $ is $ C^0 $, we have $ f^0_{\widehat{m}_0}(\cdot) \in C^1 $. The proof of \autoref{lem:C1smooth} is complete.
\end{proof}

Therefore, by \autoref{lem:C1smooth}, we conclude that $ W^{cs}_{loc}(K) \in C^1 $.

\begin{rmk}\label{rmk:diff}
	Even if such a $ C^1 $ bump function $ \Psi $ does not exist, we can show that $ f^0_{\widehat{m}_0}(\cdot) $ is differentiable at $ 0 $, with $ \widehat{K} \ni \widehat{m}_0 \mapsto Df^0_{\widehat{m}_0}(0) $ continuous, under the additional assumption $ \eta = 0 $.
	\begin{proof}
		Since $ \eta = 0 $, we know $ f^0_{\widehat{m}_0}(0) = 0 $. If $ Df^{0}_{\widehat{m}_0}(0) = K^0_{\widehat{m}_0} (0) $ exists, then it must satisfy
		\[\label{equ:zero} \tag{$ \ast $}
		\begin{cases}
		D\widehat{F}^{cs}_{m_0} ( 0, 0 ) ( \id, K^{0}_{\widehat{u}(\widehat{m}_0)}(0) R^{0}_{\widehat{m}_0} (0) ) = R^{0}_{\widehat{m}_0} (0), \\
		D\widehat{G}^{cs}_{m_0} ( 0, 0 ) ( \id, K^{0}_{\widehat{u}(\widehat{m}_0)}(0) R^{0}_{\widehat{m}_0} (0) ) = K^0_{\widehat{m}_0} (0),
		\end{cases}
		\]
		where $ m_0 = \phi(\widehat{m}_0) $. This equation has a unique solution $ K^0_{\widehat{m}_0} (0) \in L(X^{cs}_{m_0}, X^{u}_{m_0}) $ satisfying $ |K^0_{\widehat{m}_0} (0)|\leq \beta'(m_0) $ and $ \widehat{m}_0 \mapsto K^0_{\widehat{m}_0} (0) $ is $ C^0 $. This can be established using a similar but simpler approach as in \autoref{sub:pre}. Applying the same argument as in \autoref{lem:C1smooth} shows that $ \limsup_{x' \to 0}\frac{|\mathcal{Q}(\widehat{m}_0, x', 0)|}{|x'|} = 0 $, where
		\[
		\mathcal{Q}(\widehat{m}_0, x', 0) = f^0_{\widehat{m}_0} (x') - K^0_{\widehat{m}_0} (0) x'.
		\]
		See also \cite[Section 6.3]{Che18a}.
	\end{proof}
\end{rmk}

\begin{lem}
	Suppose $ \eta = 0 $ and $ X^{cs} $ is invariant under $ DH $ (i.e., $ D_{x^{cs}}\widehat{G}^{cs}_{m}( 0, 0, 0 ) = 0 $ for all $ m \in K $), then $ Df^{0}_{\widehat{m}_0} (0) = 0 $, i.e., $ T_{\widehat{m}_0} W^{cs}_{loc}(K) = X^{cs}_{\phi(\widehat{m}_0)} $, $ \widehat{m}_0 \in \widehat{K} $.
\end{lem}
\begin{proof}
	It suffices to show $ K^0_{\widehat{m}_0}(0) = 0 $ in equation \eqref{equ:zero}. Indeed, $ K^0_{\widehat{m}_0}(0) = 0 $ is the solution of \eqref{equ:zero} since $ D_{x^{cs}}\widehat{G}^{cs}_{m_0}( 0, 0, 0 ) = 0 $.
\end{proof}

\begin{proof}[Proof of \autoref{thm:smooth}]
	Consider the ($ \bullet 1 $) case: $ \widehat{H} \approx (\widehat{F}^{cs}, \widehat{G}^{cs}) $ satisfies the (A$ ' $)($ \alpha $, $ \lambda_{u} $) and (B)($ \beta; \beta', \lambda_{cs} $) condition in $ cs $-direction at $ K $. For the ($ \bullet 2 $) case, see \autoref{sec:tri}.
	The result now follows from \autoref{lem:C1smooth} and \autoref{rmk:diff}.
\end{proof}

\begin{lem}[$ C^{1,\alpha} $ smoothness]\label{lem:C1a}
	Suppose the following conditions hold, where $ \tilde{\eta} $ satisfies $ 0 < \tilde{\eta} \leq \eta_2 $, $ e_0 = \min\{ e_1, c_1 \} $, and $ X^{cs}_{m_0}(e_0\tilde{\eta})^{\complement} \triangleq X^{c}_{m_0}(c_2\epsilon_{*}) \oplus X^{s}_{m_0}(e_2\sigma_{*}) \setminus X^{cs}_{m_0} (e_0\tilde{\eta}) $.
	\begin{enumerate}[(a)]
		\item $ X^s_{\widehat{K}_{\epsilon_{*}}}(\sigma_{*}) \setminus X^s_{\widehat{K}_{\tilde{\eta}}}(\tilde{\eta}) \in C^{1, 1} $, i.e., for the local representations $ g^{cs}_{\widehat{m}_0}(\cdot) $ (see \eqref{equ:cslocal}), $ \widehat{m}_0 \in \widehat{K} $ of $ X^s_{\widehat{K}_{\epsilon_{*}}}(\sigma_{*}) $, they are $ C^{1,1} $ outside $ X^{cs}_{\phi(\widehat{m}_0)} (e_0\tilde{\eta}) $ uniformly for $ \widehat{m}_0 $, or more precisely,
		\[
		\sup_{\widehat{m}_0 \in \widehat{K}} \lip Dg^{cs}_{\widehat{m}_0}(\cdot)|_{X^{cs}_{\phi(\widehat{m}_0)}(e_0\tilde{\eta})^{\complement}} < \infty;
		\]

		\item $ \Psi \in C^{1,1}(X^s_{\widehat{K}_{\epsilon_{*}}}(\sigma_{*}) \setminus X^s_{\widehat{K}_{\tilde{\eta}}}(\tilde{\eta}), [0,1]) $, i.e., for the local representations
		\begin{equation}\label{equ:localbump}
		\widetilde{\Psi}_{\widehat{m}_0}(x^{cs}) = \Psi(\overline{\Phi}^{-c}_{m_0,\gamma}(x^{cs}, g^{cs}_{\widehat{m}_0}(x^{cs})), \overline{\Phi}^{-s}_{m_0,\gamma}(x^{cs}, g^{cs}_{\widehat{m}_0}(x^{cs}))), \quad x^{cs} \in X^{cs}_{m_0},
		\end{equation}
		where $ \widehat{m}_0 = \phi_{m_0, \gamma}^{-1}(m_0) $, it holds that $ \sup_{\widehat{m}_0 \in \widehat{K}} \lip D\widetilde{\Psi}_{\widehat{m}_0}(\cdot)|_{X^{cs}_{\phi(\widehat{m}_0)}(e_0\tilde{\eta})^{\complement}} < \infty $;

		\item $ \widehat{F}^{cs}_{m}(\cdot), \widehat{G}^{cs}_{m}(\cdot) $ are $ C^{1,\theta} $ uniformly for $ m \in K $, i.e., there is a constant $ C_0 > 0 $ such that for all $ m \in K $ and $ z_1, z_2 \in X^{cs}_{m} (r_0) \oplus X^{u}_{u(m)} (r_0) $, 
		\[
		|D\widehat{F}^{cs}_{m}(z_1) - D\widehat{F}^{cs}_{m}(z_2)| \leq C_0|z_1 - z_2|^{\theta}, \quad |D\widehat{G}^{cs}_{m}(z_1) - D\widehat{G}^{cs}_{m}(z_2)| \leq C_0|z_1 - z_2|^{\theta};
		\]

		\item (Spectral gap condition) $ \sup_{m \in K} \lambda^{\theta}_{cs}(m) \lambda_{cs}(m) \lambda_{u}(m) \vartheta_1(m) < 1 $, where $ \vartheta_1(\cdot) $ is defined by \eqref{equ:qqq}.
	\end{enumerate}
	Then $ W^{cs}_{loc}(K) \in C^{1,\theta} $, i.e., there is a constant $ C_1 > 0 $ such that for all $ \widehat{m}_0 \in \widehat{K} $ and $ x_1, x_2 \in X^{cs}_{\phi(\widehat{m}_0)} (\sigma_{0}) $, one has $ |Df^{0}_{\widehat{m}_0}(x_1) - Df^{0}_{\widehat{m}_0}(x_2)| \leq C_1|x_1 - x_2|^{\theta} $.
\end{lem}

\begin{proof}
	Let $ \hol_{\theta} h $ denote the $ \theta $-H\"older constant of $ h $ when $ h $ is a map between two metric spaces.

	Let $ K^1 = 0 $. Then $ K^1 \in E_{K} $. Let $ K^{n+1} = (\varGamma_0)^{n} K^1 $ for $ n \geq 1 $. Take $ \widehat{m}_0 \in \widehat{K} $, $ m_0 = \phi(\widehat{m}_0) $. We have $ R^{n}_{\widehat{m}_0}(x^{cs}) \in L(X^{cs}_{m_0}, X^{cs}_{u(m_0)}) $ and $ \widetilde{R}^{0}_{\widehat{m}_0}(x^{cs}) \in L(X^{cs}_{m_0}, X^{cs}_{m_0}) $ such that the following equations hold: for $x^{cs} \in X^{cs}_{m_0}(\sigma^c_{*})$,
	\begin{equation*}\label{equ:vlocala} \tag{\S}
	\begin{cases}
	D\widehat{F}^{cs}_{m_0} ( x^{cs}, f^0_{\widehat{u}(\widehat{m}_0)} ( x_{\widehat{m}_0} (x^{cs}) ) ) ( \id, K^{n}_{\widehat{u}(\widehat{m}_0)}(x_{\widehat{m}_0} (x^{cs})) R^{n}_{\widehat{m}_0} (x^{cs}) ) = R^{n}_{\widehat{m}_0} (x^{cs}), \\
	D\widehat{G}^{cs}_{m_0} ( x^{cs}, f^0_{\widehat{u}(\widehat{m}_0)} ( x_{\widehat{m}_0} (x^{cs}) ) ) ( \id, K^{n}_{\widehat{u}(\widehat{m}_0)}(x_{\widehat{m}_0} (x^{cs})) R^{n}_{\widehat{m}_0} (x^{cs}) ) \triangleq \widetilde{K}^n_{\widehat{m}_0} (x^{cs}),
	\end{cases}
	\end{equation*}
	and for $x^{cs} \in X^{cs}_{m_0}(e_0\eta'_1)$,
	\begin{equation*}\label{equ:vexa} \tag{\S\S}
	\begin{cases}
	DF^1_{\widehat{m}_0} (x^{cs}, \widetilde{f}^0_{\widehat{m}_0}(\widetilde{x}_{\widehat{m}_0}(x^{cs}))) ( \id, \widetilde{K}^n_{\widehat{m}_0} (\widetilde{x}_{\widehat{m}_0}(x^{cs})) \widetilde{R}^{n}_{\widehat{m}_0}(x^{cs}) ) = \widetilde{R}^{n}_{\widehat{m}_0}(x^{cs}),\\
	DG^1_{\widehat{m}_0} (x^{cs}, \widetilde{f}^0_{\widehat{m}_0}(\widetilde{x}_{\widehat{m}_0}(x^{cs}))) ( \id, \widetilde{K}^n_{\widehat{m}_0} (\widetilde{x}_{\widehat{m}_0}(x^{cs})) \widetilde{R}^{n}_{\widehat{m}_0}(x^{cs}) ) \triangleq K^{n+1}_{\widehat{m}_0} (x^{cs}).
	\end{cases}
	\end{equation*}

	In what follows, the constant $ \widetilde{C} > 0 $ is independent of $ \widehat{m}_0 \in \widehat{K} $ and $ n \in \mathbb{N} $ but may vary from line to line. Note that $ DF^1_{\widehat{m}_0}(\cdot), DG^1_{\widehat{m}_0}(\cdot) \in C^{0,1} $ uniformly for $ \widehat{m}_0 $ by (a)(b).
	We use induction to show that $ \sup_{\widehat{m}_0 \in \widehat{K}} \hol_{\theta} K^{n}_{\widehat{m}_0}(\cdot) \leq C_* $, where $ C_* > 0 $ is independent of $ n $ and will be chosen later (see \eqref{equ:c*}). Clearly, $ \sup_{\widehat{m}_0 \in \widehat{K}} \hol_{\theta} K^{1}_{\widehat{m}_0}(\cdot) = 0 < C_* $. Assume the result holds for $ n $ and consider the case $ n + 1 $.
	For $ x_1, x_2 \in X^{cs}_{\widehat{m}_0} (\sigma^c_*) $, by \eqref{equ:vlocala}, we see
	\begin{align*}
	&~ |\widetilde{K}^n_{\widehat{m}_0}(x_1) - \widetilde{K}^n_{\widehat{m}_0}(x_2)| \\
	= &~ \left| D\widehat{G}^{cs}_{m_0} ( x_1, f^0_{\widehat{u}(\widehat{m}_0)} ( x_{\widehat{m}_0} (x_1) ) ) ( \id, K^{n}_{\widehat{u}(\widehat{m}_0)}(x_{\widehat{m}_0} (x_1)) R^{n}_{\widehat{m}_0} (x_1) ) \right. \\
	& \quad - \left. D\widehat{G}^{cs}_{m_0} ( x_2, f^0_{\widehat{u}(\widehat{m}_0)} ( x_{\widehat{m}_0} (x_2) ) ) ( \id, K^{n}_{\widehat{u}(\widehat{m}_0)}(x_{\widehat{m}_0} (x_2)) R^{n}_{\widehat{m}_0} (x_2) ) \right| \\
	\leq &~ \left| \left\{D\widehat{G}^{cs}_{m_0} ( x_1, f^0_{\widehat{u}(\widehat{m}_0)} ( x_{\widehat{m}_0} (x_1) ) ) - D\widehat{G}^{cs}_{m_0} ( x_2, f^0_{\widehat{u}(\widehat{m}_0)} ( x_{\widehat{m}_0} (x_2) ) )\right\} \right. \\
	& \quad \quad \quad \left. ( \id, K^{n}_{\widehat{u}(\widehat{m}_0)}(x_{\widehat{m}_0} (x_1)) R^{n}_{\widehat{m}_0} (x_1) ) \right| \\
	& \quad + \left| D_{x^{u}}\widehat{G}^{cs}_{m_0} ( x_2, f^0_{\widehat{u}(\widehat{m}_0)} ( x_{\widehat{m}_0} (x_2) ) ) \right. \\
	& \quad \quad \quad \left. \left\{ K^{n}_{\widehat{u}(\widehat{m}_0)}(x_{\widehat{m}_0} (x_1)) R^{n}_{\widehat{m}_0} (x_1) - K^{n}_{\widehat{u}(\widehat{m}_0)}(x_{\widehat{m}_0} (x_2)) R^{n}_{\widehat{m}_0} (x_2)\right\} \right| \\
	\leq &~ \widetilde{C}|x_1 - x_2|^{\theta} + \lambda_{u}(m_0) \left| K^{n}_{\widehat{u}(\widehat{m}_0)}(x_{\widehat{m}_0} (x_1)) R^{n}_{\widehat{m}_0} (x_1) - K^{n}_{\widehat{u}(\widehat{m}_0)}(x_{\widehat{m}_0} (x_2)) R^{n}_{\widehat{m}_0} (x_2) \right| \\
	\leq &~ \widetilde{C}|x_1 - x_2|^{\theta} + \lambda_{cs}(m_0) \lambda_{u}(m_0) \left| K^{n}_{\widehat{u}(\widehat{m}_0)}(x_{\widehat{m}_0} (x_1)) - K^{n}_{\widehat{u}(\widehat{m}_0)}(x_{\widehat{m}_0} (x_2)) \right| \\
	& \quad + \lambda_{u}(m_0) \mu_1(\widehat{u}(\widehat{m}_0)) \left| R^{n}_{\widehat{m}_0} (x_1) - R^{n}_{\widehat{m}_0} (x_2) \right|,
	\end{align*}
	and analogously,
	\begin{multline*}
	|R^n_{\widehat{m}_0}(x_1) - R^n_{\widehat{m}_0}(x_2)| \\
	 \leq \widetilde{C}|x_1 - x_2|^{\theta} + \alpha(m_0) \lambda_{cs}(m_0) \left| K^{n}_{\widehat{u}(\widehat{m}_0)}(x_{\widehat{m}_0} (x_1)) - K^{n}_{\widehat{u}(\widehat{m}_0)}(x_{\widehat{m}_0} (x_2)) \right| \\
	+ \alpha(m_0) \mu_1(\widehat{u}(\widehat{m}_0)) \left| R^{n}_{\widehat{m}_0} (x_1) - R^{n}_{\widehat{m}_0} (x_2) \right|,
	\end{multline*}
	which yields
	\begin{multline*}
	|\widetilde{K}^n_{\widehat{m}_0}(x_1) - \widetilde{K}^n_{\widehat{m}_0}(x_2)| \\
	\leq \widetilde{C}|x_1 - x_2|^{\theta} + \frac{\lambda_{cs}(m_0) \lambda_{u}(m_0)}{1 - \alpha(m_0) \mu_1(\widehat{u}(\widehat{m}_0))} \left| K^{n}_{\widehat{u}(\widehat{m}_0)}(x_{\widehat{m}_0} (x_1)) - K^{n}_{\widehat{u}(\widehat{m}_0)}(x_{\widehat{m}_0} (x_2)) \right|.
	\end{multline*}
	For $ x_1, x_2 \in X^{cs}_{\widehat{m}_0} (e_0\eta'_1) $, applying the same computation to \eqref{equ:vexa} gives
	\[
	|K^{n+1}_{\widehat{m}_0}(x_1) - K^{n+1}_{\widehat{m}_0}(x_2)| \leq \widetilde{C}|x_1 - x_2|^{\theta} + \varpi^*_1 \left| \widetilde{K}^{n}_{\widehat{u}(\widehat{m}_0)}(\widetilde{x}_{\widehat{m}_0} (x_1)) - \widetilde{K}^{n}_{\widehat{u}(\widehat{m}_0)}(\widetilde{x}_{\widehat{m}_0} (x_2)) \right|.
	\]
	Thus, for $ x_1, x_2 \in X^{cs}_{\widehat{m}_0} (e_0\eta'_1) $, we get
	\[
	|K^{n+1}_{\widehat{m}_0}(x_1) - K^{n+1}_{\widehat{m}_0}(x_2)| \leq \widetilde{C}_*|x_1 - x_2|^{\theta} \\
	+ \theta_1(m_0) \left| K^{n}_{\widehat{u}(\widehat{m}_0)}(\widehat{x}_{\widehat{m}_0}(x_1)) - K^{n}_{\widehat{u}(\widehat{m}_0)}(\widehat{x}_{\widehat{m}_0}(x_2)) \right|,
	\]
	where $ \widetilde{C}_* > 0 $ is a constant independent of $ \widehat{m}_0 \in \widehat{K} $ and $ n \in \mathbb{N} $, $ \widehat{x}_{\widehat{m}_0}(x) = x_{\widehat{m}_0} (\widetilde{x}_{\widehat{m}_0} (x)) $, and $ \theta_1(m_0) = \varpi^*_1 \frac{\lambda_{cs}(m_0) \lambda_{u}(m_0)}{1 - \alpha(m_0) \mu_1(\widehat{u}(\widehat{m}_0))} $. Note that
	\[
	\lip \widehat{x}_{\widehat{m}_0}(\cdot)|_{X^{cs}_{\phi(\widehat{m}_0)} (e_0\eta'_1)} \leq \varpi^*_1 \lambda_{cs}(m_0) \triangleq \nu(m_0).
	\]
	Take $ C_* > 0 $ such that
	\begin{equation}\label{equ:c*}
	C_* \geq \max_{t_0 \in [0,1]} \left\{ \frac{\widetilde{C}_*(1 - t_0)^{\theta} + C'_1t^{\theta}_0}{1 - \rho(1 - t_0)^{\theta}} \right\},
	\end{equation}
	where 
	\[
	C'_1 = \sup_{\widehat{m}_0 \in \widehat{K}} \lip Dg^{cs}_{\widehat{m}_0}(\cdot)|_{X^{cs}_{\phi(\widehat{m}_0)}(e_0\tilde{\eta})^{\complement}} \cdot (2\sigma_{0})^{1 - \theta}, \quad \text{and} \quad \rho = \sup_{m_0 \in K}(\nu(m_0))^{\theta} \theta_1(m_0) < 1.
	\]
	Then
	\[
	\sup_{x_1, x_2 \in X^{cs}_{\widehat{m}_0} (e_0\eta'_1)} \frac{|K^{n+1}_{\widehat{m}_0}(x_1) - K^{n+1}_{\widehat{m}_0}(x_2)|}{|x_1 - x_2|^{\theta}} \leq \widetilde{C}_{*} + (\nu(m_0))^{\theta} \theta_1(m_0) C_* \leq C_*.
	\]

	If $ x_1, x_2 \in X^{cs}_{\phi(\widehat{m}_0)} (\sigma_{0}) \setminus X^{cs}_{\phi(\widehat{m}_0)} (e_0\eta'_1) $, then $ K^{n+1}_{\widehat{m}_0}(x_i) = g^{cs}_{\widehat{m}_0}(x_i) $ for $ n \geq 1 $, and by $ X^s_{\widehat{K}_{\epsilon_{*}}}(\sigma_{*}) \setminus X^s_{\widehat{K}_{\tilde{\eta}}}(\tilde{\eta}) \in C^{1, 1} $, we know
	\[
	|K^{n+1}_{\widehat{m}_0}(x_1) -K^{n+1}_{\widehat{m}_0}(x_2) | \leq \sup_{\widehat{m}_0 \in \widehat{K}} \lip Dg^{cs}_{\widehat{m}_0}(\cdot)|_{X^{cs}_{\phi(\widehat{m}_0)}(e_0\tilde{\eta})^{\complement}} |x_1 - x_2| \leq C_* |x_1 - x_2|^{\theta}.
	\]
	If $ x_1 \in X^{cs}_{\phi(\widehat{m}_0)} (e_0\eta'_1) $ and $ x_2 \in X^{cs}_{\phi(\widehat{m}_0)} (\sigma_{0}) \setminus X^{cs}_{\phi(\widehat{m}_0)} (e_0\eta'_1) $, then there is $ 0 < t_0 < 1 $ such that for $ x'_2 = t_0x_1 + (1 - t_0)x_2 $, we have $ |x'_2| < e_0\eta'_1 $ and $ K^{n+1}_{\widehat{m}_0}(x'_2) = g^{cs}_{\widehat{m}_0}(x'_2) $, which gives
	\begin{align*}
	|K^{n+1}_{\widehat{m}_0}(x_1) -K^{n+1}_{\widehat{m}_0}(x_2) | & \leq |K^{n+1}_{\widehat{m}_0}(x_1) -K^{n+1}_{\widehat{m}_0}(x'_2)| + |Dg^0_{\widehat{m}_0}(x'_2) -Dg^0_{\widehat{m}_0}(x_2)| \\
	& \leq (\widetilde{C}_{*} + \rho C_*) |x_1 - x'_2|^{\theta} + C'_1 |x'_2 - x_2|^{\theta} \\
	& \leq (\widetilde{C}_{*} + \rho C_*)(1 - t_0)^{\theta}|x_1 - x_2|^{\theta} + C'_1t^{\theta}_0 |x_1 - x_2|^{\theta} \\
	& \leq C_* |x_1 - x_2|^{\theta}.
	\end{align*}
	As $ K^{n} \to K^0 $ in the metric $ d_{K} $, one gets $ \sup_{\widehat{m}_0 \in \widehat{K}} \hol_{\theta} K^0_{\widehat{m}_0}(\cdot) < \infty $. The proof is complete.
\end{proof}

\begin{lem}[Uniform differential]\label{lem:udiff}
	Let $ \tilde{\eta} $ and $ X^{cs}_{m_0}(e_0\tilde{\eta})^{\complement} $ be as given in \autoref{lem:C1a}.
	Suppose that $ X^s_{\widehat{K}_{\epsilon_{*}}}(\sigma_{*}) \setminus X^s_{\widehat{K}_{\tilde{\eta}}}(\tilde{\eta}) \in C^{1, u} $, i.e., $ Dg^{cs}_{\widehat{m}_0}(\cdot)|_{X^{cs}_{\phi(\widehat{m}_0)}(e_0\tilde{\eta})^{\complement}} $ (see \eqref{equ:cslocal}), $ \widehat{m} \in \widehat{K} $, are equicontinuous, and that $ \Psi \in C^{1,u} $ outside $ X^s_{\widehat{K}_{\tilde{\eta}}}(\tilde{\eta}) $, i.e., $ D\widetilde{\Psi}_{\widehat{m}_0}(\cdot)|_{X^{cs}_{\phi(\widehat{m}_0)}(e_0\tilde{\eta})^{\complement}} $ (see \eqref{equ:localbump}), $ \widehat{m} \in \widehat{K} $, are equicontinuous.
	If $ D\widehat{F}^{cs}_{m}(\cdot), D\widehat{G}^{cs}_{m}(\cdot) $, $ m \in K $, are equicontinuous, then $ W^{cs}_{loc}(K) \in C^{1, u} $, i.e., $ Df^{0}_{\widehat{m}_0}(\cdot) $, $ \widehat{m}_0 \in \widehat{K} $, are equicontinuous.
\end{lem}

\begin{proof}
	We need to show $ \mathcal{L}_{u} |K^0_{\widehat{m}_0}(x_1) - K^{0}_{\widehat{m}_0}(x_2)| = 0 $, where $ \mathcal{L}_{u} $ denotes the limit:
	\[
	\mathcal{L}_{u} \triangleq \lim_{r \to 0^+} \sup_{\widehat{m}_0 \in \widehat{K}} \sup_{ \substack{|x_1 - x_2 | \leq r \\ x_1, x_2 \in X^{cs}_{\phi(\widehat{m}_0)} (\sigma_{0}) } }.
	\]
	The proof is very similar to that of \autoref{lem:C1a}, and we provide the details for the reader's convenience. We use the notation $ |h_{\widehat{m}_0}(x_1) - h_{\widehat{m}_0}(x_2)| \leq o_{u}(1) $ if $ \mathcal{L}_{u}|h_{\widehat{m}_0}(x_1) - h_{\widehat{m}_0}(x_2)| = 0 $. Note that $ DF^1_{\widehat{m}_0}(\cdot), DG^1_{\widehat{m}_0}(\cdot) $, $ \widehat{m}_0 \in \widehat{K} $, are equicontinuous by the lemma's conditions.
	For $ x_1, x_2 \in X^{cs}_{\widehat{m}_0} (\sigma^c_*) $, by \eqref{equ:vlocal}, we see
	\begin{align*}
	&~ |\widetilde{K}^0_{\widehat{m}_0}(x_1) - \widetilde{K}^0_{\widehat{m}_0}(x_2)| \\
	= &~ \left| D\widehat{G}^{cs}_{m_0} ( x_1, f^0_{\widehat{u}(\widehat{m}_0)} ( x_{\widehat{m}_0} (x_1) ) ) ( \id, K^{0}_{\widehat{u}(\widehat{m}_0)}(x_{\widehat{m}_0} (x_1)) R^{0}_{\widehat{m}_0} (x_1) ) \right. \\
	& \quad - \left. D\widehat{G}^{cs}_{m_0} ( x_2, f^0_{\widehat{u}(\widehat{m}_0)} ( x_{\widehat{m}_0} (x_2) ) ) ( \id, K^{0}_{\widehat{u}(\widehat{m}_0)}(x_{\widehat{m}_0} (x_2)) R^{0}_{\widehat{m}_0} (x_2) ) \right| \\
	\leq &~ \left| \left\{D\widehat{G}^{cs}_{m_0} ( x_1, f^0_{\widehat{u}(\widehat{m}_0)} ( x_{\widehat{m}_0} (x_1) ) ) - D\widehat{G}^{cs}_{m_0} ( x_2, f^0_{\widehat{u}(\widehat{m}_0)} ( x_{\widehat{m}_0} (x_2) ) )\right\} \right. \\
	& \quad \quad \quad \left. ( \id, K^{0}_{\widehat{u}(\widehat{m}_0)}(x_{\widehat{m}_0} (x_1)) R^{0}_{\widehat{m}_0} (x_1) ) \right| \\
	& \quad + \left| D_{x^{u}}\widehat{G}^{cs}_{m_0} ( x_2, f^0_{\widehat{u}(\widehat{m}_0)} ( x_{\widehat{m}_0} (x_2) ) ) \right. \\
	& \quad \quad \quad \left. \left\{ K^{0}_{\widehat{u}(\widehat{m}_0)}(x_{\widehat{m}_0} (x_1)) R^{0}_{\widehat{m}_0} (x_1) - K^{0}_{\widehat{u}(\widehat{m}_0)}(x_{\widehat{m}_0} (x_2)) R^{0}_{\widehat{m}_0} (x_2)\right\} \right| \\
	\leq &~ o_{u}(1) + \lambda_{u}(m_0) \left| K^{0}_{\widehat{u}(\widehat{m}_0)}(x_{\widehat{m}_0} (x_1)) R^{0}_{\widehat{m}_0} (x_1) - K^{0}_{\widehat{u}(\widehat{m}_0)}(x_{\widehat{m}_0} (x_2)) R^{0}_{\widehat{m}_0} (x_2) \right| \\
	\leq &~ o_{u}(1) + \lambda_{cs}(m_0) \lambda_{u}(m_0) \left| K^{0}_{\widehat{u}(\widehat{m}_0)}(x_{\widehat{m}_0} (x_1)) - K^{0}_{\widehat{u}(\widehat{m}_0)}(x_{\widehat{m}_0} (x_2)) \right| \\
	& \quad + \lambda_{u}(m_0) \mu_1(\widehat{u}(\widehat{m}_0)) \left| R^{0}_{\widehat{m}_0} (x_1) - R^{0}_{\widehat{m}_0} (x_2) \right|,
	\end{align*}
	and
	\begin{multline*}
	|R^0_{\widehat{m}_0}(x_1) - R^0_{\widehat{m}_0}(x_2)| \leq o_{u}(1) + \alpha(m_0) \lambda_{cs}(m_0) \left| K^{0}_{\widehat{u}(\widehat{m}_0)}(x_{\widehat{m}_0} (x_1)) - K^{0}_{\widehat{u}(\widehat{m}_0)}(x_{\widehat{m}_0} (x_2)) \right| \\
	+ \alpha(m_0) \mu_1(\widehat{u}(\widehat{m}_0)) \left| R^{0}_{\widehat{m}_0} (x_1) - R^{0}_{\widehat{m}_0} (x_2) \right|,
	\end{multline*}
	which yields
	\begin{multline*}
	|\widetilde{K}^0_{\widehat{m}_0}(x_1) - \widetilde{K}^0_{\widehat{m}_0}(x_2)| \\
	\leq o_{u}(1) + \frac{\lambda_{cs}(m_0) \lambda_{u}(m_0)}{1 - \alpha(m_0) \mu_1(\widehat{u}(\widehat{m}_0))} \left| K^{0}_{\widehat{u}(\widehat{m}_0)}(x_{\widehat{m}_0} (x_1)) - K^{0}_{\widehat{u}(\widehat{m}_0)}(x_{\widehat{m}_0} (x_2)) \right|.
	\end{multline*}
	For $ x_1, x_2 \in X^{cs}_{\widehat{m}_0} (e_0\eta'_1) $, from \eqref{equ:vex} we also get
	\[
	|K^0_{\widehat{m}_0}(x_1) - K^0_{\widehat{m}_0}(x_2)| \leq o_{u}(1) + \varpi^*_1 \left| \widetilde{K}^{0}_{\widehat{u}(\widehat{m}_0)}(\widetilde{x}_{\widehat{m}_0} (x_1)) - \widetilde{K}^{0}_{\widehat{u}(\widehat{m}_0)}(\widetilde{x}_{\widehat{m}_0} (x_2)) \right|.
	\]
	Thus, for $ x_1, x_2 \in X^{cs}_{\widehat{m}_0} (e_0\eta'_1) $, we get
	\[
	|K^0_{\widehat{m}_0}(x_1) - K^0_{\widehat{m}_0}(x_2)| \leq o_{u}(1) \\
	+ \theta_1 \left| K^{0}_{\widehat{u}(\widehat{m}_0)}(\widehat{x}_{\widehat{m}_0}(x_1)) - K^{0}_{\widehat{u}(\widehat{m}_0)}(\widehat{x}_{\widehat{m}_0}(x_2)) \right|,
	\]
	where $ \widehat{x}_{\widehat{m}_0}(x) = x_{\widehat{m}_0} (\widetilde{x}_{\widehat{m}_0} (x)) $ and $ \theta_1 = \varpi^*_1 \sup_{\widehat{m}_0 \in \widehat{K}} \frac{ \lambda_{cs}(\phi(\widehat{m}_0)) \lambda_{u}(\phi(\widehat{m}_0)) }{ 1 - \alpha(\phi(\widehat{m}_0)) \mu_1(\widehat{u}(\widehat{m}_0)) } < 1 $. Finally, we obtain
	\begin{align*}
	\mathcal{L}_{u}|K^0_{\widehat{m}_0}(x_1) - K^0_{\widehat{m}_0}(x_2)| \leq \theta_1 \mathcal{L}_{u}|K^0_{\widehat{m}_0}(x_1) - K^0_{\widehat{m}_0}(x_2)| < \infty,
	\end{align*}
	yielding $ \mathcal{L}_{u}|K^0_{\widehat{m}_0}(x_1) - K^0_{\widehat{m}_0}(x_2)| = 0 $. The proof is complete.
\end{proof}

\section{Higher order smoothness of the center-stable manifold}\label{sub:highersmooth}
\begin{lem}[$ C^{k} $ smoothness]\label{lem:high}
	Assume the following conditions hold, where $ \tilde{\eta}, X^{cs}_{m_0}(e_0\tilde{\eta})^{\complement} $ are as given in \autoref{lem:C1a}, and $ k \in \mathbb{N} $.
	\begin{enumerate}[(a)]
		\item $ X^s_{\widehat{K}_{\epsilon_{*}}}(\sigma_{*}) \setminus X^s_{\widehat{K}_{\tilde{\eta}}}(\tilde{\eta}) \in C^{k-1, 1} \cap C^{k} $, i.e., for the local representations $ g^{cs}_{\widehat{m}_0}(\cdot) $ (see \eqref{equ:cslocal}) of $X^s_{\widehat{K}_{\epsilon_{*}}}(\sigma_{*})$ at $ \widehat{m}_0 \in \widehat{K} $, $ g^{cs}_{\widehat{m}_0}(\cdot)|_{X^{cs}_{\phi(\widehat{m}_0)}(e_0\tilde{\eta})^{\complement}} \in C^k$ and
		\[
		\sup_{\widehat{m}_0 \in \widehat{K}} \lip D^{i}g^{cs}_{\widehat{m}_0}(\cdot)|_{X^{cs}_{\phi(\widehat{m}_0)}(e_0\tilde{\eta})^{\complement}} < \infty, \quad i = 1, 2, \ldots, k-1;
		\]

		\item $ \Psi \in C^{k-1,1} \cap C^{k} $ outside $ X^s_{\widehat{K}_{\tilde{\eta}}}(\tilde{\eta}) $, i.e., for the local representations $ \widetilde{\Psi}_{\widehat{m}_0}(\cdot) $ (see \eqref{equ:localbump}), $ \widehat{m}_0 \in \widehat{K} $, $ \widetilde{\Psi}_{\widehat{m}_0}(\cdot)|_{X^{cs}_{\phi(\widehat{m}_0)}(e_0\tilde{\eta})^{\complement}} \in C^{k} $ and $ \sup_{\widehat{m}_0 \in \widehat{K}} \lip D^i\widetilde{\Psi}_{\widehat{m}_0}(\cdot)|_{X^{cs}_{\phi(\widehat{m}_0)}(e_0\tilde{\eta})^{\complement}} < \infty $, $ i = 1, 2, \ldots, k-1 $;

		\item $ \widehat{F}^{cs}_{m}(\cdot), \widehat{G}^{cs}_{m}(\cdot) $ are $ C^{k-1,1} \cap C^{k} $ uniformly for $ m \in K $, specifically
		\[
		\max\left\{\sup_{m \in {K}} \lip D^{i}\widehat{F}^{cs}_{m}(\cdot), \sup_{m \in {K}} \lip D^{i}\widehat{G}^{cs}_{m}(\cdot): i = 1, 2, \ldots, k-1\right\} < \infty;
		\]

		\item (Spectral gap condition) $ \sup_{m \in K} \lambda^{k}_{cs}(m) \lambda_{u}(m) \vartheta_1(m) < 1 $, where $ \vartheta_1(\cdot) $ is defined by \eqref{equ:qqq}.
	\end{enumerate}
	Then $ W^{cs}_{loc}(K) \in C^{k-1,1} \cap C^{k} $, i.e., $ f^{0}_{\widehat{m}_0}(\cdot) \in C^{k} $ and $ \sup_{\widehat{m}_0 \in \widehat{K}} \lip D^{i}f^{0}_{\widehat{m}_0}(\cdot) < \infty $, $ i = 1, 2, \ldots, k-1 $.
\end{lem}
\begin{proof}
	For brevity, we use the notation 
	\[
	L^{k}_{h}(Z_1, Z_2) \triangleq L_{h}(\underbrace{Z_1 \times \cdots \times Z_1}_{k}, Z_2),
	\]
	where $ Z_i $ are vector bundles over $ M_{i} $, $ i = 1,2 $, and $ h: M_1 \to M_2 $ is a map. Here, $ L_{h}(Z_1, Z_2) $ denotes the space of all vector bundle maps from $ Z_1 $ to $ Z_2 $ over $ h $.

	First, note that by (a) and (b), $ DF^1_{\widehat{m}_0}(\cdot), DG^1_{\widehat{m}_0}(\cdot) \in C^{k-1,1} \cap C^{k} $ outside $ X^{cs}_{\phi(\widehat{m}_0)} (e_0\tilde{\eta}) $ uniformly for $ \widehat{m}_0 \in \widehat{K} $. Moreover, we have $ \sup_{m \in K} \lambda^{i}_{cs}(m) \lambda_{cs}(m) \lambda_{u}(m) \vartheta_1(m) < 1 $, $ i = 1,2,\ldots,k-1 $. The proof proceeds by induction and follows essentially the same approach as \autoref{lem:C1smooth} and \autoref{lem:C1a}. We provide a sketch below.

	As before, if $ D^{k}f^{0}_{\widehat{m}_0}(x^{cs}) = K^{[k]}_{\widehat{m}_0}(x^{cs}) $ exists, then it must satisfy the following ``variant equations'' obtained by taking the $ k $th order derivative of \eqref{equ:local0} and \eqref{equ:extension} with respect to $ x^{cs} $:
	\[\tag{\dag}\label{equ:vv1}
	\begin{cases}
	\begin{split}
	W^k_{1,\widehat{m}_0}(x^{cs}) + D_{x^{u}}\widehat{F}_{m_0}( x^{cs}, x^{u} )  & \left\{   K^{[k]}_{\widehat{u}(\widehat{m}_0)}  (x_{\widehat{m}_0}(x^{cs})) (R^0_{\widehat{m}_0})^k(x^{cs}) \right.  \\
	&\quad ~ \left. + K^0_{\widehat{u}(\widehat{m}_0)} (x_{\widehat{m}_0}(x^{cs})) R^{[k]}_{\widehat{m}_0}(x^{cs})  \right\} = R^{[k]}_{\widehat{m}_0}(x^{cs}),
	\end{split}\\
	\begin{split}
	W^k_{1,\widehat{m}_0}(x^{cs}) + D_{x^{u}}\widehat{G}_{m_0}( x^{cs}, x^{u} ) & \left\{ K^{[k]}_{\widehat{u}(\widehat{m}_0)} \right.  (x_{\widehat{m}_0}(x^{cs})) (R^0_{\widehat{m}_0})^k(x^{cs}) \\
	&\quad ~ \left. + K^0_{\widehat{u}(\widehat{m}_0)} (x_{\widehat{m}_0}(x^{cs})) R^{[k]}_{\widehat{m}_0}(x^{cs})  \right\} \triangleq \widetilde{K}^{[k]}_{\widehat{m}_0}(x^{cs}),
	\end{split}
	\end{cases}
	\]
	where $ x^{cs} \in X^{cs}_{m_0}(\sigma^c_{*}) $, $ x^{u} = f^0_{\widehat{u}(\widehat{m}_0)} ( x_{\widehat{m}_0} (x^{cs}) ) $, $ \widehat{m}_0 \in \widehat{K} $, and $ m_0 = \phi(\widehat{m}_0) $; and
	\[\tag{\ddag}\label{equ:vv2}
	\begin{cases}
	\begin{split}
	\widetilde{W}^k_{1,\widehat{m}_0}(x^{cs}) + D_{\widetilde{x}^{u}}{F}^{1}_{\widehat{m}_0}( x^{cs}, \widetilde{x}^{u} ) & \left\{ \widetilde{K}^{[k]}_{\widehat{m}_0} \right.  (\widetilde{x}_{\widehat{m}_0}(x^{cs})) (\widetilde{R}^0_{\widehat{m}_0})^k(x^{cs})  \\
	&\quad ~ \left. + \widetilde{K}^0_{\widehat{m}_0} (\widetilde{x}_{\widehat{m}_0}(x^{cs})) \widetilde{R}^{[k]}_{\widehat{m}_0}(x^{cs})  \right\} = \widetilde{R}^{[k]}_{\widehat{m}_0}(x^{cs}),
	\end{split}\\
	\begin{split}
	\widetilde{W}^k_{1,\widehat{m}_0}(x^{cs}) + D_{\widetilde{x}^{u}}{G}^1_{\widehat{m}_0}( x^{cs}, \widetilde{x}^{u} ) & \left\{ \widetilde{K}^{[k]}_{\widehat{m}_0} \right.  (\widetilde{x}_{\widehat{m}_0}(x^{cs})) (\widetilde{R}^0_{\widehat{m}_0})^k(x^{cs}) \\
	&\quad ~ \left. + \widetilde{K}^0_{\widehat{m}_0} (\widetilde{x}_{\widehat{m}_0}(x^{cs})) \widetilde{R}^{[k]}_{\widehat{m}_0}(x^{cs})  \right\} = K^{[k]}_{\widehat{m}_0}(x^{cs}),
	\end{split}
	\end{cases}
	\]
	where $ x^{cs} \in X^{cs}_{m_0}(e_0\eta'_1) $, $ \widetilde{x}^{u} = \widetilde{f}^0_{\widehat{m}_0} ( \widetilde{x}_{\widehat{m}_0} (x^{cs}) ) $. Here,
	\begin{enumerate}[(1)]
		\item $ (R^0_m)^k(x) = (R^0_m(x), \cdots, R^0_m(x)) $ ($ k $ components), $ (\widetilde{R}^0_m)^k(x) = (\widetilde{R}^0_m(x), \cdots, \widetilde{R}^0_m(x)) $ ($ k $ components), $ R^{[k]}_{\widehat{m}_0}(x^{cs}) \in L^{k}(X^{cs}_{m_0}, X^{cs}_{u(m_0)}) $, $ \widetilde{R}^{[k]}_{\widehat{m}_0}(x^{cs}) \in L^{k}(X^{cs}_{m_0}, X^{cs}_{m_0}) $, and $ K^{[k]} \in L^{k}_{f^0} (\Upsilon^{cs}, \Upsilon^{u}) $ (we write $ K^{[k]}_{\widehat{m}_0}(x^{cs}) = K^{[k]}(\widehat{m}_0, x^{cs}) $);

		\item $ W^k_{i,{\widehat{m}_0}} $, $ i = 1,2 $, consist of finite sums of terms that can be explicitly calculated with the help of the Fa\`{a} di Bruno formula (see e.g. \cite{FdlLM06}); the non-constant factors include $ D^j \widehat{F}_{m_0} $, $ D^j \widehat{G}_{m_0} $ ($ 1 \leq j \leq k $), $ D^j f^{0}_{\widehat{m}_0} $ ($ 1 \leq j < k $), and $ D^j x_{\widehat{m}_0} $ ($ 1 \leq j < k $);

		\item Similar for $ \widetilde{W}^k_{i,{\widehat{m}_0}} $, $ i = 1,2 $.
	\end{enumerate}
	The proof proceeds in the following steps. We already know that the lemma holds for $ k = 1 $ by \autoref{lem:C1smooth}. Now assume it holds for $ k - 1 $ and consider the case $ k \geq 2 $.

	(I) The first step is to show $ \sup_{\widehat{m}_0} \lip K^{[k-1]}_{\widehat{m}_0}(\cdot) < \infty $, which follows essentially the same argument as in \autoref{lem:C1a}.

	(II) Following the same procedure as in \autoref{sub:pre}, we construct a graph transform and find a unique $ K^{[k]} \in L^{k}_{f^0} (\Upsilon^{cs}, \Upsilon^{u}) $ satisfying \eqref{equ:vv1} and \eqref{equ:vv2}, such that 
	\[
	\sup_{\widehat{m}_0} \sup_{x^{cs}} |K^{[k]}_{\widehat{m}_0}(x^{cs})| < \infty,
	\]
	and $ x^{cs} \mapsto K^{[k]}_{\widehat{m}_0}(x^{cs}) $ is continuous. Note the following:
	\begin{enumerate}[(i)]
		\item By induction, the terms $ W^k_{i,\widehat{m}_0} $ and $ \widetilde{W}^k_{i,{\widehat{m}_0}} $, $ i = 1,2 $, are bounded uniformly for $ \widehat{m}_0 \in \widehat{K} $.
		\item This graph transform is Lipschitz with Lipschitz constant less than 
		\[
		\varpi^*_1 \sup_{m \in K} \lambda^{k}_{cs}(m) \lambda_{u}(m) \vartheta_1(m) < 1 
		\]
		in the supremum metric on $ E^{[k]}_{K, C_1} $, where
		\[
		E^{[k]}_{K, C_1} = \{ K^{[k]} \in L^{k,c}_{f^0} (\Upsilon^{cs}, \Upsilon^{u}): \sup_{(\widehat{m}_0, x^{cs})} |K^{[k]}_{\widehat{m}_0}(x^{cs})| \leq C_1, K^{[k]}_{\widehat{m}_0}(\cdot) \in C^0,  \widehat{m}_0 \in \widehat{K} \},
		\]
		for some suitable constant $ C_1 > 0 $.
		\item In fact, the case $ k \geq 2 $ is simpler than $ k = 1 $ due to the separation of the term 
		\[
		K^{[k]}_{\widehat{u}(\widehat{m}_0)} (x_{\widehat{m}_0}(x^{cs})) (R^0_{\widehat{m}_0})^k(x^{cs}) + K^0_{\widehat{u}(\widehat{m}_0)} (x_{\widehat{m}_0}(x^{cs})) R^{[k]}_{\widehat{m}_0}(x^{cs}) ,
		\] 
		so only estimates of $ |D\widehat{F}_{m_0}|, |D\widehat{G}_{m_0}| $ are needed (i.e., the (B) condition is not used here).
	\end{enumerate}

	(III) Finally, using an analogous argument to that for \autoref{lem:C1smooth}, we obtain $ D^{k}f^{0}_{\widehat{m}_0}(x^{cs}) = K^{[k]}_{\widehat{m}_0}(x^{cs}) $ by showing that
	\begin{multline*}
		\limsup_{x' \to x} \frac{|\mathcal{Q}^{[k]}(\widehat{m}_0, x', x)|}{|x' - x|} \\
		\leq \varpi^*_1 \frac{\lambda^{k}_{cs}(m_0)\lambda_{u} (m_0)}{1 - \alpha(m_0) \mu_1(\widehat{u}(\widehat{m}_0))} \sup_{x \in X^{cs}_{u(m_0)}(\sigma_{0})} \limsup_{x' \to x}  \frac{|\mathcal{Q}^{[k]}(\widehat{u}(\widehat{m}_0), x', x)|}{|x' - x|},
	\end{multline*}
	
	where
	\[
	\mathcal{Q}^{[k]}(\widehat{m}_0, x', x) = D^{k-1}f^0_{\widehat{m}_0} (x') - D^{k-1}f^{0}_{\widehat{m}_0} (x) - K^{[k]}_{\widehat{m}_0} (x) (x' - x), \quad x', x \in X^{cs}_{m_0}(\sigma_{0}).
	\]
	This estimate can be established through a computation similar to that of \autoref{slem:aa}, combined with the induction hypothesis for $ k - 1 $. Note that the Lipschitz continuity of $ D^{k-1}f^{0}_{\widehat{m}_0}(\cdot) = K^{[k-1]}_{\widehat{m}_0}(\cdot) $ is essential for obtaining this bound.
\end{proof}

\begin{rmk}
	Under \autoref{lem:high} (a) (b) with (c) replaced by that $ \widehat{F}^{cs}_{m}(\cdot), \widehat{G}^{cs}_{m}(\cdot) $ are $ C^{k - 1, r} $ uniform $ m \in K $ and $ \sup_{m \in K} \lambda^{k - 1 + r}_{cs}(m) \lambda_{u}(m) \vartheta_1(m) < 1 $, then $ W^{cs}_{loc}(K) \in C^{k-1, r} $ (see e.g. \autoref{lem:C1a}), where $ r \in (0,1] $.
\end{rmk}

\begin{rmk}[Limited case]\label{rmk:inflowingSmooth}
	\begin{enumerate}[(a)]
		\item Under assumption ($ \star\star $) in \autoref{sub:limited}, using a $ C^1 $ and $ C_1 $-Lipschitz bump function $ \Psi $ in $ \widehat{\Sigma} $ near $ K $ (i.e., $ \Psi $ satisfies (a) (b) in \autoref{sub:LipC1bm} with $ X^s_{m} = \{0\} $ for all $ m \in \Sigma $), and assuming the same conditions as in \autoref{sub:LipC1bm} hold, we can also choose $ W^{cs}_{loc}(K) \in C^1 $. Moreover, in \autoref{lem:C1a}, \autoref{lem:udiff}, and \autoref{lem:high}, the corresponding conditions (a) and (b) are replaced as follows:
		\begin{itemize}
			\item In (a), let $ X^{c}_{m_0}(e_0\tilde{\eta})^{\complement} \triangleq X^{c}_{m_0}(c_2\epsilon_{*}) \oplus X^{s}_{m_0}(e_2\sigma_{*}) \setminus X^{c}_{\widehat{m}_0} (e_0\tilde{\eta}) $, and replace $ X^s_{\widehat{K}_{\epsilon_{*}}}(\sigma_{*}) \setminus X^s_{\widehat{K}_{\tilde{\eta}}}(\tilde{\eta}) $ by $ X^s_{\widehat{K}_{\epsilon_{*}}}(\sigma_{*}) \setminus X^s_{\widehat{K}_{\tilde{\eta}}}(\sigma_{*}) $;
			\item In (b), let $ X^{c}_{m_0}(e_0\tilde{\eta})^{\complement} \triangleq X^{c}_{m_0}(c_2\epsilon_{*}) \setminus X^{c}_{\widehat{m}_0} (e_0\tilde{\eta}) $, and replace $ X^s_{\widehat{K}_{\epsilon_{*}}}(\sigma_{*}) \setminus X^s_{\widehat{K}_{\tilde{\eta}}}(\tilde{\eta}) $ by $ \widehat{K}_{\epsilon_{*}} \setminus \widehat{K}_{\tilde{\eta}} $.
		\end{itemize}
		Note that in this case, \eqref{equ:local0} holds for all $ x^{cs} = (x^c, x^s) \in X^{c}_{m_0}(\sigma^c_{*}) \oplus X^{s}_{m_0}(e_2\sigma_{*}) $, \eqref{equ:vlocal} is considered for all $ x^{cs} = (x^c, x^s) \in X^{c}_{m_0}(\sigma^c_{*}) \oplus X^{s}_{m_0}(e_2\sigma_{*}) $, and \eqref{equ:vex} is considered for all $ x^{cs} = (x^c, x^s) \in X^{c}_{m_0}(e_0\eta'_1) \oplus X^{s}_{m_0}(e_2\sigma_{*}) $. The proofs are barely changed.

	\item (Strictly inflowing) Under assumption ($ \bullet\bullet $) in \autoref{sub:limited0}, together with (A1)--(A3) and (A4) (i)(ii) with only $ \varsigma_0 \geq 1 $ (in (A3) (a)), if $ \epsilon_{*}, \chi(\epsilon_{*}), \eta $ are small, then the center-stable manifold $ W^{cs}_{loc} (K) $ given in \autoref{rmk:inflowing} \eqref{it:i2} is $ C^1 $. Note that in this case, the regularity results are very similar to those for normal hyperbolicity studied in \cite{Che18b}; no smooth truncation and no smoothness condition on the space $ \Sigma $ are required. In \autoref{lem:C1a}, \autoref{lem:udiff}, and \autoref{lem:high}, conditions (a)(b) are no longer needed.

	The proofs are analogous to the unlimited case, where \eqref{equ:local0} holds for all $ x^{cs} = (x^c, x^s) \in X^{c}_{m_0}(c_2\epsilon_{*}) \oplus X^{s}_{m_0}(e_2\sigma_{*}) $ with $ \widetilde{f}^0_{\widehat{m}_0} = f^0_{\widehat{m}_0} $, \eqref{equ:vlocal} is considered for all $ x^{cs} = (x^c, x^s) \in X^{c}_{m_0}(c_2\epsilon_{*}) \oplus X^{s}_{m_0}(e_2\sigma_{*}) $, and \eqref{equ:vex} is not needed. These results are also direct consequences of \cite[Section 6.3]{Che18a}.

	\item \label{it:CS3} Under Assumption (CS) in \autoref{sub:limited*} with (A2) (a)(ii) and (A4) (i)(ii) in \autoref{subsec:main} (for $ H^{\delta} $), the manifold $ W^{cs, \delta}_{loc} (K) $ given in \autoref{rmk:inflowing} \eqref{it:i3} is $ C^1 $. Note also that in \autoref{lem:C1a}, \autoref{lem:udiff}, and \autoref{lem:high}, conditions (a)(b) are not needed. All the regularity results are also direct consequences of \cite[Section 6.3]{Che18a}.
	\end{enumerate}
\end{rmk}

\section{Existence of $ C^{0,1} \cap C^{1} $ bump functions: brief discussion} \label{sub:bump}

Throughout this section, we make the following assumption.

\vspace{.5em}
\noindent{Assumption}. $ \Sigma \in C^1 $ and $ m \mapsto \Pi^{\kappa}_{m} $ ($ \kappa = s, c, u $) are $ C^1 $ (in the immersed topology of $ \Sigma $).
\vspace{.5em}

As before, let $ \epsilon_{*}, \chi(\epsilon_{*}) $ be sufficiently small such that \autoref{lem:lip2} holds. Let $ \chi_{*} < 1/16 $ and take $ c_i, e_i $, $ i = 1,2 $, as in \autoref{lem:lip2}. Note that $ \chi_{*} \to 0 $ and $ c_i, e_i \to 1 $ as $ \epsilon_{*}, \chi(\epsilon_{*}) \to 0 $. Take $ \sigma_{*} = 2\chi_{*} \epsilon_{*} $. In the following, we focus on the existence of $ C^{0,1} \cap C^{1} $ bump functions in $ X^s_{\widehat{K}_{\epsilon_{*}}}(\sigma_{*}) $; the existence in $ \widehat{\Sigma} $ is simpler by taking $ \Pi^{s}_{m} = 0 $ for all $ m $ (see \autoref{exa:case2}).

Without loss of generality, assume $ \widehat{\Sigma} = \widehat{K}_{2\epsilon_{*}} $. For brevity, write
\[
\Sigma^s \triangleq X^s_{\widehat{K}_{\epsilon_{*}}}(\sigma_{*}).
\]
By our assumption, we know $ \Sigma^s \in C^1 $, and so the local representations $ g^{cs}_{\widehat{m}_0}(\cdot) $ (see \eqref{equ:cslocal}) are $ C^{1} $ for all $ \widehat{m}_0 \in \widehat{K} $.
Note that $ \Sigma $ with $ K_{\epsilon_{*}} $ also satisfies (H1)--(H4) in \autoref{subsec:main} for small $ \epsilon_{*} $. So we can assume that $ g^{cs}_{\widehat{m}} : X^{cs}_{\phi(\widehat{m})}(2\epsilon_{*}) \to X^{u}_{\phi(\widehat{m})} $ can be defined for $ \widehat{m} \in \widehat{K}_{\epsilon_{*}} $, i.e.,
\begin{equation}\label{equ:cslocal00}
\phi(\widehat{m}') + \overline{x}^s{'} = \phi(\widehat{m}) + x^{cs} + g^{cs}_{\widehat{m}}(x^{cs}),
\end{equation}
where $ \widehat{m}' \in \widehat{U}_{\widehat{m}} $, $ \overline{x}^{s}{'} \in X^{s}_{\phi(\widehat{m}')} $, and $ x^{cs} \in X^{cs}_{\phi(\widehat{m})} $. Moreover, $ \Sigma^s \subset \bigcup_{\widehat{m} \in \widehat{K}_{\epsilon_{*}}}\graph g^{cs}_{\widehat{m}} $, and $ g^{cs}_{\widehat{m}}(\cdot) \in C^1 $ for $ \widehat{m} \in \widehat{K}_{\epsilon_{*}} $. For $ (\widehat{m}, \overline{x}^s) \in \Sigma^s $, set
\[
A_{(\widehat{m}, \overline{x}^s)} \triangleq Dg^{cs}_{\widehat{m}}(0, \overline{x}^s) \in L(X^{cs}_{\phi(\widehat{m})}, X^{u}_{\phi(\widehat{m})}),
\]
and
\[
\mathbb{X}^{cs} = \bigsqcup_{(\widehat{m}, \overline{x}^s) \in \Sigma^s} \graph A_{(\widehat{m}, \overline{x}^s)} \subset \mathbb{G}(X).
\]
As before, the norm of $ X^{cs}_{\phi(\widehat{m})} $ is given by $ |x^{cs}| = \max\{ |x^s|, |x^c| \} $, where $ x^{cs} = (x^c , x^s) \in X^{c}_{\phi(\widehat{m})} \oplus X^{s}_{\phi(\widehat{m})} $.

\begin{lem}\label{lem:csP}
	$ T\Sigma^s = \mathbb{X}^{cs} $, and there is a Finsler structure on $ \mathbb{X}^{cs} $ such that $ \Sigma^s $ is a Finsler manifold (in the sense of Palais; cf. \cite{Pal66} or \cite[Appendix D.2]{Che18a}). Furthermore, if $ d_{F} $ denotes the Finsler metric on each component of $ \Sigma^s $, then for $ \widehat{m}_0 \in \widehat{K} $ and $ (\widehat{m}_i, \overline{x}^s_i) \in X^s_{\widehat{U}_{\widehat{m}_0}(\epsilon_*)} (\sigma_*) $, $ i = 1,2 $, with $ m_j = \phi(\widehat{m}_j) $, $ j = 0,1,2 $, we have
	\begin{multline*}
	(1 + \chi_{*})^{-1} \max\{ |\Pi^c_{m_0} ({m}_1 - {m}_2) |, |\Pi^s_{m_0} (\overline{x}^s_1 - \overline{x}^s_2)| \} \\
	\leq d_{F}((\widehat{m}_1, \overline{x}^s_1), (\widehat{m}_2, \overline{x}^s_2)) \leq (1 + \chi_{*}) \max\{ |\Pi^c_{m_0} ({m}_1 - {m}_2) |, |\Pi^s_{m_0} (\overline{x}^s_1 - \overline{x}^s_2)| \}.
	\end{multline*}
\end{lem}

\begin{proof}
	The proof is straightforward. First, from \autoref{lem:lip2} (2), we know $ \lip g^{cs}_{\widehat{m}}(\cdot) \leq \chi_{*} $, and in particular
	\[
	|A_{(\widehat{m}, \overline{x}^s)}| \leq \chi_{*}, \quad \mathbb{X}^{cs}_{(\widehat{m}, \overline{x}^s)} \oplus X^u_{\phi(\widehat{m})} = X.
	\]
	Set $ \widetilde{\Pi}^{cs}_{(\widehat{m}, \overline{x}^s)} = \Pi_{X^u_{\phi(\widehat{m})}} (\mathbb{X}^{cs}_{(\widehat{m}, \overline{x}^s)}) $. For $ (\widehat{m}, \overline{x}^s) \in \Sigma^s $, by \eqref{equ:cslocal00}, we get
	\begin{align*}
	&~ \frac{|\phi(\widehat{m}') + \overline{x}^s{'} - \phi(\widehat{m}) - \overline{x}^s - \widetilde{\Pi}^{cs}_{(\widehat{m}, \overline{x}^s)}(\phi(\widehat{m}') + \overline{x}^s{'} - \phi(\widehat{m}) - \overline{x}^s) |}{|\phi(\widehat{m}') + \overline{x}^s{'} - \phi(\widehat{m}) - \overline{x}^s|} \\
	= &~ \frac{|x^{cs} + g^{cs}_{\widehat{m}}(x^{cs}) - \overline{x}^{s} - g^{cs}_{\widehat{m}}(0, \overline{x}^s) - (x^{cs} - \overline{x}^{s}) - Dg^{cs}_{\widehat{m}}(0, \overline{x}^s)(x^{cs} - \overline{x}^{s})|}{|x^{cs} + g^{cs}_{\widehat{m}}(x^{cs}) - \overline{x}^{s} - g^{cs}_{\widehat{m}}(0, \overline{x}^s)|} \\
	\leq &~ (1 - \chi_{*})^{-1} \frac{ |g^{cs}_{\widehat{m}}(x^{cs}) - g^{cs}_{\widehat{m}}(0, \overline{x}^s) - Dg^{cs}_{\widehat{m}}(0, \overline{x}^s)(x^{cs} - \overline{x}^{s}) |}{|x^{cs} - \overline{x}^{s}|} \to 0,
	\end{align*}
	as $ x^{cs} \to \overline{x}^s $ ($ \Leftrightarrow \phi(\widehat{m}') + \overline{x}^s{'} \to \phi(\widehat{m}) + \overline{x}^s $). This shows that $ T_{(\widehat{m}, \overline{x}^s)} \Sigma^s = \mathbb{X}^{cs}_{(\widehat{m}, \overline{x}^s)} $.

	Define a Finsler structure on $ \mathbb{X}^{cs} $ by setting $ |v|_{(\widehat{m}, \overline{x}^s)} = |x| $, where $ v = (x, A_{(\widehat{m}, \overline{x}^s)}x) \in T_{(\widehat{m}, \overline{x}^s)} \Sigma^s = \mathbb{X}^{cs}_{(\widehat{m}, \overline{x}^s)} $. This makes $ \Sigma^s $ a Finsler manifold in the sense of \emph{Palais} (see \cite{Pal66}). Indeed, for any $ \varepsilon > 0 $, there is $ \delta > 0 $ such that for any $ x^{cs}, \overline{x}^s \in X^{cs}_{\phi(\widehat{m})} $ with $ |x^{cs} - \overline{x}^s| \leq \delta $, if we write
	\[
	y + Dg^{cs}_{\widehat{m}} (x^{cs})y = v + A_{(\widehat{m}', \overline{x}^s{'})}v, \quad y \in X^{cs}_{\phi(\widehat{m})}, v \in X^{cs}_{\phi(\widehat{m}')},
	\]
	where $ (\widehat{m}', \overline{x}^s{'}) $ is defined by \eqref{equ:cslocal00}, then
	\[
	|y| \leq 1 \Rightarrow 1 - \varepsilon \leq |v| \leq 1 + \varepsilon;
	\]
	here, note that
	\[
	v = \Pi^{cs}_{\phi(\widehat{m}')} (y + Dg^{cs}_{\widehat{m}} (x^{cs})y) = y + (\Pi^{cs}_{\phi(\widehat{m}')} - \Pi^{cs}_{\phi(\widehat{m})})y + (\Pi^{cs}_{\phi(\widehat{m}')} - \Pi^{cs}_{\phi(\widehat{m})}) Dg^{cs}_{\widehat{m}} (x^{cs})y.
	\]
	(In fact, $ \mathbb{X}^{cs} $ is $ C^{0,1} $-uniform in the sense of \cite{Che18a}.) In particular, for $ \widehat{m} \in \widehat{K} $, the $ C^1 $ local chart $ \psi_{\widehat{m}} $ at $ \widehat{m} $ ($ = (\widehat{m}, 0) $) defined by
	\[
	\psi_{\widehat{m}} : X^s_{\widehat{U}_{\widehat{m}}(\epsilon_*)} (\sigma_*) \to X^{cs}_{\phi(\widehat{m})}, \quad (\phi(\widehat{m}'), \overline{x}^s{'}) \mapsto x^{cs},
	\]
	where $ x^{cs} $ is given by \eqref{equ:cslocal00}, satisfies
	\[
	|D\psi_{\widehat{m}} (\phi(\widehat{m}'), \overline{x}^s{'})| \leq 1 + \chi_{*}, \quad |D\psi^{-1}_{\widehat{m}} (\psi_{\widehat{m}} (\phi(\widehat{m}'), \overline{x}^s{'}))| \leq 1 + \chi_{*}.
	\]
	This easily yields the estimate give in the lemma (see e.g. \cite[Lemma 2.4]{JS11}). The proof is complete.
\end{proof}

The existence of a $ C^{0,1} \cap C^1 $ bump function $ \Psi $ satisfying (a) and (b) in \autoref{sub:LipC1bm} relies on the following regular extension property of $ \mathbb{X}^{cs} $; see also \cite{JS13, JS11}.

\begin{defi}\label{def:property*}
	A vector bundle $ \mathbb{X} $ over $ \mathcal{M} $, where each fiber $ \mathbb{X}_{m} $ is endowed with a norm $ |\cdot|_{m} $ ($ m \in \mathcal{M} $), satisfies the \emph{$ C_1 $-uniform property ($ *^{k} $)} if there is a constant $ C_1 \geq 1 $ (independent of $ m \in \mathcal{M} $) such that for any $ m \in \mathcal{M} $, any Lipschitz function $ f_{m}: \mathbb{X}_{m} \to \mathbb{R} $, and any $ \varepsilon > 0 $, there is a $ C^{0,1} \cap C^{k} $ function $ K_{m}: \mathbb{X}_{m} \to \mathbb{R} $ satisfying
	\[
	|f_{m}(x) - K_{m}(x)| < \varepsilon \quad \forall x \in \mathbb{X}_{m}, \quad \text{and}~ \lip K_{m}(\cdot) \leq C_1 \lip f_{m}(\cdot).
	\]
	If $ \mathcal{M} = \{ m \} $, we say $ \mathbb{X}_{m} $ satisfies the \emph{$ C_1 $-property ($ *^{k} $)}. Let $ X_0 $ be a Banach space and $ \mathcal{M} $ be the set of all equivalent norms $ |\cdot|_\gamma $ of $ X_0 $, with $ \mathbb{X}_{|\cdot|_\gamma} = X_0 $ endowed with norm $ |\cdot|_\gamma $. If $ \mathbb{X} $ satisfies the $ C_1 $-uniform property ($ *^{k} $), then we say $ X_0 $ admits the \emph{$ C_1 $-uniform property ($ *^{k} $)}.
\end{defi}

\begin{exa}\label{exa:property*}
	\begin{enumerate}[(1)]
		\item $ \mathbb{R}^{n} $ satisfies the $ 1 $-uniform property ($ *^{k} $) for any $ n \in \mathbb{N} $. In fact, $ K_{m}(\cdot) $ can be constructed using convolutions.

		\item A Hilbert space endowed with any Hilbert norm admits the $ 1 $-property ($ *^{1} $) (see e.g. \cite{LL86}). Moreover, one can require $ K_{m}(\cdot) \in C^{1,1} $.

		\item For any set $ \Gamma $, $ c_0(\Gamma) $ with the usual sup norm satisfies the $ 1 $-property ($ *^{\infty} $) (see e.g. \cite[Chapter 7, Theorem 74]{HJ14}).

		\item If a separable Banach space $ X_0 $ admits a $ C^{k} \cap C^{0,1} $ bump function, then $ X_0 $ fulfills the $ (3+\epsilon) $-uniform property ($ *^{k} $) for any $ \epsilon > 0 $. This fact was announced in \cite[Remark 3.2 (3)]{JS11} for a universal constant $ C_1 \leq 602 $, that is, $ X_0 $ satisfies the $ C_1 $-property ($ *^{k} $) with $ C_1 \leq 602 $, independent of the choice of equivalent norm on $ X_0 $. We conjecture that this constant can be taken as $ 1 $, but currently only have a proof for $ C_1 = 3 + \epsilon $ (see \autoref{app:separable}).

		\item Assume for each $ \mathbb{X}_{m} $, there is a bi-Lipschitz map $ \Phi_{m} : \mathbb{X}_{m} \to \Phi_{m}(\mathbb{X}_{m}) \subset c_0(\Gamma) $ with $ \lip \Phi^{\pm 1}_{m} \leq \sqrt{C_1} $ and $ e^{*}_{\gamma} \circ \Phi_{m} \in C^{k} $ for all $ \gamma \in \Gamma $. Then $ \mathbb{X} $ satisfies the $ C_1 $-uniform property ($ *^k $); see \cite[Theorem 7]{HJ10}.
	\end{enumerate}
\end{exa}

In general, aside from the cases mentioned in \autoref{exa:property*}, we do not know any results concerning the uniform property $ (*^k) $ for non-separable Banach spaces.

\begin{thm}\label{thm:general}
	Let $ M $ be a $ C^k $ Finsler manifold in the sense of Palais (cf. \cite{Pal66}) with Finsler metric $ d_{F} $ in each component of $ M $ (possibly with boundary). If $ TM $ satisfies the $ C_1 $-uniform property ($ *^k $) (see \autoref{def:property*}), then for any $ C'_1 > C_1 $, any Lipschitz function $ f: M \to \mathbb{R} $, and any $ \varepsilon > 0 $, there is a $ C^{0,1} \cap C^{k} $ function $ g: M \to \mathbb{R} $ such that
	\[
	|f(m) - g(m)| < \varepsilon, \quad  \forall m \in M, \quad \text{and} \quad \lip g(\cdot) \leq C'_1 \lip f(\cdot).
	\]
\end{thm}

\begin{proof}
	Since $ M $ is a $ C^k $ Finsler manifold in the sense of Palais, for any $ K > 1 $, $ M $ is a $ C^1 $ Finsler $ K $-weak-uniform manifold in the sense of \cite{JS11}. One can then apply \cite[Theorem 3.4]{JS11}, although our theorem setting is slightly more general: in that theorem, $ M $ is modeled on a Banach space $ X $ admitting the $ C_1 $-uniform property ($ *^k $). See also \autoref{app:general} for a sketch of the proof.
\end{proof}

We say a norm is $ C^{k} $ if it is $ C^{k} $ on $ X \setminus \{ 0 \} $, and it is $ C^{k-1,1} $ if it is $ C^{k-1,1} $ on $ X(r) \setminus X(r_1) $ ($ 0 < r_1 < r $). In the following, we consider the existence of a $ C^{0,1} \cap C^1 $ bump function $ \Psi $ satisfying (a)--(b) in \autoref{sub:LipC1bm}, focusing on two cases: the special case $ \Sigma = K $ (see \autoref{exa:case1}) and the general case $ K \subset \Sigma $ (see \autoref{exa:case2}).

\begin{exa}[$ \Sigma = K $]\label{exa:case1}
	If $ \Sigma = K $ (see e.g. \cite{CLY00, NS12, KNS15, JLZ17}), it suffices to consider $ (\widehat{m}, \overline{x}^s) \mapsto |\overline{x}^s| $. There exist at least the following approaches to construct a $ C^{0,1} \cap C^1 $ bump function $ \Psi $:
	\begin{enumerate}[(a)]
		\item If the norm of $ X $ is $ C^1 $, then the function defined in \eqref{equ:cutoff}, i.e., $ \Psi (\widehat{m}, \overline{x}) = \ell(|\overline{x}|) $, already fulfills (a)--(b) in \autoref{sub:LipC1bm} with $ C_1 $ sufficiently close to 1. Examples include the Hilbert norm in Hilbert spaces and the usual norms in $ L^p(\Omega) $ ($ 1 < p < \infty $) or $ W^{k, p}(\Omega) $ ($ \Omega $ open in $ \mathbb{R}^n $).

		\item If $ X^* $ is separable, then the original norm of $ X $ can be approximated by a $ C^1 $ norm (see e.g. \cite[Chapter 7, Theorem 103]{HJ14}), and thus case (a) applies.

		\item If $ X^* $ is weakly compactly generated (e.g., $ X $ is reflexive), then $ X $ admits a $ C^1 $ norm (see e.g. \cite{HJ14}). Therefore, such $ \Psi $ always exists (though the constant $ C_1 $ may not be close to $ 1 $).

		\item The above cases do not apply to $ X = C[0,1] $ or $ L^1(\Omega) $, since $ C[0,1] $ and $ L^1(\Omega) $ do not admit $ C^1 $ norms. If $ \mathbb{X}^{cs} = T\Sigma^s $ satisfies the $ C_0 $-uniform property ($ *^1 $), then using \autoref{thm:general} and \autoref{lem:csP}, one can obtain such $ \Psi $ with $ C_1 $ sufficiently close to $ C_0 $ by approximating $ (\widehat{m}, \overline{x}^s) \mapsto |\overline{x}^s| $ and composing with $ \ell(\cdot) $. In particular, if $ X^{cs}_{m} $ ($ m \in K $) are all separable Banach spaces admitting $ C^1 $ bump functions (i.e., the duals of $ X^{cs}_{m} $ are separable), then $ C_1 $ can be taken sufficiently close to $ 3 $; if $ X^{cs}_{m} $ ($ m \in K $) are all finite-dimensional, then $ C_1 $ can be close to $ 1 $.
	\end{enumerate}

	In the above cases, we have $ \widehat{\psi} \in C^1(\Sigma^s, [0,1]) $ such that
	\[
	\widehat{\psi} (\widehat{m}, \overline{x}^s) = \begin{cases}
	1, & \quad (\widehat{m}, \overline{x}^s) \in X^s_{\widehat{\Sigma}} (\eta_2),\\
	0, & \quad (\widehat{m}, \overline{x}^s) \in X^s_{\widehat{\Sigma}} (\sigma_*) \setminus X^s_{\widehat{\Sigma}} (\eta_1),
	\end{cases}
	\]
	where $ 0 < \eta_2 < \eta_1 < \sigma_{*} $, and if $ \widehat{m}_1, \widehat{m}_2 \in \widehat{U}_{\widehat{m}_0}(\epsilon_{*}) $ and $ \widehat{m}_0 \in \widehat{K} $, then
	\[
	|\widehat{\psi} (\widehat{m}_1, \overline{x}^s_1) - \widehat{\psi} (\widehat{m}_2, \overline{x}^s_2) | \leq \frac{C_1}{\eta_1 - \eta_2} \max\{ |\Pi^c_{\phi(\widehat{m}_0)}(\phi(\widehat{m}_1) - \phi(\widehat{m}_2))|, |\Pi^s_{m_0} (\overline{x}^s_1 - \overline{x}^s_2)| \}.
	\]
	Here, $ C_1 $ can be close to $ 1 $ in cases (a) (b), and $ C_1 \to C_0 $ as $ \epsilon_{*}, \chi(\epsilon_*) \to 0 $ in case (d).
\end{exa}

\begin{exa}[$ K \subset \Sigma $]\label{exa:case2}
	Consider the general case $ K \subset \Sigma $ (see e.g. \cite{BC16, CLY00a}). Here, we focus on approximating $ \widehat{d}(\cdot, \widehat{K}) $ (see \eqref{equ:dP}), i.e., the existence of $ C^{0,1} \cap C^{1} $ bump functions in $ \widehat{\Sigma} $.
	
	If $ T\widehat{\Sigma} \approx X^{c}_{\widehat{\Sigma}} $ (with Finsler structure given by $ |x|_{m} = |x| $, $ x \in X^{c}_{m} $) satisfies the $ C_0 $-uniform property ($ *^1 $), then using \autoref{thm:general} and \autoref{lem:csP} to approximate $ \widehat{d}(\cdot, \widehat{K}) $, one obtains $ \widehat{\varphi} \in C^1(\widehat{\Sigma}, [0,1]) $ (by composing with a suitable bump function of $ \mathbb{R} $) such that 
	\[
	\widehat{\varphi} (\widehat{m}) = \begin{cases}
	1, & \quad \widehat{m} \in \widehat{K}_{\eta_2},\\
	0, & \quad \widehat{m} \in \widehat{\Sigma} \setminus \widehat{K}_{\eta_1},
	\end{cases}
	\]
	and if $ \widehat{m}_1, \widehat{m}_2 \in \widehat{U}_{\widehat{m}_0}(\epsilon_{*}) $ and $ \widehat{m}_0 \in \widehat{K} $, then
	\[
	|\widehat{\varphi} (\widehat{m}_1) - \widehat{\varphi} (\widehat{m}_2) | \leq \frac{C_1}{\eta_1 - \eta_2} |\Pi^c_{\phi(\widehat{m}_0)}(\phi(\widehat{m}_1) - \phi(\widehat{m}_2))|,
	\]
	where $ C_1 \to C_0 $ as $ \epsilon_{*}, \chi(\epsilon_*) \to 0 $ and $ 0 < \eta_2 < \eta_1 $.

	In particular, if $ \widehat{\Sigma} $ is finite-dimensional, or it is a Riemannian manifold with Riemannian metric $ \approx d_{F}|_{\Sigma} $ ($ d_{F} $ is the Finsler metric given in \autoref{lem:csP}) (for example, when $ X $ is a Hilbert space), then $ C_1 $ can be taken close to $ 1 $. Alternatively, if $ X^{c}_{m} $ ($ m \in \Sigma $) are all separable Banach spaces admitting $ C^1 $ bump functions (i.e., the duals of $ X^{c}_{m} $ are separable), then the constant $ C_1 $ can be taken sufficiently close to $ 3 $. It is worth noting that such functions $ \widehat{\varphi} $ in Riemannian manifolds, first introduced as \emph{uniformly bumpable} functions in \cite{AFL05}, have proven to be important and were further investigated in \cite{JS11}.
\end{exa}

Combining \autoref{exa:case1} and \autoref{exa:case2} (together with \autoref{exa:property*}), we have the following statements.
\begin{cor}\label{cor:spaces}
	Consider the following cases:
	\begin{enumerate}[(a)]
		\item \label{0a} $ X $ is a Hilbert space;
		\item the original norm of $ X $ is $ C^1 $ or $ X^* $ is separable, with finite-dimensional $ \Sigma $ or $ K = \Sigma $;
		\item \label{0c} $ X^{cs}_{m} $, $ m \in \Sigma $, are all finite-dimensional;
		\item $ X^* $ is weakly compactly generated, and one of the following holds: $ K = \Sigma $, or $ (X^c_{m})^* $ ($ m \in \Sigma $) are separable, or $ X^{c}_{\widehat{\Sigma}} $ satisfies the $ C_0 $-uniform property ($ *^1 $) (see \autoref{def:property*});
		\item $ X^* $ is separable;
		\item there is a bi-Lipschitz map $ \Phi : X \to \Phi(X) \subset c_0(\Gamma) $ with $ e^{*}_{\gamma} \circ \Phi \in C^{1} $ for all $ \gamma \in \Gamma $ (where $ \Gamma $ is a set);
		\item $ (X^{cs}_{m})^* $, $ m \in \Sigma $, are separable;
		\item $ X^{cs}_{\widehat{\Sigma}} $ satisfies the $ C_0 $-uniform property ($ *^1 $) (see \autoref{def:property*}).
	\end{enumerate}
	Then there is a function $ \Psi $ satisfying (a)--(b) in \autoref{sub:LipC1bm}. Moreover, if one of cases \eqref{0a}--\eqref{0c} holds, then the constant $ C_1 $ can be taken sufficiently close to $ 1 $.
\end{cor}

\begin{proof}
	All cases have been discussed in \autoref{exa:case1} and \autoref{exa:case2}, which imply the existence of $ \widehat{\psi} $ and $ \widehat{\varphi} $ mentioned in those examples. Note that for any $ \epsilon > 0 $, if $ p $ is sufficiently large, then
	\[
	(1+\epsilon)^{-1} \max\{ a, b \} \leq (a^{p} + b^{p})^{1/p} \leq (1+\epsilon)\max\{ a, b \}, \quad  0 \leq a, b \leq 2.
	\]
	Thus, define $ \Psi(\widehat{m}, \overline{x}^s) = 2^{-1/p} (\widehat{\varphi}(\widehat{m})^{p} + \widehat{\psi}(\widehat{m}, \overline{x}^s)^{p})^{1/p} $, and consequently $ \Psi $ satisfies (a) (b) in \autoref{sub:LipC1bm}. The proof is complete.
\end{proof}

Next, we consider under what conditions the function $ \Psi $ satisfying (a)--(b) in \autoref{sub:LipC1bm} can be $ C^{k-1,1} \cap C^{k} $. In the existing literature, we cannot find a general theory to address this problem. So we concentrate only on some special cases.

\begin{rmk}\label{rmk:highL}
	(I). ($ C^k $ + compactness) Consider the following two cases.
	\begin{enumerate}[(a)]
		\item Assume $ \Sigma^s \in C^k $, $ \Psi \in C^{k} $, and $ \widehat{K} $ is compact in $ \widehat{\Sigma} $. Then $ \Psi \in C^{k-1,1} \cap C^{k} $ in $ X^{s}_{\widehat{K}_{\epsilon}}(\epsilon) $ for small $ \epsilon > 0 $.

		(i) A special case where $ \Sigma^s \in C^k $ is when $ \widehat{\Sigma} \in C^k $ and $ \widehat{m} \mapsto \Pi^{s}_{\phi(\widehat{m})} $ is $ C^k $. A priori, the map $ \widehat{m} \mapsto \Pi^{s}_{\phi(\widehat{m})} $ might be only $ C^0 $, but in some cases this can be achieved by $ C^k $ approximation of $ \widehat{m} \mapsto \Pi^{s}_{\phi(\widehat{m})} $; for example, when $ \widehat{\Sigma} $ is a separable and finite-dimensional manifold (see e.g. \cite[Theorem 6.9]{BLZ08} or \autoref{sub:AE} for more general results). (ii) $ \Psi \in C^{k} $ can be obtained by assuming that $ X^{cs}_{\widehat{\Sigma}} $ satisfies the $ C_0 $-uniform property ($ *^k $). So $ \Sigma^s \in C^k $ and $ \Psi \in C^{k} $ can be obtained if $ X^{cs}_{\widehat{\Sigma}} $ satisfies the $ C_0 $-uniform property ($ *^k $) and $ \widehat{\Sigma} \in C^k $ (when $ \widehat{K} $ is compact); see \autoref{sub:AE}. Here, we note that if $ \widehat{\Sigma} $ is embedded, then $ \widehat{K} $ is compact in $ \widehat{\Sigma} $ if and only if $ K $ is compact in $ X $.

		\item In many cases, one can, without loss of generality, assume $ X^s_{\widehat{\Sigma}} = 0 $. In case (a), $ \widehat{\Sigma} \in C^k $ sometimes cannot be satisfied, especially when we need to construct $ \Sigma $ from $ K $ (see e.g. \cite{BC16, CLY00a} or \autoref{thm:whitney} for general results). Usually, we only know that $ \widehat{\Sigma} \setminus \widehat{K} \in C^{k} $. Since we only need $ \Psi \in C^{k-1,1} \cap C^{k} $ in $ \widehat{\Sigma} \setminus \widehat{K}_{\tilde{\eta}} $, a natural way for $ C^k $ to imply $ C^{k-1, 1} $ is that $ \widehat{K} $ is bounded in $ \widehat{\Sigma} $ and $ \widehat{\Sigma} $ is locally compact (i.e., $ \widehat{\Sigma} $ is finite-dimensional).
		This situation was also studied in \cite{BC16, CLY00a} in the finite-dimensional setting with $ K $ compact. However, it also frequently occurs in infinite-dimensional dynamical systems. For instance, (i) $ K $ is an isolated equilibrium and the essential spectrum bound of the linearized dynamic at $ K $ is strictly smaller than $ 1 $ (see e.g. \cite{DPL88, MR09a}); (ii) $ K $ is a (non-trivial) periodic orbit of some dynamical system with some compactness (see e.g. \cite{SS99, HR13}).
	\end{enumerate}

	(II). ($ K = \Sigma $) A special case where $ \Psi \in C^{k-1,1} \cap C^{k} $ is when $ K = \Sigma \in C^{k-1,1} \cap C^{k} $, $ \widehat{m} \mapsto \Pi^{s}_{\phi(\widehat{m})} $ is $ C^{k-1,1} \cap C^{k} $, and $ X $ is a Hilbert space (or the original norm of $ X $ is $ C^{k-1,1} \cap C^{k} $, e.g., $ X = L^p(\Omega) $ or $ W^{k,p}(\Omega) $ where $ p $ is even or $ p \geq k $ is odd). Although from an abstract viewpoint this is very restrictive, for applications in some practical problems, this condition is somewhat favorable; see also \cite{NS12, KNS15, JLZ17}.

	(III). In general, one needs to consider $ \widehat{\psi} $ and $ \widehat{\varphi} $ from \autoref{exa:case1} and \autoref{exa:case2} separately on a case-by-case basis. We have no idea how to address the existence of a $ C^{k-1,1} \cap C^k $ bump function with (a)--(b) in \autoref{sub:LipC1bm} in general Finsler manifolds (and even in finite-dimensional but non-precompact Finsler manifolds).
\end{rmk}

\begin{rmk}
	(a) We note that the existence of a $ C^{k-1,1} \cap C^{k} $ bump function $ \Psi $ is not necessary, and what we actually need is the following \emph{blip map}, introduced in \cite{BR17} for a Banach space, i.e., a $ C^{k-1,1} \cap C^{k} $ map $ \Psi_b: \Sigma^s \to \Sigma^s $ such that $ \Psi_b(x) = x $ if $ x \in X^s_{\widehat{K}_{\eta_2}} (\eta_2) $, $ \Psi_b(\Sigma^s) \subset X^s_{\widehat{K}_{\eta_1}} (\eta_1) $, and the Lipschitz property (b) in \autoref{sub:LipC1bm} holds. In the construction of the graph transform $ \varGamma $, one can use $ \Psi_b $ instead of $ \Psi $, i.e., for $ \widetilde{h} = \widetilde{\varGamma}(h) $, define $ \widehat{h} = \widetilde{h} \circ \Psi_b $ (which is well-defined); now $ \varGamma: h \mapsto \widehat{h} $.

	(b) Consider the special situation $ K = \Sigma $. Then the existence of such $ \Psi $ implies the existence of such $ \Psi_b $, but not vice versa.

	(c) For a Banach space $ B $, a blip map is a global \emph{b}ounded \emph{l}ocal \emph{id}entity at zero $ C^{k,\alpha} $ map $ b_{\varepsilon}: B \to B $ where $ b_{\varepsilon}(x) = x $ if $ x \in B(\varepsilon) $ (see \cite{BR17}). Any Banach space admitting a $ C^{k,\alpha} $ bump function possesses such a blip map, and for $ B = C[0,1] $ (which does not contain $ C^1 $ bump functions), it has natural blip maps defined by, for example, $ b_{\varepsilon}(x)(t) = \ell(x(t)) x(t) $, where $ \ell(\cdot) $ is a suitable bump function on $ \mathbb{R} $.

	(d) Thus, if $ \Sigma $ is a complemented subspace of $ C[0,1] $, $ K = \Sigma(\epsilon) $, and $ X^s_{\Sigma} = 0 $, then we always have such $ \Psi_b $ but not $ \Psi $ (if the dual of $ \Sigma $ is not separable). However, we do not know any results about the existence of blip maps in Finsler manifolds as in (a); note that one way to construct such blip maps is using $ C^{k-1,1} \cap C^{k} $ smooth partitions of unity, but this implies the existence of $ C^{k-1,1} \cap C^{k} $ bump functions.
\end{rmk}

\chapter{Invariant case: proof of \autoref{thm:invariant}} \label{sec:invariant}

Assume the conditions in \autoref{thm:invariant} hold.

We consider the case where $ \widehat{H} \approx (\widehat{F}^{cs}, \widehat{G}^{cs}) $ satisfies the (A$ ' $)($ \alpha $, $ \lambda_{u} $) (B)($ \beta; \beta', \lambda_{cs} $) condition in $ cs $-direction at $ K $. The ($ \bullet 2 $) case can be proved using a similar argument as in \autoref{sec:tri}.

All the constants defined in \autoref{sub:preparation} are used here, but now $ \eta = 0 $ and $ \Sigma = K $.

For the existence results, the construction of the graph transform is the same as in \autoref{sub:unlimited}. Let us define a metric space (see also \eqref{equ:space0}):
\begin{multline*}
\varSigma_{\mu, 0, \epsilon_{*}, \sigma_{*}, \varrho_* } = \{ h: X^s_{\widehat{\Sigma}} (\sigma_*) \to \overline{X^u_{\widehat{\Sigma}} (\varrho_*)} \text{ is a bundle map over } \id: \\
h(\widehat{m}, 0) = 0, \widehat{m} \in \widehat{\Sigma},
\graph h \cap X^{su}_{\widehat{K}_{\epsilon_*}} (\sigma_*, \varrho_*) \text{ is $ \mu $-Lip in the $ u $-direction}\},
\end{multline*}
and its metric
\[
d_{2}(h_1, h_2) = \sup\{ |h_1(\widehat{m},\overline{x}^s) - h_2(\widehat{m},\overline{x}^s)| / |\overline{x}^s|: (\widehat{m},\overline{x}^s) \in X^s_{\widehat{\Sigma}} (\sigma_*) \}.
\]
Since $ h(\widehat{m}, 0) = 0 $ for all $ \widehat{m} \in \widehat{\Sigma} $ and $ \lip h(\widehat{m}, \cdot) \leq \mu_{1}(\widehat{m}) $ (with $ \sup_{\widehat{m}}\mu_{1}(\widehat{m}) < \infty $), the metric $ d_{2} $ is well defined and $ \varSigma_{\mu, 0, \epsilon_{*}, \sigma_{*}, \varrho_* } $ is complete under $ d_{2} $.

For $ h \in \varSigma_{\mu, 0, \epsilon_{*}, \sigma_{*}, \varrho_* } $, let $ f_{\widehat{m}_0} $ be its local representation of $ \graph h \cap X^{su}_{\widehat{\Sigma}} (\sigma_*, \varrho_*) $ at $ \widehat{m}_0 $. Define $ \widetilde{f}_{\widehat{m}} (\cdot) $ as in \eqref{equ:local}.

\begin{enumerate}[$ \bullet $]
	\item Note that since $ \eta = 0 $, we have $ \widetilde{f}_{\widehat{m}} (0) = 0 $.

	\item \autoref{lem:locLip} also holds since the proof relies only on the (A$ ' $) (B) condition of $ H(m_0 + \cdot) - u({m}_2) \sim (\widetilde{F}^{cs}, \widetilde{G}^{cs}) $.

	\item \autoref{lem:belong} holds with $ \eta = 0 $; here take $ K_1 = K'_1 = 0 $. Now we can define the \emph{graph transform} as
	\[
	\varGamma: \varSigma_{\mu, 0, \epsilon_{*}, \sigma_{*}, \varrho_* } \to \varSigma_{\mu, 0, \epsilon_{*}, \sigma_{*}, \varrho_* }, h \mapsto \widehat{h} = \Psi \cdot \widetilde{h}.
	\]
	Note also that $ \widetilde{h} (\widehat{m}, \overline{x}^s) = \widetilde{f}_{\widehat{m}}(0, \overline{x}^s) $ for all $ (\widehat{m}, \overline{x}^s) \in X^s_{\widehat{\Sigma}} (\sigma^1_*) $ due to $ \Sigma = K $.

	\item Finally, let us show $ \lip \varGamma \leq \varpi^*_1 \sup_{\widehat{m}_0 \in \widehat{\Sigma}} \frac{ \lambda_{cs}(\phi(\widehat{m}_0)) \lambda_{u}(\phi(\widehat{m}_0)) }{ 1 - \alpha(\phi(\widehat{m}_0)) \mu_1(\widehat{u}(\widehat{m}_0)) } < 1 $ in the metric $ d_{2} $, where $ \varpi_1^* \to 1 $ as $ \epsilon_{*}, \chi(\epsilon_{*}) \to 0 $.
	\begin{proof}
		We continue the proof in \autoref{lem:contractive}. Since $ \Sigma = K $, we can let $ \widehat{m}_0 = \widehat{m} $, so the equation \eqref{equ:abc} can be rewritten as
		\[
		\begin{cases}
		m + \overline{x}^{s} + \widetilde{h}^{1}(\widehat{m}, \overline{x}^{s}) = m + \overline{x}^{s} + \widetilde{f}^1_{\widehat{m}}(0, \overline{x}^{s}), \\
		m + \overline{x}^{s} + \widetilde{h}^{2}(\widehat{m}, \overline{x}^{s}) = m + x^{cs}_2 + \widetilde{f}^2_{\widehat{m}}(x^{cs}_2),
		\end{cases}
		\]
		i.e., now $ x^{cs}_{1} = (0, \overline{x}^{s}) $.
		We have shown in \autoref{lem:contractive} that
		\begin{equation}\label{equ:mid}
		|\widetilde{h}^1(\widehat{m}, \overline{x}^s) - \widetilde{h}^2(\widehat{m}, \overline{x}^s)| \leq \frac{\varpi^*_1\lambda_{u}(m)}{1 - \alpha(m) \mu_1(\widehat{m}_1)} |{h}^1(\widehat{\overline{m}}_1, \overline{x}^s_1) - {h}^2(\widehat{\overline{m}}_1, \overline{x}^s_1)|.
		\end{equation}
		Since
		\[
		m_1 + \widetilde{x}^1 + f^{1}_{\widehat{m}_1}(\widetilde{x}^1) = \overline{m}_1 + \overline{x}^{s}_1 + h^{1}(\widehat{\overline{m}}_1, \overline{x}^{s}_1),
		\]
		where $ \widetilde{x}^1 = x^{1}_{\widehat{m}}(0, \overline{x}^s) $ is defined in \eqref{equ:zz}, from \eqref{equ:estimates} (in \autoref{lem:lip2}), we get
		\[
		|\overline{x}^s_1| \leq \varpi^*_1 |x^{1}_{\widehat{m}}(0, \overline{x}^s)| \leq \varpi^*_1 \lambda_{cs}(m) |\overline{x}^s|.
		\]
		Combining this with \eqref{equ:mid}, we obtain $ \lip \varGamma \leq \varpi^*_1 \sup_{\widehat{m}_0 \in \widehat{\Sigma}} \frac{ \lambda_{cs}(\phi(\widehat{m}_0)) \lambda_{u}(\phi(\widehat{m}_0)) }{ 1 - \alpha(\phi(\widehat{m}_0)) \mu_1(\widehat{u}(\widehat{m}_0)) } < 1 $, which completes the proof.
	\end{proof}
\end{enumerate}

We have thus established the existence results, i.e., the existence of $ W^{cs}_{loc}(\Sigma) $ satisfying the representation in \autoref{thm:I} \eqref{it:h0}, with $ \Sigma \subset W^{cs}_{loc}(\Sigma) $ and local invariance under $ H $. For the regularity results, these have already been proved in \autoref{sec:smooth} (see \autoref{lem:C1smooth} and \autoref{rmk:diff}), and the proofs given there require no modification.

Now consider the partial characterization given in \autoref{thm:invariant} \eqref{it:invPartial} (see also \cite[Section 4.4]{Che18a} for a similar argument).

Let $ \{z_{k} = (\widehat{m}_{k}, x^{s}_{k}, x^{u}_{k})\}_{k \geq 0} \subset X^s_{\widehat{\Sigma}}(\varepsilon_0) \oplus X^u_{\widehat{\Sigma}} (\varrho_{*}) $ be given as in \autoref{thm:invariant} \eqref{it:invPartial}. Then
\begin{equation*}
m_{k} + x^{s}_{k} + x^{u}_{k} = u(m_{k-1}) + \hat{x}^{cs}_{k} + \hat{x}^{u}_{k} \in H(m_{k-1} + x^{s}_{k-1} + x^{u}_{k-1}),
\end{equation*}
where $ m_{j} = \phi(\widehat{m}_{j}) $, $ \widehat{m}_{k} \in \widehat{U}_{\widehat{u}(\widehat{m}_{k-1})} (\epsilon_{*}) $ and $ \hat{x}^{\kappa}_{k} \in X^{\kappa}_{u(m_{k-1})} $, $ \kappa = cs, u $.

Since
\[
\beta'(u(\phi(\widehat{m}))) + \hat{\chi} < \widetilde{\beta}_{0}(\widehat{m}) < \beta(\phi(\widehat{m})) - \hat{\chi},
\]
we can assume $ \widetilde{\beta}_{0}(\widehat{m}_{k}) (1 + \chi_{*}) + \chi_{*} \leq \widetilde{\beta}(\widehat{m}_{k}) $ (by making $ \epsilon_{*}, \chi(\epsilon_{*}) $ small).

First, note that if $ |x^{u}_{k}| \leq \widetilde{\beta}_{0}(\widehat{m}_{k-1}) |x^{s}_{k}| $ for all $ k \geq 0 $, then by \autoref{lem:lip2},
\[
|\hat{x}^{u}_{k}| \leq (\widetilde{\beta}_{0}(\widehat{m}_{k-1}) (1 + \chi_{*}) + \chi_{*}) |\hat{x}^{cs}_{k}| \leq \widetilde{\beta}(\widehat{m}_{k-1}) |\hat{x}^{cs}_{k}|,
\]
which, by the (A$ ' $) (B) condition of $ \widehat{H}_{m_{k-1}} = H(m_{k-1} + \cdot) - u(m_{k-1}) $, yields
\[
|\hat{x}^{cs}_{k}| \leq \lambda_{cs}(m_{k-1}) |x^{s}_{k-1}| \Rightarrow |x^{s}_{k}| \leq \widetilde{\lambda}_{cs}(m_{k-1}) |x^{s}_{k-1}|.
\]
In particular,
\[
|x^{s}_{k}| \leq \widetilde{\lambda}_{cs}(m_{k-1}) |x^{s}_{k-1}| \leq \widetilde{\lambda}_{cs}(m_{k-1}) \widetilde{\lambda}_{cs}(m_{k-2}) \cdots \widetilde{\lambda}_{cs}(m_{0}) |x^{s}_{0}|,
\]
and so $ \sup_{k}\{\varepsilon_{s}(\widehat{m}_0) \varepsilon_{s}(\widehat{m}_1) \cdots \varepsilon_{s}(\widehat{m}_{k-1})\}^{-1} (|x^{s}_{k}| + |x^{u}_{k}|) < \infty $.

Therefore, it suffices to show that if 
\[
\sup_{k}\{\varepsilon_{s}(\widehat{m}_0) \varepsilon_{s}(\widehat{m}_1) \cdots \varepsilon_{s}(\widehat{m}_{k-1})\}^{-1} (|x^{s}_{k}| + |x^{u}_{k}|) < \infty,
\]
then $ x^{u}_{k} = h_0(\widehat{m}_{k}, x^{s}_{k}) $, where $ h_0 $ is given in \autoref{thm:invariant} (i.e., \autoref{thm:I} \eqref{it:h0}).
From the proof of \eqref{equ:mid} (see also \autoref{lem:partial}), and noting that $ (\widehat{m}_{k-1}, \overline{x}^s_{k-1}) \in X^s_{\widehat{\Sigma}}(\varepsilon_0) $, we have
\[
|h(\widehat{m}_{k-1}, x^s_{k-1}) - x^{u}_{k-1}| \leq \frac{\varpi^*_1\lambda_{u}(m_{k-1})}{1 - \alpha(m_{k-1}) \mu_1(\widehat{m}_k)} |h(\widehat{m}_k, x^s_k) - x^{u}_{k}|.
\]
Set
\[
\widehat{\lambda}^{(k)}_{u} (m_{0}) = \frac{\varpi^*_1\lambda_{u}(m_{0})}{1 - \alpha(m_{0}) \mu_1(\widehat{m}_1)} \frac{\varpi^*_1\lambda_{u}(m_{1})}{1 - \alpha(m_{1}) \mu_1(\widehat{m}_2)} \cdots \frac{\varpi^*_1\lambda_{u}(m_{k-1})}{1 - \alpha(m_{k-1}) \mu_1(\widehat{m}_k)},
\]
and $ \varepsilon^{(k)}_{s}(\widehat{m}_0) = \varepsilon_{s}(\widehat{m}_0) \varepsilon_{s}(\widehat{m}_1) \cdots \varepsilon_{s}(\widehat{m}_{k-1}) $.
Since
\[
0 < \varepsilon_{s}(\widehat{m}) < \lambda^{-1}_{u} (\phi(\widehat{m})) \vartheta(\phi(\widehat{m})) - \hat{\chi},
\]
we can assume $ \sup_{k} \frac{\varpi^*_1\varepsilon_{s}(\widehat{m}_{k-1})\lambda_{u}(m_{k-1})}{1 - \alpha(m_{k-1}) \mu_1(\widehat{m}_k)} < 1 $ (if $ \xi_1 $ and $ \epsilon_{*}, \chi_{*} $ are small). Then we get
\[
|h(\widehat{m}_{0}, x^s_{0}) - x^{u}_{0}| \leq \widehat{\lambda}^{(k)}_{u} (m_{0}) |h(\widehat{m}_k, x^s_k) - x^{u}_{k}| \leq \widetilde{C} \widehat{\lambda}^{(k)}_{u} (m_{0}) \varepsilon^{(k)}_{s}(\widehat{m}_0)  \to 0 \quad \text{as } k \to \infty,
\]
where $ \widetilde{C} $ is a constant independent of $ k $. This completes the proof of \autoref{thm:invariant}. \qed

\chapter{Trichotomy case: proof of \autoref{thm:tri0}} \label{sec:tri}

Assume the conditions in \autoref{thm:tri0} hold. (The following arguments in fact provide a proof of the ($ \bullet 1 $) case in \autoref{thm:I}, i.e., $ \widehat{H} \approx (\widehat{F}^{cs}, \widehat{G}^{cs}) $ satisfies the (A) ($ \alpha $, $ \lambda_{u} $) (B) ($ \beta; \beta', \lambda_{cs} $) condition in $ cs $-direction at $ K $.)

Let $ \widetilde{\iota}(\cdot) $ be given as in \autoref{lem:qq} (by applying to $ \widehat{H} \approx (\widehat{F}^{\kappa}, \widehat{G}^{\kappa}) $ with the (A) ($ {\alpha}_{\kappa_2} $; ${\alpha}_{\kappa_2}' $, $ {\lambda}_{\kappa_2} $) (B) ($ {\beta}_{\kappa_1} $; $ {\beta}_{\kappa_1}' $, $ {\lambda}_{\kappa_1} $) condition), where $ \iota $ stands for $ \alpha_{\kappa_2} $, $ \alpha'_{\kappa_2} $, $ \beta_{\kappa_1} $, $ \beta'_{\kappa_1} $, $ \lambda_{\kappa_1} $, $ \lambda_{\kappa_2} $, and $ \kappa_1 = cs $, $ \kappa_2 = u $, $ \kappa = cs $ (or $ \kappa_1 = s $, $ \kappa_2 = cu $, $ \kappa = cu $). Note also that (A) ($ {\alpha}_{u} $; ${\alpha}_{u}' $, $ {\lambda}_{u} $) condition implies (A$ ' $) (${\alpha}_{u}' $, $ {\lambda}_{u} $) condition. Use $ {\alpha}_{u}', \beta_{cs}, \beta'_{cs}, $ instead of $ \alpha,\beta, \beta' $ and $ \widetilde{\alpha}_{u}' $ instead of $ \widetilde{\alpha} $ in Observations \eqref{OI}--\eqref{OVII}. We take all the constants defined in \autoref{sub:preparation}.

Since $ \sup_m \lambda_{u}(m) (1 - \alpha'_{u}(m) \beta'_{cs}(u(m)))^{-1} < 1 $ is not assumed, in general, the graph transform $ \varGamma $ (defined in \autoref{sub:graph}) does not map $ \varSigma_{lip, \mu, K_1} $ into $ \varSigma_{lip, \mu, K_1} $. Note that if $ \eta = 0 $, then it is clear that $ \varGamma \varSigma_{lip, \mu, 0} \subset \varSigma_{lip, \mu, 0} $.

As $ \sup_m \lambda_{u}(m) < 1 $, choose $ n \in \mathbb{N} $ large such that
\[
(\sup_m \lambda_{u}(m))^n  < 1 - 2\sup_{m}\alpha'_{u}(m) \beta'_{cs}(u(m)).
\]
Then there is a constant $ K_1 \geq 1 $ such that $ \varGamma^{n} \varSigma_{lip, \mu, K_1} \subset \varSigma_{lip, \mu, K_1} $; moreover, $ \lip \varGamma^{n} < 1 $.
\begin{proof}
	For simplicity, let $ n = 2 $. We consider only the case (1) and leave the case (2) to the readers.
	Take $ \widehat{m}_0 \in \widehat{K} $.
	Let $ \widehat{I}_{\widehat{m}_0} $ be the correspondence determined by $ F^1_{\widehat{m}_0} $ and $ G^1_{\widehat{m}_0} $, so that $ \widehat{I}_{\widehat{m}_0} \sim (F^1_{\widehat{m}_0}, G^1_{\widehat{m}_0}) $.
	Since $ \lip F^1_{\widehat{m}_0}(x^{cs}, \cdot) $ is small, it is well defined that $ \overline{H}_{\widehat{m}_0} \triangleq \widehat{I}_{\widehat{m}_0} \circ \widehat{H}_{\widehat{m}_0} \sim (\widetilde{F}^{(1)}_{\widehat{m}_0}, \widetilde{G}^{(1)}_{\widehat{m}_0}) $, and so $ \overline{H}_{\widehat{u}(\widehat{m}_0)} \circ \overline{H}_{\widehat{m}_0}  \sim (\widetilde{F}^{(2)}_{\widehat{m}_0}, \widetilde{G}^{(2)}_{\widehat{m}_0}) $ with
	\[
	\widetilde{F}^{(i)}_{\widehat{m}_0}: X^{cs}_{m_0} (r_0) \times X^{u}_{u^{(i)}(m_0)} (r_0) \to X^{cs}_{u^{(i)}(m_0)}, ~
	\widetilde{G}^{(i)}_{\widehat{m}_0}: X^{cs}_{m_0} (r_0) \times X^{u}_{u^{(i)}(m_0)} (r_0) \to X^{u}_{m_0}, 
	\]
	for $i = 1,2$, where $ m_0 = \phi(\widehat{m}_0) $ and $ r_0 $ is given in Observation \eqref{OI}; moreover, $ |\widetilde{F}^{(2)}_{\widehat{m}_0} (0, 0)| \leq K_{00}\eta $, $ |\widetilde{G}^{(2)}_{\widehat{m}_0} (0, 0)| \leq K_{00}\eta $ for some fixed constant $ K_{00} > 1 $ (see e.g. \autoref{lem:first2}) if $ \eta $ is small.

	Let
	\[
	\overline{\lambda}'_{u} = \sup_{\widehat{m} \in \widehat{K}}\frac{\lambda_{u}(\phi(\widehat{u}(\widehat{m})))\lambda_{u}(\phi(\widehat{m}))}{1 - \alpha(\phi(\widehat{u}(\widehat{m}))) \mu_1(\widehat{u}^2(\widehat{m}))}, ~\quad K_1 = \frac{ \overline{\lambda}'_u\hat{\beta} + 1 }{1 - \overline{\lambda}'_u} K_{00}.
	\]
	For $ h^{(2)} \in \varSigma_{lip, \mu, K_1} $, write $ \varGamma^{i} h^{(2)} = h^{(2 - i)} $, $ i = 1,2 $. Let $ f^{(i)}_{\widehat{m}_0} $ be the local representation of $ \graph h^{(i)} \cap X^{su}_{\widehat{K}_{\epsilon_{*}}} (\sigma_*, \varrho_*) $ at $ \widehat{m}_0 \in \widehat{K} $. Note that for $ |x^{cs}| > \eta''_1 $, $ f^{(i)}_{\widehat{m}_0} (x^{cs}) = g^0_{\widehat{m}_0} (x^{cs}) $, where $ g^0_{\widehat{m}_0} $ is defined by \eqref{equ:cslocal}.
	Then for $ i = 1,2 $,
	\begin{equation}\label{equ:localaa}
	\begin{cases}
	\widetilde{F}^{(i)}_{\widehat{m}_0} ( x^{cs}, f^{(i)}_{\widehat{u}^{i}(\widehat{m}_0)} ( x^{(i)}_{\widehat{m}_0} (x^{cs}) ) ) = x^{(i)}_{\widehat{m}_0} (x^{cs}), \\
	\widetilde{G}^{(i)}_{\widehat{m}_0} ( x^{cs}, f^{(i)}_{\widehat{u}^{i}(\widehat{m}_0)} ( x^{(i)}_{\widehat{m}_0} (x^{cs}) ) ) = f^{(0)}_{\widehat{m}_0} (x^{cs}),
	\end{cases}
	x^{cs} \in X^{cs}_{\phi(\widehat{m}_0)}(e_0\eta'_1).
	\end{equation}

	We first need to show $ |f^{(0)}_{\widehat{m}_0} (0)| \leq K_1 \eta $.

	Let $ K_{00}\eta \leq \epsilon_{0,*} \eta_0 $ (where $ \epsilon_{0,*} = O_{\epsilon_{*}}(1) $ is defined in \eqref{equ:o(1)}) be small such that \eqref{equ:small000} holds; also, \eqref{equ:small000}  holds when $ K_1, K_2 = K_2(K_1) $ are replaced by $ K'_1 = \overline{\lambda}_u ( \hat{\beta} + K_1 ) + 1, K_2(K'_1) $.

	Note that by \autoref{lem:belong} (1), $ h^{(1)} \in \varSigma_{\mu, K'_1, \epsilon_{*}, \sigma_{*}, \varrho_* } $ and $ \varGamma h^{(1)} = h^{(0)} \in \varSigma_{\mu, K''_1, \epsilon_{*}, \sigma_{*}, \varrho_* } $, and so
	\[
	|f^{(1)}_{\widehat{m}_0} (0)| \leq K'_1 \eta, ~\quad |f^{(0)}_{\widehat{m}_0} (0)| \leq K''_1 \eta,
	\]
	where $ K'_1 = \overline{\lambda}'_u ( \hat{\beta} + K_1 ) + 1 $ and $ K''_1 = \overline{\lambda}'_u ( \hat{\beta} + K'_1 ) + 1 $.

	Due to \autoref{lem:est000}, by taking $ \epsilon_{0,*} $ (i.e., $ \epsilon_{*} $) small, we can assume that $ \overline{H}_{\widehat{u}(\widehat{m}_0)} \circ \overline{H}_{\widehat{m}_0}  \sim (\widetilde{F}^{(2)}_{\widehat{m}_0}, \widetilde{G}^{(2)}_{\widehat{m}_0}) $ satisfies the (A$ ' $) $( \widetilde{\alpha}'_{u}(m_1)$, $\widetilde{\lambda}_u(m_0) \widetilde{\lambda}_u(m_1) )$ (B) $( \widetilde{\beta}_{cs}(m_1)$; $\mu_1(\widehat{m}_0)$, $\widetilde{\lambda}_{cs}(m_0) \widetilde{\lambda}_{cs}(m_1) )$ condition, where $ m_0 = \phi(\widehat{m}_0) $ and $ m_1 = u(m_0) $. Now applying the same argument as in \autoref{lem:first2} to \eqref{equ:localaa} (for $ i = 2 $), we obtain that $ |f^{(0)}_{\widehat{m}_0} (0)| \leq K_1 \eta $ as $ \overline{\lambda}'_{u} < 1 $. This means that $ \varGamma^{2} \varSigma_{lip, \mu, K_1} \subset \varSigma_{lip, \mu, K_1} $.

	Furthermore, by the same argument given in \autoref{lem:contractive} (but in this case considering
	\[
	x^{i}_{\widehat{m}_0}(x^{cs}_i) + f^{i}_{\widehat{u}^2(\widehat{m}_0)}(x^{i}_{\widehat{m}_0}(x^{cs}_i)) \in \overline{H}_{\widehat{u}(\widehat{m}_0)} \circ \overline{H}_{\widehat{m}_0}(x^{cs}_i + {f}^i_{\widehat{m}_0}(x^{cs}_i)),
	\]
	instead of \eqref{equ:zz}), we have $ \lip \varGamma^2 < 1 $.
\end{proof}

In particular, we have a unique $ h^{cs}_0 $ belonging to $ \varSigma_{lip, \mu, K_1} $ such that $ \varGamma^{n} h^{cs}_0 = h^{cs}_0 $. Let us show $ \varGamma h^{cs}_0 = h^{cs}_0 $. Write $ \varGamma h^{cs}_0 = h'_0 $. By \autoref{lem:belong} (1), $ h'_0 \in \varSigma_{lip, \mu, K'_1} $. Without loss of generality, let $ K'_1 \geq K_1 $. Since we also have $ \varGamma^{n} \varSigma_{lip, \mu, K'_1} \subset \varSigma_{lip, \mu, K'_1} $ and $ \varGamma^{n} h'_0 = h'_0 $, and noting that $ h^{cs}_0 \in \varSigma_{lip, \mu, K'_1} $ and $ \varGamma^{n} $ has only one fixed point in $ \varSigma_{lip, \mu, K'_1} $, we see that $ h'_0 = h_0 $, i.e., $ \varGamma h^{cs}_0 = h^{cs}_0 $.

Now we have a center-stable manifold $ W^{cs}_{loc}(K) = \graph h^{cs}_0 \subset X^s_{\widehat{\Sigma}} (\sigma_{*}) \oplus X^u_{\widehat{\Sigma}} (\varrho_{*}) $ of $ K $, where
\[
\graph h^{cs}_0 \triangleq \{ (\widehat{m}, x^s, h_0(\widehat{m}, x^s)) \triangleq \phi(\widehat{m}) + x^s + h^{cs}_0(\widehat{m}, x^s): \widehat{m} \in \widehat{\Sigma}, x^s \in X^{s}_{\phi(\widehat{m})}(\sigma_{*}) \},
\]
with $ \graph h^{cs}_0 $ being $ \mu'_{cs} $-Lip in $ u $-direction near $ K $ (see \autoref{def:lip}) and $ \mu'_{cs}(\cdot) \approx \beta_{cs}(\cdot) $; moreover, $ \Omega_{cs} \subset H^{-1} W^{cs}_{loc}(K) $, where
\[
\Omega_{cs} = \graph h^{cs}_0|_{X^{s}_{\widehat{K}_{\varepsilon_{0}}} (\varepsilon_{0})},
\]
and $ \varepsilon_{0} > 0 $ is small as given in \autoref{lem:existence}. Also, by \autoref{lem:locLip}, $ \Omega_{cs} $ is $ \mu_{cs} $-Lip in $ u $-direction near $ K $ (see \autoref{def:lip}) with $ \mu_{cs}(\cdot) \approx \beta'_{cs}(\cdot) $.

To construct a center-unstable manifold of $ K $, consider the dual correspondences $ \widetilde{H}^{cu}_m $ (see \autoref{defi:dual}) of $ H^{cu}_{m} \triangleq H(m+ \cdot) - u(m): \widehat{X}^{s}_m(r) \times \widehat{X}^{cu}_{m} (r_1) \to \widehat{X}^{s}_{u(m)}(r_2) \times \widehat{X}^{cu}_{u(m)} (r) $, $ m \in K $; the ``center-stable direction'' results for $ \widetilde{H}^{cu}_m $ will give us the desired results. That is, we have $ h^{cu}_0 : X^u_{\widehat{\Sigma}} (\sigma_{*}) \to X^s_{\widehat{\Sigma}} (\varrho_{*}) $ satisfying \autoref{thm:tri0} \eqref{it:tri1} and $ W^{cu}_{loc}(K) = \graph h^{cu}_0 \subset X^s_{\widehat{\Sigma}} (\varrho_{*}) \oplus X^u_{\widehat{\Sigma}} (\sigma_{*}) $ with $ \graph h^{cu}_0 $ being $ \mu'_{cu} $-Lip in $ s $-direction near $ K $ (see \autoref{def:lip}), where $ \mu'_{cu}(\cdot) \approx \alpha_{cu}(\cdot) $; furthermore, $ \Omega_{cu} \subset H W^{cu}_{loc}(K) $, where
\[
\Omega_{cu} = \graph h^{cu}_0|_{X^{u}_{\widehat{K}_{\varepsilon_{0}}} (\varepsilon_{0})},
\]
which is $ \mu_{cu} $-Lip in $ s $-direction near $ K $ (see \autoref{def:lip}) with $ \mu_{cu}(\cdot) \approx \alpha'_{cu}(\cdot) $.
Now we have
\begin{equation}\label{equ:c00}
\max_{\kappa = s, u} \left\{ \sup_{\widehat{m} \in \widehat{K}}|h^{c\kappa}_0(\widehat{m}, 0)| \right\} \leq K'_0 \eta,
\end{equation}
for some constant $ K'_0 > 0 $.

As $ \sup_m \alpha_{cu}(m) \beta_{cs}(m) < 1 $ (in (B3) (a) (i)), we can assume $ \sup_{m} \mu'_{cs}(m)\mu'_{cu}(m) < 1 $. Assertion: for $ \widehat{m} \in \widehat{K}_{\eta_1} $, the equation
\[
\begin{cases}
x^u = h^{cs}_0(\widehat{m}, x^s), \\
x^s = h^{cu}_0(\widehat{m}, x^u),
\end{cases}
\]
has a unique solution $ x^u = x^u(\widehat{m}) $, $ x^s = x^s(\widehat{m}) $ if $ \eta $ is further reduced. Note also that by construction, if $ \widehat{m} \in \widehat{\Sigma} \setminus \widehat{K}_{\eta_1} $, then $ h^{cs}_{0}(\widehat{m}, \cdot) = h^{cu}_{0}(\widehat{m}, \cdot) = 0 $.
\begin{proof}[Proof of the assertion]
	Letting 
	\[
	\nu_1 = K'_0 \eta + \sup_{m}\mu'_{cs}(m) K'_0 \eta \quad \text{and} \quad \nu_0 = (1 - \sup_m\mu'_{cs}(m)\mu'_{cu}(m))^{-1} \nu_1,
	\]
	we choose $ \eta $ to satisfy $ K'_0 \eta + \sup_{m}\mu'_{cu}(m) \nu_0 \leq \sigma_{*} $ in order to solve $ x^u = h^{cs}_0(\widehat{m}, h^{cu}_0(\widehat{m}, x^u)) $ in $ X^{u}_{\phi(\widehat{m})} (\nu_0) $.
\end{proof}

Define $ h^c_0(\widehat{m}) = x^s(\widehat{m}) + x^u(\widehat{m}) $. This gives $ \Sigma^c = W^{cs}_{loc}(K) \cap W^{cu}_{loc}(K) = \graph h^{c}_0 $ satisfying \autoref{thm:tri0} \eqref{it:tri1} and $ \Omega_c \triangleq \Omega_{cs} \cap \Omega_{cu} \subset H^{\pm1} \Sigma^c $. Also, note that $ \sup_{\widehat{m}}|h^{c}_{0}(\widehat{m})| \leq K''_0  \eta $ for some $ K''_0 > 0 $.

Conclusion (2) in \autoref{thm:tri0} follows from \autoref{lem:unique}. Conclusions (2) (3) in \autoref{thm:tri0} can be proved in the same way as in \autoref{sec:smooth}, but in this case we need to consider equation \eqref{equ:localaa}.

Therefore, we complete the proof of \autoref{thm:tri0}. \qed

\begin{rmk}
	The above argument in fact implies more, which we discuss below.
	\begin{asparaenum}
		\item We say $ \widehat{H} \approx (\widehat{F}^{\kappa}, \widehat{G}^{\kappa}) $ satisfies the \emph{(A$ '_1 $) ($ \alpha; \lambda_1; c $)} (resp. \emph{(A$ _1 $)($ \alpha; \alpha_1, \lambda_1; c $)}) \emph{condition in $ \kappa $-direction at $ K $}, if for all $ m^{\rho} \in K $ and $ n, \rho \in \mathbb{N} $,
		\[
		\widehat{H}^{(n)}_{m} \sim (\widehat{F}^{(n)}_{m}, \widehat{G}^{(n)}_{m}): \widehat{X}^{\kappa}_{m}(r_{n}) \oplus \widehat{X}^{\kappa_1}_{m}(r_{n,1}) \to \widehat{X}^{\kappa}_{u^{n}(m)}(r_{n,2}) \oplus \widehat{X}^{\kappa_1}_{u^{n}(m)} (r_{n}),
		\]
		where $ r_{n}, r_{n,i} > 0 $ ($ i = 1, 2 $) are independent of $ m \in K $, and $ m = u^{\rho}(m^{\rho}) $,
		such that $ (\widehat{F}^{(n)}_{m}, \widehat{G}^{(n)}_{m}) $ satisfies the (A$ ' $)$(\alpha(m^{\rho})$, $c(m^{\rho}) {\lambda}^n_1(m^{\rho}))$ (resp. (A) $(\alpha(m^{\rho})$; $ \alpha'(m^{\rho}) $, $c(m^{\rho}) {\lambda}^n_1(m^{\rho}))$) condition (see \autoref{defAB}).
		
		Similarly, \emph{(B$ '_1 $) ($ \beta; \lambda_2; c $)} (resp. \emph{(B$ _1 $) ($ \beta; \beta_1, \lambda_2; c $)}) \emph{condition in $ \kappa $-direction at $ K $} can be defined.
		
		\item Let us replace (A3) (a) (i) (ii) by the following (A3$ ' $) (a) (i) (ii).
		
			(A3$ ' $) (a) $ \widehat{H} \approx (\widehat{F}^{cs}, \widehat{G}^{cs}) $ satisfies the (A$ '_1 $)($ \alpha $; $ \lambda_{u} $; $ c $) (B$ _1 $)($ \beta $; $ \beta', \lambda_{cs} $; $ c $) condition in $ cs $-direction at $ K $ (see \autoref{defi:ABk}) with $ \sup_{m} c(m) < \infty $. Moreover,
			
			\noindent (i) (Angle condition) $ \sup_m \alpha(m) \beta'(m) < 1/(2\varsigma_0) $, $ \inf_m \{\beta(m) - \varsigma_0\beta'(m)\} > 0 $;
			
			\noindent (ii) (Spectral condition) $ \sup_m \lambda_{u}(m) < 1 $.
		
		Then all the results in \autoref{thm:I}, case (1), hold. In addition, if we replace the spectral gap condition $ \sup_{m\in K} \lambda_{cs}(m) \lambda_u(m) \vartheta(m) < 1 $ (in (A4) (ii)) by $ \sup_{m\in K} \lambda_{cs}(m) \lambda_u(m) < 1 $, then the results in \autoref{thm:smooth}, case (1), still hold.
		
		In fact, in this case $ \overline{H}_{\widehat{m}_0} \triangleq \widehat{I}_{\widehat{m}_0} \circ \widehat{H}_{\widehat{m}_0} $ satisfies the (A$ ' $) ($ \widetilde{\alpha}(m^{\rho}) $, $ c(m^{\rho})\widetilde{\lambda}^{n}_{u}(m^{\rho}) $) (B) ($ \widetilde{\beta}(m^{\rho}) $; $ \widetilde{\beta}'(m^{\rho}) $, $ c(m^{\rho})\widetilde{\lambda}^{n}_{cs}(m^{\rho}) $) condition, where $ m_0 = u^{\rho}(m^{\rho}) = \phi(\widehat{m}_0) \in K $. First choose $ n \in \mathbb{N} $ large such that
		\[
		\hat{c}\{\sup_m \widetilde{\lambda}_{u}(m)\}^n  < 1 - \sup_{m}\widetilde{\alpha}(m) \widetilde{\beta}'(m),
		\]
		where $ \hat{c} = \sup_{m} c(m) $. Now the graph transform $ \varGamma $ (defined in \autoref{sub:graph}) satisfies $ \varGamma^{n} \varSigma_{lip, \mu, K_1} \subset \varSigma_{lip, \mu, K_1} $ and $ \lip \varGamma^{n} < 1 $.
		
		Note that if we consider $ \widehat{H}^{(n)}_{m} $ for large $ n $ instead of $ \widehat{H}_{m} $, then we can assume that $ \inf_{m}\{\beta(m) - \beta'(m) \} $ is large (see e.g. \autoref{lem:a4} \eqref{it:ab1}); but in this case, what we obtain is the invariance under $ H^{n} $ (not for $ H $).
		
		\item Similarly, if (A3$ ' $) holds with $ \varsigma_0 \geq 1 $ and \eqref{equ:est} holds with $ 0 < \gamma^{*}_{u} < 1 $ and small $ \gamma_0 > 0 $, then the results in \autoref{thm:smooth}, case (2), are still true.
		
		\item The same remark can be made for \autoref{thm:invariant} and \autoref{thm:tri0}.
	\end{asparaenum}
\end{rmk}

\chapter{Invariant manifolds of approximately invariant sets} \label{sec:whitney}

In this chapter, we study the existence and smoothness of invariant manifolds for approximately invariant sets, which were also obtained by Chow, Liu and Yi for ODEs in $ \mathbb{R}^n $ (see \cite{CLY00a}), and by Bonatti and Crovisier for diffeomorphisms on smooth Riemannian manifolds (see \cite{BC16}). The main tool is the geometric version of the Whitney extension theorem in infinite dimensions.

\vspace{.5em}
\noindent{\textbf{Notations.}}
Throughout this chapter, we use the following notations:

\begin{enumerate}[$ \bullet $]
	\item $ \supp f = \overline{\{ x : f(x) \neq 0 \}} $ if $ f: M \to Y $ where $ M $ is a topological space;

	\item when $ Y $ is a linear space, for $ A_1 \subset Y $, the convex hull of $ A_1 $ is defined by
	\[
	\mathrm{co} (A_1) = \left\{ \sum_{i=1}^{n} a_i x_i: x_i \in A_1, a_i \in [0,1], \sum_{i=1}^{n} a_i = 1, n \in \mathbb{N} \right\};
	\]

	\item for $ W \subset M $, write $ W^{\complement} = M \setminus W = \{ x \in M: x \notin W \} $;

	\item $ V \Subset W $ if $ \overline{V} \cap \overline{W^{\complement}} = \emptyset $ where $ V, W $ are subsets of a topological space;

	\item $ O_{\epsilon}(A) = \{ x \in M: d(x, A) < \epsilon \} $ if $ A $ is a subset of a metric space $ M $ with metric $ d $;

	\item $ \mathbb{B}_{\epsilon}(x) = \mathbb{B}_{\epsilon}(\{x\}) $, and we also write $ \mathbb{B}_{\epsilon} = \mathbb{B}_{\epsilon}(0) $.
\end{enumerate}

\section{Smooth analysis in Banach spaces: background}

Let us review some results about approximation and extension problems in the vector-valued setting. The first is Dugundji's generalization of the Tietze extension theorem for continuous vector-valued functions; the following is a special case of this well known result (see \cite[Theorem 6.1]{Dug66}).
\begin{thm}[Dugundji extension theorem]\label{thm:dug}
	 Let $ M $ be a metric space and $ Y $ a normed space. Assume $ A \subset M $ is closed and $ f: A \to Y $ is continuous. Then there is a continuous function $ g : M \to Y $ such that $ g|_{A} = f $ and $ g(M) \subset \mathrm{co} f(A) $.
\end{thm}

Next, consider the smooth approximation of continuous functions, which relies on the existence of smooth partitions of unity.

Let $ M $ be a $ C^{k} $ manifold locally modeled on a Banach space $ X $, where $ k \in \mathbb{N} $ or $ k = \infty $. A \emph{$ C^{k} $ partition of unity} on $ M $ is a collection $ \{ (V_\gamma, \phi_{\gamma}) \} $ such that
\begin{enumerate}[(i)]
	\item $ \{V_\gamma\} $ is a \emph{locally finite open covering} of $ M $, i.e., $ V_\gamma $ is open, $ \bigcup_{\gamma} V_\gamma = M $, and for any $ m \in M $, there is a neighborhood $ U $ of $ m $ such that $ U \cap V_\gamma = \emptyset $ except for finitely many $ \gamma $;
	\item $ \phi_{\gamma} \in C^{k}(M, \mathbb{R}_+) $ and $ \supp \phi_{\gamma} \subset U_{\gamma} $ for all $ \gamma $;
	\item $ \sum_{\gamma} \phi_{\gamma} (m) = 1 $ for all $ m \in M $.
\end{enumerate}
We say $ M $ admits \emph{$ C^{k} $ partitions of unity} if for any $ C^{k} $ atlas $ \{(U_{\alpha}, \varphi_{\alpha})\} $ of $ M $, if there is a $ C^{k} $ partition of unity $ \{ (V_\gamma, \phi_{\gamma}) \} $ such that for each $ \gamma $, $ V_{\gamma} \subset U_{\alpha(\gamma)} $ for some $ \alpha(\gamma) $; sometimes we also say $ \{ (V_\gamma, \phi_{\gamma}) \} $ is \emph{subordinate} to the open cover $ \{ U_{\alpha} \} $. We collect some classical facts about manifolds admitting $ C^k $ partitions of unity; see \cite[Section 5.5]{AMR88} and \cite[Section 7.5]{HJ14}.
\begin{thm}
	\begin{enumerate}[(1)]
		\item (R. Palais) For a paracompact $ C^{k} $ manifold $ M $ locally modeled on a Banach space $ X $, $ M $ admits $ C^k $ partitions of unity if and only if $ X $ admits $ C^k $ partitions of unity.

		\item (H. Toru\'{n}czyk) If $ X $ is weakly compactly generated (e.g., $ X $ is separable or reflexive) and $ X $ admits a $ C^{k} $ bump function, then $ X $ admits $ C^k $ partitions of unity. In particular, (R. Bonic and J. Frampton) if $ X^* $ is separable, then $ X $ admits $ C^1 $ partitions of unity; (H. Toru\'{n}czyk) if $ X $ is a Hilbert space, then $ X $ admits $ C^\infty $ partitions of unity.
	\end{enumerate}
\end{thm}

The following is a well known result concerning the $ C^{k} $ approximation of continuous vector-valued functions; here, we give a slightly generalized version for our purpose.

\begin{thm}\label{thm:C1app0}
	Let $ M $ be a $ C^k $ manifold admitting $ C^k $ partitions of unity and $ Y $ a normed space. If
	\begin{enumerate}[(a)]
		\item $ f: M \to Y $ is continuous and there are two open subsets $ V \Subset \widetilde{V} \subset M $ such that $ f $ is $ C^k $ in $ \widetilde{V} $,
		\item $ O \Subset \widetilde{O} \subset M $ where $ O, \widetilde{O} $ are open, 
		\item $ \delta(\cdot): M \to (0, \infty) $ is any continuous function;
	\end{enumerate}
	then there is a continuous map $ g: M \to Y $ such that
	
	\begin{enumerate}[(1)]
		\item $ g $ is $ C^k $ in $ \widetilde{V} \cup O $;
		
		\item $ g|_{V \cup \widetilde{O}^{\complement}} = f|_{V \cup \widetilde{O}^{\complement}} $;
		
		\item $ |g(x) - f(x)| < \delta(x) $ for all $ x \in M $; 
		
		\item $ g(M) \subset \mathrm{co} (f(M)) $.
	\end{enumerate}
\end{thm}

\begin{proof}
	The proof is straightforward. For each $ x $, since $ \delta(x) > 0 $ is continuous at $ x $, there exists an open neighborhood $ U_{x} $ of $ x $ such that $ \delta(x) / 2 < \delta(y) $ for all $ y \in U_{x} $; in addition, as $ f $ is continuous, we can further assume $ |f(y) - f(x)| < \delta(x) / 2 $ and so $ |f(y) - f(x)| < \delta(y) $ for all $ y \in U_{x} $. As $ \{ U_{x} \}_{x \in M} $ is an open covering of $ M $, by the assumption on $ M $, there is a $ C^k $ partition of unity $ \{ (V_{\gamma}, \phi_{\gamma}) \} $ subordinate to this open covering. For each $ \gamma $, choose one $ x_{\gamma} $ such that $ V_{\gamma} \subset U_{x_{\gamma}} $. Define
	\[
	f_{1}(y) = \sum_{\gamma} \phi_{\gamma}(y) f(x_{\gamma}), \quad y \in M.
	\]
	Then $ f_{1} \in C^k $, $ f_1(M) \subset \mathrm{co} (f(M)) $, and
	\[
	|f_1(y) - f(y)| \leq \sum_{\gamma} \phi_{\gamma}(y)|f(x_{\gamma}) - f(y)| < \sum_{\gamma} \phi_{\gamma}(y) \delta(y) = \delta(y).
	\]
	Let $ \theta_i: M \to [0,1] $, $ i = 1,2 $, be $ C^k $ functions such that $ \theta_1(V) \equiv 1 $, $ \theta_1(\widetilde{V}^{\complement}) \equiv 0 $, and $ \theta_2(O) \equiv 0 $, $ \theta_2(\widetilde{O}^{\complement}) \equiv 1 $; see e.g. \cite[Proposition 5.5.8]{AMR88}. Define
	\[
	g(x) = (1 - \theta_2(x))( \theta_1(x)f(x) + (1 - \theta_1(x)) f_1(x) ) + \theta_2(x) f(x), \quad x \in M,
	\]
	which gives the desired properties (1)--(4). The proof is complete.
\end{proof}

Finally, let us focus on the $ C^1 $ extension problem in infinite dimensions, i.e., the Whitney extension problem in Banach spaces. The Whitney extension theorem in $ \mathbb{R}^n $ (see \cite{Whi34}) with later generalizations due to G. Glaeser, E. Bierstone, P. D. Milman, C. Fefferman, etc., is a celebrated result in mathematical analysis. However, such Whitney extension theorem generally fails for $ C^3 $ (or $ C^{2,1} $) functions even in separable Hilbert spaces (see \cite{Wel73}). A positive result for $ C^1 $ functions (with bounded derivatives) was recently obtained by M. Jim\'enez-Sevilla and L. S\'anchez-Gonz\'alez \cite{JS13} based on techniques due to D. Azagra, R. Fry and L. Keener \cite{AFK10}. This was achieved in \cite{JS13} by introducing \emph{property ($ * $)} of a pair of Banach spaces (see also \autoref{def:property*} for a special case); this property is also demonstrated to be a \emph{necessary} condition.
\begin{defi}\label{def:p*}
	The pair of Banach spaces $ (X, Z) $ is said to have \emph{$ C_* $-property ($ * $)} (or for short \emph{property ($ * $)}) if there is a constant $ C_* \geq 1 $ (depending only on $ X $ and $ Z $) such that for each closed subset $ A \subset X $, each Lipschitz function $ f: A \to Z $, and each $ \varepsilon > 0 $, there is a $ C^1 $ smooth and Lipschitz function $ g: X \to Z $ such that $ |f(x) - g(x)| < \varepsilon $ for all $ x \in A $ and $ \lip g \leq C_* \lip f $.
\end{defi}

It is straightforward to see that if $ (X, Z) $ has $ C_* $-property ($ * $), then $ X $  satisfies the $ C_* $-property ($ *^1 $) in the sense of \autoref{def:property*}; see also \cite[Remark 1.3 (4)]{JS13}.

\begin{exa}\label{exa:p*}
	We list some examples of $ (X, Z) $ having property ($ * $) that are taken from \cite[Section 2]{JS13}.
	\begin{enumerate}[(a)]
		\item Let $ X $ be finite-dimensional and $ Z $ a Banach space. Then $ (X, Z) $ has $ C_* $-property ($ * $) with $ C_* $ depending only on the dimension of $ X $.
		\item Let $ X, Z $ be Hilbert spaces with $ X $ separable. Then $ (X, Z) $ has $ C_* $-property ($ * $) for some fixed $ C_* > 0 $ independent of $ X, Z $. In fact, we can choose $ C_* = 2 + \epsilon $ for any small $ \epsilon > 0 $ (see \autoref{exa:propertyA} and \cite[Example 2.2]{JS13}).
		\item The pairs $ (L_2, L_{p}) $ ($ 1 < p < 2 $) and $ (L_q, L_2) $ ($ 2 < q < \infty $) have $ C_* $-property ($ * $) with $ C_* $ depending only on $ p $ and $ q $, respectively.
		\item Let $ X, Z $ be Banach spaces such that $ X^* $ is separable and $ Z $ is an absolute Lipschitz retract (e.g., $ Z = C(K) $ for some compact metric space $ K $ or a complemented subspace of $ C(K) $). Then $ (X, Z) $ has property ($ * $).
		\item Let $ X, Z $ be Banach spaces such that $ X $ satisfies the property ($ *^1 $) (see \autoref{def:property*}) and $ Z $ is finite-dimensional. Then $ (X, Z) $ has property ($ * $). In particular, if $ X $ is a Hilbert space, then $ (X, \mathbb{R}^n) $ has property ($ * $).
		\item If $ (X, Z) $ has property ($ * $), $ X_1 $ is a subspace of $ X $, and $ Z_1 $ is complemented in $ Z $, then $ (X_1, Z_1) $ also has property ($ * $).
	\end{enumerate}
\end{exa}

The following version of the Whitney extension theorem in Banach spaces is the main result of \cite{JS13} (see Theorems 3.1 and 3.2 therein).

\begin{thm}[M. Jim\'enez-Sevilla and L. S\'anchez-Gonz\'alez]\label{thm:JS}
	Let $ (X, Z) $ be a pair of Banach spaces having $ C_* $-property ($ * $), $ A \subset X $ a closed subset of $ X $, and a function $ f: A \to Z $. Suppose $ f $ is $ C^1 $ at $ A $ in the sense of Whitney, that is, there is a continuous map $ \mathcal{D}: A \to L(X, Z) $ such that for each $ x \in A $ and each $ \varepsilon > 0 $, there is $ r > 0 $ such that
	\[
	|f(y) - f(z) - \mathcal{D}(x)(y - z)| \leq \varepsilon|y - z|, \quad \forall y, z \in A \cap \mathbb{B}_{r}(x).
	\]
	Then there is a $ C^1 $ function $ g: X \to Z $ such that $ g|_{A} = f $ and $ Dg(x) = \mathcal{D}(x) $ for all $ x \in A $.

	Furthermore, if $ C = \sup_{x \in A}\{ |\mathcal{D}(x)| \} < \infty $ and $ f $ is Lipschitz, then we can choose $ g $ such that it additionally satisfies $ \lip g \leq (1 + C_*)(C + \lip f) $.
\end{thm}

The general Whitney extension theorem in Banach spaces for $ C^2 $ (or $ C^{1,1} $) functions remains an open problem. In the following \autoref{sub:AE} and \autoref{sub:geometric}, we will present some generalizations of the above results suited for our purpose. In particular, a geometric version of the Whitney extension theorem in Banach spaces (see \autoref{thm:whitney}) based on \autoref{thm:JS} was originally obtained by \cite{CLY00a, BC16} in the finite-dimensional setting; see also \cite[Corollary 3.3]{OW92} for a similar result.

\section{Approximation and extension between two manifolds: preparation}\label{sub:AE}

In this section, we make the following assumption.

\vspace{.5em}
\noindent{{Assumption.}}
Assume $ M $ is a $ C^1 $ paracompact manifold admitting $ C^1 $ partitions of unity and modeled on a Banach space $ X $. Let $ N $ be any $ C^1 $ (boundaryless) Banach manifold with a (compatible) metric $ d $. Here, note that $ M $ admits a metric.
\vspace{.5em}

A $ C^1 $ local chart $ (V, \varphi) $ of $ N $ is said to be $ C^{0,1} $ with respect to $ d $ if for $ \varphi: V \to \varphi(V) \subset Y_{V} $ (where $ Y_{V} $ is a Banach space with norm $ |\cdot| $), there is a constant $ C_{\varphi} > 0 $ satisfying
\[\label{equ:lip0}\tag{$ \divideontimes $}
d(\varphi^{-1}(x), \varphi^{-1}(y)) \leq C_{\varphi}|x - y|, \quad x,y \in \varphi(V).
\]
Note that such a local chart exists if $ (V, \varphi) $ is a local chart at $ m \in V $ with $ V $ \emph{small} (i.e., $ V = \varphi^{-1}(\mathbb{B}_{\epsilon}) $ where $ \varphi(m) = 0 $ and $ \epsilon $ is small).

\begin{thm}[$ C^0 $ approximation by $ C^1 $ maps]\label{thm:C1app}
	Let $ f: M \to N $ and $ \delta: M \to (0,\infty) $ be continuous. Assume $ \{(V_i, \varphi_i)\} $ is a collection of $ C^{1} \cap C^{0,1} $ local charts of $ N $ (with $ \varphi_i(V_i) $ convex) such that $ f(M) \subset \bigcup_{i = 1}^{\infty} V_i $. Then there is a $ C^1 $ map $ g: M \to N $ such that $ d(g(x), f(x)) < \delta(x) $ for all $ x \in M $. If, in addition, $ f $ is $ C^1 $ in a neighborhood of a closed subset $ K \subset M $, then we can further choose $ g $ such that $ g|_{K} = f|_{K} $.
\end{thm}
\begin{proof}
	First, we prove the following local extension lemma.
	\begin{lem}\label{lem:ext}
		Take a $ C^{1} \cap C^{0,1} $ local chart $ (V, \varphi) $ of $ N $ with $ \varphi(V) $ convex. Let $ \widehat{O}_1 \Subset \widetilde{O}_1 \subset M $ and $ \widehat{O}_2 \Subset f^{-1}_0(V) $ with $ \widehat{O}_1, \widetilde{O}_1, \widehat{O}_2 $ open. Let $ f_0: M \to N $ be $ C^0 $ such that $ f_0 $ is $ C^1 $ in $ \widetilde{O}_1 $. Then there is a $ C^0 $ map $ \widetilde{f}: M \to N $ such that
		\begin{enumerate}[(i)]
			\item $ d(\widetilde{f}(x), f_0(x)) < \delta(x) $ for all $ x \in M $;
			\item $ \widetilde{f}|_{\widehat{O}_1 \cup (f^{-1}_0(V))^{\complement}} = f_0|_{\widehat{O}_1 \cup (f^{-1}_0(V))^{\complement}} $;
			\item $ \widetilde{f} $ is $ C^1 $ in $ \widetilde{O}_1 \cup \widehat{O}_2 $; and
			\item $ f^{-1}_0(V) \subset \widetilde{f}^{-1}(V) $.
		\end{enumerate}
	\end{lem}
	\begin{proof}
		Let $ U_1 = \widehat{O}_2 \setminus \widetilde{O}_1 $. Then $ \overline{U_1} \cap \overline{\widehat{O}_1} = \emptyset $, and hence there is an open set $ \widetilde{O}_2 \Subset f^{-1}_0(V) $ such that $ U_1 \Subset \widetilde{O}_2 \subset f^{-1}_0(V) $ and $ \widetilde{O}_2 \cap \widetilde{O}_1 = \emptyset $.

		Consider $ \varphi \circ f_0: f^{-1}_0(V) \to \varphi(V) \subset Y_{V} $ where $ Y_{V} $ is a Banach space. By \autoref{thm:C1app0}, there is a continuous function $ \widehat{f}: f^{-1}_0(V) \to Y_{V} $ such that
		\begin{enumerate}[(i$ ' $)]
			\item $ \widehat{f} $ is $ C^1 $ in $ (f^{-1}_0(V) \cap \widetilde{O}_1) \cup U_1 $;
			\item $ \widehat{f}|_{f^{-1}_0(V)\setminus\widetilde{O}_2} = \varphi \circ f_0|_{f^{-1}_0(V)\setminus\widetilde{O}_2} $;
			\item $ |\widehat{f}(x) - \varphi \circ f_0(x)| < \delta'(x) $ for all $ x \in f^{-1}_0(V) $, where $ \delta'(\cdot) $ is continuous such that $ 0 < \delta'(x) <  \delta(x) / C_{\varphi} $ and $ C_\varphi $ is defined by \eqref{equ:lip0}; and
			\item $ \widehat{f}(f^{-1}_0(V)) \subset \mathrm{co} (\varphi (V)) = \varphi (V) $ (as $ \varphi (V) $ is convex).
		\end{enumerate}

		Define
		\[
		\widetilde{f}(x) = \begin{cases}
		\varphi^{-1}\circ \widehat{f} (x), & x \in f^{-1}_0(V), \\
		f_0(x), & x \notin f^{-1}_0(V).
		\end{cases}
		\]
		By (ii$ ' $) and (iv$ ' $), $ \widetilde{f} $ is well defined and continuous, and satisfies (ii). By the choice of $ \delta'(\cdot) $, $ \widetilde{f} $ satisfies (i). (iv) is obviously satisfied by $ \widetilde{f} $ due to (iv$ ' $).

		If $ x \in \widehat{O}_1 $, then $ x \in \widetilde{O}^{\complement}_2 $, and so $ \widehat{f}(x) = \varphi \circ f_0(x) $, i.e., (ii) holds.
		Let $ x \in \widetilde{O}_1 $. If $ x \in f^{-1}_0(V) \cap \widetilde{O}_1 $, then by (i$ ' $), $ \widetilde{f} $ is $ C^1 $ at $ x $; if $ x \in \widetilde{O}_1 \setminus f^{-1}_0(V) $, then $ \widetilde{f}(x) = f_0(x) $.
		Let $ x \in \widehat{O}_2 $. Then $ x \in \{\widehat{O}_2 \setminus \widetilde{O}_1\} \cup \{\widehat{O}_2 \cap \widetilde{O}_1\} \subset U_1 \cup \{f^{-1}_{0}(V) \cap \widetilde{O}_1\} $. So by (i$ ' $), $ \widetilde{f} $ is $ C^1 $ at $ x $. This establishes (iii) and completes the proof.
	\end{proof}

	Now we turn to the proof of \autoref{thm:C1app}. Assume $ f $ is $ C^1 $ in a neighborhood of a closed subset $ K \subset M $.
	Set $ W_i = f^{-1}(V_i) $. Take $ O_{i,1}  \Subset \cdots \Subset O_{i,n} \Subset W_{i} $ such that $ \bigcup_{i = 1}^{\infty} O_{i,n} = W_{i} $; for instance, let $ M $ be equipped with a metric $ d_{M} $ and $ O_{i, n} = \{ x \in W_{i} : \mathbb{B}_{1/2^{n}}(x) \subset W_{i} \} $ (as $ W_{i} $ is open, $ O_{i, n} \neq \emptyset $ for large $ n $, and so without loss of generality, assume $ O_{i, n} \neq \emptyset $ for all $ n $; also note that $ \inf\{ d_{M}(x, y): x \in O_{i,n}, y \in O^{\complement}_{i, n + 1} \} \geq 1/2^{n+1} $).
	Let
	\[
	O^*_1 = O_{1,1}, ~O^*_2 = O_{1,2} \cup O_{2,2}, \cdots, ~ O^*_{n} = O_{1,n} \cup O_{2,n} \cup O_{3,n} \cup \cdots \cup O_{n,n}.
	\]
	Note that $ O^*_{n-1} \subset O^{*}_{n} $ and $ \bigcup_{i = 1}^{\infty} O^*_{i} = M $.

	For $ g_0 \in C^0(M, N) $, we say $ g_0 $ satisfies property ($ \boxtimes $) if $ W_{i} \subset g^{-1}_0(V_{i}) $ for all $ i $. We say $ g_0 $ is $ C^1 $ in $ \overline{O} $ if it is $ C^1 $ in an open set $ O' $ satisfying $ O \Subset O' $.

	For $ f $ and $ (V_1, \varphi_1) $, by \autoref{lem:ext}, we have $ f_1 \in C^0(M, N) $ such that 
	
	\begin{enumerate}[({i}1)]
		\item $ f_1 $ satisfies property ($ \boxtimes $); 
		
		\item $ f_1 $ is $ C^1 $ in $ \overline{O^*_1} $;
		
		\item  $ f_1|_{K} = f|_{K} $;
		
		\item $ d(f_1(x), f(x)) < \delta(x) $ for all $ x \in M $.
	\end{enumerate}
	Here, to verify (i1), note that by the construction of $ f_1 $, it satisfies $ W_1 \subset f^{-1}_1(V_1) $ and $ f_{1}|_{W_1^{\complement}} = f|_{W_1^{\complement}} $.

	Similarly, for $ f_1 $ and $ (V_1, \varphi_1) $, by \autoref{lem:ext}, we have $ f_{1,1} \in C^0(M, N) $ such that 
	\begin{enumerate}[({i}1)]
		\item $ f_{1,1} $ satisfies property ($ \boxtimes $); 
		
		\item $ f_{1,1} $ is $ C^1 $ in $ \overline{O^*_1 \cup O_{1,2}} $;
		
		\item $ f_{1,1}|_{K \cup O^*_1} = f_1|_{K \cup O^*_1} $;
		
		\item $ d(f_{1,1}(x), f_1(x)) < \delta(x) - d(f_1(x), f(x)) $ for all $ x \in M $, and so $ d(f_{1,1}(x), f(x)) < \delta(x) $.
	\end{enumerate}

	For $ f_{1,1} $ and $ (V_2, \varphi_2) $, by \autoref{lem:ext}, we have $ f_{2} \in C^0(M, N) $ such that 
	\begin{enumerate}[({i}1)]
		\item $ f_2 $ satisfies property ($ \boxtimes $);
		
		\item $ f_2 $ is $ C^1 $ in $ \overline{O^*_1 \cup O_{1,2} \cup O_{2,2}} = \overline{O^*_2} $;
		
		\item $ f_{2}|_{K \cup O^*_1} = f_{1,1}|_{K \cup O^*_1} $ and hence $ f_{2}|_{K \cup O^*_1} = f_{1}|_{K \cup O^*_1} $;
		
		\item $ d(f_{2}(x), f(x)) < \delta(x)  $ for all $ x \in M $.
	\end{enumerate}

	Inductively, for $ f_{n} $ and $ (V_{1}, \varphi_1) $, by \autoref{lem:ext}, we obtain $ f_{n, 1} \in C^0(M, N) $ 
	\begin{enumerate}[({i}1)]
	\item $ f_{n, 1} $ satisfies property ($ \boxtimes $);
	
	\item $ f_{n, 1} $ is $ C^1 $ in $ \overline{O^*_n \cup O_{1,n+1}} $;
	
	\item $ f_{n,1}|_{K \cup O^*_{n}} = f_{n}|_{K \cup O^*_{n}} $;
	
	\item $ d(f_{n,1}(x), f(x)) < \delta(x)  $ for all $ x \in M $.
	\end{enumerate}	
	Now we can construct $ f_{n, i} $ from $ f_{n,i-1} $ and $ (V_{i}, \varphi_i) $, $ i = 2,3,\ldots,n $. For $ f_{n,n} $ and $ (V_{n+1}, \varphi_{n+1}) $, we further obtain $ f_{n+1} \in C^0(M, N) $ such that
	\begin{enumerate}[({i}1)]
		\item $ f_{n+1} $ satisfies property ($ \boxtimes $);
		
		\item $ f_{n+1} $ is $ C^1 $ in $ \overline{O^*_{n+1}} $;
		
		\item $ f_{n + 1}|_{K \cup O^*_{n}} = f_{n}|_{K \cup O^*_{n}} $;
		
		\item $ d(f_{n+1}(x), f(x)) < \delta(x)  $ for all $ x \in M $.
	\end{enumerate}	

	Therefore, from the construction of $ f_{n} $ and its properties, we can define $ g(x) = f_{n}(x) $ if $ x \in O^*_{n} $. By (i3), $ g $ is well defined and $ g|_{K} = f|_{K} $; by (i2), $ g $ is $ C^1 $ in $ O_{n}^* $; by (i4), $ d(g(x), f(x)) < \delta(x) $. The proof is complete.
\end{proof}
\begin{rmk}\label{rmk:addcr}
	If $ M $ admits $ C^r $ partitions of unity, then we can choose $ g \in C^r $; in fact, if, in addition, $ f $ is $ C^1 $ in a neighborhood of a closed subset $ K \subset M $, then we can choose $ \widetilde{g} \in C^r (M \setminus K, N) \cap C^{1}(M, N) $ such that $ \widetilde{g}|_{K} = f|_{K} $ and $ D\widetilde{g}|_{K} = Df|_{K} $.
\end{rmk}
\begin{proof}
	We only consider the latter statement. Assume $ f $ is $ C^1 $ in $ O_0 $, where $ O_0 $ is open such that $ K \subset O_0 $.
	First, let $ N = Y $ where $ Y $ is a Banach space. Choose open sets $ K_{n+1} \Subset K_{n} \Subset O_0 $ ($ n = 0, 1, 2, \ldots $) such that $ K = \bigcap_{n \geq 0} K_{n} $ and let $ \Omega_{n} = M \setminus \overline{K_{n-1}} $. Now we can choose $ g_n \in C^{k}(\Omega_{n+1}, Y) \cap C^1(X, Y) $ such that $ g_{n}|_{K_{n+1}} = f|_{K_{n+1}} $ and $ g_{n}|_{\Omega_{n-1}} = g_{n-1}|_{\Omega_{n-1}} $. Define $ \widetilde{g} (x) = g_{n}(x) $ if $ x \in \Omega_{n} \cup K $. This is well defined and $ \widetilde{g}(x) = g(x) $, $ D\widetilde{g}(x) = Dg(x) $ if $ x \in K $; moreover, if $ x \in M \setminus K $, there is $ n $ such that $ x \in \Omega_{n} $ and hence $ \widetilde{g} $ is $ C^k $ at $ x $. Now consider the general case where $ N $ is an arbitrary manifold; we argue similarly to the proof of \autoref{thm:C1app} by applying the above construction for the case $ N = Y $.
\end{proof}

Here are some consequences. The special case where $ M $ is compact was also discussed in \cite[Theorem 6.9]{BLZ08}.
\begin{cor}\label{cor:C1app}
	Let $ f \in C^0(M, N) $. Assume one of the following conditions holds: 
	\begin{enumerate}[(a)]
		\item $ f(M) $ is separable (a special case is $ M $ or $ N $ is separable);
		\item $ f(M) $ is relatively compact or $ \sigma $-compact.
	\end{enumerate}
	Then the conclusion in \autoref{thm:C1app} holds.

	In particular, if $ K \subset M $ is relatively compact, then there is a small $ \epsilon > 0 $ such that $ f $ in $ O_{\epsilon}(K) $ can be approximated by functions in $ C^{1}(O_{\epsilon}(K), N) $.
\end{cor}
\begin{proof}
	If $ f(M) $ is separable, i.e., $ \overline{f(M)} = \overline{\{x_{i}\}} $, then for every $ x_{i} $, we choose a $ C^{1} $ local chart $ \varphi_{i} $ such that $ \varphi_{i}(x_i) = 0 $. Let $ V_{i} = \varphi^{-1}_{i} (\mathbb{B}_{\epsilon}) $ for sufficiently small $ \epsilon $ (depending on $ x_i $) such that $ \varphi_{i} $ is $ C^{0,1} $ in $ V_{i} $ (as $ \varphi_{i} $ is $ C^1 $). Then $ f(M) \subset \bigcup_{i = 1}^{\infty} V_{i} $.

	For every $ x \in \overline{f(M)} $, one can find a $ C^{0,1} \cap C^1 $ local chart $ (V_{x}, \varphi_{x}) $ such that $ \varphi_{x}(x) = 0 $ and $ \varphi_{x}(V_{x}) = \mathbb{B}_{\epsilon_x} $. If $ \overline{f(M)} $ is compact, then there are $ x_{1}, x_2, \ldots, x_n $ such that $ f(M) \subset \bigcup_{i=1}^{n}V_{x_{i}} $. If $ \overline{f(M)} $ is $ \sigma $-compact, i.e., $ \overline{f(M)} = \bigcup_{n = 1}^{\infty} M_{n} $ where $ M_{n} $ ($ 1 \leq n < \infty $) are compact, then for each $ M_n $, we have $ M_n \subset \bigcup_{i=1}^{s_{n}}V_{x_{i, n}} $ and so $ f(M) \subset \bigcup_{n = 1}^{\infty}\bigcup_{i=1}^{s_{n}}V_{x_{i, n}} $. The proof is complete.
\end{proof}

Finally, let us consider the continuous extension of a function with the range taken in a manifold.

\begin{thm}[$ C^0 $ extension]\label{thm:extc0}
	Let $ M $ be a metric space and $ N $ a $ C^0 $ Banach manifold. Assume $ A \subset M $ is closed and $ f: A \to N $ is continuous with $ f(A) $ relatively compact. Then there are an open set $ \Omega $ and a continuous function $ \widetilde{f}: \Omega \to N $ such that $ A \subset \Omega $ and $ \widetilde{f}|_{A} = f|_{A} $. If, in addition, $ A $ is compact, then we can take $ \Omega = O_{\epsilon}(A) $ for some small $ \epsilon > 0 $.
\end{thm}
\begin{proof}
	Let $ \mathcal{P} $ denote the set such that $ (g, U) \in \mathcal{P} $ if and only if $ g: M \to N $ is $ C^0 $ in a neighborhood of $ U \cap A $, $ U $ is an open subset of $ M $ and $ g(x) = f(x) $ for all $ x \in U \cap A $. We first show the following.
	\begin{slem}\label{slem:localc0}
		For $ (f_1, U) \in \mathcal{P} $, a $ C^0 $ local chart $ (V, \varphi) $ of $ N $ with $ \varphi(V) $ convex, $ O_2 \cap A \Subset f^{-1}(V) $ and $ O_1 \Subset U $ with $ O_2, O_{1} $ open, there is $ (\widetilde{f}, \widetilde{U}) \in \mathcal{P} $ such that $ O_1 \cup O_2 \Subset \widetilde{U} $.
	\end{slem}
	\begin{proof}
		To simplify our writing, in the following, denote by $ O^{\rhd} (\Omega) $ an open set such that $ \Omega \Subset O^{\rhd} (\Omega) $.

		Let $ O_{1,0}, O_{2,0} $ be open such that $ O_{1} \Subset O_{1,0} \Subset U $, $ O_2 \Subset O_{2,0} $ and $ O_{2,0} \cap A \Subset f^{-1}(V) $; here, note that $ f^{-1}(V) \cap A $ is open in $ A $.
		Set $ A_0 = \overline{O_{1,0}} \cap \overline{O_{2,0}} \cap A $. If $ A_0 = \emptyset $, then there are two open sets $ O^{\rhd}(\overline{O_{i,0}} \cap A) $ ($ i = 1,2 $) such that
		\[\label{equ:vv}\tag{$ \circledast $}
		O^{\rhd}(\overline{O_{1,0}} \cap A) \cap O^{\rhd}(\overline{O_{2,0}} \cap A) = \emptyset, ~ O^{\rhd}(\overline{O_{1,0}} \cap A) \Subset U.
		\]
		Note that $ f: \overline{O_{2,0}} \cap A \to V $ and so $ \varphi \circ f: \overline{O_{2,0}} \cap A \to \varphi(V) \subset Y_{V} $ for some Banach space $ Y_{V} $. By the Dugundji extension theorem (see \autoref{thm:dug}), there is a $ C^0 $ function $ f'_2: M \to Y_{V} $ such that
		\[
		f'_2|_{\overline{O_{2,0}} \cap A} = \varphi \circ f|_{\overline{O_{2,0}} \cap A}, ~f'_2(M) \subset \mathrm{co} (\varphi \circ f (\overline{O_{2,0}} \cap A)) \subset \varphi(V) \quad \text{(as $ \varphi(V) $ is convex)}.
		\]
		Let $ \widetilde{f}(x) $ be equal to $ f_1(x) $ if $ x \in U \setminus O^{\rhd}(\overline{O_{2,0}} \cap A) $ and $ (\varphi^{-1} \circ f'_2)(x) $ otherwise. By \eqref{equ:vv}, $ \widetilde{f} $ is $ C^0 $ in $ O^{\rhd}(\overline{O_{1,0}} \cap A) \cup O^{\rhd}(\overline{O_{2,0}} \cap A) $. Let $ \widetilde{U} = O_{1,0} \cup O_{2,0} $. Then $ (\widetilde{f}, \widetilde{U}) \in \mathcal{P} $ and $ O_1 \cup O_2 \Subset \widetilde{U} $.

		Now assume $ A_0 \neq \emptyset $. Since $ (f_1, U) \in \mathcal{P} $, we can assume $ f_1 $ is $ C^0 $ in $ O^{\rhd} (U \cap A) $ and thus $ O^{\rhd} (A_0) $, where $ O^{\rhd} (U \cap A), O^{\rhd} (A_0) $ are two open sets such that $ O^{\rhd} (A_0) \subset O^{\rhd} (U \cap A) $; furthermore, we can let $ O^{\rhd}(A_0) \Subset U $ and $ O^{\rhd}(A_0) \Subset f^{-1}_1(V) $.

		Since $ \{(\overline{O_{2,0}} \cap A) \setminus O^{\rhd}(A_0)\} \cap \{\overline{O_{1,0}} \cap A\} = \emptyset $, we can take two open sets $ O^{\rhd}(\overline{O_{i,0}} \cap A) $ ($ i = 1,2 $) such that
		\[
		\{O^{\rhd}(\overline{O_{2,0}} \cap A) \setminus O^{\rhd}(A_0)\} \cap O^{\rhd}(\overline{O_{1,0}} \cap A) = \emptyset, \quad O^{\rhd}(\overline{O_{1,0}} \cap A) \Subset U;
		\]
		in particular, $ O^{\rhd}(\overline{O_{1,0}} \cap A) \cap O^{\rhd}(\overline{O_{2,0}} \cap A) \subset O^{\rhd}(A_0) $.
		As $ f_1(\overline{O^{\rhd}(A_0)}) \subset V $, by the Dugundji extension theorem (\autoref{thm:dug}), we have a $ C^0 $ map $ \widetilde{f}_1: M \to V $ such that $ \widetilde{f}_1|_{\overline{O^{\rhd}(A_0)}} = f_1|_{\overline{O^{\rhd}(A_0)}} $. Take an open set $ O^{\rhd}(O_{2,0}) $ such that $ O^{\rhd}(O_{2,0} \cap A) \subset O^{\rhd}(O_{2,0}) $ and $ f(\overline{O^{\rhd}(O_{2,0})} \cap A) \subset V $ (due to $ O_{2,0} \cap A \Subset f^{-1}(V) $).
		Again applying the Dugundji extension theorem (\autoref{thm:dug}), we obtain $ \widetilde{f}_2 : M \to V $ such that $ \widetilde{f}_2|_{\overline{O^{\rhd}(O_{2,0})} \cap A} = f|_{\overline{O^{\rhd}(O_{2,0})} \cap A} $.

		Take further two open sets $ O^{\rhd}_i(\overline{O_{2,0}} \cap A) $ ($ i = 1,2 $) such that $ \overline{O_{2,0}} \cap A \Subset O^{\rhd}_1(\overline{O_{2,0}} \cap A) \Subset O^{\rhd}_2(\overline{O_{2,0}} \cap A) \Subset O^{\rhd}(O_{2,0} \cap A) $; let $ \theta: M \to [0,1] $ be a $ C^0 $ function such that $ \theta|_{O^{\rhd}_1(\overline{O_{2,0}} \cap A)} = 1 $ and $ \theta|_{(O^{\rhd}_2(\overline{O_{2,0}} \cap A))^{\complement}} = 0 $. Define
		\[
		\widehat{f}(x) = \varphi^{-1}((1 - \theta(x)) \varphi \circ \widetilde{f}_1(x) + \theta(x) \varphi \circ \widetilde{f}_2(x)), ~ x \in M,
		\]
		and
		\[
		\widetilde{f}(x) = \begin{cases}
		f_1(x), & x \in O^{\rhd}(\overline{O_{1,0}} \cap A) \setminus O^{\rhd}_2(\overline{O_{2,0}} \cap A), \\
		\widehat{f}(x), & \text{otherwise}.
		\end{cases}
		\]
		Let $ \widetilde{U} = O_{1,0} \cup O_{2,0} $. Finally, we show $ (\widetilde{f}, \widetilde{U}) \in \mathcal{P} $.

		(1) $ \widetilde{f}|_{(O_{1,0} \cup O_{2,0}) \cap A} = f|_{(O_{1,0} \cup O_{2,0}) \cap A} $. To see this, let $ x \in (O_{1,0} \cup O_{2,0}) \cap A $.
		\begin{enumerate}[$ \bullet $]
			\item If $ x \in O_{2,0} \cap A $, then $ \widetilde{f}(x) = \widehat{f}(x) = \widetilde{f}_2(x) = f(x) $; else let $ x \in (O_{1,0} \cap A) \setminus O_{2,0} $.
			\item If $ x \in (O_{1,0} \cap A) \setminus O^{\rhd}_2(\overline{O_{2,0}} \cap A) $, then $ \widetilde{f}(x) = f_1(x) = f(x) $.
			\item Otherwise, $ x \in (O_{1,0} \cap A) \cap O^{\rhd}_2(\overline{O_{2,0}} \cap A) \subset O^{\rhd}(A_0) $ and so $ x \in O^{\rhd}(O_{2,0}) \cap A $; in this case we have $ \widetilde{f}_2(x) = f(x) $ and $ \widetilde{f}_1(x) = f_1(x) = f(x) $, which yields $ \widehat{f}(x) = f(x) $ and then $ \widetilde{f}(x) = f(x) $.
		\end{enumerate}

		(2) $ \widetilde{f} $ is $ C^0 $ in a neighborhood of $ (O_{1,0} \cup O_{2,0}) \cap A $, for instance, in $ O^{\rhd}(\overline{O_{1,0}} \cap A) \cup O^{\rhd}(\overline{O_{2,0}} \cap A) $. If $ x \in \{O^{\rhd}(\overline{O_{1,0}} \cap A) \setminus O^{\rhd}_2(\overline{O_{2,0}} \cap A)\} \cap O^{\rhd}(\overline{O}_{2,0} \cap A) $, then $ x \in O^{\rhd}(A_0) $ and so $ \widehat{f}(x) = \widetilde{f}_1(x) = f_1(x) $ by the construction of $ \widehat{f} $ and $ \widetilde{f}_1 $, which shows $ \widetilde{f} $ is $ C^0 $ at $ x $. This completes the proof of the \autoref{slem:localc0}.
	\end{proof}
	Since $ \overline{f(A)} $ is compact, there exist $ C^0 $ local charts $ \{(V_{i}, \varphi_{i})\} $ of $ N $ such that $ \varphi_{i}(V_{i}) = \mathbb{B}_{\varepsilon_i} $ and $ f(A) \subset \bigcup_{i = 1}^{n} V_{i, 0} $, where $ V_{i, 0} = \varphi^{-1}_{i} (\mathbb{B}_{\varepsilon_i / 4}) $ and $ \varphi^{-1}_{i}(0) \in f(A) $.
	
	Since $ f: A \to N $ is $ C^0 $, there are open sets $ O_i $ and $ U_1 $ such that $ O_{i} \cap A = f^{-1}(V_{i,0}) \cap A $ and $ U_1 \cap A = f^{-1} \circ \varphi^{-1}_{1} (\mathbb{B}_{\varepsilon_1 / 2}) \cap A $.
	
	Applying the Dugundji extension theorem (\autoref{thm:dug}), we get a $ C^0 $ function $ f_1: M \to V_{1} $ such that $ f_1|_{\overline{U_1} \cap A} = f|_{\overline{U_1} \cap A} $. Then $ (f_1, U_1) \in \mathcal{P} $ and $ O_{1} \Subset U_{1} $. By \autoref{slem:localc0}, we have $ (f_2, U_2) \in \mathcal{P} $ such that $ O_{1} \cup O_{2} \Subset U_2 $. Proceeding inductively, we obtain $ (f_{n}, U_{n}) \in \mathcal{P} $ with $ \bigcup_{i = 1}^{n} O_{i} \Subset U_{n} $. Note that $ A \subset \bigcup_{i = 1}^{n} O_{i} $. We see that $ f_{n}|_{A} = f|_{A} $ and $ \widetilde{f} \triangleq f_{n} $ is $ C^0 $ in some neighborhood of $ A $, completing the proof of \autoref{thm:extc0}.
\end{proof}

\begin{rmk}
	(a) In general, one cannot expect that in \autoref{thm:extc0}, $ \widetilde{f} $ is $ C^0 $ on all of $ M $. For example, let $ N = (0,2) \cup (3, 5) $, $ M = (-2,2) $, and $ A = \{ 0,1 \} $; define $ f: A \to N $ such that $ f(0) = 1 $ and $ f(1) = 4 $. Then there is no $ C^0 $ function $ \widetilde{f}: M \to N $ such that $ \widetilde{f}|_{A} = f|_{A} $, since $ M $ is connected but $ \widetilde{f}(M) $ is not.

	(b) In fact, in \autoref{thm:extc0}, what we need is $ f(A) \subset \bigcup_{i = 1}^{n} V_{i} $, where $ (V_i, \varphi_{i}) $ ($ 1 \leq i \leq n < \infty $) are $ C^0 $ local charts of $ N $. It remains unknown whether the theorem holds when $ f(A) \subset \bigcup_{i = 1}^{\infty} V_{i} $.
\end{rmk}

\section{A geometric version of the Whitney extension theorem}\label{sub:geometric}

Assume $ X $ is a Banach space. Let $ U \subset X $, $ m \in U $, and $ X^c \in \mathbb{G}(X) $.
Write $ U(\epsilon) = U \cap \mathbb{B}_\epsilon(m) $. Take a projection $ \Pi^c \in \overline{\Pi}(X) $ such that $ X^c = R(\Pi^c) $. Consider the following conditions:
\begin{enumerate}[($ \bullet $a)]
	\item For any $ \epsilon > 0 $, there exists $ \chi(\epsilon, m) > 0 $ such that
	\[
	\sup\left\{ \frac{|m_1 - m_2 - \Pi^c(m_1 - m_2)|}{|m_1 - m_2|} : m_1 \neq m_2 \in U(\epsilon) \right\} \leq \chi_{U}(\epsilon, m) < 1;
	\]

	\item There exist $ \delta_0(m), \epsilon_m > 0 $ such that
	\[
	X^c(\delta_0(m)) \subset \Pi^c (U(\epsilon_m) - m).
	\]
\end{enumerate}
\begin{defi}\label{def:tangent}
	If ($ \bullet $a) holds with $ \chi_{U}(\epsilon, m) \to 0 $ as $ \epsilon \to 0^+ $, then we say $ X^c $ is a \emph{pre-tangent space} (in the sense of Whitney) of $ U $ at $ m $, denoted by $ T_m U \subset X^c $; if, in addition, ($ \bullet $b) holds, then $ X^c $ is called a \emph{tangent space} of $ U $ at $ m $, denoted by $ T_m U = X^c $, and we say $ U $ is \emph{differentiable} at $ m $ (in the sense of Whitney).

	If $ \phi: \widehat{\Sigma} \to X $ is a $ C^0 $ map where $ \widehat{\Sigma} $ is a topological space, then we also define $ T_{\widehat{m}} \widehat{\Sigma} = T_{\phi(\widehat{m})} \phi(\widehat{U}_{\widehat{m}}) $, where $ \widehat{U}_{\widehat{m}} $ is the component of $ \widehat{\Sigma} $ containing $ \widehat{m} $.
\end{defi}

Note that the (pre-)tangent space $ X^c $ in the above definition does not depend on the choice of the projection $ \Pi^c $. We write $ TU \subset X^{c} $ (resp. $ TU = X^{c} $) if $ T_{m} U \subset X^{c}_{m} $ (resp. $ T_{m} U = X^{c}_{m} $) for all $ m \in U $, where $ X^{c}_{m} \in \mathbb{G}(X) $, $ m \in U $.

\begin{thm}\label{thm:whitney}
	Let $ K \subset X $ be compact and $ \{ \Pi^{c}_{m} \}_{m \in K} \subset \overline{\Pi}(X) $. Set $ \Pi^{h}_{m} = I - \Pi^c_{m} $ and $ X^{\kappa}_{m} \triangleq R(\Pi^{\kappa}_{m}) $, $ \kappa = c,h $. Assume the following conditions hold:
	\begin{enumerate}[(a)]
		\item The map $ m \mapsto \Pi^{c}_{m} $ is continuous and $ T_{m} K \subset X^{c}_{m} $ for all $ m \in K $;

		\item For each $ m \in K $, the pair $ (X^c_{m}, X^{h}_{m}) $ has property (*) (see \autoref{def:p*} and \autoref{exa:p*}).
	\end{enumerate}
	Then the following statements hold.
	\begin{enumerate}[(1)]
		\item There is a $ C^1 $ submanifold $ \Sigma $ of $ X $ such that $ K \subset \Sigma $ and $ T_m \Sigma = X^c_{m} $ for all $ m \in K $. Moreover, there are projections $ \widetilde{\Pi}^c_{m} $, $ m \in \Sigma $, such that $ R(\widetilde{\Pi}^c_{m}) = T_m \Sigma $ and $ m \mapsto \widetilde{\Pi}^c_{m} $ is continuous; in addition, if $ m \in K $ then $ \widetilde{\Pi}^c_{m} = \Pi^c_{m} $.
		\item \label{it:w2} For any $ \epsilon > 0 $, there are projections $ \widehat{\Pi}^c_{m} $ such that $ m \mapsto \widehat{\Pi}^c_{m}: \Sigma \to L(X) $ is $ C^1 $ and $ \sup_{m \in \Sigma}|\widehat{\Pi}^c_{m} - \widetilde{\Pi}^c_{m}| \leq \epsilon $. In particular, there are $ r > 0 $ and $ C_2 > 0 $ such that for any $ m \in K $,
		\[
		|\widehat{\Pi}^c_{m_1} - \widehat{\Pi}^c_{m_2}| \leq C_2 |m_1 - m_2|, \quad m_1, m_2 \in \Sigma \cap \mathbb{B}_{r}(m).
		\]
	\end{enumerate}
\end{thm}

\begin{proof}
	We will adapt some arguments originally due to \cite{CLY00a, BC16}.
	Let
	\[
	C_{1,*} = \sup_{m \in K} \{|\Pi^c_{m}|, |\Pi^h_{m}|\} < \infty.
	\]
	Since $ T_{m} K \subset X^{c}_{m} $ and $ K $ is compact, for each small $ 0 < \varepsilon < 1/2 $, there is $ r = r(\varepsilon) > 0 $ such that for any $ m \in K $ and $ m_1, m_2 \in K \cap \mathbb{B}_{r}(m) $, we have
	\[
	|m_1 - m_2 - \Pi^c_{m}(m_1 - m_2)| \leq \varepsilon|m_1 - m_2| < 1/2 |m_1 - m_2|.
	\]
	Set $ \Omega_{m_0}(r) \triangleq \Pi^c_{m_0}(\overline{\mathbb{B}_{r}(m_0)} \cap K - m_0) \subset X^{c}_{m_0} $, which is a compact set. For each $ m \in \overline{\mathbb{B}_{r}(m_0)} \cap K $, there is unique $ x^c_0 \in \Omega_{m_0}(r) $ and $ \omega_{m_0}(x^c_0) \in X^h_{m_0} $ such that
	\[
	m = m_0 + x^c_0 + \omega_{m_0}(x^c_0).
	\]
	In particular, we obtain
	\[
	|\omega_{m_0}(x^c_1) - \omega_{m_0}(x^c_2)| \leq \frac{\varepsilon}{1 - \varepsilon}|x^c_1 - x^c_2| \leq 2\varepsilon |x^c_1 - x^c_2|, \quad \forall x^c_1, x^c_2 \in \Omega_{m_0}(r).
	\]
	Let
	\[
	A_{m_0}(x^c_0) = (\id|_{X^h_{m_0}} - \Pi^h_{m_0}\Pi^c_{m}\Pi^h_{m_0})^{-1} \circ \Pi^h_{m_0} \Pi^c_{m} \Pi^c_{m_0},
	\]
	where $ m = m_0 + x^c_0 + \omega_{m_0}(x^c_0) $. One can see that $ (\id|_{X^c_{m_0}} + A_{m_0}(x^c_0)) X^{c}_{m_0} = X^c_{m} $. Fix $ r_1 = r(1/4) $.
	
	\begin{slem}\label{slem:ab}
		There are a constant $ C \geq 1 $ (independent of $ \varepsilon \leq 1/4 $ and $ m_0 $) and $ 0 < 2\tilde{r} < r = r(\varepsilon) \leq r_1 $ such that for each $ x^c_0 \in \Omega_{m_0}(r_1) $ and $ x^c_1, x^c_2 \in \Omega_{m_0}(r_1) $ with $ |x^c_i - x^c_0| < \tilde{r} $ ($ i = 1,2 $), the following inequality holds:
		\[
		|\omega_{m_0}(x^c_1) - \omega_{m_0}(x^c_2) - A_{m_0}(x^c_0)(x^c_1 - x^c_2)| \leq C \varepsilon |x^c_1 - x^c_2|.
		\]
	\end{slem}
	\begin{proof}
		Since $ m \mapsto \Pi^c_{m} $ is $C^0$, we may assume that $ |\Pi^c_{m} - \Pi^c_{m_0}| < 1/(4C_{1,*}) $ for all $ m \in K \cap \mathbb{B}_{r_1}(m_0) $ (by further reducing $ r_1 $ if necessary). Consequently, the operator $ \id|_{X^h_{m_0}} - \Pi^h_{m_0}\Pi^c_{m}\Pi^h_{m_0} $ is invertible, and its inverse satisfies $ |(\id|_{X^h_{m_0}} - \Pi^h_{m_0}\Pi^c_{m}\Pi^h_{m_0})^{-1}| \leq 2 $ for each $ m \in K \cap \mathbb{B}_{r_1}(m_0) $.
		
		Write $ m_1 = m_0 + x^c_0 + \omega_{m_0}(x^c_0) $, and
		\[
		m_0 + x^c_i + \omega_{m_0}(x^c_i) = m_1 + \tilde{x}^c_i + \tilde{x}^h_i, \quad i = 1,2,
		\]
		where $ \tilde{x}^\kappa_i \in X^\kappa_{m_1} $ for $ \kappa = c,h $. Then
		\[
		|\tilde{x}^c_i| = |\Pi^c_{m_1}(x^c_i - x^c_0) + \Pi^c_{m_1} (\omega_{m_0}(x^c_i) - \omega_{m_0}(x^c_0))| \leq 2|x^c_i - x^c_0| \leq 2\tilde{r} < r.
		\]
		This gives that $ \tilde{x}^h_i = \omega_{m_1}(\tilde{x}^c_i) $ and $ \tilde{x}^c_i \in \Omega_{m_1}(r) $. Similarly, $ |\tilde{x}^c_1 - \tilde{x}^c_2| \leq 2|x^c_1 - x^c_2| $. Since
		\[
		x^c_1 - x^c_2 + \omega_{m_0}(x^c_1) - \omega_{m_0}(x^c_2) = \tilde{x}^c_1 - \tilde{x}^c_2 + \omega_{m_1}(\tilde{x}^c_1) - \omega_{m_1}(\tilde{x}^c_2),
		\]
		one gets
		\[
		\Pi^h_{m_0}(\tilde{x}^c_1 - \tilde{x}^c_2) - \Pi^h_{m_0}\Pi^c_{m_1} \Pi^h_{m_0} (\omega_{m_0}(x^c_1) - \omega_{m_0}(x^c_2)) - \Pi^h_{m_0} \Pi^c_{m_1} \Pi^c_{m_0} (x^c_1 - x^c_2) = 0,
		\]
		which yields
		\begin{align*}
		&~ |\omega_{m_0}(x^c_1) - \omega_{m_0}(x^c_2) - A_{m_0}(x^c_0)(x^c_1 - x^c_2)| \\
		= &~ |\Pi^h_{m_0}(\tilde{x}^c_1 - \tilde{x}^c_2) + \Pi^h_{m_0} (\omega_{m_1}(\tilde{x}^c_1) - \omega_{m_1}(\tilde{x}^c_2)) - A_{m_0}(x^c_0)(x^c_1 - x^c_2)| \\
		\leq &~ |(\id|_{X^h_{m_0}} - \Pi^h_{m_0}\Pi^c_{m_1}\Pi^h_{m_0})^{-1}||\Pi^h_{m_0} (\omega_{m_1}(\tilde{x}^c_1) - \omega_{m_1}(\tilde{x}^c_2))| \\
		\leq &~ 6\varepsilon|\tilde{x}^c_1- \tilde{x}^c_2| \leq 12 \varepsilon|x^c_1 - x^c_2|.
		\end{align*}
		The proof of \autoref{slem:ab} is complete.
	\end{proof}

	By the above sublemma and the infinite-dimensional Whitney extension theorem due to M. Jim\'enez-Sevilla and L. S\'anchez-Gonz\'alez (see \autoref{thm:JS}), we have a $ C^1 $ map $ \widetilde{\omega}_{m_0}: X^c_{m_0} \to X^{h}_{m_0} $ such that
	\[
	\widetilde{\omega}_{m_0}|_{\Omega_{m_0}(r_1)} = \omega_{m_0}|_{\Omega_{m_0}(r_1)}, \quad D\widetilde{\omega}_{m_0}(x^c_0) = A_{m_0}(x^c_0), \quad x^c_0 \in \Omega_{m_0}(r_1).
	\]
	In particular, $ T_{m} \graph \widetilde{\omega}_{m_0} = X^c_{m} $, $ m \in \overline{\mathbb{B}_{r_1}(m_0)} \cap K $, where
	\[
	\graph \widetilde{\omega}_{m_0} = \{ m_0 + x^c_0 + \widetilde{\omega}_{m_0}(x^c_0): x^c_0 \in X^c_{m_0} \}.
	\]
	As $ x^c_0 \mapsto D\widetilde{\omega}_{m_0}(x^c_0) $ is $ C^0 $, we see there is $ \delta_{m_0} > 0 $ such that
	\[
	|D\widetilde{\omega}_{m_0}(x^c_0)| = |D\widetilde{\omega}_{m_0}(x^c_0) - D\widetilde{\omega}_{m_0}(0)| < 1 / (4C_{1,*}), \quad \forall x^c_0 \in X^c_{m_0}(\delta_{m_0}).
	\]
	
	By the compactness of $ K $, we obtain the following:

	There are $ r > 0 $, points $ m_1, m_2, \ldots, m_s \in K $, and $ C^1 $ maps $ \widetilde{\omega}_{m_i}: X^c_{m_i} \to X^h_{m_i} $ ($ i = 1,2,\ldots,s $) such that
	\begin{enumerate}[(i)]
		\item For $ U_i = \mathbb{B}_{r}(m_i) $,  $ K \subset \bigcup_{i = 1}^{s} \mathbb{B}_{r/2}(m_i) \subset \bigcup_{i = 1}^{s} U_i $;
		\item \label{it:rep} For $ \Sigma_i \triangleq \graph \widetilde{\omega}_{m_i} $,  $ U_i \cap K \subset \Sigma_i $ and $ T_{m}\Sigma_i = X^c_{m} $ for all $ m \in U_i \cap K $;
		\item $ \widehat{d}(T_{m}\Sigma_i, X^c_{m_i}) < 1/(2C_{1,*}) $ (by \autoref{lem:granlip}) for all $ m \in U_i $.
	\end{enumerate}

	Define $ \mathcal{P} $ to be the set of pairs $ (W, \Sigma) $ where $ W $ is an open subset of $ X $, $ \Sigma $ is a $ C^1 $ submanifold of $ X $ satisfying $ K \cap W \subset \Sigma $, and $ T_{m} \Sigma = X^c_{m} $ for all $ m \in K \cap W $.
	
	\begin{lem}\label{lem:con}
		Let $ (U_i, \Sigma_i) $ be given. For any $ (W, \Sigma) \in \mathcal{P} $ and open sets $ \widetilde{U} \Subset U_i $ and $ \widetilde{W} \Subset W $, there is $ (W', \Sigma') \in \mathcal{P} $ such that $ \widetilde{U} \cup \widetilde{W} \Subset W' $.
	\end{lem}
	\begin{proof}
		Set $ \Sigma_{\epsilon} \triangleq O_{\epsilon}(K \cap \overline{\widetilde{W}}) \cap \Sigma $ and $ K_0 = \overline{\widetilde{U}} \cap K \cap \overline{\widetilde{W}} $ (compact).

		First, consider the case where $ K_0 = \emptyset $. Then there is $ \varepsilon > 0 $ such that 
		\[
		O_{\varepsilon}(\overline{\widetilde{U}}) \cap O_{\varepsilon}(K \cap \overline{\widetilde{W}}) = \emptyset \quad \text{and} \quad O_{\varepsilon}(\overline{\widetilde{U}} \cap K) \cap O_{\varepsilon}(\overline{\widetilde{W}}) = \emptyset.
		\]
		Let $ \varepsilon' > 0 $ be sufficiently small such that
		\[
		O_{\varepsilon'}(\widetilde{U}) \cap K \subset O_{\varepsilon}(\overline{\widetilde{U}} \cap K) \quad \text{and} \quad O_{\varepsilon'}(\widetilde{W}) \cap K \subset O_{\varepsilon}(\overline{\widetilde{W}} \cap K).
		\]
		Set
		\[
		\Sigma' = \{O_{\varepsilon}(\overline{\widetilde{U}} \cap K) \cap O_{\varepsilon'}(\widetilde{U}) \cap \Sigma_i\} \cup \{O_{\varepsilon}(K \cap \overline{\widetilde{W}}) \cap O_{\varepsilon'}(\widetilde{W}) \cap \Sigma\}, \quad W' = O_{\varepsilon'}(\widetilde{U}) \cup O_{\varepsilon'}(\widetilde{W}).
		\]
		Then $ (W', \Sigma') \in \mathcal{P} $ and $ \widetilde{U} \cup \widetilde{W} \Subset W' $.

		Next assume $ K_0 \neq \emptyset $. Without loss of generality, we can let $ U_i = m_i + X^c_{m_i}(r) \oplus X^h_{m_i}(r) $. For brevity, we identify $ m_i + X^c_{m_i}(r) \oplus X^h_{m_i}(r) $ with $ X^c_{m_i}(r) \times X^h_{m_i}(r) $.
		
		\begin{slem}\label{slem:rep}
			For any open set $ \widetilde{U} \Subset U_{i} $ and sufficiently small $ \epsilon > 0 $, the intersection $ \widetilde{U} \cap \Sigma_{\epsilon} $ can be represented as a $ C^1 $ graph. That is, there are an open subset $ \Omega_{0,i} \subset X^c_{m_i} $ and a $ C^1 $ map $ \phi_{i}: \Omega_{0,i} \to X^h_{m_i} $ such that $ \widetilde{U} \cap \Sigma_{\epsilon} = \graph \phi_{i}|_{\Omega_{0,i}} $.
		\end{slem}
		\begin{proof}
			For each $ m \in K_0 $, we have $ m \in \Sigma \cap \Sigma_i $. Since $ T_m \Sigma = X^c_m = T_m\Sigma_i $ and $ \widehat{d}(T_{m}\Sigma_i, X^c_{m_i}) < 1/(2C_{1,*}) $ (which implies $ T_{m}\Sigma \oplus X^h_{m_i} = X $ by \autoref{lem:gram2}), there is an open set $ O_{m} \subset \Sigma $ such that $ m \in O_{m} \subset \Sigma \cap W $ and $ O_{m} \subset U_i $. Moreover, $ O_{m} = \graph \phi_{m}|_{\widetilde{\Omega}_{m}} $, where $ \phi_{m}: \widetilde{\Omega}_{m} \subset X^c_{m_i} \to X^h_{m_i} $ is $ C^1 $ and $ \widetilde{\Omega}_{m} $ is open in $ X^c_{m_i} $; here,
			\[
			\graph \phi_{m}|_{\widetilde{\Omega}_{m}} = \{ m_i + x^c_0 + \phi_{m}(x^c_0): x^c_0 \in \widetilde{\Omega}_{m} \}.
			\]
			By \eqref{it:rep}, for each $ m \in K_0 $, there is $ x^c_0 = x^c(m) \in X^c_{m_i} $ such that $ m = m_i + x^c_0 + \widetilde{\omega}_{i}(x^c_0) $.
			
			Observe that $ \{\{x^c(m)\} \times \overline{X^h_{m_i}(r)} \setminus O_m\} \cap K = \emptyset $, since $ U_i \cap K \cap \{x\} \times \overline{X^h_{m_i}(r)} $ contains at most one point. Therefore, there are an open set $ V_m \subset X $ and a small $ \varepsilon(m) > 0 $ such that 
			\[
			\{\{x^c(m)\} \times \overline{X^h_{m_i}(r)}\} \setminus O_m \subset V_m, V_{m} \cap K = \emptyset, \quad \text{and} \quad V_{m} \cap \Sigma_{\varepsilon(m)} = \emptyset.
			\]
			Choose an open set $ O_{c,m} \subset X^c_{m_i} $ such that $ x^c(m) \in O_{c,m} $ and $ O_{c,m} \times \overline{X^h_{m_i}(r)} \subset O_{m} \cup V_{m} $. Then $ \{O_{c,m} \times \overline{X^h_{m_i}(r)}\} \cap \Sigma_{\varepsilon(m)} \subset O_m $, which implies that this intersection is a $ C^1 $ graph: $ \{O_{c,m} \times \overline{X^h_{m_i}(r)}\} \cap \Sigma_{\varepsilon(m)} = \graph \phi_{m}|_{\Omega_{c,m}} $ for some open subset $ \Omega_{c,m} \subset \widetilde{\Omega}_{m} $.

			By the compactness of $ K_0 $, there are an open set $ O_{i} \subset X^c_{m_i} $, $ \varepsilon_i > 0 $, and a finite set $ \{m_{i_j}\} \subset K_0 $ such that
			\begin{enumerate}[(p1)]
				\item $ O_i = \bigcup_{m_{i_j}} O_{c,m_{i_{j}}} $ and $ K_0 \subset O_{i} \times \overline{X^h_{m_i}(r)} $;
				\item $ \{O_{i} \times \overline{X^h_{m_i}(r)}\} \cap \Sigma_{\varepsilon_i} \subset U_{i} $ is a $ C^1 $ graph, i.e., $ \{O_{i} \times \overline{X^h_{m_i}(r)}\} \cap \Sigma_{\varepsilon_i} = \graph \phi_i $ for some $ C^1 $ map $ \phi_{i}: \Omega_{i} \subset X^c_{m_i} \to X^h_{m_i} $ with $ \Omega_{i} $ open in $ X^c_{m_i} $.
			\end{enumerate}
			To show (p2), note that $ \{O_{i} \times \overline{X^h_{m_i}(r)}\} \cap \Sigma_{\varepsilon_i} = \bigcup_{m_{i_j}} \graph \phi_{m_{i_j}}|_{\Omega_{c,m_{i_j}}} $. For any $ m_1, m_2 \in K_0 $, we have
			\[
			\phi_{m_{1}}|_{\Omega_{c,m_{1}}\cap \Omega_{c,m_{2}}}  = \phi_{m_{2}}|_{ \Omega_{c,m_{1}} \cap \Omega_{c,m_{2}} },
			\]
			which ensures the existence of a well defined $ C^1 $ map $ \phi_{i} $ on $ \Omega_i = \bigcup_{m_{i_j}} \Omega_{c,m_{i_j}} $ such that $ \phi_{i}(x) = \phi_{m_{i_j}}(x) $ for $ x \in \Omega_{c,m_{i_j}} $.

			Since $ K_0 \subset O_{i} \times \overline{X^h_{m_i}(r)} $ (see (p1)), we get
			\[
			K \cap \overline{\widetilde{W}} \cap (\overline{\widetilde{U}} \setminus \{O_{i} \times \overline{X^h_{m_i}(r)}\} ) \subset K_0 \setminus \{O_{i} \times \overline{X^h_{m_i}(r)}\} = \emptyset.
			\]
			Thus, there is $ 0 < \varepsilon_{0,i} < \varepsilon_i $ such that $ \Sigma_{\varepsilon_{0,i}} \cap (\overline{\widetilde{U}} \setminus \{O_{i} \times \overline{X^h_{m_i}(r)}\} ) = \emptyset $, and hence
			\[
			\widetilde{U} \cap \Sigma_{\varepsilon_{0,i}} \subset (\{O_{i} \times \overline{X^h_{m_i}(r)}\} \cup \{\overline{\widetilde{U}} \setminus \{O_{i} \times \overline{X^h_{m_i}(r)}\}\}) \cap \Sigma_{\varepsilon_{0,i}} \subset \{O_{i} \times \overline{X^h_{m_i}(r)}\}.
			\]
			This implies that $ \widetilde{U} \cap \Sigma_{\varepsilon_{0,i}} = \graph \phi_{i}|_{\Omega_{0,i}} $ for some open subset $ \Omega_{0,i} \subset \Omega_i $.
			
			This completes the proof of \autoref{slem:rep}.
		\end{proof}

		Let $ \widetilde{U} \Subset X^c_{m_i}(\delta) \times X^h_{m_i}(\delta) $, where $ \delta $ is sufficiently close to but smaller than $ r $. Write $ U_{i,\eta} = X^c_{m_i}(\eta) \times X^h_{m_i}(\eta) $. Take $ \delta < \eta' < \eta < r $ and a $ C^1 $ map $ \theta_i : X^c_{m_i} \to [0,1] $ such that $ \theta_i(x) = 1 $ if $ |x| \leq \eta' $ and $ \theta_i(x) = 0 $ if $ |x| \geq \eta $. By \autoref{slem:rep}, for small $ \epsilon > 0 $, $ U_{i,\eta} \cap \Sigma_{\epsilon} $ can be represented as the graph of a $ C^1 $ map $ \phi_{i}: \Omega_{0,i} \to X^h_{m_i} $. Set
		\[
		\widehat{\phi}_i(x) = \theta_i(x) \widetilde{\omega}_i(x) + (1 - \theta_i(x))\phi_{i}(x), \quad x \in \Omega_{0,i},
		\]
		and
		\[
		\Sigma' = \{\Sigma_{\epsilon} \setminus U_{i,\eta}\} \cup \graph \widehat{\phi}_{i}|_{\Omega_{0,i}} \cup \{\graph \widetilde{\omega}_i \cap U_{i,\eta'}\}.
		\]
		By construction, $ \Sigma' $ is a $ C^1 $ submanifold.
		
		Choose $ \epsilon' > 0 $ sufficiently small so that $ O_{\epsilon'}(\widetilde{W}) \cap K \subset O_{\epsilon}(\overline{\widetilde{W}} \cap K) $, which implies $ O_{\epsilon'}(\widetilde{W}) \cap K \subset \Sigma_{\epsilon} $. Let $ W' = O_{\epsilon'}(\widetilde{W}) \cup U_{i,\delta} $, so that $ \widetilde{W} \cup \widetilde{U} \Subset W' $.
		
		Now consider $ m \in W' \cap K = \{O_{\epsilon'}(\widetilde{W})\cap K\} \cup \{U_{i,\delta} \cap K\} $:
		\begin{enumerate}[$ \bullet $]
			\item If $ m \in U_{i,\eta'} \cap K $ ($ \subset \graph \widetilde{\omega}_i $), then $ m \in \graph \widetilde{\omega}_i \cap U_{i,\eta'} \subset \Sigma' $. So $ T_{m} \Sigma' = T_{m} \graph \widetilde{\omega}_i = X^c_{m} $.

			\item If $ m \notin U_{i,\eta'} \cap K $, then $ m \in (O_{\epsilon'}(\widetilde{W}) \cap K) \setminus U_{i,\eta'} $ (by the choice of $\eta'$), so $ m \in \Sigma_{\epsilon} $. If $ m \in \Sigma_{\epsilon} \setminus U_{i,\eta} $, then $ m \in \Sigma' $ and $ T_{m}\Sigma' = T_{m}\Sigma_{\epsilon} = X^c_{m} $. If $ m \in \Sigma_{\epsilon} \cap U_{i,\eta} $ (i.e., $ m \in \graph \phi_{i}|_{\Omega_{0,i}} $), then $ T_{m}\graph \phi_{i} = X^c_{m} $. Moreover, since $ m \in U_{i,\eta} \cap K $, we have $ m \in \graph \widetilde{\omega}_{i} \cap U_{i,\eta} $ and $ T_{m} \graph \widetilde{\omega}_{i} = X^c_{m} $. Therefore, $ m \in \graph \widehat{\phi}_{i}|_{\Omega_{0,i}} $, i.e., $ m \in \Sigma' $, and $ T_{m} \Sigma' = T_{m}\graph \widehat{\phi}_{i} = X^c_{m} $.
		\end{enumerate}
		This shows that $ (W', \Sigma') \in \mathcal{P} $, completing the proof of \autoref{lem:con}.
	\end{proof}

	\emph{Proof of \autoref{thm:whitney} (continued).} Starting with $ (W_1, \Sigma'_1) \triangleq (U_1, \Sigma_1) \in \mathcal{P} $ and applying \autoref{lem:con} with $ (U_2, \Sigma_2) $, we have $ (W_2, \Sigma'_2) \in \mathcal{P} $ such that $ \widetilde{W}_2 \triangleq \mathbb{B}_{r/2}(m_1) \cup \mathbb{B}_{r/2}(m_2) \Subset W_2 $. Repeating this process for $ (W_2, \Sigma'_2) $ and $ (U_3, \Sigma_3) $ yields $ (W_3, \Sigma'_3) \in \mathcal{P} $ with $ \widetilde{W}_3 \triangleq \widetilde{W}_2 \cup \mathbb{B}_{r/2}(m_3) \Subset W_3 $. After finitely many steps, we have $ (W, \Sigma) \triangleq (W_s, \Sigma'_{s}) \in \mathcal{P} $ such that $ K \subset \bigcup_{i = 1}^{s} \mathbb{B}_{r/2}(m_i) \subset W_{s} = W $. In particular, $ T_m \Sigma = X^c_{m} $, $ m \in K $.

	Next, we establish the existence of the projections $ \widetilde{\Pi}^c_{m} $. Since $ \Sigma $ is a $ C^1 $ submanifold, the map $ \Sigma \to \mathbb{G}(X): m \mapsto T_{m}\Sigma $ is continuous.
	Since the map $ K \to \mathbb{G}(X): m \mapsto X^h_{m} $ is continuous and $ K $ is compact, by \autoref{thm:extc0}, we have a continuous map $ O_{\epsilon}(K) \cap \Sigma \to \mathbb{G}(X): m \mapsto \widetilde{X}^h_{m} $ such that $ \widetilde{X}^h_{m} = X^h_{m} $, $ m \in K $, for some small $ \epsilon > 0 $. Using the compactness of $ K $, by further reducing $ \epsilon $ if necessary, we may assume that for $ m \in \mathbb{B}_{m_0}(\epsilon) \cap \Sigma $ with $ m_0 \in K $, we have $ \widehat{d}(\widetilde{X}^h_{m}, X^h_{m_0}) < 1/(4C_{1,*}) $ and $ \widehat{d}(T_{m}\Sigma, T_{m_0}\Sigma) < 1/(2C_{1,*}) $. This gives that $ T_{m}\Sigma \oplus X^h_{m_0} = X $ (by \autoref{lem:gram2}). Moreover, by \autoref{lem:gram2} \eqref{gram2b}, we see $ \alpha(T_{m}\Sigma, X^h_{m_0}) \geq 1/(2C_{1,*}) $, and so
	\[
	\alpha(X^h_{m_0}, T_{m}\Sigma) \geq \alpha(T_{m}\Sigma, X^h_{m_0}) / 2 \geq 1/(4C_{1,*}) > \widehat{d}(\widetilde{X}^h_{m}, X^h_{m_0}),
	\]
	which again implies that $ T_{m}\Sigma \oplus \widetilde{X}^h_{m} = X $ (by \autoref{lem:gram2}).
	Let us define its corresponding projection $ \widetilde{\Pi}^c_{m} $ such that $ R(\widetilde{\Pi}^c_{m}) = T_{m}\Sigma $ and $ R(\id - \widetilde{\Pi}^c_{m}) = \widetilde{X}^h_{m} $. Because $ m \mapsto T_{m}\Sigma $ and $ m \mapsto \widetilde{X}^h_{m} $ are continuous, we get $ m \mapsto \widetilde{\Pi}^c_{m} $ is continuous at $ O_{\epsilon}(K) \cap \Sigma $. Use $ O_{\epsilon}(K) \cap \Sigma $ instead of $ \Sigma $, completing the proof of conclusion (1).

	Finally, we consider conclusion (2). There is a natural \emph{Finsler structure} on $ T\Sigma $ given by $ |x|_{m} = |x| $, $ x \in T_{m}\Sigma $, $ m \in \Sigma $, which induces the natural \emph{Finsler metric} on $ \Sigma $, denoted by $ d_{\Sigma} $. The metric $ d_{\Sigma} $ satisfies that for any $ \zeta > 0 $, there is a small $ \delta_{\zeta} > 0 $ such that for any $ m \in K $,
	\[
	(1 - \zeta)|m_1 - m_2| \leq d_{\Sigma}(m_1, m_2) \leq (1 + \zeta)|m_1 - m_2|, ~ m_1, m_2 \in \Sigma \cap \mathbb{B}_{\delta_{\zeta}}(m).
	\]
	Note that by \autoref{lem:proj}, we know $ \overline{\Pi}(X) $ is a $ C^\infty $ submanifold of $ L(X) $.
	Since $ K $ is compact, by applying \autoref{cor:C1app} to the map $ m \mapsto \widetilde{\Pi}^c_{m}: \Sigma \to \overline{\Pi}(X) $, we know for any small $ \epsilon_{0,1} > 0 $, there is a small $ \epsilon_{0,2} > 0 $ such that there are projections $ \widehat{\Pi}^c_{m} $ satisfying $ m \mapsto \widehat{\Pi}^c_{m}: \Sigma \cap O_{\epsilon_{0,2}}(K) \to \overline{\Pi}(X) \subset L(X) $ is $ C^1 $ and $ |\widehat{\Pi}^c_{m} - \widetilde{\Pi}^c_{m}| \leq \epsilon_{0,1} $ for all $ m \in \Sigma \cap O_{\epsilon_{0,2}}(K) $. Also, due to the compactness of $ K $, we can take $ \epsilon_{0,2} $ smaller such that $ |D_{m} \widehat{\Pi}^c_{m}| \leq C_3' $ for all $ m \in \Sigma \cap O_{\epsilon_{0,2}}(K) $ and some constant $ C_3 > 0 $, which implies the Lipschitz continuity of $ \widehat{\Pi}^c_{m} $ in $ \Sigma \cap O_{\epsilon_{0,2}}(K) $ with respect to the metric $ d_{\Sigma} $ (see e.g. \cite[Proposition 2.3]{JS11}), and thus in each $ \Sigma \cap \mathbb{B}_{\delta_{\zeta}}(m) $ for $ m \in K $ (with respect to the norm $ |\cdot| $ of $ X $). Taking $ O_{\epsilon_{0,2}}(K) \cap \Sigma $ instead of $ \Sigma $, we complete the proof of \autoref{thm:whitney}.
\end{proof}

\begin{rmk}\label{rmk:Cr}
	If each $ X^c_{m} $ admits $ C^r $ partitions of unity, then in \autoref{thm:whitney}, we can further choose $ \Sigma $ such that $ \Sigma \setminus K \in C^r $. Also, the projections $ \widehat{\Pi}^c_{m} $ can be chosen such that $ m \mapsto \widehat{\Pi}^c_{m} $ is $ C^r $ in $ \Sigma \setminus K $.
\end{rmk}
\begin{proof}
	In the proof of \autoref{thm:whitney}, by \autoref{rmk:addcr}, we can choose $ \widetilde{\omega}_{m_0} \in C^r(X^c_{m_0} \setminus \Omega_{m_0}(r_1), X^h_{m_0}) $. The set $ \mathcal{P} $ is now chosen such that $ (W, \Sigma) \in \mathcal{P} $ if and only if $ W $ is open, $ \Sigma \in C^1 $ and $ \Sigma \setminus K \in C^{r} $, such that $ K \cap W \subset \Sigma $ and $ T_{m} \Sigma = X^c_{m} $, $ m \in K \cap W $. And so in \autoref{slem:rep}, we further get $ \phi_i \in C^r(\Omega_{0,i} \setminus K, X^h_{m_i}) $. In particular, $ \widehat{\phi}_i \in C^r(\Omega_{0,i} \setminus K, X^h_{m_i}) $ since we can let $ \theta_i \in C^r $, and hence $ \Sigma' \setminus K \in C^r $. Once again applying \autoref{rmk:addcr}, we have projections $ \widehat{\Pi}^c_{m} $ with an additional property: $ m \mapsto \widehat{\Pi}^c_{m} $ is $ C^r $ in $ \Sigma \setminus K $.
	This gives the desired result.
\end{proof}

\section{Proof of \autoref{cor:compact}} \label{sub:geo}

\
\

\begin{proof}[Proof of \autoref{thm:geo}]
	By \autoref{thm:whitney}, we know there is a $ C^1 $ submanifold $ \Sigma $ of $ X $ such that $ K \subset \Sigma $ and $ T_{m} \Sigma = X^{c}_{m} $, $ m \in K $. Since $ \Sigma $ is $ C^{1} $, it is also a $ C^{0,1} $ submanifold of $ X $ (see also \cite{Pal66}), i.e. (H1) holds with $ \{ \Pi^c_{m} \} $ and $ \{U_{m}(\epsilon_{m}) = \Sigma \cap \mathbb{B}_{\epsilon_{m}}(m)\} $ (here note that $ \Sigma $ is embedding and so $ \Lambda(m) = \{ m \} $). Since $ K $ is compact, by a standard compactness argument, (H3)--(H4) are satisfied; see also \cite[Lemma 4.4]{BLZ98}. In order to verify (H2), we need to extend and then approximate $ \{ \Pi^{\kappa}_{m} \} $ by $ C^1 $ ones. Let $ \varepsilon > 0 $ be sufficiently small. Applying the same argument as in the proof of \autoref{thm:whitney} \eqref{it:w2}, we obtain projections $ \{ \widetilde{\Pi}^{\kappa}_{m}: m \in \Sigma, \kappa = s, c, u \} $ such that $ \sup_{m \in K}|\widetilde{\Pi}^{\kappa}_{m} - \Pi^{\kappa}_{m}| \leq \varepsilon $, and the map $ m \mapsto \widetilde{\Pi}^{\kappa}_{m}: \Sigma \to L(X) $ is $ C^1 $, $ \kappa = s, c, u $. Now we have (H2) holds for $ \{\widetilde{\Pi}^{\kappa}_{m}: m \in \Sigma\} $ (if necessary, replacing $ \Sigma $ by $ \Sigma \cap O_{\epsilon'}(K) $ with $ \epsilon' $ small), and hence $ (\Sigma, K, \{ \widetilde{\Pi}^{\kappa}_{m} \}_{\kappa = s, c, u}, \{ \Sigma \cap \mathbb{B}_{\epsilon}(m) \}) $ satisfies (H1)--(H4) in \autoref{sub:setup}.
\end{proof}

\begin{proof}[Proof of \autoref{cor:compact}]
	From \autoref{thm:geo}, we see that $ (\Sigma, K, \{ \widetilde{\Pi}^{\kappa}_{m} \}_{\kappa = s, c, u}, \{ \Sigma \cap \mathbb{B}_{\epsilon}(m) \}) $ satisfies (H1)--(H4).
	Now conclusions (1) (2) are consequences of \autoref{thm:tri0} \eqref{it:tri1} \eqref{it:tri04}. Consider item (3) in \autoref{thm:geo}. Since $ H $ satisfies the \emph{strong $ s $-contraction} and \emph{strong $ u $-expansion}, one needs to show (B4) (iv$ ' $) holds, i.e., there is a bump function $ \Psi $ satisfying (a)--(b) in \autoref{sub:LipC1bm} (in this case $ X^s_{m} = \{0\} $ for all $ m \in \Sigma $). Also, since $ K $ is compact, there are finitely many $ C^1 $ local charts with their domains covering $ K $ in $ \Sigma $, and so by \autoref{exa:case2}, it suffices to know that $ X^c_{m} $ satisfies property ($ *^1 $) (without uniformity). But this follows from (Aa1) (ii); that the pair $ (X^c_{m}, X^{h}_{m}) $ has property (*) (see \autoref{def:p*}) obviously implies that $ (X^c_{m}, \mathbb{R}) $ has property (*), i.e., $ X^c_{m} $ satisfies the property ($ *^1 $) (see \autoref{def:property*}). Thus, conclusion (3) (i) is a consequence of \autoref{thm:tri0} \eqref{it:tri0smooth}.

	Let us consider item (3) (ii). Since in this case $ X^c_{m} $ is finite-dimensional, it admits $ C^{\infty} $ partitions of unity. By \autoref{rmk:Cr}, we can assume $ \Sigma \setminus K \in C^{\infty} $ and the maps $ \Sigma \setminus K \to L(X), m \mapsto \widetilde{\Pi}^{\kappa}_{m} $ ($ \kappa = s, c, u $) are $ C^{\infty} $; also $ \Psi \in C^{\infty}(\Sigma \setminus K, [0,1]) $. As $ K $ is compact and $ \Sigma $ is finite-dimensional, we see \autoref{lem:high} holds (see also \autoref{rmk:highL} (I) (b)), which means conclusion (3) (ii) holds. The proof is complete.
\end{proof}

\chapter{Invariant manifolds for $ C^1 $ maps}\label{sec:C1map}

\section{Proof of \autoref{cor:mapT} and \autoref{cor:mapG}}

Since $ u: K \to K $ is $ C^0 $, we know (B2) holds with any small $ \xi_1 > 0 $. Note that $ \mathfrak{A}(\epsilon') \to 0 $ (defined in (IV$ ' $)) as $ \epsilon' \to 0 $ since $ K $ is compact. From \autoref{lem:mapAB} with $ \sup_{m \in K}|\widehat{\Pi}^\kappa_m - \Pi^\kappa_m| \leq \xi_2 $, $ \kappa = s, c, u $, we have (B2) (a$ ' $) (b) hold. Now the conclusions of \autoref{cor:mapT} follow from \autoref{cor:compact}. Here, note that if $ H $ is $ C^{k} $ in $ \mathbb{B}_{\epsilon'} (K) $, which implies $ C^{k-1,1} $ in $ \mathbb{B}_{\epsilon'} (K) $ (when $ \epsilon' $ is small), then for
\[
\widehat{H}_{m} = H(m + \cdot) - u(m) \sim (F^{cs}_{m}, G^{cs}_{m}) : \widehat{X}^{cs}_{m}(r) \oplus \widehat{X}^{u}_{m}(r_1) \to \widehat{X}^{cs}_{u(m)}(r_2) \oplus \widehat{X}^{u}_{u(m)}(r)
\]
we have $ F^{cs}_{m}(\cdot), G^{cs}_{m}(\cdot) \in C^{k} $ and thus $ F^{cs}_{m}(\cdot), G^{cs}_{m}(\cdot) \in C^{k-1,1} $ uniformly for $ m \in K $ since $ K $ is compact.

\autoref{cor:mapG} is a consequence of \autoref{lem:mapAB} and \autoref{thm:tri0}. The proof is complete. \qed

\section{Dynamical characterization of pre-tangent spaces}\label{sub:dyntan}

Let us give a characterization of $ T_{m} K \subset \widehat{X}^{c}_{m} $ (see \autoref{def:tangent}) based on the dynamical property of $ H $.

\vspace{.5em}
\noindent{{Assumptions}}.
\begin{enumerate}[(i)]
	\item Let $ K \subset X $ be compact.
	\item Suppose $ H \in C^{1}(\mathbb{B}_{\epsilon'} (K), X) $ is an invertible map admitting (uniformly) \emph{partial hyperbolicity} at $ K $ (i.e., (IV$ ' $) (a$ ' $) in \autoref{sub:maps} holds, with the additional conditions that $ u = H|_{K} $ (i.e., $ \eta_* = 0 $), $ H(K) = K $, $ \xi_0 = 0 $, and that $ m \mapsto \widehat{\Pi}^{\kappa}_{m} $ is continuous, $ \kappa = s, c, u $).
\end{enumerate}
\vspace{.5em}

Under the above assumptions, there are \emph{strong stable} and \emph{strong unstable laminations} of $ K $, whose leaves through $ m \in K $ are denoted by $ W^{ss}(m) $, $ W^{uu}(m) $ (called the \emph{strong stable} and \emph{strong unstable manifolds} of $ m $, respectively), such that (i) $ W^{\kappa \kappa}(m) \in C^1 $ and $ T_{m}W^{\kappa\kappa}(m) = \widehat{X}^{c\kappa}_{m} $, $ \kappa = s, u $, and (ii) $ H(W^{ss}(m)) \subset W^{ss}(H(m)) $, $ H^{-1}(W^{ss}(m)) \subset W^{ss}(H^{-1}(m)) $; for details, see e.g. \cite{HPS77} or \cite[Section 7.2]{Che18a}.

\begin{lem}[See {\cite[Proposition 3.10]{BC16}}]\label{lem:bc}
	If for all $ m \in K $, $ W^{ss}(m) \cap K = \{m\} $ and $ W^{uu}(m) \cap K = \{m\} $, then $ T_{m} K \subset \widehat{X}^{c}_{m} $ for all $ m \in K $.
\end{lem}

The proof is the same as \cite[Proposition 3.10]{BC16}, and we give a sketch here.

\begin{proof}
	First note the following characterizations of $ W^{ss}(m) $; similar for $ W^{uu}(m) $.

	\begin{slem}\label{slem:stable-m}
		The following characterizations of strong stable manifolds are equivalent:
		\begin{enumerate}[(i)]
		\item $ m' \in W^{ss}(m) $.
		\item For some (or equivalently, for all) $ n \in \mathbb{N} $, $ H^{n}(m') \in W^{ss}(H^{n}(m)) $.
		\item There are $ \varepsilon > 0 $ and $ N \in \mathbb{N} $ such that for all $ n \geq N $,
		\begin{itemize}
		\item $ H^{n}(m') \in \mathbb{B}_{\varepsilon} (H^{n}(m)) $
		\item $ |H^{n}(m') - H^{n}(m)| \leq C_1 \widetilde{\lambda}^{(n)}_{s}(m) $
		\end{itemize}
		where $ \widetilde{\lambda}_{s}(m) = \lambda_{s}'(m) + \varsigma_* < \gamma_{s} < 1 $, $ C_1 > 0 $ is a constant, and $ \varsigma_* > 0 $ is small.
		\item For any $ \beta_0 > 0 $, there are $ \varepsilon > 0 $ and $ N \in \mathbb{N} $ such that for all $ n \geq N $,
		\begin{itemize}
		\item $ H^{n}(m') \in \mathbb{B}_{\varepsilon} (H^{n}(m)) $
		\item In the decomposition $ H^{n}(m') = H^{n}(m) + x^s_{n} + x^{cu}_{n} $ with $ x^{\kappa}_{n} \in X^{\kappa}_{H^{n}(m)} $, we have $ |x^{cu}_{n}| \leq \beta_0|x^s_{n}| $
		\end{itemize}
		(which implies $ |H^{n}(m') - H^{n}(m)| \leq \widetilde{\lambda}^{(n)}_{s}(m) |m' - m| $ by the (B) condition).
		\end{enumerate}
		Here, $\widetilde{\lambda}^{(n)}_{s}(m)$ is defined as:
		\[
		\widetilde{\lambda}^{(n)}_{s}(m) = \widetilde{\lambda}_{s}(m) \widetilde{\lambda}_{s}(H(m)) \cdots \widetilde{\lambda}_{s}(H^{n-1}(m)).
		\]
	\end{slem}
	\begin{proof}
		We only show that (i) implies (iv).
		The other characterizations are well known (see e.g. \cite[Theorem 5.1]{HPS77} or \cite[Section 4.4]{Che18a}).
		Here, we should mention that this characterization only holds for $ H $ admitting (uniformly) partial hyperbolicity at $ K $, not for the general setting in \cite[Section 7.2]{Che18a}. Without loss of generality, let $ N = 0 $. We need to show that $ |x^{cu}_{n}| \leq \beta_0|x^s_{n}| $ implies $ |x^{cu}_{n}| \leq \beta|x^s_{n}| $ for small $ \beta > 0 $, then by the (B) condition (using \autoref{lem:mapAB}), this yields $ |H^{n}(m') - H^{n}(m)| \leq \widetilde{\lambda}^{(n)}_{s}(m) |m' - m| $.

		More precisely, for any $ \beta_0 > 0 $, if $ \epsilon'' $ is small and $ n \in \mathbb{N} $ is large, then
		\[
		H^{n}(m+\cdot) - H^{n}(m): X^{s}_{m}(\epsilon'') \oplus X^{cu}_{m}(\epsilon') \to X^{s}_{H^{n}(m)}(\epsilon') \oplus X^{cu}_{H^{n}(m)} (\epsilon''),
		\]
		satisfies (A) ($ \alpha $, $ \lambda^{(n)}_{cu}(m) $) (B) ($ \beta_0 $; $ \beta $, $ \lambda^{(n)}_{s}(m) $) with $ \alpha, \beta $ small and $ \lambda^{(n)}_{cu}(m), \lambda^{(n)}_{s}(m) $ close to $ 2(\lambda'_{cu})^{(n)}(m) $, $ 2(\lambda'_{s})^{(n)}(m) $.

		In fact, due to $ \sup_{m}\lambda'_s(m) \lambda'_{cu}(m) < 1 $, let $ n $ be large such that 
		\[
		(\sup_{m}\lambda'_s(m) \lambda'_{cu}(m))^{-n/2} > 2\beta_0.
		\]
		Note that there is a small $ \epsilon'' > 0 $ such that $ H^{n}(m+\cdot) - H^{n}(m) \sim (F^{(n)}_{m}, G^{(n)}_{m}) $ satisfies for all $x_1, x'_1 \in X^{s}_{m}(\epsilon''),~ y_2, y'_2 \in X^{cu}_{H^{n}(m)} (\epsilon'')$,
		\begin{gather*}
		|F^{(n)}_{m}(x_1, y_2)-F^{(n)}_{m}(x'_1, y'_2)| \leq \max\{ 2\lambda^{(n)}_s(m) |x_1 - x'_1|,~ \alpha' |y_2 - y'_2| \}, \\
		|G^{(n)}_{m}(x_1, y_2)-G^{(n)}_{m}(x'_1, y'_2)| \leq \max\{ \beta' |x_1 - x'_1|,~ 2\lambda^{(n)}_u(m) |y_2 - y'_2| \},
		\end{gather*}
		where $ \alpha', \beta' $ are small (since $ \xi_0 = 0 $ and $ \epsilon'', \mathfrak{A}(\epsilon'') $ can be small).
		
		Since $ (\sup_{m}\lambda_s(m)\lambda_u(m))^{(n/2)} $ is small, we can let $ \beta' \geq (\sup_{m}\lambda_s(m)\lambda_u(m))^{n/2} $. Now applying \autoref{lem:a4} \eqref{it:ab0}, we get $ |x^{cu}_{0}| \leq \beta'|x^s_{0}| $ and then (by replacing $ m $ with $ H^{k}(m) $) $ |x^{cu}_{k}| \leq \beta'|x^s_{k}| $ for all $ k \geq 0 $. The proof is complete.
	\end{proof}

	We first show $ T_{m} K \subset \widehat{X}^{cu}_{m} $ for all $ m \in K $. Otherwise, there are a constant $ \beta_0 > 0 $, $ m_0 \in K $ and $ m^1_{n}, m^2_{n} \in K $ ($ n \in \mathbb{N} $) such that $ m^1_{n} \neq m^2_{n} $, $ m^1_{n}, m^2_{n} \to m_0 $ and
	\[
	|\widehat{\Pi}^{cu}_{m^1_{n}} (m^1_{n} - m^2_{n})| \leq \beta_0 |\widehat{\Pi}^{s}_{m^1_{n}} (m^1_{n} - m^2_{n})|.
	\]
	Without loss of generality, $ m^{2}_{n} \in \mathbb{B}_{\varepsilon/4}(m^{1}_{n}) $ for all $ n $, where $ \varepsilon $ is given in \autoref{slem:stable-m}.

	By the (B) condition, if for all $ 0 \leq j \leq k $, $ H^{-j}(m^{2}_{n}) \in \mathbb{B}_{\varepsilon}(H^{-j}(m^{1}_{n})) $, then $ |x^{cu}_{n, j}| \leq \beta_0 |x^{s}_{n, j}| $ and $ |H^{-k}(m^{1}_{n}) - H^{-k}(m^{2}_{n})| \geq \gamma^{-k}_{s}|m^{1}_{n} - m^{2}_{n}| $, where $ H^{-j}(m^{2}_{n}) = H^{-j}(m^{1}_{n}) + x^{s}_{n, j} + x^{cu}_{n, j} $ and $ x^{\kappa}_{n, j} \in X^{\kappa}_{u^{-j}(m^{1}_{n})} $.

	As $ m^1_{n}, m^2_{n} \to m_0 $, for any $ n > 0 $, we can choose $ k_{n} > 0 $ such that 
	\[
	H^{-j}(m^{2}_{n}) \in \mathbb{B}_{\varepsilon/2}(H^{-j}(m^{1}_{n})), 0 \leq j \leq k_{n}, \quad \text{and}~ |H^{-k_{n}}(m^{1}_{n}) - H^{-k_{n}}(m^{2}_{n})| \geq \varepsilon / 8,
	\]
	since $ |m^1_{n} - m^{2}_{n}| \neq 0 $ and $ \gamma_{s} < 1 $ by the (B) condition. Note that $ k_{n} \to \infty $ as $ n \to \infty $. Since $ K $ is compact, without loss of generality, assume $ H^{-k_{n}}(m^{i}_{n}) \to m_i \in K $ as $ n \to \infty $. Then $ H^{k}(m_{2}) \in \mathbb{B}_{\varepsilon}(H^{k}(m_1)) $ and $ |\widehat{\Pi}^{cu}_{H^{k}(m_1)} (H^{k}(m_{1}) - H^{k}(m_{2}))| \leq \beta_0 |\widehat{\Pi}^{s}_{H^{k}(m_1)} (H^{k}(m_{1}) - H^{k}(m_{2}))| $ for all $ k \in \mathbb{N} $, which yields $ m_2 \in W^{ss}(m_1) $ and so $ m_1 = m_2 $ by the assumption. However, by our construction, $ |m_1 - m_2| \geq \varepsilon / 8 $, a contradiction.

	Similarly, we have $ T_{m} K \subset \widehat{X}^{cs}_{m} $ for all $ m \in K $ and so $ T_{m} K \subset \widehat{X}^{c}_{m} $, completing the proof of \autoref{lem:bc}.
\end{proof}

Now we can state the following corollary which generalizes the corresponding finite-dimensional result due to Bonatti and Crovisier \cite{BC16}.

\begin{thm}\label{thm:eqInvariant}
	Under assumptions (i)--(ii) and the following:
	(iii) for each $ m \in K $, the pair $ (\widehat{X}^c_{m}, \widehat{X}^{h}_{m}) $ has property (*) (see \autoref{def:p*} and \autoref{exa:p*}; for examples (a) $ X $ is a Hilbert space with $ X^c_{m} $ separable for each $ m \in K $, or (b) $ X^c_{m} $ is finite-dimensional for each $ m \in K $), there is a $ C^1 $ submanifold $ \Sigma^c $ of $ X $ such that it contains $ K $ in its interior, $ T_{m}\Sigma^c = \widehat{X}^{c}_{m} $ for all $ m \in K $, and it is locally invariant under $ H $ if and only if for all $ m \in K $, $ W^{ss}(m) \cap K = \{m\} $ and $ W^{uu}(m) \cap K = \{m\} $.
\end{thm}

See \cite{BC16} for more deep dynamical consequences of the above result in the finite-dimensional setting. The invertibility of $ H $ on $\mathbb{B}_{\epsilon'} (K)$ is not necessary; in fact, what we need is that $ H: K \to K $ is invertible. Moreover, the above result also holds for $ H $ being a correspondence with $ C^1 $ generating map, which is left to the readers.

\section{Dependence on external parameters}

Now we study the dependence on external parameters in a special setting. Under the context in \autoref{cor:mapT}, consider $ \mathbb{H} \to C^{1}(\mathbb{B}_{\epsilon'} (K), X), \lambda \mapsto H_{\lambda} $ with $ H_{0} = H $. Suppose $ \lambda \mapsto H_{\lambda} $ is continuous with respect to the $ C^1 $-topology. If $ |\lambda| $ is small, then the conclusions in \autoref{cor:mapT} also hold when $ H $ is replaced by $ H_{\lambda} $. Moreover, we can choose the corresponding local center manifold of $ K $ for $ H_{\lambda} $, denoted by $ \Sigma^c_{\lambda} $, such that $ \Sigma^c_{\lambda} \to \Sigma^c, T\Sigma^c_{\lambda} \to T\Sigma^c $ as $ \lambda \to 0 $; similar for local center-(un)stable manifolds. These assertions can be established using arguments similar to those in \autoref{sub:C1smooth} and \autoref{sub:highersmooth}.

Next, we give a sufficient condition such that $ \lambda \mapsto \Sigma^c_{\lambda} $ is $ C^1 $ with respect to the $ C^0 $-topology: Let $ \eta_* = 0 $ (for $ H_0 = H $) and assume $ (\lambda, x) \mapsto H_{\lambda}(x) $ is $ C^{1,1} $; furthermore, (i) $ X $ is a Hilbert space with $ X^c_{m} $ separable for each $ m \in K $ and $ \mathbb{H} $ is a separable Hilbert, or (ii) $ X^c_{m} $ is finite-dimensional for each $ m \in K $ and $ \mathbb{H} = \mathbb{R}^{n} $.

To establish this result, consider $ \widetilde{H}(x, \lambda) = (H_{\lambda}(x), \lambda): X \times \mathbb{H} \to X \times \mathbb{H} $, $ \widetilde{K} = K \times \mathbb{H}(\epsilon'') $, and $ \widetilde{\Sigma} = \Sigma \times \mathbb{H}(2\epsilon'') $ where $ \Sigma $ is given by \autoref{thm:geo} and $ \epsilon''>0 $ is small; see also \cite[Section 6.3]{BLZ08} and \cite[Remark 3.13]{Che18b}.
Now the assumptions in \autoref{thm:tri0} are satisfied for $ \widetilde{H} $ and $ \widetilde{K} $, and so we obtain, in a neighborhood of $ \widetilde{K} $, a $ C^{1,u} $ locally invariant center manifold $ \widetilde{\Sigma}^{c} $ of $ \widetilde{K} $ for $ \widetilde{H} $ (see also \autoref{lem:udiff}). Due to the special form of $ \widetilde{H} $, we can express $ \widetilde{\Sigma}^{c} $ as $ \widetilde{\Sigma}^{c} = (\Sigma^c_{\lambda}, \lambda) $. Consequently, $ \Sigma^c_{\lambda} $ is the desired $ C^1 $ center manifold of $ K $ for $ H_{\lambda} $, and $ \lambda \mapsto \Sigma^c_{\lambda} $ is $ C^1 $ (in fact $ C^{1,u} $) with respect to the $ C^0 $-topology.

\begin{appendices}
	\setcounter{equation}{0}
	\renewcommand{\theequation}{\Alph{section}.\arabic{equation}}

\chapter{Appendix. Miscellaneous}

\subsection{} \label{app:mapAB}

\begin{proof}[Proof of \autoref{lem:mapAB}]
	This result was in fact already proved in \cite[Lemma 3.13]{Che18a}. We give a sketch for item (a) here. For $ (x, y) \in \widehat{X}^{cs}_m \oplus \widehat{X}^u_m $, write
	\[
	(f_{m}(x, y), g_{m}(x, y)) = (\widehat{\Pi}^{cs}_{u(m)}\{H(m + x + y) - u(m)\}, \widehat{\Pi}^{u}_{u(m)}\{H(m + x + y) - u(m)\}).
	\]
	Let $ C_1 = \sup \{|\widehat{\Pi}^{\kappa}_{m} |: m \in K, \kappa = s, c, u \} $.
	It is not hard to see that if $ \mathfrak{A}(\epsilon') \leq C^{-1}_1 / 4 $ and $ \eta_* \leq \epsilon' C^{-1}_1 / 8 $, then for $ r = \epsilon' /2 $ and $ |x| \leq r $,
	\[
	G_{m}(x, z) \triangleq g^{-1}_{m}(x, \cdot)|_{X^{u}_{u(m)}(r)} ~\text{exists and}~ g_{m}(x, G_{m}(x, z)) = z, ~ z \in X^{u}_{u(m)}(r);
	\]
	for details, see \cite[Lemma 3.13]{Che18a}.
	Let $ F_{m}(x, z) = f_{m}(x, G_{m}(x, z)) $. Now there are $ r_1, r_2 > 0 $ such that (i)
	\begin{enumerate}[(i)] 
		\item  $\widehat{H}_{m} = H(m + \cdot) - u(m) \sim (F_{m}, G_{m}) : \widehat{X}^{cs}_{m}(r) \oplus \widehat{X}^{u}_{m}(r_1) \to \widehat{X}^{cs}_{u(m)}(r_2) \oplus \widehat{X}^{u}_{u(m)}(r)$, 
		
		\item $ |G_{m}(0, 0)| \leq 2C_1 \eta_* $,
		
		\item for $ |x|, |y| \leq r $,
		\[
		\lip G_{m}(\cdot, y) \leq 2(\epsilon + \xi_0), \quad \lip G_{m}(x, \cdot) \leq (\epsilon + \lambda'_{u}(m))(1 - \epsilon)^{-1},
		\]
		where $ \epsilon = C_1\mathfrak{A}(\epsilon') \leq 1/4 $;
	\end{enumerate}
	
	This also yields
	\begin{gather*}
	\lip F_{m}(\cdot, y) \leq \epsilon + \lambda'_{cs}(m) + 4(\epsilon + \xi_0)^2, \quad 
	\lip F_{m}(x, \cdot) \leq 5(\epsilon + \xi_0),
	\end{gather*}
	and $ |F_{m}(0, 0)| \leq (2C_1 + \epsilon + \xi_0) \eta_* $; in addition, for $ \gamma_0 = 2( C_1\mathfrak{A}(\epsilon') + \xi_0) $,
	\[
	|G_{m}(x, 0)| \leq |G_{m}(x, 0) - G_{m}(0, 0)| + |G_{m}(0, 0)| \leq \gamma_0|x| + |G_{m}(0, 0)|.
	\]
	Now by \autoref{lem:a4} \eqref{it:ab1}, there is a small $ 1 > \omega_0 > 0 $ such that if $ C_1\mathfrak{A}(\epsilon') + \xi_0 \leq \omega_0 / 6 $, then the following (1) and (2) hold; note that in this case, $ \widehat{H}_{m} \sim (F_{m}, G_{m}) $ satisfies the (A$ ' $) ($ \omega_0, (\omega_0 + \lambda'_{u}(m))(1 - \omega_0)^{-1} $) (B$ ' $) ($ \omega_0, \omega_0 + \lambda'_{cs}(m) $) condition.
	\begin{enumerate}[(1)]
		\item There are constants $ \alpha > 0 $ (small), $ 0 < c < 1 $ (depending on $ \sup_{m}\lambda'_{cs}(m)\lambda'_{u}(m) $ with $ c \to 0 $ as $ \sup_{m}\lambda'_{cs}(m)\lambda'_{u}(m) \to 0 $), and
		\[
		\lambda_{cs}(m) = (\omega_0 + \lambda'_{cs}(m)) (1 - \alpha\omega_0)^{-1}, ~\lambda_{u}(m) = (\omega_0 + \lambda'_{u}(m))(1 - \omega_0)^{-1} (1 - \alpha\omega_0)^{-1},
		\]
		such that $ \widehat{H}_{m} \sim (F_{m}, G_{m}) $ satisfies the (A) ($ \alpha; c\alpha, \lambda_{u}(m) $) (B) ($ \alpha; c\alpha, \lambda_{cs}(m) $) condition;
		\item except (A3) (a) (iii), we have (A3) (a$ ' $) (b) (i) hold with
		\[
		\eta = 2C_1\eta_*, \quad \varsigma_0 = 1,\quad \gamma_0 = 2( C_1\mathfrak{A}(\epsilon') + \xi_0).
		\]
		Note also that we can take $ \varsigma_0 $ large if $ \sup_{m}\lambda'_{cs}(m)\lambda'_{u}(m) $ is small.
	\end{enumerate}

	Conclusion (a) (i) follows from \autoref{rmk:conexp}. Under the condition in (a) (ii), it suffices to show that $ m \mapsto \lambda'_{cs}(m) $ and $ m \mapsto \lambda'_{u}(m) $ are $ \xi' $-almost uniformly continuous at $ K $ (in the immersed topology), where $ \xi' $ is small if $ \xi_*, \mathfrak{A} (\epsilon') $ are small. Only consider $ m \mapsto \lambda'_{cs}(m) $.
	Note that
	\begin{align*}
	\sup_{m \in K}\|DH(m)\widehat{\Pi}^{cs}_{m}\| & \leq \sup_{m \in K}\|\widehat{\Pi}^{cs}_{u(m)}DH(m)\widehat{\Pi}^{cs}_{m}\| + \sup_{m \in K}\|\widehat{\Pi}^{u}_{u(m)}DH(m)\widehat{\Pi}^{cs}_{m}\| \\
	& \leq \sup_{m \in K}\lambda'_{cs}(m) + \xi_0 < \widetilde{C}_0 < \infty,
	\end{align*}
	and similarly $ \sup_{m \in K}\|\widehat{\Pi}^{cs}_{u(m)}DH(m)\| \leq \widetilde{C}_0 $.
	By the almost uniform continuity of $ u $, $ \widehat{\Pi}^{cs}_{m}$, $ DH(m) $ at $ K $, if $ m_1, m_2 \in U_{m, \gamma}(\epsilon) \cap K $, then
	\begin{align*}
	&~ |\widehat{\Pi}^{cs}_{u(m_1)} DH(m_1)\widehat{\Pi}^{cs}_{m_1} - \widehat{\Pi}^{cs}_{u(m_2)} DH(m_2)\widehat{\Pi}^{cs}_{m_2}| \\
	\leq &~ |\widehat{\Pi}^{cs}_{u(m_1)} - \widehat{\Pi}^{cs}_{u(m_2)}|| DH(m_1)\widehat{\Pi}^{cs}_{m_1}| + | \widehat{\Pi}^{cs}_{u(m_2)} || DH(m_1) -  DH(m_2) | | \widehat{\Pi}^{cs}_{m_1} | \\
	& \quad + |\widehat{\Pi}^{cs}_{u(m_2)} DH(m_2)||\widehat{\Pi}^{cs}_{m_1} - \widehat{\Pi}^{cs}_{m_2}| \\
	\leq &~ \xi_1\xi_* \widetilde{C}_0 + C^2_1 \mathfrak{A}(\epsilon') + \xi_* \widetilde{C}_0 \triangleq \xi' , ~\text{as}~ \epsilon \to 0.
	\end{align*}
	The proof is complete.
\end{proof}

\subsection{} \label{app:GramBasic}
\begin{proof}[Proof of \autoref{lem:gram1}]
	(1). Set $ \Pi_{X_2}(X_1) = \Pi_0 $. If $ X_1 \oplus X_2 $ is closed, then by the closed graph theorem, $ |\Pi_0| < \infty $. We have
	\[
	|\Pi_0 x| = \inf_{y \in X_2} |\Pi_0(x - y)| \leq |\Pi_0| \inf_{y \in X_2} |\Pi_0 x - y|,
	\]
	and so $ \alpha(X_1, X_2) \geq |\Pi_0|^{-1} > 0 $. For any $ \epsilon > 0 $, take $ z \in X_1 \oplus X_2 $ such that $ (1 - \epsilon)|\Pi_0| \leq |\Pi_0z| $ and $ |z| = 1 $. Then
	\[
	\inf_{y \in X_2}\left| \frac{\Pi_0z}{|\Pi_0z|} - y \right| = \frac{1}{|\Pi_0z|} \inf_{y \in X_2} |\Pi_0z - y| \leq \frac{1}{|\Pi_0z|} |z| \leq \frac{1}{(1 - \epsilon)|\Pi_0|}.
	\]
	We see that $ \alpha(X_1, X_2) = |\Pi_0|^{-1} $.

	Assume $ \alpha(X_1, X_2) > 0 $, i.e.,
	\[
	|x_1 - x_2| \geq \alpha(X_1, X_2)|x_1|, ~ \forall x_i \in X_i \tag{$ \ast $}. \label{equ:110}
	\]
	Then we have (i) $ X_1 \cap X_2 = \{ 0 \} $ (since for $ x \in X_1 \cap X_2 $, $ 0 = |x - x| \geq \alpha(X_1, X_2)|x| $); and (ii) $  X_1 \oplus X_2 $ is closed. We show (ii) as follows. If $ x^i_n \in X_i $, $ i = 1, 2 $, $ n = 1,2,\ldots $, such that $ x^1_n + x^2_n \to z \in X $, then $ \{ x^1_n + x^2_n \} $ is a Cauchy sequence. By \eqref{equ:110}, one gets that $ \{ x^1_n \} $ is also a Cauchy sequence, and so is $ \{ x^2_n \} $. Since $ X_i $, $ i = 1,2 $ are closed, we obtain $ z \in X_1 \oplus X_2 $, which yields (ii).

	(2). This is a restatement of the Riesz lemma \cite{Meg98}, and (3) is trivial.

	(4). Let $ x_i \in X_i $, $ i = 2,3 $, where $ |x_3| = 1 $. Then
	\[
	\inf_{x_1 \in \mathbb{S}_{X_1}} |x_1 - x_2| \leq \inf_{x_1 \in \mathbb{S}_{X_1}} (|x_1 - x_3| + |x_2 - x_3|) \leq d(X_3, X_1) + |x_2 - x_3|.
	\]
	So (4) follows.

	(5). Since for $ x_3 \in X_3 $,
	\begin{align*}
	|\Pi_{X_3}(X_2)|_{X_1}| & \triangleq \sup_{x_1 \in \mathbb{S}_{X_1} } | \Pi_{X_3}(X_2) x_1 |  = \sup_{x_1 \in \mathbb{S}_{X_1} } | \Pi_{X_3}(X_2) (x_1 - x_3) | \\
	& \leq |\Pi_{X_3}(X_2)| \sup_{x_1 \in \mathbb{S}_{X_1}} |x_1 - x_3|,
	\end{align*}
	we have $ |\Pi_{X_3}(X_2)|_{X_1}| \leq |\Pi_{X_3}(X_2)| \delta(X_1, X_3) \leq |\Pi_{X_3}(X_2)| d(X_1, X_3) $.

	(6). Since for $ x \in R(\Pi_1) $ with $ |x| = 1 $,
	\[
	d(x, R(\Pi_2)) \leq |\Pi_1 x - \Pi_2 \Pi_1 x| \leq |\Pi_1 - \Pi_2|,
	\]
	this gives $ \delta(R(\Pi_1), R(\Pi_2)) \leq |\Pi_1 - \Pi_2| $ and then $ \widehat{d}(R(\Pi_1), R(\Pi_2)) \leq 2 |\Pi_1 - \Pi_2| $.
	The proof is complete.
\end{proof}

\begin{proof}[Proof of \autoref{lem:gram2}]
	By \autoref{lem:gram1} (4), we have 
	\[
	\alpha(X_3, X_1) \geq \alpha(X_1, X_2) - d(X_3, X_1) > 0;
	\]
	and by \autoref{lem:gram1} (3), we get
	\[
	\delta(X, X_3 \oplus X_2) = \delta(X_1 \oplus X_2, X_3 \oplus X_2) \leq |\Pi_{X_2}(X_1) | \delta (X_1, X_3) < 1.
	\]
	Hence, $ X_3 \oplus X_2 = X $. Finally,
	\begin{align*}
	|\Pi_{X_2}(X_1) - \Pi_{X_2}(X_3)| = & \sup_{|x| = 1} |(\Pi_{X_2}(X_1) - \Pi_{X_2}(X_3))x| \\
	= & \sup_{|x| = 1} |(\Pi_{X_2}(X_1) - \Pi_{X_2}(X_3))\Pi_{X_2}(X_1)x| \\
	= & \sup_{|x| = 1} |(\Pi_{X_1}(X_2) - \Pi_{X_3}(X_2))\Pi_{X_2}(X_1)x| \\
	\leq & \sup_{|x| = 1} |\Pi_{X_3}(X_2)\Pi_{X_2}(X_1)x| \leq |\Pi_{X_2}(X_1)| \cdot |\Pi_{X_3}(X_2)|_{X_1}| \\
	\leq & |\Pi_{X_2}(X_1)| \cdot |\Pi_{X_3}(X_2)| \delta(X_1, X_3) \\
	\leq & |\Pi_{X_2}(X_1)| (1+|\Pi_{X_2}(X_3)|) \delta(X_1, X_3),
	\end{align*}
	and combining with the estimate of $ |\Pi_{X_2}(X_3)| $ given in (2), we get the estimate for $ |\Pi_{X_2}(X_1) - \Pi_{X_2}(X_3)| $. The proof is complete.
\end{proof}

\subsection{} \label{app:separable}

In this appendix, we show the following fact.
\begin{lem}\label{lem:ap0}
	A separable Banach space $ X $ admitting a $ C^{k} \cap C^{0,1} $ bump function fulfills $ (3 + \epsilon) $-uniform property ($ *^{k} $) (see \autoref{def:property*}) for any $ \epsilon > 0 $.
\end{lem}

Such a space $ X $ is introduced in order to study $ C^k $-smooth approximations of Lipschitz mappings preserving the Lipschitz condition (see \cite[Chapter 7]{HJ14}), i.e., there is a constant $ C \geq 1 $ such that for any Lipschitz function $ f: X \to \mathbb{R} $ and any $ \varepsilon > 0 $, there is a $ C^{k} $-smooth and Lipschitz function $ g: X \to \mathbb{R} $ such that
\[\tag{$ \clubsuit $} \label{equ:up}
|f(x) - g(x)| < \varepsilon, ~\text{and}~ \lip g \leq C \lip f.
\]
This result was established in works such as \cite{AFK10, HJ10}, where it was demonstrated that such a constant $ C $ exists, though it may depend on the particular Banach space $ X $.
In \cite[Remark 3.2 (3)]{JS11}, the authors noticed that $ C $ can be universal and less than $ 602 $. We will show that the construction given in \cite{AFK10, HJ10} (see also \cite{HJ14}) in fact implies that $ C \leq 3 + \epsilon $ (for any $ \epsilon > 0 $) by a careful examination of each step in their proofs. The details are given in the following for the convenience of the readers.

\begin{proof}[Proof of \autoref{lem:ap0}]
	(Step 1). First consider the case $ k = 1 $. Note that a well known result (see \cite{HJ14}) says that if $ X $ is separable, then $ X $ has a $ C^1 $ bump function if and only if $ X $ has an equivalent $ C^1 $ norm (if and only if $ X^* $ is separable). Take a $ C^1 $ norm $ |\cdot| $ of $ X $. Let $ 0 < \delta < r < 1 $ and $ 0 < \eta < \sigma \leq 1 $. Choose $ \theta_i \in C^{\infty}(\mathbb{R}_+, [0,1]) $, $ i = 1,2,3 $, such that $ \lip \theta_1 \leq \frac{\sigma}{1 - r} $, $ \lip \theta_2 \leq \frac{\sigma}{r - \delta} $, $ \lip \theta_3 \leq \frac{1}{\sigma - \eta} $ and
	\[
	\theta_1(t) = \begin{cases}
	\sigma, ~t \leq r,\\
	0,~t \geq 1,
	\end{cases}
	\theta_2(t) = \begin{cases}
	0, ~t \leq \delta,\\
	\sigma,~t \geq r,
	\end{cases}
	\theta_3(t) = \begin{cases}
	0, ~t \leq \eta,\\
	1,~t \geq \sigma.
	\end{cases}
	\]
	Since $ X $ is separable, we can take $ \{ x_j \}_{j \geq 1} $ such that $ \overline{\{ x_j \}_{j \geq 1}} = X $; without loss of generality, assume $ x_1 = 0 $. Set
	\[
	f_j(x) = \theta_1(|x - x_j|), ~g_j(x) = \theta_2(|x - x_j|), ~j \geq 1.
	\]
	Furthermore, take $ \varphi_{j} \in C^{\infty}(\mathbb{R}^j) $ such that $ \lip \varphi_{j} \leq 1 $ (with respect to the max norm) and
	\[
	\min\{ \omega_1, \ldots, \omega_j \} \leq \varphi_{j}(\omega) \leq \min\{ \omega_1, \ldots, \omega_j \} + \eta, \quad \omega = (\omega_1, \ldots, \omega_j) \in \mathbb{R}^j.
	\]
	Define
	\[
	\psi_j(x) = \theta_3( \varphi_{j}(g_1(x), \ldots, g_{j-1}(x), f_{j}(x)) ), \quad j = 1,2,\ldots.
	\]
	It is straightforward to verify the following properties about $ \{ \psi_j \} $, where $ \mathbb{B}_{\varrho} (x) = \{ y \in X: |y - x| < \varrho \} $.
	\begin{enumerate}[(1)]
		\item For any $ x \in X $, if $ x \in \mathbb{B}_{\delta} (x_{k}) $, then $ \psi_{j}(x) = 0 $ for all $ j > k $.
		\item If $ n $ is the smallest index such that $ x \in \mathbb{B}_{r} (x_{n}) $, then $ \psi_{n}(x) = 1 $.
		\item If $ |x - x_j| \geq 1 $, then $ \psi_{j} (x) = 0 $.
		\item $ \lip \psi_j \leq \frac{1}{\sigma - \eta} \max\{ \frac{\sigma}{1-r}, \frac{\sigma}{r-\delta} \} $ (with respect to $ |\cdot| $), $ \psi_{j} \in C^1(X) $, and $ 0 \leq \psi_{j} \leq 1 $ for all $ j \geq 1 $.
	\end{enumerate}

	(Step 2). Choose any bijection $ \rho: \mathbb{Z} \times \{ j \in \mathbb{N}: j\geq 2 \} \to \mathbb{Z} $. Let $ \beta > 1 + r $ be arbitrary and $ t = t_{\beta} = \frac{\beta}{1 + r} > 1 $. Note that $ \bigcup_{m \in \mathbb{Z}} [ (1+r)t^m, \beta t^m ] = \mathbb{R}_+ \setminus \{ 0 \} $. Define $ \Phi: X \to \mathbb{R}^{\mathbb{Z}} $ by $ \Phi(x)_{\rho(n,j)} = t^{n} \psi_{j}(\frac{x}{t^n}) $. Then we have the following.
	\begin{enumerate}[(a)]
		\item $ \Phi(X) \subset c_0(\mathbb{Z}) $. That is, $ t^{n} \psi_{j}(\frac{x}{t^n}) \to 0 $ as $ |n| \to \infty $ or $ j \to \infty $. Note that 
		\[
		\sup_{j \geq 2}|t^{n} \psi_{j}(\frac{x}{t^n})| \leq t^{n} \to 0, \quad \text{as} \quad n \to -\infty.
		\]
		Fix $ x $ and choose $ N = N(x) $ such that $ |x| \leq t^{n} \delta $ for all $ n \geq N $. Then $ \frac{x}{t^{n}} \in \mathbb{B}_{\delta} (x_1) $ and so $ t^{n} \psi_{j}(\frac{x}{t^n}) = 0 $ if $ j \geq 2 $ by (1). Finally, note that $ \{ j \geq 2: \psi_{j}(\frac{x}{t^n}) \neq 0: n = 0,1,\ldots,N \} $ is finite by (1).

		\item $ \lip \Phi \leq \frac{1}{\sigma - \eta} \max\{ \frac{\sigma}{1-r}, \frac{\sigma}{r-\delta} \} $ (with respect to the usual sup norm in $ c_0(\mathbb{Z}) $).

		\item $ \Phi: X \to \Phi(X) $ is bi-Lipschitz with $ \lip \Phi^{-1} \leq \beta $. Let $ x, y \in X $ and $ x \neq y $. Choose $ m \in \mathbb{Z} $ such that $ (1+r) t^m \leq |x-y| \leq \beta t^m $. Without loss of generality, $ |\frac{x}{t^m}| \geq \frac{1+r}{2} $ ($ > r $). Let $ k $ be the smallest one such that $ \frac{x}{t^m} \in \mathbb{B}_{r}(x_{k}) $. Then $ k \geq 2 $. By (2), $ \psi_{k}(\frac{x}{t^m}) = 1 $. Since
		\[
		|\frac{y}{t^m} - \frac{x_{k}}{t^m}| \geq |\frac{y}{t^m} - \frac{x}{t^m}| - |\frac{x}{t^m} - \frac{x_{k}}{t^m}| \geq 1+r - r \geq 1,
		\]
		by (3), one gets $ \psi_{k}(\frac{y}{t^m}) = 0 $. Thus,
		\[
		|\Phi(x) - \Phi(y)|_{\infty} \geq |t^m\psi_{k}(\frac{x}{t^m}) - t^m\psi_{k}(\frac{y}{t^m})| = t^m > \frac{1}{\beta} |x - y|.
		\]
	\end{enumerate}

	(Step 3). By \cite[Theorem 7]{HJ10} (see also \cite[Chapter 7, Theorem 79]{HJ14}), we know that for any Lipschitz function $ f: X \to \mathbb{R} $ and $ \varepsilon > 0 $, there is $ g \in C^{1}(X) $ such that \eqref{equ:up} holds with
	\[
	C \leq \lip \Phi \lip \Phi^{-1} \leq \frac{1}{\sigma - \eta} \max\left\{ \frac{\sigma}{1-r}, \frac{\sigma}{r-\delta} \right\} \beta.
	\]
	Choose $ \sigma = 1 $, $ r = 1/2 $, $ \eta, \delta \to 0 $ and $ \beta \to 1 + r $; we know that $ C $ can be taken sufficiently close to $ 3 $, independent of the choice of $ C^1 $ norms in $ X $.

	(Step 4). Let $ |\cdot|_{\sim} $ be any equivalent norm in $ X $. By \cite[Chapter 7, Theorem 103]{HJ14}, for any $ \epsilon > 0 $, we can take a $ C^1 $ norm $ |\cdot| $ in $ X $ such that $ (1+\epsilon)^{-1} |x| \leq |x|_{\sim} \leq (1+\epsilon) |x| $ for all $ |x|_{\sim} \leq 1 $. This shows that \eqref{equ:up} also holds for any $ C > 3 $ sufficiently close to $ 3 $.

	(Step 5). Consider the general case $ k $. First, for any Lipschitz function $ f: X \to \mathbb{R} $ and $ \varepsilon > 0 $, choose $ g \in C^{1}(X) $ such that \eqref{equ:up} holds. By \cite[Corollary 19]{HJ10}, there is $ g' \in C^{k}(X) $ such that $ |g(x) - g'(x)| \leq \varepsilon $ and $ |Dg(x) - Dg'(x)| \leq \varepsilon $. So we have $ |f(x) - g'(x)| \leq 2 \varepsilon $ and $ \lip g' \leq (C + \varepsilon) \lip f $. That is, the constant $ C $ can be made arbitrarily close to $ 3 $. The proof is complete.
\end{proof}

The main constructions in (Step 1) and (Step 2) come from \cite[Chapter 7, Theorem 64 (due to R. Fry), Theorem 63 (due to P. H\'{a}jek and M. Johanis)]{HJ14}, where a slight difference of the construction $ \Phi $ is that we do not use bump functions. (Step 4) and (Step 5) follow from \cite[Remark 3.2 (3)]{JS11}.

\chapter{Appendix. Approximation by $ C^{1} \cap C^{0,1} $ maps between two manifolds} \label{app:general}

Here we give a sketch of the proof of \autoref{thm:general} due to \cite{JS11}.

\begin{proof}[Proof of \autoref{thm:general}]
	Let 
	\[
	U_{m} (\varrho) = \{ m' ~\text{in the component of $ M $ containing}~ m: d_{F}(m', m) < \varrho\}. 
	\]
	For any two sets $ A, B $ in a metric space, let 
	\[
	d_{H} (A, B) = \inf \{ d(x, y): x \in A, y \in B \}.
	\]
	
	(I). Since $ M $ is a $ C^k $ Finsler manifold in the sense of Palais, for any $ K > 1 $ and every $ m \in M $, there is a $ C^k $ local chart $ \varphi_{m}: U_{m} (r_{m}) \to T_{m} M $ such that $ D\varphi_{m}(m) = \id $ and for every $ m' \in U_{m} (r_{m}) $
	\[
	|D\varphi_{m}(m')| \leq K, ~|D\varphi^{-1}_{m}(\varphi_{m}(m'))| \leq K,
	\]
	where $ |D\varphi_{m}(m')| = \sup\{ |D\varphi_{m}(m')x|_{m}: |x|_{m'} \leq 1 \} $ and $ |\cdot|_{m} $ is the norm of $ T_{m} M $ induced by the Finsler structure of $ TM $; in particular (if $ r_{m} $ is smaller),
	\[
	K^{-1} d_{F}(m', m'') \leq |\varphi_{m}(m') - \varphi_{m}(m'')| \leq K d_{F}(m', m''), \quad m', m'' \in U_{m} (r_{m}).
	\]
	Here, when $ m \in \partial M \neq \emptyset $, see \cite[Definition 7.2.2]{AMR88} for the meaning of $ C^{k} $ local chart $ \varphi_{m} $.

	(II). ($ C^{k} \cap C^{0,1} $ partition of unity) Now there is $ \{ \phi_{i, m} \}_{(i, m)\in \mathbb{N} \times M} $ such that
	\begin{enumerate}[(1)]
		\item $ \phi_{j} \in C^{k} \cap C^{0,1} $, $ j \in \mathbb{N} \times M $;
		\item $ \sum_{j \in \mathbb{N} \times M} \phi_{j}(m) = 1 $ for $ \forall m \in M $;
		\item $ \{ \mathrm{supp}(\phi_{j}) \}_{j \in \mathbb{N} \times M} $ is locally finite subordinated to $ \{ U_{m}(r_{m}) \}_{m \in M} $. Here $ \mathrm{supp}(\phi_{j}) $ denotes the closure of $ \{ m' \in M: \phi_{j}(m') \neq 0 \} $.
	\end{enumerate}

	The construction of $ \{ \phi_{i, m} \}_{(i, m)\in \mathbb{N} \times M} $ depends on very good open refinements of 
	\[
	\{U_{m} (r_{m})\}_{m \in M}
	\]
	due to M. E. Rudin (see also \cite[Lemma 8 in Chapter 7]{HJ14}). Usually, the general open refinement cannot be used to obtain $ \phi_{j} \in C^{0,1} $, though $ \phi_{j} $ can be locally Lipschitz (see e.g. \cite[Theorem 1.6]{Pal66}). The following property (ii) is the key. Note that each component of $ M $ is a metric space. So there are open refinements $ \{ V_{i, m} \}_{(i, m)\in \mathbb{N} \times M} $ and $ \{ W_{i, m} \}_{(i, m)\in \mathbb{N} \times M} $ such that
	\begin{enumerate}[(i)]
		\item $ V_{i, m} \subset W_{i, m} \subset U_{m}(r_{m}) $;
		\item
		$ d_{H} (V_{i, m}, W^{\complement}_{i, m}) \geq C_{i, m} $ for all $ i \in \mathbb{N} $ and $ m \in M $ where $ C_{i, m} > 0 $
		\footnote{In \cite[Lemma 8 in Chapter 7]{HJ14} (or \cite[Lemma 3.3]{HJ14}), $ C_{i, m} = 1 / 2^{i+1} $ for all $ i $, but this is not always true; for instance if $ r_{m} < 1/4 $, then there is no $ V_{0,m}, W_{0,m} \subset U_{m}(r_{m}) $ such that $ d_{H} (V_{0, m}, W^{\complement}_{0, m}) \geq 1/2 $. However, from the proof given in \cite[Lemma 8 in Chapter 7]{HJ14}, for each $ m \in M $, there is $ i_{m} \in \mathbb{N} $ such that $ C_{i, m} = 1 / 2^{i+1} $ if $ i \geq i_{m} $. Now, we can take $ V_{i + i_{m},m} $ and $  W_{i + i_{m},m} $ instead of $ V_{i,m} $ and $ W_{i,m} $; what we really need is that $ C_{i, m} > 0 $.},
		$ W^{\complement}_{i, m} = \mathcal{M}_{m} \setminus W_{i, m} $ and $ \mathcal{M}_{m} $ is the component of $ M $ containing $ m $;
		\item $ W_{i, m} \cap W_{i, m'} = \emptyset $ if $ m \neq m' $;
		\item for every $ m_0 \in M $, there is $ U_{m_0}(s_{m_0}) $ ($ s_{m_0} > 0 $) such that for at most one $ (i,m) \in \mathbb{N} \times M $ we have $ W_{i,m} \cap U_{m_0}(s_{m_0}) \neq \emptyset $.
	\end{enumerate}

	For each $ i, m $, since $ d_{H} (V_{i, m}, W^{\complement}_{i, m}) \geq C_{i, m} > 0 $ and $ X_{m} $ satisfies $ C_1 $-property ($ *^k $), there is a $ C^k $ function $ \psi_{i, m}: M \to [0,1] $ such that $ \psi_{i, m}(V_{i, m}) = 1 $, $ \psi_{i, m}(W^{\complement}_{i, m}) = 0 $ and $ \lip \psi_{i, m}(\cdot) \leq \widetilde{C}_{i, m} $ for some constant $ \widetilde{C}_{i, m} > 0 $; see also \cite[Lemma 3.6]{JS11}. Define
	\begin{gather*}
	h_{i} (m') = \sum_{m \in M} \psi_{i, m}(m') ~(\leq 1), ~\text{which is well defined by (iii) (iv)},\\
	\phi_{1, m} = \psi_{1, m}, \quad \phi_{n, m} = \psi_{n, m} (1 - h_1) \cdots (1 - h_{n-1}), ~n \geq 2.
	\end{gather*}
	Then $ \{ \phi_{i, m} \}_{(i, m)\in \mathbb{N} \times M} $ satisfies conditions (1)--(3).

	(III). For any given $ \varepsilon > 0 $, since $ X_{m} $ satisfies $ C_1 $-property ($ *^k $), there is a $ C^{k} $ function $ g_{i, m} : X_{m} \to \mathbb{R} $ such that
	\[
	|f(\varphi^{-1}_{m}(x)) - g_{i, m}(x)| \leq \frac{\varepsilon \min\{ 1, \lip f \}}{2^{i+2} (\lip\phi_{i, m} + 1)}, ~ x \in \varphi_{m}(U_{m}(r_{m})),
	\]
	and $ \lip g_{i, m}(\cdot) \leq C_1 \lip f(\varphi^{-1}_{m}(\cdot)) \leq C_1 K \lip f $. Define
	\[
	g(m) = \sum_{i \in \mathbb{N}, m' \in M} \phi_{i, m'} (m) g_{i, m'} (\varphi_{m'}(m)), ~ m \in M.
	\]
	Now one can see that $ g \in C^{k} $,
	\[
	|g(m) - f(m) | = \left|\sum_{i \in \mathbb{N}, m' \in M} \phi_{i, m'} (m) \left\{ g_{i, m'} (\varphi_{m'}(m)) - f( \varphi^{-1}_{m'}( \varphi_{m'} (m) ) ) \right\}\right| \leq \varepsilon,
	\]
	and
	\begin{multline*}
	\lip g \leq \sup_{m \in M} |Dg(m)| = \sup_{m \in M} \left| \sum_{i \in \mathbb{N}, m' \in M} \left\{  \phi_{i, m'} (m) D g_{i, m'} (\varphi_{m'}(m)) \right. \right. \\
	\left. \left.  + ~  D\phi_{i, m'} (m) \left( g_{i, m'} (\varphi_{m'}(m)) - f( \varphi^{-1}_{m'}( \varphi_{m'} (m) ) ) \right) \right\} \right| \leq (C_1K^2 + \varepsilon) \lip f.
	\end{multline*}
\end{proof}

Below, let us give a vector-valued version of \autoref{thm:general}.

\begin{defi}
	For two Banach spaces $ X, Z $, we say the pair $ (X, Z) $ satisfies \emph{$ C_1 $-property (A)} if for every $ \varepsilon>0 $ and every Lipschitz map $ f: X \to Z $, there is a $ C^1 $ and Lipschitz map $ g: X \to Z $ such that $ |f(x) - g(x)| < \varepsilon $ and $ \lip g \leq C_1 \lip f $. (See also \cite{JS13}.)
	We say $ X $ satisfies \emph{strong $ C_1 $-property (A)} if for every Banach space $ Z $, $ (X, Z) $ satisfies \emph{$ C_1 $-property (A)}.
\end{defi}

It is easy to see that if $ (X, Z) $ satisfies $ C_1 $-property (A), then $ X $ satisfies $ C_1 $-property ($ *^1 $) in the sense of \autoref{def:property*}.

\begin{exa}\label{exa:propertyA}
	\begin{enumerate}[(a)]
		\item If $ X $ is finite-dimensional, then $ X $ satisfies strong $ 1 $-property (A).

		\item Let $ X $ be a Banach space with an unconditional Schauder basis $ (x_{n}) $ which admits a $ C^{1} $ and Lipschitz bump function $ \varphi: X \to [0,1] $ such that $ \varphi|_{\mathbb{B}_{r}} = 1 $, $ \varphi|_{\mathbb{B}^{\complement}_1} = 0 $ where $ 0 < r < 1 $. Then there is a constant $ C > 0 $ only depending on the \emph{unconditional constant} of $ (x_{n}) $ (see e.g. \cite[Definition 4.2.28]{Meg98}) and $ \lip \varphi $ such that for any Banach space $ Z $, $ (X, Z) $ satisfies $ C_1 $-property (A) and so $ X $ admits strong $ C_1 $-property (A). In particular, if $ X $ is a separable Hilbert space, then $ C_1 $ can be taken as $ 2 + \varepsilon $ for any small $ \varepsilon > 0 $.
	\end{enumerate}

\end{exa}

\begin{proof}
	(a) This can be rapidly seen by standard convolution techniques.

	(b) The result was first announced by R. Fry (see also \cite[Corollary 87 in Chapter 7]{HJ14} or \cite[Theorem H]{HJ10}). That the constant $ C > 0 $ depends only on the unconditional constant of $ (x_{n}) $ and $ \lip \varphi $ follows from the details in the proofs of \cite[Lemma 83, Theorem 84 in Chapter 7]{HJ14} (due to N. Moulis).
	More precisely, write $ K_{b} $, $ K_{ub} $, and $ K_{u} $ as the basis constant, unconditional basis constant, and unconditional constant of $ (x_{n}) $, respectively; see \cite[Definition 4.2.28]{Meg98}. Define the bounded multiplier unconditional $ (x_{n}) $ norm of $ X $ as
	\[
	\left|\sum \alpha_{n} x_{n}\right|_{bmu} = \sup\left\{\left|\sum \beta_{n} \alpha_{n} x_{n}\right|: (\beta_{n}) \in l_{\infty}, \sup_{n}|\beta_{n}| \leq 1 \right\}.
	\]
	Then by the definition of the unconditional constant $ K_{u} $ of $ (x_{n}) $, we have $ |x| \leq |x|_{bmu} \leq K_{u}|x| $; in addition, if $ X $ is renormed with $ |\cdot|_{bmu} $, then $ K_{u} = K_{ub} = K_{b} = 1 $ (see \cite[Proposition 4.2.31]{Meg98}).

	Let $ X $ be equipped with $ |\cdot|_{bmu} $. Now, in the proof of \cite[Lemma 83 in Chapter 7]{HJ14}, we have $ K = \lip \Psi \leq 1 + \lip \varphi $ (in the new norm $ |\cdot|_{bmu} $, the Lipschitz constant of $ \Psi $ can be calculated more precisely); also, in \cite[Theorem 84 in Chapter 7]{HJ14}, the constant $ C $ can be taken as $ K $. Since \cite[Theorem H]{HJ10} is a consequence of Lemma 69 and Theorem 84 in \cite[Chapter 7]{HJ14}, we see $ C_1 $ can be taken as $ K(1 + \varepsilon) $ for the norm $ |\cdot|_{bmu} $ and so $ K_{u}(1 + \lip\varphi)(1 + \varepsilon) $ for the original norm $ |\cdot| $ of $ X $, where $ \varepsilon > 0 $ is any small constant.
\end{proof}

We refer readers to \cite{JS13, HJ10} and \cite[Chapter 7]{HJ14} for more details about property (A). The approximation of Lipschitz maps by smooth and Lipschitz maps between two general Banach spaces remains an open problem, and related results are much less developed; see also \cite[Problem 110, p. 464]{HJ14}.

Let $ N $ be a $ C^1 $ Banach manifold with a (compatible) metric $ d $. Assume there is a $ C^1 \cap C^{0,1} $ embedding $ j : N \to Y $, where $ Y $ is a Banach space, and a $ C^1 \cap C^{0,1} $ retraction $ r: \mathbb{B}_{\rho}(N) \subset Y \to N $ for some $ \rho > 0 $ (i.e., $ r|_{N} = \id $). See also \autoref{sub:setup} and \autoref{sub:granTwo} for some examples. Let $ M $ be a $ C^1 $ Finsler manifold (possibly with boundary) in the sense of Palais (cf. \cite{Pal66}) endowed with the Finsler metric in each of its components.

\begin{thm}\label{thm:LipManifold}
	Let $ M, N $ be given as above. Let $ f: M \to N $ be $ C^{0,1} $. Assume for any $ m \in M $, $ T_{m}M $ satisfies strong $ C_1 $-property (A) (see \autoref{exa:propertyA}). Then there is a constant $ \widetilde{C}' > 0 $ such that for any $ \varepsilon > 0 $, there is a $ C^1 \cap C^{0,1} $ function $ g: M \to N $ such that $ d(f(x), g(x)) < \varepsilon $ for all $ x \in M $ and $ \lip g \leq \widetilde{C}'(1 + \varepsilon) \lip f $; here $ \widetilde{C}' = C_1 \lip r \lip j $.
	In addition, if $ f $ is $ C^1 $ at $ V $ with $ V $ satisfying $ d_{H}(V, M \setminus V_1) > 0 $ for some $ V_1 $ such that $ V \subset V_1 \subset M $, then $ g $ can be chosen such that $ g|_{V} = f|_{V} $.
\end{thm}
\begin{proof}
	Let $ j: N \to Y $ be the $ C^1 \cap C^{0,1} $ embedding and $ r: \mathbb{B}_{\rho}(N) \subset Y \to N $ the $ C^1 \cap C^{0,1} $ retraction.

	First, note that for any given $ \varepsilon_1 > 0 $ and a Lipschitz map $ f_0: M \to Y $, the same argument as in \autoref{thm:general} shows that there is a $ C^1 \cap C^{0,1} $ map $ g_0: M \to Y $ such that $ |f_0(x) - g_0(x)| < \varepsilon_1 $ for all $ x \in M $ and $ \lip g_0 \leq C_1(1+\varepsilon_1) \lip f_0 $.

	In particular, for $ f_0 = j \circ f: M \to Y $ and $ 0 < \varepsilon_1 < \min \{ \varepsilon, (\lip r)^{-1} \varepsilon, \rho \} $, we have such a map $ g_0: M \to Y $. Note that $ g_0(M) \subset \mathbb{B}_{\rho}(N) $ as $ f_0(M) = N $. Set $ g = r \circ g_0 $. Then $ g: M \to N $,
	\[
	\lip g \leq \lip r \cdot C_1(1 + \varepsilon_1) \lip f_0 \leq \widetilde{C}'(1 + \varepsilon) \lip f,
	\]
	where $ \widetilde{C}' = C_1 \lip r \lip j $, and
	\[
	d(f(x), g(x)) = d(r\circ f_0(x), r \circ g_0(x)) \leq \lip r |f_0(x) - g_0(x)| < \varepsilon.
	\]
	For the last conclusion, one can use suitable bump function to get the desired $ g $ (see e.g. \autoref{thm:C1app0}). This completes the proof.
\end{proof}

Here, we do not try to give a more general version of \autoref{thm:LipManifold} (e.g. the size of tubular neighborhoods of $ N $ can be non-uniform or $ M $ is a Finsler manifold in the sense of Neeb-Upmeier (see e.g. \cite[Appendix D.2]{Che18a} or \cite{JS11})).

\begin{cor}\label{cor:appLip}
	Let $ X, \Sigma, K, \{ \Pi^{\kappa}_{m} \} $ ($ \kappa = s, c, u $) be given as in \autoref{sub:setup} (I), and assume that (H1)--(H4) (in \autoref{sub:setup}) hold with $ \chi(\epsilon) $ small if $ \epsilon $ is small. Suppose (i) $ \Sigma \in C^1 $, and (ii) $ X $ is a separable Hilbert space or $ \Sigma $ is finite-dimensional. Then there is a small $ \epsilon_{*} > 0 $ such that for any small $ \varepsilon_0 > 0 $, there are projections $ \{ \widetilde{\Pi}^{\kappa}_{m} \} $, $ \kappa = s, c, u $, such that
	\begin{enumerate}[(1)]
		\item $ m \mapsto \widetilde{\Pi}^{\kappa}_{m} $ is $ C^1 $ in $ U_{m_0, \gamma}(\epsilon_*) $ ($ m_0 \in K, \gamma \in \phi^{-1}(m_0) $),
		\item $ \{ \widetilde{\Pi}^{\kappa}_{m} \} $ satisfies (H2) (in \autoref{sub:setup}), and
		\item $ |\Pi^{\kappa}_{m} - \widetilde{\Pi}^{\kappa}_{m}|< \varepsilon_0 $ for all $ m \in U_{m_0, \gamma}(\epsilon_*) $ and $ m_0 \in K, \gamma \in \phi^{-1}(m_0) $.
	\end{enumerate}
\end{cor}

In fact, we can let $ m \mapsto \widetilde{\Pi}^{\kappa}_{m} $ be $ C^\infty $ in $ U_{m_0, \gamma}(\epsilon_*) $ ($ m_0 \in K $, $ \gamma \in \phi^{-1}(m_0) $).

\begin{proof}
	Without loss of generality, take $ \widehat{\Sigma} = \widehat{K}_{\epsilon_*} $ and $ \Sigma = K_{\epsilon_*} $ for small $ \epsilon_* > 0 $. By (H2), we have $ \sup_{m \in \Sigma} |\Pi^{\kappa}_{m}| < \infty $. By (ii), we see that for each $ \widehat{m} \in \widehat{\Sigma} $, $ T_{\widehat{m}} \widehat{\Sigma} $ satisfies the strong $ C_1 $-property (A) (see \autoref{exa:propertyA}). $ \widehat{\Sigma} $ can be considered as a Finsler manifold in the sense of Palais, and the Finsler metric in each component of $ \widehat{\Sigma} $ is locally almost the same as in $ X $; see also \autoref{lem:csP}. By \autoref{lem:Piret}, any small neighborhood of $ \{ \Pi^{\kappa}_{m} \} $ in $ \overline{\Pi}(X) $, denoted by $ N $, satisfies the condition in \autoref{thm:LipManifold}. Thus, the conclusion follows from \autoref{thm:LipManifold} by applying it to the map $ \widehat{\Sigma} \to N \subset L(X): \widehat{m} \mapsto \Pi^{\kappa}_{\phi(\widehat{m})} $.
\end{proof}

\chapter{Appendix. Invariant manifolds in the invariant and trichotomy case} \label{app:im}

Consider the parallel results as \autoref{thm:invariant} for the trichotomy case.
In this case, we take $ X^c_{m} \oplus X^{c_0}_{m} = X^{c_1}_{m} $ and $ X^{s_1}_{m} \subset X^{s}_{m} $, $ X^{u_1}_{m} \subset X^{u}_{m} $ with $ X^{c_1}_{m} \oplus X^{s_1}_{m} \oplus X^{u_1}_{m} = X $ (the corresponding projections also satisfy (H2) in \autoref{sub:setup}); similarly for $ \widehat{X}^{\kappa_0}_{m} $, $ \kappa_0 = s_1, u_1, c_1, c_0 $. In (B3) (a) and (B4) (i) (ii) (in \autoref{sub:tri}), $ c, s, u $ are changed as $ c_1, s_1, u_1 $, respectively. Let $ \Sigma = K $, $ \eta = 0 $, and replace the spectral condition in (B3) (a) (ii) by the spectral gap condition in (B4) (ii) (in \autoref{sub:tri}).

There are $ 0 < c_{*} < 1 $ and small positive constants $ \xi_{1,*}, \xi_{2,*}, \epsilon_{*} $, and $ \chi_{0,*}, \gamma_{0,*} = O_{\epsilon_{*}}(1) $, $ \eta_* = o(\epsilon_{*}) $ such that if
\begin{enumerate}[$ \bullet $]
	\item Case (1): (B1), (B2), (B3) (a) (b) hold with the constants satisfying $ \xi_{i} \leq \xi_{i,*} $ ($ i = 1,2 $), $ \chi(\epsilon) \leq \chi_{0,*} $, $ \epsilon \leq \epsilon_{*} $, $ \eta \leq \eta_* $,

	\item Case (2): \emph{or} (B1), (B2), (B3) (a$ ' $) (b) hold with the constants satisfying $ \xi_{i} \leq \xi_{i,*} $ ($ i = 1,2 $), $ \chi(\epsilon) \leq \chi_{0,*} $, $ \epsilon \leq \epsilon_{*} $, $ \eta \leq \eta_* $, $ \gamma_0 \leq \gamma_{0,*} $,
\end{enumerate}
then there exist $ \varepsilon = O(\epsilon_*) $ and $ \sigma, \varrho = o(\epsilon_*) $ with $ \varepsilon_0 = c_{*}\sigma $ such that the following statements hold.

We have $ W^{cs}_{loc}(\Sigma) \subset X^{c_0s_1}_{\widehat{\Sigma}}(\sigma) \oplus X^{u_1}_{\widehat{\Sigma}} (\varrho) $, $ W^{cu}_{loc}(\Sigma) \subset X^{c_0u_1}_{\widehat{\Sigma}}(\sigma) \oplus X^{s_1}_{\widehat{\Sigma}} (\varrho) $, $ \Sigma^c = W^{cs}_{loc}(\Sigma) \cap W^{cu}_{loc}(\Sigma) $ called a \emph{local center-stable manifold}, a \emph{local center-unstable manifold}, and a \emph{local center manifold} of $ \Sigma $ satisfying the following properties.

\begin{enumerate}[(1)]
	\item (Representation) $ W^{cs}_{loc}(\Sigma), W^{cu}_{loc}(\Sigma), \Sigma^c $ can be represented as graphs of Lipschitz maps.
	That is, there are maps $ h^{\kappa}_0 $, $ \kappa = cs, cu, c $, such that for $ \widehat{m} \in \widehat{\Sigma} $,
	\begin{gather*}
	h^{cs}_0(\widehat{m}, \cdot): X^{c_0s_1}_{\phi(\widehat{m})} (\sigma) \to X^{u_1}_{\phi(\widehat{m})}(\varrho), \\
	h^{cu}_0(\widehat{m}, \cdot): X^{c_0u_1}_{\phi(\widehat{m})} (\sigma) \to X^{s_1}_{\phi(\widehat{m})}(\varrho), \\
	h^{c}_0 (\widehat{m}, \cdot): X^{c_0}_{\phi(\widehat{m})}(\sigma) \to X^{s_1}_{\phi(\widehat{m})}(\varrho) \oplus X^{u_1}_{\phi(\widehat{m})}(\varrho),
	\end{gather*}
	and $ W^{cs}_{loc}(\Sigma) = \graph h^{cs}_0 $, $ W^{cu}_{loc}(\Sigma) = \graph h^{cu}_0 $, $ \Sigma^{c} = \graph h^{c}_0 $;
	there are functions $ \mu_{\kappa}(\cdot) $, $ \kappa = cs, cu, c $, such that $ \mu_{cs}(m) = (1+\chi_{*}) \beta'_{c_1s_1}(m) + \chi_{*} $, $ \mu_{cu}(m) = (1+\chi_{*}) \alpha'_{c_1u_1}(m) + \chi_{*} $ and $ \mu_{c} = \max\{ \mu_{cs}, \mu_{cu} \} $, with $ \chi_{*} = O_{\epsilon_{*}}(1) $, and for every $ m \in \Sigma $, we have
	\begin{multline*}
		|\Pi^{u_1}_{m}( h^{cs}_0(\widehat{m}_1, x^{c_0s_1}_1) - h^{cs}_0(\widehat{m}_2, x^{c_0s_1}_2)  )| \\
		 \leq \mu_{cs}(m) \max \{| \Pi^c_{m}(  \phi(\widehat{m}_1) - \phi(\widehat{m}_2)  ) |, | \Pi^{c_0s_1}_{m}( x^{c_0s_1}_1 - x^{c_0s_1}_2 ) |\},
	\end{multline*}
	\begin{multline*}
	|\Pi^{s_1}_{m}( h^{cu}_0(\widehat{m}_1, x^{c_0u_1}_1) - h^{cu}_0(\widehat{m}_2, x^{c_0u_1}_2)  )| \\
	\leq \mu_{cu}(m) \max \{| \Pi^c_{m}(  \phi(\widehat{m}_1) - \phi(\widehat{m}_2)  ) |, | \Pi^{c_0u_1}_{m}( x^{c_0u_1}_1 - x^{c_0u_1}_2 ) |\},
	\end{multline*}
	\begin{multline*}
	\max\{|\Pi^{u_1}_{m}( h^{c}_0(\widehat{m}_1, x^{c_0}_1) - h^{c}_0(\widehat{m}_2, x^{c_0}_2))|, |\Pi^{s_1}_{m}( h^{c}_0(\widehat{m}_1, x^{c_0}_1) - h^{c}_0(\widehat{m}_2, x^{c_0}_2) )|\} \\
	\leq \mu_{c}(m) \max\{ | \Pi^c_{m}(  \phi(\widehat{m}_1) - \phi(\widehat{m}_2)  ) |, |\Pi^{c_0}_{m}(x^{c_0}_1 - x^{c_0}_2)| \} ,
	\end{multline*}
	where $ \widehat{m}_i \in \widehat{U}_{\widehat{m}}(\varepsilon_0) $, $ x^{\kappa_0}_i \in X^{\kappa_0}_{\phi(\widehat{m}_i)} (\varepsilon_0) $ ($ \kappa_0 = c_0s_1, c_0u_1, c_0 $), $ i = 1,2 $, $ \widehat{m} \in \phi^{-1}(m) $.

	\item (Local invariance) For
	\[
	\Omega_{cs} = W^{cs}_{loc}(\Sigma) \cap \{X^{c_0s_1}_{\widehat{\Sigma}}(\varepsilon_0) \oplus X^{u_1}_{\widehat{\Sigma}} (\varrho)\}, ~\Omega_{cu} = W^{cu}_{loc}(\Sigma) \cap \{X^{c_0u_1}_{\widehat{\Sigma}}(\varepsilon_0) \oplus X^{s_1}_{\widehat{\Sigma}} (\varrho)\},
	\]
	and $ \Omega_{c} = \Omega_{cs} \cap \Omega_{cu} $,
	we have $ \Omega_{cs} \subset H^{-1}W^{cs}_{loc}(\Sigma) $, $ \Omega_{cu} \subset HW^{cu}_{loc}(\Sigma) $ and $ \Omega_{c} \subset H^{\pm 1}\Sigma^{c} $; moreover, $ H: \Omega_{cs} \to W^{cs}_{loc}(\Sigma) $, $ H^{-1}: \Omega_{cu} \to W^{cu}_{loc}(\Sigma) $ induce Lipschitz maps and $ H: \Omega_{c} \to \Sigma^{c} $ induces a bi-Lipschitz map.

	\item \label{it:inv} (Partial characterization) 
	
	\noindent(i) $ \Sigma \subset \Sigma^c $.

	\noindent(ii) If $ \{z_{k} = (\widehat{m}_{k}, x^{c_0s_1}_{k}, x^{u_1}_{k})\}_{k \geq 0} \subset X^{c_0s_1}_{\widehat{\Sigma}}(\varepsilon_0) \oplus X^{u_1}_{\widehat{\Sigma}} (\varrho) $ satisfies
	\begin{enumerate}[$ \bullet $]
		\item $ z_{k} \in H(z_{k-1}) $ with $ \widehat{m}_{k} \in \widehat{U}_{\widehat{u}(\widehat{m}_{k-1})} (\varepsilon) $ for all $ k \geq 0 $; and
		\item $ |x^{u_1}_{k}| \leq \widetilde{\beta}_{0}(\widehat{m}_{k-1}) |x^{c_0s_1}_{k}| $ for all $ k \geq 1 $ or 
		\[
		\sup_{k}\{\varepsilon_{s}(\widehat{m}_0) \varepsilon_{s}(\widehat{m}_1) \cdots \varepsilon_{s}(\widehat{m}_{k-1})\}^{-1} (|x^{c_0s_1}_{k}| + |x^{u_1}_{k}|) < \infty,
		\]
	\end{enumerate}
	then $ \{z_{k}\}_{k \geq 0} \subset W^{cs}_{loc}(\Sigma) $ and $ |x^{c_0s_1}_{k+1}| \leq (\lambda_{c_1s_1}(\phi(\widehat{m}_{k})) + \hat{\chi}) |x^{c_0s_1}_{k}| $.

	\noindent(iii) If $ \{z_{-k} = (\widehat{m}_{-k}, x^{s_1}_{-k}, x^{c_0u_1}_{-k})\}_{k \geq 0} \subset X^{s_1}_{\widehat{\Sigma}}(\varrho) \oplus X^{c_0u_1}_{\widehat{\Sigma}} (\varepsilon_0) $ satisfies
	\begin{enumerate}[$ \bullet $]
		\item $ z_{-k} \in H^{-1}(z_{-k+1}) $ with $ \widehat{m}_{-k} \in \widehat{U}_{\widehat{u}^{-1}(\widehat{m}_{-k+1})} (\varepsilon) $ for all $ k \geq 0 $; and
		\item $ |x^{s_1}_{-k}| \leq \widetilde{\alpha}_{0}(\widehat{m}_{-k+1}) |x^{c_0u_1}_{-k}| $ for all $ k \geq 1 $ or 
		\[
		\sup_{k}\{\varepsilon_{u}(\widehat{m}_0) \varepsilon_{u}(\widehat{m}_{-1}) \cdots \varepsilon_{u}(\widehat{m}_{-k+1})\}^{-1} (|x^{c_0u_1}_{-k}| + |x^{s_1}_{-k}|) < \infty,
		\]
	\end{enumerate}
	then $ \{z_{-k}\}_{k \geq 0} \subset W^{cu}_{loc}(\Sigma) $ and $ |x^{c_0u_1}_{-k}| \leq (\lambda_{c_1u_1}(\phi(\widehat{m}_{-k})) + \hat{\chi}) |x^{c_0u_1}_{-k+1}| $.

	\noindent(iv) If $ \{z_{k} = (\widehat{m}_{k}, x^{c_0}_{k}, x^{s_1}_{k}, x^{u_1}_{k})\}_{k \in \mathbb{Z}} \subset X^{c_0}_{\widehat{\Sigma}}(\varepsilon_0) \oplus X^{s_1}_{\widehat{\Sigma}} (\varepsilon_0) \oplus X^{u_1}_{\widehat{\Sigma}} (\varepsilon_0) $ satisfies that $ \{z_{k}\}_{k \geq 0} $ and $ \{z_{-k}\}_{k \geq 0} $ fulfill (i) and (ii), respectively, then $ \{z_{k}\}_{k \in \mathbb{N}} \subset \Sigma^{c} $.

	Here, $ \widetilde{\beta}_{0}(\cdot), \widetilde{\alpha}_{0}(\cdot), \varepsilon_{s}(\cdot), \varepsilon_{u}(\cdot): \widehat{\Sigma} \to \mathbb{R}_+ $ satisfy for all $ \widehat{m} \in \widehat{\Sigma} $,
	\begin{gather*}
	\beta'_{c_1s_1}(u(\phi(\widehat{m}))) + \hat{\chi} < \widetilde{\beta}_{0}(\widehat{m}) < \beta_{c_1s_1}(\phi(\widehat{m})) - \hat{\chi}, ~ 0 < \varepsilon_{s}(\widehat{m}) < \lambda^{-1}_{u_1} (\phi(\widehat{m})) - \hat{\chi}, \\
	\alpha'_{c_1u_1}(u^{-1}(\phi(\widehat{m}))) + \hat{\chi} < \widetilde{\alpha}_{0}(\widehat{m}) < \alpha_{c_1u_1}(\phi(\widehat{m})) - \hat{\chi}, ~ 0 < \varepsilon_{u}(\widehat{m}) < \lambda^{-1}_{s_1} (\phi(\widehat{m})) - \hat{\chi},
	\end{gather*}
	where $ \hat{\chi} > 0 $ is some small constant depending on $ \xi_1, \epsilon_{*} $.

	\item If (B4) (i) holds, then $ W^{cs}_{loc}(\Sigma), W^{cu}_{loc}(\Sigma), \Sigma^c $ are differentiable at $ \widehat{m} \in \widehat{\Sigma} $ in the sense of Whitney (see \autoref{def:tangent}) with $ \widehat{\Sigma} \to \mathbb{G}(X): \widehat{m} \mapsto T_{\widehat{m}} W^{cs}_{loc}(\Sigma) $, $ T_{\widehat{m}} W^{cs}_{loc}(\Sigma) $, $ T_{\widehat{m}} \Sigma^c $ continuous; if, in addition, $ \widehat{X}^{c_1s_1}, \widehat{X}^{c_1u_1} $ are invariant under $ DH $ (meaning that 
	\[
	D_{x^{c_1s_1}}\widehat{G}^{c_1s_1}_{m}( 0, 0, 0 ) = 0, \quad \text{and} \quad D_{x^{c_1u_1}}\widehat{F}^{c_1u_1}_{m}( 0, 0, 0 ) = 0
	\]
	for all $ m \in \Sigma $), then (see \autoref{def:tangent} for the meaning)
	\[
	T_{\widehat{m}} W^{cs}_{loc}(\Sigma) = \widehat{X}^{c_1s_1}_{\phi(\widehat{m})},
	~T_{\widehat{m}} W^{cu}_{loc}(\Sigma) = \widehat{X}^{c_1u_1}_{\phi(\widehat{m})},
	~T_{\widehat{m}} \Sigma^c = \widehat{X}^{c_1}_{\phi(\widehat{m})}, ~\widehat{m} \in \widehat{\Sigma}.
	\]

	\item Assume the following assumptions hold:
	\begin{enumerate}[(i)]
		\item (B4) (i);
		\item $ \Sigma \in C^1 $ and $ m \mapsto \Pi^{\kappa}_{m} $ is $ C^1 $ (in the immersed topology), $ \kappa = s_1, c_1, u_1, c_0 $;
		\item There exist \emph{$ C^{1} $ and $ C_1 $-Lipschitz bump functions} in $ X^{c_0s_1}_{\widehat{\Sigma}} $ and in $ X^{c_0u_1}_{\widehat{\Sigma}} $, respectively, where $ C_1 \geq 1 $ (see \autoref{def:C1Lbump} and \autoref{exa:case1}), \emph{or}
		\item[(iii$ ' $)] $ H $ satisfies the \emph{strong $ s_1 $-contraction} and \emph{strong $ u_1 $-expansion} (see assumption ($ \star\star $) in \autoref{sub:limited} (e.g., $ \sup_{m} \alpha'_{c_1u_1} (m) $ and $ \sup_{m} \beta'_{c_1s_1}(m) $ are sufficiently small) and there exist \emph{$ C^{1} $ and $ C_1 $-Lipschitz bump functions} in $ X^{c_0}_{\widehat{\Sigma}} $ (see \autoref{def:C1Lbump}), where $ C_1 \geq 1 $.
	\end{enumerate}
	Then we also can choose $ W^{cs}_{loc}(\Sigma), W^{cu}_{loc}(\Sigma), \Sigma^c $ such that they are $ C^1 $ immersed submanifolds of $ X $.
\end{enumerate}

\begin{rmk}
	\begin{enumerate}[(a)]
		\item If $ \sup_{m}\lambda_{s_1}(m) < 1, \sup_{m}\lambda_{u_1}(m) < 1 $ in item \eqref{it:inv}, then we take $ \varepsilon_{s}(\cdot), \varepsilon_{u}(\cdot) \equiv 1 $; this means e.g. if in item \eqref{it:inv}, $ \{z_{k} = (\widehat{m}_{k}, x^{c_0}_{k}, x^{s_1}_{k}, x^{u_1}_{k})\}_{k \in \mathbb{Z}} \subset X^{c_0}_{\widehat{\Sigma}}(\varepsilon_0) \oplus X^{s_1}_{\widehat{\Sigma}} (\varepsilon_0) \oplus X^{u_1}_{\widehat{\Sigma}} (\varepsilon_0) $ satisfies $ z_{k} \in H(z_{k-1}) $ with $ \widehat{m}_{k} \in \widehat{U}_{\widehat{u}(\widehat{m}_{k-1})} (\varepsilon) $ and $ \widehat{m}_{k} \in \widehat{U}_{\widehat{u}^{-1}(\widehat{m}_{k-1})} (\varepsilon) $ for all $ k \in \mathbb{Z} $, then $ \{z_{k}\}_{k \in \mathbb{Z}} \subset \Sigma^{c} $.

		In particular, we see that the \emph{strong-stable lamination} (resp. \emph{strong-unstable lamination}) of $ \Sigma $ for $ H $, whose leaves are (uniformly) Lipschitz immersed submanifolds of $ X $ locally modeled on $ X^{s_1}_{m} $ (resp. $ X^{u_1}_{m} $) ($ m \in \Sigma $), belongs to $ W^{cs}_{loc}(\Sigma) $ (resp. $ W^{cu}_{loc}(\Sigma) $). However, it should be noted that, in general, the strong-stable lamination is not open in $ W^{cs}_{loc}(\Sigma) $. For a detailed discussion of the existence and regularity of strong-(un)stable laminations, we refer to \cite{HPS77} or \cite[Section 7.2]{Che18a}.

		\item If (i) $ c_1 = c $ (and so $ s_1 = s, u_1 = u $), (ii) $ \sup_{m}\lambda_{s_1}(m) < 1 $, $ \sup_{m}\lambda_{u_1}(m) < 1 $, and (iii) $ H $ satisfies the \emph{strong $ s $-contraction} and \emph{strong $ u $-expansion} (see assumption ($ \star\star $) in \autoref{sub:limited}, e.g., when $ \sup_{m} \alpha'_{c_1u_1} (m) $ and $ \sup_{m} \beta'_{c_1s_1}(m) $ are sufficiently small), then $ \Sigma $ is a normally hyperbolic invariant manifold of $ H $ (in the sense of \cite{Che18b}) and also $ \Sigma = \Sigma^c $, $ W^{cs}_{loc}(\Sigma) \subset H^{-1}W^{cs}_{loc}(\Sigma) $, $ W^{cu}_{loc}(\Sigma) \subset HW^{cu}_{loc}(\Sigma) $ (i.e., we can take $ \varepsilon_{0} = \varepsilon $). Moreover, in this case, if (B4) (i) holds, then $ \Sigma \in C^1 $ (as well as $ W^{cs}_{loc}(\Sigma), W^{cu}_{loc}(\Sigma) \in C^1 $).
	\end{enumerate}
\end{rmk}

\end{appendices}


\end{document}